\documentclass{memo-l}

\usepackage[utf8]{inputenc}
\usepackage[T1]{fontenc}
\usepackage[english]{babel}

\usepackage{amsfonts}
\usepackage{amsmath}
\usepackage{amsrefs}
\usepackage{amssymb}
\usepackage{amstext}
\usepackage{amsthm}

\usepackage{enumitem}
\usepackage{graphicx}
\usepackage{hyperref}
\usepackage{listings}
\usepackage{mathrsfs}
\usepackage{mathtools}
\usepackage{multicol}
\usepackage{shuffle}

\usepackage{tikz}
\usepackage{tikz-cd}
\usetikzlibrary{arrows,backgrounds,chains,decorations.pathreplacing,patterns,shapes}

\usepackage{chngcntr}
\counterwithout{figure}{chapter}

\newcommand{\area}{\mathsf{area}}
\newcommand{\dinv}{\mathsf{dinv}}
\newcommand{\bounce}{\mathsf{bounce}}
\newcommand{\pbounce}{\mathsf{pbounce}}
\newcommand{\pmaj}{\mathsf{pmaj}}

\newcommand{\Val}{\mathsf{Val}}
\newcommand{\Rise}{\mathsf{Rise}}
\newcommand{\Peak}{\mathsf{Peak}}
\newcommand{\DVal}{\mathsf{DVal}}

\newcommand{\DPeak}{\mathsf{DPeak}}

\newcommand{\D}{\mathsf{D}} 
\newcommand{\LD}{\mathsf{LD}} 
\newcommand{\DD}{\mathsf{DD}} 
\newcommand{\LDD}{\mathsf{LDD}} 
\newcommand{\DDd}{\mathsf{DDd}} 
\newcommand{\DDb}{\mathsf{DDb}} 
\newcommand{\DDp}{\mathsf{DDp}} 
\newcommand{\SQE}{\mathsf{SQ^{E}}} 
\newcommand{\SQN}{\mathsf{SQ^{N}}} 
\newcommand{\PSQE}{\mathsf{PSQ^{E}}} 
\newcommand{\PSQN}{\mathsf{PSQ^{N}}} 
\newcommand{\PF}{\mathsf{PF}} 
\newcommand{\PLD}{\mathsf{PLD}} 
\newcommand{\PP}{\mathsf{PP}} 
\newcommand{\RP}{\mathsf{RP}} 
\newcommand{\LPP}{\mathsf{LPP}} 

\newcommand{\qbinom}[2]{\genfrac{[}{]}{0pt}{}{#1}{#2}}

\makeatletter

\pgfkeys{
	/tikz/sharp angle/.code={%
		\pgfsetarrowoptions{sharp >}{#1}%
		\pgfsetarrowoptions{sharp <}{-#1}%
	},
	/tikz/sharp > angle/.code={%
		\pgfsetarrowoptions{sharp >}{#1}%
	},
	/tikz/sharp < angle/.code={%
		\pgfsetarrowoptions{sharp <}{#1}%
	},
	/tikz/sharp protrude/.code=\csname if#1\endcsname\qrr@tikz@sharp@z@-0.05\p@\else\qrr@tikz@sharp@z@\z@\fi,
	/tikz/sharp protrude/.default=true
}

\newdimen\qrr@tikz@sharp@z@
\qrr@tikz@sharp@z@\z@
\pgfarrowsdeclare{sharp >}{sharp >}{%
	\edef\pgf@marshal{\noexpand\pgfutil@in@{and}{\pgfgetarrowoptions{sharp >}}}%
	\pgf@marshal
	\ifpgfutil@in@
	\edef\pgf@tempa{\pgfgetarrowoptions{sharp >}}
	\expandafter\qrr@tikz@sharp@parse\pgf@tempa\@qrr@tikz@sharp@parse
	\else
	\qrr@tikz@sharp@parse\pgfgetarrowoptions{sharp >}and-\pgfgetarrowoptions{sharp >}\@qrr@tikz@sharp@parse
	\fi
	\pgfmathparse{max(\pgf@tempa,\pgf@tempb,0)}%
	\let\qrr@tikz@sharp@max\pgfmathresult
	\pgfmathsetlength\pgf@xa{.5*\pgflinewidth * tan(\qrr@tikz@sharp@max)}%
	\pgfarrowsleftextend{+\pgf@xa}%
	\pgfarrowsrightextend{+\pgf@xa}%
}{%
	\edef\pgf@marshal{\noexpand\pgfutil@in@{and}{\pgfgetarrowoptions{sharp >}}}%
	\pgf@marshal
	\ifpgfutil@in@
	\edef\pgf@tempa{\pgfgetarrowoptions{sharp >}}
	\expandafter\qrr@tikz@sharp@parse\pgf@tempa\@qrr@tikz@sharp@parse
	\else
	\qrr@tikz@sharp@parse\pgfgetarrowoptions{sharp >}and-\pgfgetarrowoptions{sharp >}\@qrr@tikz@sharp@parse
	\fi
	\pgfmathsetlength\pgf@ya{.5*\pgflinewidth * tan(max(\pgf@tempa,\pgf@tempb,0))}%
	\pgfmathsetlength\pgf@xa{-.5*\pgflinewidth * tan(\pgf@tempa)}%
	\pgfmathsetlength\pgf@xb{-.5*\pgflinewidth * tan(\pgf@tempb)}%
	\advance\pgf@xa\pgf@ya
	\advance\pgf@xb\pgf@ya
	\ifdim\pgf@xa>\pgf@xb
	\pgftransformyscale{-1}%
	\pgf@xc\pgf@xb
	\pgf@xb\pgf@xa
	\pgf@xa\pgf@xc
	\fi
	\pgfpathmoveto{\pgfqpoint{\qrr@tikz@sharp@z@}{.5\pgflinewidth}}%
	\pgfpathlineto{\pgfqpoint{\pgf@xa}{.5\pgflinewidth}}%
	\pgfpathlineto{\pgfqpoint{\pgf@ya}{+0pt}}%
	\pgfpathlineto{\pgfqpoint{\pgf@xb}{-.5\pgflinewidth}}%
	\pgfpathlineto{\pgfqpoint{\qrr@tikz@sharp@z@}{-.5\pgflinewidth}}%
	\pgfusepathqfill
}
\pgfarrowsdeclare{sharp <}{sharp <}{%
	\edef\pgf@marshal{\noexpand\pgfutil@in@{and}{\pgfgetarrowoptions{sharp <}}}%
	\pgf@marshal
	\ifpgfutil@in@
	\edef\pgf@tempa{\pgfgetarrowoptions{sharp <}}
	\expandafter\qrr@tikz@sharp@parse\pgf@tempa\@qrr@tikz@sharp@parse
	\else
	\expandafter\qrr@tikz@sharp@parse\pgfgetarrowoptions{sharp <}and-\pgfgetarrowoptions{sharp <}\@qrr@tikz@sharp@parse
	\fi
	\pgfmathparse{max(\pgf@tempa,\pgf@tempb,0)}%
	\let\qrr@tikz@sharp@max\pgfmathresult
	\pgfmathsetlength\pgf@xa{.5*\pgflinewidth * tan(\qrr@tikz@sharp@max)}%
	\pgfarrowsleftextend{+\pgf@xa}%
	\pgfarrowsrightextend{+\pgf@xa}%
}{%
	\edef\pgf@marshal{\noexpand\pgfutil@in@{and}{\pgfgetarrowoptions{sharp <}}}%
	\pgf@marshal
	\ifpgfutil@in@
	\edef\pgf@tempa{\pgfgetarrowoptions{sharp <}}
	\expandafter\qrr@tikz@sharp@parse\pgf@tempa\@qrr@tikz@sharp@parse
	\else
	\expandafter\qrr@tikz@sharp@parse\pgfgetarrowoptions{sharp <}and-\pgfgetarrowoptions{sharp <}\@qrr@tikz@sharp@parse
	\fi
	\pgfmathsetlength\pgf@ya{.5*\pgflinewidth * tan(max(\pgf@tempa,\pgf@tempb,0))}%
	\pgfmathsetlength\pgf@xa{-.5*\pgflinewidth * tan(\pgf@tempa)}%
	\pgfmathsetlength\pgf@xb{-.5*\pgflinewidth * tan(\pgf@tempb)}%
	\advance\pgf@xa\pgf@ya
	\advance\pgf@xb\pgf@ya
	\ifdim\pgf@xa>\pgf@xb
	\pgftransformyscale{-1}%
	\pgf@xc\pgf@xb
	\pgf@xb\pgf@xa
	\pgf@xa\pgf@xc
	\fi
	\pgfpathmoveto{\pgfqpoint{\qrr@tikz@sharp@z@}{.5\pgflinewidth}}%
	\pgfpathlineto{\pgfqpoint{\pgf@xa}{.5\pgflinewidth}}%
	\pgfpathlineto{\pgfqpoint{\pgf@ya}{+0pt}}%
	\pgfpathlineto{\pgfqpoint{\pgf@xb}{-.5\pgflinewidth}}%
	\pgfpathlineto{\pgfqpoint{\qrr@tikz@sharp@z@}{-.5\pgflinewidth}}%
	\pgfusepathqfill
}
\def\qrr@tikz@sharp@parse#1and#2\@qrr@tikz@sharp@parse{\def\pgf@tempa{#1}\def\pgf@tempb{#2}}

\makeatother

\newcommand\multiset[2]%
{\mathchoice{\left(\kern-0.4em{\binom{#1}{#2}}\kern-0.4em\right)}
	{\bigl(\kern-0.2em{\binom{#1}{#2}}\kern-0.2em\bigr)}
	{\bigl(\kern-0.2em{\binom{#1}{#2}}\kern-0.2em\bigr)}
	{\bigl(\kern-0.2em{\binom{#1}{#2}}\kern-0.2em\bigr)}}

\let\existstemp\exists \renewcommand*{\exists}{\mathop \existstemp}
\let\foralltemp\forall \renewcommand*{\forall}{\mathop \foralltemp}

\def\quotient#1#2{\raise1ex\hbox{$#1$}\Big/\lower1ex\hbox{$#2$}}

\newcommand{\<}{\langle}
\renewcommand{\>}{\rangle}

\newtheorem{theorem}{Theorem}[chapter]
\newtheorem{lemma}[theorem]{Lemma}
\newtheorem{proposition}[theorem]{Proposition}
\newtheorem{corollary}[theorem]{Corollary}
\newtheorem{conjecture}[theorem]{Conjecture}

\theoremstyle{definition}
\newtheorem{definition}[theorem]{Definition}

\newtheorem{example}[theorem]{Example}

\theoremstyle{remark}
\newtheorem{remark}[theorem]{Remark}

\numberwithin{section}{chapter}
\numberwithin{equation}{chapter}

\makeindex

\begin{document}

\frontmatter

\title{Decorated Dyck paths, polyominoes, and the Delta conjecture}


\author{Michele D'Adderio}
\address{Universit\'e Libre de Bruxelles (ULB)\\D\'epartement de Math\'ematique\\ Boulevard du Triomphe, B-1050 Bruxelles\\ Belgium}\email{mdadderi@ulb.ac.be}
\thanks{}

\author{Alessandro Iraci}
\address{Universit\'a di Pisa and Universit\'e Libre de Bruxelles (ULB)\\Dipartimento di Matematica\\ Largo Bruno Pontecorvo 5, 56127 Pisa\\ Italia}\email{iraci@student.dm.unipi.it}

\author{Anna Vanden Wyngaerd}
\address{Universit\'e Libre de Bruxelles (ULB)\\D\'epartement de Math\'ematique\\ Boulevard du Triomphe, B-1050 Bruxelles\\ Belgium}\email{anvdwyng@ulb.ac.be}

\date{\today}




\dedicatory{To the memory of Jeff Remmel}

\begin{abstract}
We discuss the combinatorics of decorated Dyck paths and decorated parallelogram polyominoes, extending to the decorated case the main results of both \cite{Haglund-Schroeder-2004} and \cite{Aval-DAdderio-Dukes-Hicks-LeBorgne-2014}. This settles in particular the cases $\<\cdot,e_{n-d}h_d\>$ and $\<\cdot,h_{n-d}h_d\>$ of the Delta conjecture of Haglund, Remmel and Wilson \cite{Haglund-Remmel-Wilson-2015}. Along the way, we introduce some new statistics, formulate some new conjectures, prove some new identities of symmetric functions, and answer a few open problems in the literature (e.g. from \cites{Haglund-Remmel-Wilson-2015,Zabrocki-4Catalan-2016,Aval-Bergeron-Garsia-2015}). The main technical tool is a new identity in the theory of Macdonald polynomials that extends a theorem of Haglund in \cite{Haglund-Schroeder-2004}.
\end{abstract}

\maketitle

\tableofcontents

\chapter*{Introduction}

Motivated by the problem of proving the Schur positivity of the modified Macdonald polynomials, Garsia and Haiman introduced the $\mathfrak{S}_n$-module of diagonal harmonics, and they conjectured that its Frobenius characteristic is $\nabla e_n$, where $\nabla$ is the so called nabla operator introduced by Bergeron and Garsia \cite{Bergeron-Garsia-ScienceFiction-1999}. For any symmetric function $f$, a so called Delta operator $\Delta_f$ (and its sibling $\Delta_f'$) is defined. These operators have been introduced by Bergeron, Garsia, Haiman and Tesler in \cite{Bergeron-Garsia-Haiman-Tesler-Positivity-1999}, and it turns out that the nabla operator restricted to homogeneous symmetric functions of degree $n$ coincides with the Delta operator $\Delta_{e_n}$ (and with $\Delta_{e_{n-1}}'$).

In \cite{HHLRU-2005} a combinatorial formula for $\nabla e_n$ has been conjectured, in terms of labelled Dyck paths (i.e. parking functions). This formula has been known as the \emph{shuffle conjecture}, and it has been recently proved by Carlsson and Mellit in \cite{Carlsson-Mellit-ShuffleConj-2015}, where they actually proved a refinement of it, known as the \emph{compositional shuffle conjecture}, proposed in \cite{Haglund-Morse-Zabrocki-2012}.

A generalization of the shuffle conjecture, known as the \emph{Delta conjecture}, has been stated in \cite{Haglund-Remmel-Wilson-2015}. This predicts a combinatorial interpretation for $\Delta_{e_{n-k-1}}'e_n$ in terms of decorated labelled Dyck paths (for $k=0$ it reduces to the case of no decorations, i.e.\ to the shuffle conjecture). More recently in \cite{Zabrocki_Delta_Module} Zabrocki constructed bigraded $\mathfrak{S}_n$-modules which extend the diagonal harmonics, and he conjectured that their Frobenius characteristic is given by $\Delta_{e_{n-k-1}}'e_n$, completing in this way the extension of the all framework of the shuffle conjecture.

Some special cases and consequences of the Delta conjecture have been proved: see Section~\ref{sec:Delta_Conj} for the precise statement and the state of the art of this conjecture. To this date, the general problem remains open.

In the same paper \cite{Haglund-Remmel-Wilson-2015}, Haglund, Remmel and Wilson proposed a generalization of the Delta conjecture, which predicts a combinatorial interpretation of $\Delta_{h_m}\Delta_{e_{n-k-1}}'e_n$ in terms of decorated partially labelled Dyck paths: see Section~\ref{sec:gen_Delta} for the precise statement and the state of the art of this conjecture.

We mention here that in \cite{Zabrocki-4Catalan-2016}, Zabrocki established a consequence of the Delta conjecture, the so called \emph{$4$-variable Catalan conjecture}. In fact, this conjecture, stated in \cite{Haglund-Remmel-Wilson-2015}, makes other predictions that have not been explained in \cite{Zabrocki-4Catalan-2016}. One of them is the symmetry 
\begin{equation} \label{eq:HRW_symmetry}
	\langle\Delta_{h_{\ell}}\nabla e_{n-\ell},s_{k+1,1^{n-\ell-k-1}}\rangle=\langle\Delta_{h_{k}}\nabla e_{n-k},s_{\ell+1,1^{n-k-\ell-1}}\rangle\qquad \text{for }n>k+\ell.
\end{equation} 

In \cite{Haglund-Remmel-Wilson-2015} and \cite{Wilson-PhD-2015}, the authors asked for a proof of a decorated version of both the famous $q,t$-Schr\"{o}der theorem of Haglund \cite{Haglund-Schroeder-2004}, i.e. a combinatorial interpretation of the formula
\begin{equation} 
	\langle \Delta_{e_{a+b-k-1}}'e_{a+b},e_a h_{b}\rangle
\end{equation}
as a sum of weights $q^{\mathsf{dinv}(D)}t^{\mathsf{area}(D)}$ of decorated Dyck paths, and the related interpretation of the formula
\begin{equation} 
	\langle \Delta_{e_{a+b-k-1}}'e_{a+b},h_a h_{b}\rangle.
\end{equation}

\medskip

In \cite{Aval-DAdderio-Dukes-Hicks-LeBorgne-2014}, statistics $\area$, $\bounce$ and $\dinv$ have been defined on parallelogram polyominoes, and some of their $q,t$-enumerators have been shown to equal 
\begin{equation}
	\< \Delta_{e_{m+n}} e_{m+n}, h_m h_n \> = \< \Delta_{h_m} e_{n+1}, s_{1^{n+1}} \>.
\end{equation}
In particular, the first author proposed a combinatorial interpretation of the full symmetric function $\Delta_{h_m} e_{n+1}$ at $q=1$ in terms of labelled parallelogram polyominoes, and asked for a $\dinv$ statistic that could give the more general $\Delta_{h_m} e_{n+1}$ (cf. \cite{Aval-Bergeron-Garsia-2015}*{Equations (8.1) and (8.14)}). 

\medskip

The present work can be thought of as a continuation of both \cite{Haglund-Schroeder-2004} (cf. also \cite{Haglund-Book-2008}) and \cite{Aval-Bergeron-Garsia-2015}. 

In particular we extend the results in \cite{Haglund-Schroeder-2004} by proving a decorated version of the $q,t$-Schr\"{o}der theorem, which gives the case $\<\cdot , e_{n-d}h_d\>$ of the Delta conjecture. In fact our proof will parallel the one in \cite{Haglund-Schroeder-2004}, as we will give a recursion for these polynomials, and prove that both the symmetric function side and the combinatorial side satisfy this recursion.

Also, following a similar strategy, we extend the results in \cite{Aval-Bergeron-Garsia-2015} by providing a combinatorial interpretation of the more general $\< \Delta_{h_m} e_{n+1}, s_{k+1,1^{n-k}} \>$ in terms of decorated parallelogram polyominoes.

Surprisingly, this last polynomial turns out to equal $\< \Delta_{e_{m+n-k-1}}' e_{m+n}, h_m h_n \>$, and this will allow us to prove the case $\<\cdot , h_{n-d}h_d\>$ of the Delta conjecture, completing the extension of the results in \cite{Haglund-Schroeder-2004} to the decorated case.

Along the way, we will solve all the aforementioned problems. Moreover, we will introduce some new statistics on our combinatorial objects, state some conjectures and prove some results about them.

\medskip

We realize that this work contains a lot of definitions and results, hence we prefer not to detail them in this introduction; moreover, several theorems turn out to have quite technical and long proofs. Instead, we decided to organize the rest of our work in the following way.

In Part~1 we outline most of the basic definitions (Chapter~1), we state the conjectures and explain the state of the art of them (Chapter~2), we state our main results and indicate some further open problems (Chapter~3).

In Part~2 we will provide all the proofs, both on the symmetric function side (mainly Chapter~4) and the combinatorial side (Chapter~5-8), leaving only the proof of some elementary lemmas for the Appendix.

Two remarks are in order:
\begin{itemize}
	\item Not all of our results are mentioned in Part~1: e.g. some of the recursions to prove the cases $\<\cdot , e_{n-d}h_d\>$ and $\<\cdot , h_{n-d}h_d\>$ of the Delta conjecture appear only in Part~2 (cf. Theorem~\ref{thm:reco_first_bounce_S_R}, Theorem~\ref{thm:reco_first_bounce_R_S}, Theorem~\ref{thm:reco1_first_bounce}, Theorem~\ref{thm:reco_old_bounce}, Theorem~\ref{th:dinvrecursion}, Theorem~\ref{th:bouncerecursion}, Theorem~\ref{th:redbouncerecursion}, Theorem~\ref{th:recursion2cpf}). In Part~1 we mainly focused on ``global'' results.
	
	\item We decided to include a full statement of the (Delta) conjectures in \cite{Haglund-Remmel-Wilson-2015}, and a brief discussion of shuffles taken from \cite{HHLRU-2005} (our Section~\ref{sec:link_delta}), in order to make at least Part~1 of our text more self-contained.
\end{itemize}

With this structure, we aimed at giving to the interested reader the option to jump on Chapters 2 and 3 to look at the conjectures and (most of) our results, while referring to Chapter~1 for definitions and notations; if an appetite for more details kicks in, references to Part~2 with the relevant details are timely provided.

We hope that in this way the reader will find our work more accessible.

%

\section*{Acknowledgments}

The first author is pleased to thank Adriano Garsia, who, after countless hours of explanations along several years, finally converted him to this fascinating subject.

All the authors are pleased to thank the anonymous referee for the patient and careful work of pinpointing countless mistakes and providing useful observations and suggestions on a first draft: all this greatly improved the readability of the present work.

\mainmatter

\part{Definitions and results}

\chapter{Background and definitions}

In this chapter we introduce some of the basic definitions and notations for symmetric functions, and most of the basic combinatorial definitions and notations. It can be read immediately but also used as a reference.

\section{Symmetric and quasisymmetric functions} \label{sec:symmetric_fcts_intro}

The main references that we will use for symmetric functions
are \cite{Macdonald-Book-1995} and \cite{Stanley-Book-1999}. 

The standard bases of the symmetric functions that will appear in our
calculations are the monomial $\{m_{\lambda}\}_{\lambda}$, complete homogeneous $\{h_{\lambda}\}_{\lambda}$, elementary $\{e_{\lambda}\}_{\lambda}$, power $\{p_{\lambda}\}_{\lambda}$ and Schur $\{s_{\lambda}\}_{\lambda}$ bases.

\emph{We will use the usual convention that $e_0 = h_0 = 1$ and $e_k = h_k = 0$ for $k < 0$.}

The ring $\Lambda$ of symmetric functions can be thought of as the polynomial ring in the power sum generators $p_1, p_2, p_3,\dots$. This ring has a grading $\Lambda=\bigoplus_{n\geq 0}\Lambda^{(n)}$ given by assigning degree $i$ to $p_i$ for all $i\geq 1$. As we are working with Macdonald symmetric functions involving two parameters $q$ and $t$, we will consider this polynomial ring over the field $\mathbb{Q}(q,t)$. We will make extensive use of the \emph{plethystic notation} (see \cite{Loehr-Remmel-plethystic-2011} for a nice exposition).

With this notation we will be able to add and subtract alphabets, which will be represented as sums of monomials $X = x_1 + x_2 + x_3+\cdots $. Then, given a symmetric function $f$, and thinking of it as an element of $\Lambda$, we denote by $f[X]$ the expression $f$ with $p_k$ replaced by $x_{1}^{k}+x_{2}^{k}+x_{3}^{k}+\cdots$, for all $k$. More generally, given any expression $Q(z_1,z_2,\dots)$, we define the plethystic substitution $f[Q(z_1,z_2,\dots)]$ to be $f$ with $p_k$ replaced by $Q(z_1^k,z_2^k,\dots)$.

We denote by $\<\, , \>$ the \emph{Hall scalar product} on symmetric functions, which can be defined by saying that the Schur functions form an orthonormal basis. We denote by $\omega$ the fundamental algebraic involution which sends $e_k$ to $h_k$, $s_{\lambda}$ to $s_{\lambda'}$ and $p_k$ to $(-1)^{k-1}p_k$.

For a partition $\mu\vdash n$, we denote by
\begin{align}
	\widetilde{H}_{\mu} \coloneqq \widetilde{H}_{\mu}[X]=\widetilde{H}_{\mu}[X;q,t]=\sum_{\lambda\vdash n}\widetilde{K}_{\lambda \mu}(q,t)s_{\lambda}
\end{align}
the \emph{(modified) Macdonald polynomials}, where
\begin{align}
	\widetilde{K}_{\lambda \mu} \coloneqq \widetilde{K}_{\lambda \mu}(q,t)=K_{\lambda \mu}(q,1/t)t^{n(\mu)}\quad \text{ with }\quad n(\mu)=\sum_{i\geq 1}\mu_i(i-1)
\end{align}
are the \emph{(modified) Kostka coefficients} (see \cite{Haglund-Book-2008}*{Chapter~2} for more details). 

The set $\{\widetilde{H}_{\mu}[X;q,t]\}_{\mu}$ is a basis of the ring of symmetric functions $\Lambda$ with coefficients in $\mathbb{Q}(q,t)$. This is a modification of the basis introduced by Macdonald \cite{Macdonald-Book-1995}, and they are the Frobenius characteristic of the so called Garsia-Haiman modules (see \cite{Garsia-Haiman-PNAS-1993}).

If we identify the partition $\mu$ with its Ferrers diagram, i.e. with the collection of cells $\{(i,j)\mid 1\leq i\leq \mu_j, 1\leq j\leq \ell(\mu)\}$, then for each cell $c\in \mu$ we refer to the \emph{arm}, \emph{leg}, \emph{co-arm} and \emph{co-leg} (denoted respectively as $a_\mu(c), l_\mu(c), a_\mu'(c), l_\mu'(c)$) as the number of cells in $\mu$ that are strictly to the right, above, to the left and below $c$ in $\mu$, respectively (see Figure~\ref{fig:notation}).

\begin{figure}[h]
	\centering
	\begin{tikzpicture}[scale=0.4]
	\draw[gray,opacity=.4](0,0) grid (15,10);
	\fill[white] (1,10)|-(3,9)|- (5,7)|-(9,5)|-(13,2)--(15.2,2)|-(1,10.2);
	\draw[gray]  (1,10)|-(3,9)|- (5,7)|-(9,5)|-(13,2)--(15,2)--(15,0)-|(0,10)--(1,10);
	\fill[blue, opacity=.2] (0,3) rectangle (9,4) (4,0) rectangle (5,7); 
	\fill[blue, opacity=.5] (4,3) rectangle (5,4);
	\draw (7,4.5) node {\tiny{Arm}} (3.25,5.5) node {\tiny{Leg}} (6.25, 1.5) node {\tiny{Co-leg}} (2,2.5) node {\tiny{Co-arm}} ;
	\end{tikzpicture}
	\caption{}
	\label{fig:notation}
\end{figure}

We define for every partition $\mu$
\begin{align}
	B_{\mu} & \coloneqq B_{\mu}(q,t)=\sum_{c\in \mu}q^{a_{\mu}'(c)}t^{l_{\mu}'(c)},\text{ and } \\
	T_{\mu} & \coloneqq T_{\mu}(q,t)=\prod_{c\in \mu}q^{a_{\mu}'(c)}t^{l_{\mu}'(c)} .
\end{align}
The following linear operators were introduced in \cites{Bergeron-Garsia-ScienceFiction-1999,Bergeron-Garsia-Haiman-Tesler-Positivity-1999}, and they are at the basis of the conjectures relating symmetric function coefficients and $q,t$-combinatorics in this area. 

We define the \emph{nabla} operator on $\Lambda$ by
\begin{align} 
	\nabla \widetilde{H}_{\mu} \coloneqq T_{\mu} \widetilde{H}_{\mu} \quad \text{ for all } \mu,
\end{align}
and we define the \emph{Delta} operators $\Delta_f$ and $\Delta_f'$ on $\Lambda$ by
\begin{align} \label{eq:deltaop_def}
	\Delta_f \widetilde{H}_{\mu} \coloneqq f[B_{\mu}(q,t)] \widetilde{H}_{\mu} \quad \text{ and } \quad 
	\Delta_f' \widetilde{H}_{\mu}  \coloneqq f[B_{\mu}(q,t)-1] \widetilde{H}_{\mu}, \quad \text{ for all } \mu.
\end{align}
Observe that on the vector space of symmetric functions homogeneous of degree $n$, denoted by $\Lambda^{(n)}$, the operator $\nabla$ equals $\Delta_{e_n}$. Moreover, for every $1\leq k\leq n$,
\begin{align}
	\label{eq:deltaprime}
	\Delta_{e_k} = \Delta_{e_k}' + \Delta_{e_{k-1}}' \quad \text{ on } \Lambda^{(n)},
\end{align}
and for any $k > n$, $\Delta_{e_k} = \Delta_{e_{k-1}}' = 0$ on $\Lambda^{(n)}$, so that $\Delta_{e_n}=\Delta_{e_{n-1}}'$ on $\Lambda^{(n)}$.

For $n, k\in \mathbb{N}$, we set
\begin{align}
	[0]_q \coloneqq 0, \quad \text{ and } \quad [n]_q \coloneqq \frac{1-q^n}{1-q} = 1+q+q^2+\cdots+q^{n-1} \quad \text{ for } n \geq 1,
\end{align}
\begin{align}
	[0]_q! \coloneqq 1 \quad \text{ and }\quad [n]_q! \coloneqq [n]_q[n-1]_q \cdots [2]_q [1]_q \quad \text{ for } n \geq 1,
\end{align}
and
\begin{align}
	\qbinom{n}{k}_q  \coloneqq \frac{[n]_q!}{[k]_q![n-k]_q!} \quad \text{ for } n \geq k \geq 0, \quad \text{ and } \quad \qbinom{n}{k}_q \coloneqq 0 \quad \text{ for } n < k.
\end{align}
For every $S\subseteq \{1,2,\dots,n-1\}$, let $Q_{S,n}$ denote the \emph{Gessel fundamental quasisymmetric function} of degree $n$ indexed by $S$ , i.e.
\begin{equation}
	Q_{S,n} \coloneqq \mathop{\sum_{i_1\leq i_2\leq \cdots \leq i_n}}_{i_j<i_{j+1}\text{ if }j\in S}x_{i_1} x_{i_2}  \cdots   x_{i_n}.
\end{equation}

\section{Combinatorial definitions}
\label{section:combinatorial definitions}

In this section we define several classes of combinatorial objects that will appear in the results and conjectures of the next chapters. We tried to indicate at the beginning of each section the origin of the relevant definitions. 

\subsection{Labelled Dyck paths}
\label{section: Labelled Dyck paths}

The bounce statistic on Dyck paths was introduced by Haglund in \cite{Haglund-Catalan-stats}, while Haiman introduced the dinv statistic shortly after. The dinv statistic for labelled Dyck paths appears on paper for the first time probably in \cite{HHLRU-2005}. We refer to \cite{Haglund-history-statistics} for an account on the origins of these statistics. The pmaj statistic was first introduced in \cite{Loehr-Remmel-2004} and \cite{Loehr-2005}. The area statistic is classical.

\begin{definition}
	A \emph{Dyck path} of size $n$ is a lattice path going from $(0,0)$ to $(n,n)$, using only north and east unit steps and staying weakly above the line $x=y$, also called the \emph{main diagonal}. A \emph{labelled Dyck path} is a Dyck path whose vertical steps are labelled with (not necessarily distinct) positive integers such that the labels appearing in each column are strictly increasing from bottom to top.
\end{definition}

The set of all Dyck paths of size $n$ is denoted by $\D(n)$, while the set of labelled Dyck paths of size $n$ is denoted by $\LD(n)$. For $D \in \LD(n)$ we set $l_i(D)$ to be the label of the $i$-th vertical step (the first step being the lowest one).

\begin{figure}[!ht]
	\centering
	
	\begin{tikzpicture}[scale=0.6]
	\draw[gray!60, thin] (0,0) grid (8,8) (0,0) -- (8,8);
	\draw[blue!60, line width = 1.6pt] (0,0)|-(2,3)|-(5,5)|-(6,7)|-(8,8);
	
	\draw
	(0.5,0.5) circle(0.4 cm) node {$2$}
	(0.5,1.5) circle(0.4 cm) node {$4$}
	(0.5,2.5) circle(0.4 cm) node {$5$}
	(2.5,3.5) circle(0.4 cm) node {$1$}
	(2.5,4.5) circle(0.4 cm) node {$3$}
	(5.5,5.5) circle(0.4 cm) node {$2$}
	(5.5,6.5) circle(0.4 cm) node {$6$}
	(6.5,7.5) circle(0.4 cm) node {$1$};
	\end{tikzpicture}
	
	\caption{A labelled Dyck path.}
	\label{fig:LDP}
\end{figure}
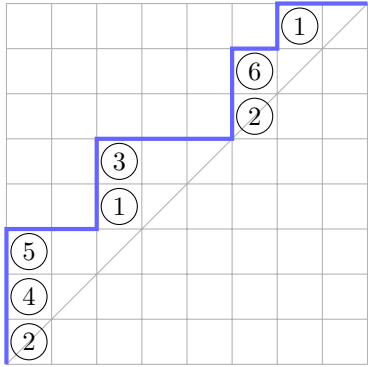

\begin{definition}
	A parking function of size $n$ is a function $f \colon [n] \rightarrow [n]$ such that $ \# \{ 1 \leq j \leq n \mid f(j) \geq i \} \leq n+1-i$.
\end{definition}

We denote by $\PF(n)$ the set of all the parking functions of size $n$. These are in bijective correspondence with the subset of labelled Dyck paths of size $n$ whose labels are exactly the numbers from $1$ to $n$. In particular, a parking function $f$ corresponds to the labelled Dyck path with label $i$ in column $f(i)$ for $1 \leq i \leq n$. We will often identify parking functions with the corresponding labelled Dyck paths.

\begin{definition}
	An \emph{area word} is a (finite) string of symbols $a_1 a_2 \cdots a_n$ in a well-ordered alphabet such that if $a_i < a_{i+1}$ then $a_{i+1}$ is the successor of $a_i$ in the alphabet.
\end{definition}

Other than the usual alphabet $\mathbb{N} = 0 < 1 < 2 < \dots$ it will be useful for our purposes to introduce the alphabet $\overline{\mathbb{N}} \coloneqq 0 < \bar{0} < 1 < \bar{1} < 2 < \dots$.

\begin{definition}
	Let $D$ be a (labelled) Dyck path of size $n$. We define its \emph{area word} to be the string of integers $a(D) = a_1(D) \cdots a_n(D)$ where $a_i(D)$ is the number of whole squares in the $i$-th row (from the bottom) between the path and the main diagonal.
\end{definition}

\begin{definition}
	We define the statistic \emph{area} on $\D(n)$ and $\LD(n)$ as \[ \area(D) \coloneqq \sum_{i=1}^n a_i(D). \]
\end{definition}

For example, the area word of the path in Figure \ref{fig:LDP} is $01212011$ and its area is $8$. Notice that the area word of a Dyck path is an area word in the alphabet $\mathbb{N}$.

\begin{definition}
	Let $D \in \D(n)$. We define its \emph{bounce path} as a lattice path from $(0,0)$ to $(n,n)$ computed in the following way: it starts in $(0,0)$ and travels north until it encounters the beginning of an east step of $D$, then it turns east until it hits the main diagonal, then it turns north again, and so on; thus it continues until it reaches $(n,n)$. 
\end{definition}

We label the vertical steps of the bounce path starting from $0$ and increasing the labels by $1$ every time the path hits the main diagonal (so the steps in the first vertical segment of the path are labelled with $0$, the ones in the next vertical segment are labelled with $1$, and so on). We define the \emph{bounce word} of $D$ to be the string $b(D) = b_1(D) \cdots b_n(D)$ where $b_i(D)$ is the label attached to the $i$-th vertical step of the bounce path.

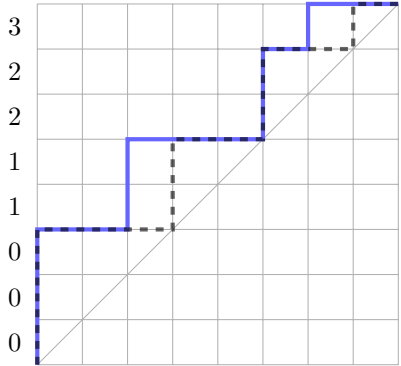
\begin{figure}[!ht]
	\centering
	
	\begin{tikzpicture}[scale=0.6]
	\draw[gray!60, thin] (0,0) grid (8,8) (0,0) -- (8,8);
	\draw[blue!60, line width = 1.6pt] (0,0)|-(2,3)|-(5,5)|-(6,7)|-(8,8);
	
	\draw[dashed, ultra thick, opacity=0.6] (0,0) |- (3,3) |- (5,5) |- (7,7) |- (8,8);
	
	\draw
	(-0.5,0.5) node {$0$}
	(-0.5,1.5) node {$0$}
	(-0.5,2.5) node {$0$}
	(-0.5,3.5) node {$1$}
	(-0.5,4.5) node {$1$}
	(-0.5,5.5) node {$2$}
	(-0.5,6.5) node {$2$}
	(-0.5,7.5) node {$3$};
	\end{tikzpicture}
	
	\caption{Construction of the bounce path and the bounce word (left).}
	\label{fig:bouncepath-dyck}
\end{figure}

\begin{definition}
	We define the statistic \emph{bounce} on $\D(n)$ and $\LD(n)$ as \[ \bounce(D) \coloneqq \sum_{i=1}^n b_i(D). \]
\end{definition}

\begin{definition}\label{def: inversion LD}
	Let $D \in \LD(n)$. For $1 \leq i < j \leq n$, we say that the pair $(i,j)$ is a \emph{diagonal inversion} if
	\begin{itemize}
		\item either $a_i(D) = a_j(D)$ and $l_i(D) < l_j(D)$ (\emph{primary inversion}),
		\item or $a_i(D) = a_j(D) + 1$ and $l_i(D) > l_j(D)$ (\emph{secondary inversion}).
	\end{itemize}
	Then we define $d_i(D) \coloneqq \# \{ i < j \leq n \mid (i,j) \; \text{is a diagonal inversion} \}$.
\end{definition}

\begin{definition}\ \label{def: dinv LD}
	We define the statistic \emph{dinv} on $\LD(n)$ as \[ \dinv(D) \coloneqq \sum_{i=1}^n d_i(D), \] and on $\D(n)$ by assuming that the inequality condition on the labels always holds.
\end{definition}

The number of primary and secondary diagonal inversions are referred to as \emph{primary} and \emph{secondary dinv} respectively. The labelled Dyck path in Figure~\ref{fig:LDP} has dinv equal to $6$. Its diagonal inversions are $(2,7), (4,7)$ (primary), and $(2,6), (3,4), (3,8), (5,8)$ (secondary).

\begin{definition} \label{def: dinv reading word LD}
	Let $D \in \LD(n)$. We define its \emph{dinv reading word} as the sequence of the labels read starting from the ones in the main diagonal from bottom to top; next the ones in the diagonal $y=x+1$ from bottom to top; then the ones in the diagonal $y=x+2$ from bottom to top, and so on.
\end{definition}

For example, the dinv reading word of the labelled Dyck path in Figure~\ref{fig:LDP} is $22416153$.

To introduce the fourth statistic, we need the following definitions.

\begin{definition}
	Let $a_1 a_2 \cdots a_k$ be a string of integers. We define its \emph{descent set} as \[ \mathsf{Des}(a_1 a_2 \cdots a_k) \coloneqq \{ 1 \leq i \leq k-1 \mid a_i > a_{i+1} \} \] and its \emph{major index} $\mathsf{maj}(a_1 a_2 \cdots a_k)$ as the sum of the elements of the descent set.
\end{definition}

\begin{definition}
	Let $D \in \LD(n)$. We define its \emph{parking word} $p(D)$ as follows.
	
	Let $C_1$ be the multiset containing the labels appearing in the first column of $D$, and let $p_1(D) \coloneqq \max C_1$. At step $i$, let $C_i$ be the multiset obtained from $C_{i-1}$ by removing $p_{i-1}(D)$ and adding all the labels in the $i$-th column of the $D$; let \[ p_i(D) \coloneqq \max \, \{x \in C_i \mid x \leq p_{i-1}(D) \} \] if this last set is non-empty, and $p_i(D) \coloneqq \max \, C_i$ otherwise. We finally define the parking word of $D$ as $p(D) \coloneqq p_1(D) \cdots p_n(D)$.
\end{definition}

\begin{definition}
	We define the statistic \emph{pmaj} on $\LD(n)$ as \[ \pmaj(D) \coloneqq \mathsf{maj}(p_n(D) \cdots p_1(D)). \]
\end{definition}

For example, the parking word of the labelled Dyck path in Figure~\ref{fig:LDP} is $54321621$. In fact, we have $C_1 = \{2,4,5\}$, $C_2 = \{2,4\}$, $C_3 = \{1,2,3\}$, and so on. The descent set of the reverse is $\{3\}$, so $\pmaj(D) = 3$. Notice that $\pmaj(D) \leq \bounce(D)$, and if $l_i(D) = i$ (or, more generally, if all the labels are in strictly increasing order from bottom to top), then the equality holds. The converse is not true in general.

\begin{definition}
	Let $D \in \LD(n)$. We define its \emph{pmaj reading word} as the sequence $l_1(D) \cdots l_n(D)$, i.e. the sequence of the labels read bottom to top.
\end{definition}

For example, the pmaj reading word of the labelled Dyck path in Figure~\ref{fig:LDP} is $24513261$. 

Finally, we associate to each labelled Dyck path a monomial in the variables $x_1,x_2,\dots$: for $D\in \LD(n)$ we set \[ x^D \coloneqq \prod_{i=1}^n x_{l_i(D)}. \]

\subsection{Decorated Dyck paths}
The idea of decorating peaks and rises (or falls) first appeared in \cite{Haglund-Remmel-Wilson-2015}, together with the corresponding statistics area and dinv (cf. also \cite{Zabrocki-4Catalan-2016}). Our pbounce statistic is essentially the bounce statistic in \cite{Egge-Haglund-Killpatrick-Kremer-2003} (cf. Remark~\ref{rem:pbounce_origin}). Our statistic bounce is new.

\medskip

We want to extend the definitions on labelled Dyck paths to more general objects, namely decorated Dyck paths. Let us give some definitions.

\begin{definition}
	The \emph{rises} of a Dyck path $D$ are the indices \[ \Rise(D) \coloneqq \{2\leq i \leq n\mid a_{i}(D)>a_{i-1}(D)\},\] or the vertical steps that are directly preceded by another vertical step.
\end{definition}

\begin{definition}
	The \emph{peaks} of a Dyck path $D$ are the indices \[ \Peak(D) \coloneqq \{1\leq i\leq n-1 \mid a_{i+1}(D) \leq a_i(D)\}\cup \{n\}, \] or the vertical steps that are followed by a horizontal step.
\end{definition}

\begin{definition}
	A \emph{decorated Dyck path} is a Dyck path where certain peaks are decorated with a symbol ${\color{blue!60} \bullet}$ (on the point joining the peak to the horizontal step following it) and certain rises are decorated with a symbol $\ast$. By $\DD(n)^{\circ a, \ast b}$ we denote the set of Dyck paths of size $n$ with $a$ decorated peaks and $b$ decorated rises.
\end{definition}

See Figure \ref{fig:example_deco_Dyck} for an example. 

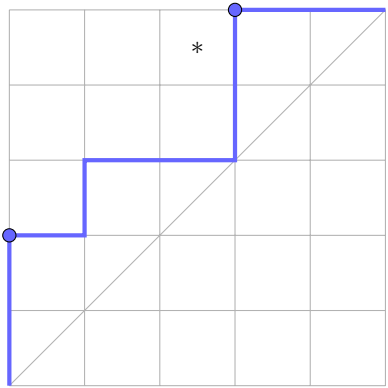
\begin{figure}[!ht]
	\centering
	\begin{tikzpicture}[scale=1]
	\draw[gray!60, thin] (0,0) grid (5,5) (0,0)--(5,5);
	\draw[blue!60, line width = 1.6pt] (0,0) -- (0,2) -- (1,2) -- (1,3) -- (3,3) -- (3,5) -- (5,5);
	
	\filldraw[fill=blue!60]
	(0,2) circle (2.5 pt)
	(3,5) circle (2.5 pt);
	
	\draw (2.5,4.5) node {$\ast$};
	\end{tikzpicture}
	\caption{Example of an element in $\DD(5)^{\circ 2, \ast 1} $.}
	\label{fig:example_deco_Dyck}
\end{figure}

\begin{remark}
	\label{rmk:rises-falls-correspondence}
	Note that if we call \emph{fall} a horizontal step that is followed by another horizontal step, then there is a natural way of mapping rises into falls bijectively. Indeed, take any rise.  The point joining the two vertical steps of this rise is a point where $D$ vertically crosses a certain diagonal parallel to the main diagonal. Since the path must end at the main diagonal, it must cross the same diagonal horizontally with a fall at least once. We map this rise to the first of such falls: this clearly yields a bijective map (see Figure~\ref{fig:rises-falls-correspondence}). Therefore we might equivalently decorate falls instead of the corresponding rises. 
\end{remark}

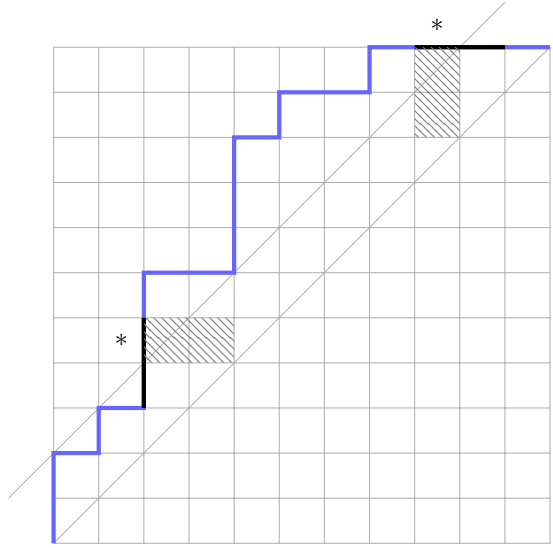
\begin{figure}[!ht]
	\centering
	\begin{tikzpicture}[scale=0.6]
	\draw[gray!60, thin] (0,0) grid (11,11) (0,0) -- (11,11) (-1,1) -- (10,12);
	\draw[blue!60, line width = 1.6pt] (0,0) -- (0,2) -- (1,2) -- (1,3) -- (2,3) -- (2,6) -- (4,6) -- (4,9) -- (5,9) -- (5,10) -- (7,10) -- (7,11) -- (11,11);
	\draw[line width = 1.6pt] (2,3) -- (2,5) (8,11) -- (10,11);
	\fill[pattern=north west lines, pattern color=gray] (2,4) rectangle (4,5) (8,11) rectangle (9,9);
	\draw  (1.5,4.5) node {$\ast$};
	\draw  (8.5,11.5) node {$\ast$};
	\end{tikzpicture} 
	\caption{Correspondence between rises and falls.}
	\label{fig:rises-falls-correspondence}
\end{figure}

We introduce here four statistics on the set $\DD(n)^{\circ a, \ast b}$ (area, dinv, bounce, and pbounce), all of which are generalizations of the usual statistics on undecorated Dyck paths. The area and dinv statistics are the ``unlabelled decorated'' versions of the statistics defined on labelled Dyck paths in Section~\ref{section: Labelled Dyck paths}. The bounce and pbounce are new.

\begin{definition}
	Let $D \in \DD(n)^{\circ a, \ast b}$, and let $a_1(D) \cdots a_n(D)$ be the area word of its underlying Dyck path. Now, let $\mathsf{DRise}(D) \subseteq \Rise(D)$ be the set of indices such that $i \in \mathsf{DRise}(D)$ if the $i$-th vertical step of $D$ is a decorated rise. We define the area of $D$ as \[ \area(D) \coloneqq \sum_{i \not \in \mathsf{DRise(D)}} a_{i}(D). \]
\end{definition}

For a more visual definition, the area is the number of whole squares that lie between the path and the main diagonal, except for the ones in the rows containing a decorated rise. Equivalently, we could define the area of $D$ as the number of whole squares that lie between the path and the main diagonal, except the ones in the columns containing falls corresponding to decorated rises. For example, the path in Figure~\ref{fig:statdef} has area equal to $6$ (gray in the picture). 

\begin{figure}[!ht]
	\centering
	
	\begin{tikzpicture}[scale=0.6]
	\draw
	(8.5,0.5) node {$0$}
	(8.5,1.5) node {$1$}
	(8.5,2.5) node {$2$}
	(8.5,3.5) node {$2$}
	(8.5,4.5) node {$3$}
	(8.5,5.5) node {$1$}
	(8.5,6.5) node {$0$}
	(8.5,7.5) node {$1$};
	
	\draw
	(-0.5,0.5) node {$0$}
	(-0.5,1.5) node {$0$}
	(-0.5,2.5) node {$0$}
	(-0.5,3.5) node {$1$}
	(-0.5,4.5) node {$1$}
	(-0.5,5.5) node {$2$}
	(-0.5,6.5) node {$3$}
	(-0.5,7.5) node {$3$};
	
	\draw[gray!60, thin] (0,0) grid (8,8) (0,0)--(8,8);
	
	\fill[gray, opacity=0.4]
	(0,1) rectangle (1,3)
	(2,3) rectangle (3,5)
	(3,5) rectangle (4,4)
	(4,5) rectangle (5,6);
	
	\draw[blue!60, line width = 1.6 pt] (0,0) -- (0,3) -- (1,3) -- (1,5) -- (4,5) -- (4,6) -- (6,6) -- (6,8) --(8,8);
	
	\draw[dashed, ultra thick, opacity=0.6] (0,0) |- (3,3) |- (5,5) |- (6,6) |- (8,8);
	
	\draw  
	(1.5,5.5) node {$\ast$}
	(6.5,8.5) node {$\ast$};			
	
	\filldraw[fill=blue!60]
	(4,6) circle (3 pt)
	(6,8) circle (3 pt);
	
	\end{tikzpicture}
	
	\caption{$D\in \DD(8)^{\circ 2, \ast 2}$, with the decorations on rises drawn above the corresponding falls. Its bounce word (left) and its area word (right) are shown.}
	\label{fig:statdef}
\end{figure}
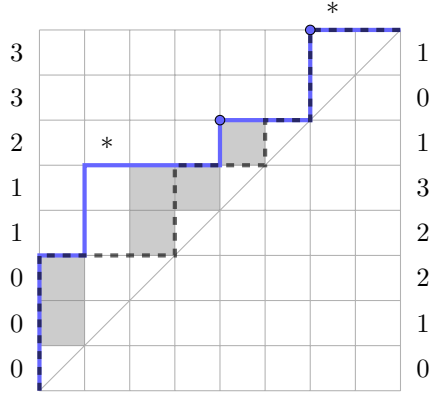

Similarly let $\mathsf{DPeak}(D) \subseteq \Peak(D)$ be the set of indices such that $i \in \mathsf{DPeak}(D)$ if the $i$-th vertical step of $D$ is a decorated peak. We define \[ \bounce(D) \coloneqq \sum_{i \not \in \DPeak(D)} b_i(D) \] and
\begin{align*}
	\dinv(D) \coloneqq & \, \# \{1\leq i<j \leq n \mid a_i(D) = a_j(D) \; \text{and} \; i \not \in \DPeak(D) \} \\
	+ & \, \# \{i<j \leq n \mid a_i(D) = a_j(D) + 1 \; \text{and} \; j \not \in \DPeak(D) \}.
\end{align*}
We refer to the first and second term as \emph{primary} and \emph{secondary dinv} respectively. For example, the path in Figure~\ref{fig:statdef} has bounce equal to $5$, and dinv also equal to $5$ ($4$ primary and $1$ secondary). This dinv statistic corresponds to the one on labelled decorated Dyck paths when the dinv reading word is a shuffle of  $1,2,\dots,n-a$ and $n,n-1,\dots,n-a+1$ where the labels from $n-a+1$ to $n$ are assigned to decorated peaks. For more implications of this fact and the definition of what it means to be a shuffle, we refer to Section \ref{sec:link_delta}.

We now describe how to compute the \emph{pbounce} statistic of a decorated Dyck path $D\in \DD(n)^{\circ a, \ast b}$. For every decorated peak, delete the horizontal step that directly follows it. When deleting the horizontal steps, move the rest of the path one square to the left, together with part of the main diagonal so that the number of squares between the path and the main diagonal in each row remains the same.  In this way we end up with a path from $(0,0)$ to $(n-a,n)$ and some line segments that form the new \emph{main diagonal}. This is what is called the corresponding \emph{leaning stack} in \cite{Haglund-Remmel-Wilson-2015}. See Figure~\ref{fig:ex_second_bounce} for an example. 

Now we draw the \emph{pbounce path} of this new object as described before: start from $(0,0)$, bounce at the beginning of horizontal steps of the path and at the main diagonal. Then label the vertical steps of this pbounce path as for the bounce path to get a \emph{pbounce word} $b'_1(D) \cdots b_n'(D)$. The pbounce of $D$ is then defined as \[ \pbounce(D) \coloneqq \sum_{i=1}^n b'_i(D). \]

In the example $D\in \DD(11)^{\circ 2, \ast 0}$ of Figure~\ref{fig:ex_second_bounce} its pbounce word is $00111111122$, so that its pbounce is $\pbounce(D)=11$.

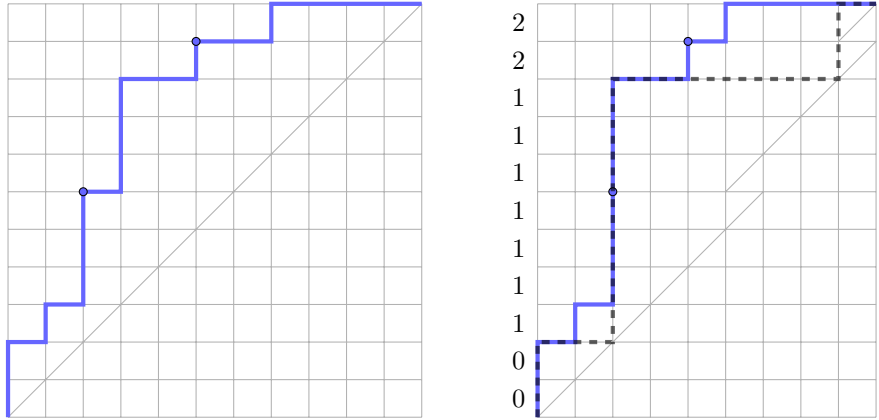
\begin{figure}[!ht]
	\centering
	\begin{minipage}{.55\textwidth}
		\centering
		\begin{tikzpicture}[scale=0.5]
		\draw[gray!60, thin] (0,0) grid (11,11) (0,0) -- (11,11);
		
		\draw[blue!60, line width = 1.6 pt] (0,0) -- (0,2) -- (1,2) -- (1,3) -- (2,3) -- (2,6) -- (3,6) -- (3,9) -- (5,9) -- (5,10) -- (7,10) -- (7,11) --(11,11);
		
		\filldraw[fill=blue!60]
		(2,6) circle (3 pt)
		(5,10) circle (3 pt);
		\end{tikzpicture} 
	\end{minipage}%
	\begin{minipage}{.45\textwidth}
		\centering
		\begin{tikzpicture}[scale=0.5]
		\draw[gray!60, thin] (0,0) grid (9,11) (0,0) -- (6,6) (5,6) -- (9,10) (8,10) -- (9,11);
		
		\draw[blue!60, line width = 1.6 pt] (0,0) -- (0,2) -- (1,2) -- (1,3) -- (2,3) -- (2,6) -- (2,9) -- (4,9) -- (4,10) -- (5,10) -- (5,11) -- (9,11);					
		
		\filldraw[fill=blue!60]
		(2,6) circle (3 pt)
		(4,10) circle (3 pt);
		
		\draw[dashed,ultra thick, opacity=0.6] (0,0) |- (2,2) |- (8,9) |- (9,11);
		\draw
		(-0.5,0.5) node {$0$}
		(-0.5,1.5) node {$0$}
		(-0.5,2.5) node {$1$}
		(-0.5,3.5) node {$1$}
		(-0.5,4.5) node {$1$}
		(-0.5,5.5) node {$1$}
		(-0.5,6.5) node {$1$}
		(-0.5,7.5) node {$1$}
		(-0.5,8.5) node {$1$}
		(-0.5,9.5) node {$2$}
		(-0.5,10.5) node {$2$};
		\end{tikzpicture} 
	\end{minipage}
	\caption{Construction of the pbounce path.}
	\label{fig:ex_second_bounce}
\end{figure}

\begin{remark} \label{rem:pbounce_origin}
	Since a peak is a vertical step followed by a horizontal step, one could think of these two steps together as a diagonal step. So $\DD(n)^{\circ a, \ast b}$ can be interpreted as the set of \emph{decorated Schr\"oder paths} with $a$ diagonal steps and $b$ decorated rises. The definitions of bounce and dinv given in \cite{Haglund-Book-2008}*{Chapter~4} coincide with our definitions of pbounce and dinv. When $b=0$ the definition of the area in \cite{Haglund-Book-2008}*{Chapter~4} and our area also coincide.
\end{remark}

\begin{remark}
	\label{rmk:bounce_pmaj}
	The notation $\pbounce$ was chosen because the $\pbounce$ statistic corresponds to a shuffle of the $\pmaj$ statistic. More precisely, the objects in $\DD(n)^{\circ a, \ast b}$ correspond essentially to labelled Dyck paths with $b$ decorated rises, whose pmaj reading word is a shuffle of $1,2,\dots,n-a$ and $n,n-1,\dots,n-a+1$ where the labels from $n-a+1$ to $n$ are assigned to decorated peaks. For the definition of what it means to be a shuffle, we refer to Section \ref{sec:link_delta}.
\end{remark}

We conclude by introducing some notation that will be convenient when discussing the combinatorics of decorated Dyck paths. For each of the three statistics that depend on the decorations on the peaks, i.e. dinv, bounce and pbounce, there is one specific peak that, when decorated, does not alter this statistic. Often, when dealing with one of these statistics, it will be useful to consider not the whole set of decorated Dyck paths, but rather the set of decorated Dyck paths where this specific peak is not decorated. We will use the following notations:
\begin{center}
	\begin{itemize}
		\item $ \DDd(n)^{\circ a, \ast b} \subseteq \DD(n)^{\circ a, \ast b}$  is the subset where the \emph{rightmost highest peak}, i.e.\ the rightmost peak that is the furthest removed from the main diagonal, is never decorated, denoted as such because decorating this peak does not alter the dinv statistic;
		\item $\DDb(n)^{\circ a, \ast b} \subseteq \DD(n)^{\circ a, \ast b}$ is the subset where the \emph{first peak}, i.e.\ the peak in the leftmost column, is never decorated, denoted as such because decorating this peak does not alter the bounce statistic; 
		\item $ \DDp(n)^{\circ a, \ast b} \subseteq \DD(n)^{\circ a, \ast b}$ is the subset where the \emph{last peak}, i.e.\ the peak in the top row is never decorated, denoted as such because decorating this peak does not alter the pbounce statistic. 
	\end{itemize}
\end{center}

\begin{remark}\label{rem: combinatorial pieri rule }
	We observe that 
	\begin{align*}
		\sum_{D\in \DD(n)^{\circ a, \ast b}} q^{\dinv(D) }t^{\area(D)}=& \sum_{D\in \DDd(n)^{\circ a, \ast b}} q^{\dinv(D) }t^{\area(D)}\\&+\sum_{D\in \DDd(n)^{\circ a-1, \ast b}} q^{\dinv(D) }t^{\area(D)}.
	\end{align*} 
	Indeed, a decoration on the rightmost highest peak does not alter the dinv or area of a path. It follows that we can split the sum over $\DD(n)^{\circ a, \ast b}$ into two parts, one with $a$ decorations on peaks, never decorating the rightmost hightest peak, and one with $a-1$ decorations on peaks, always decorating the rightmost highest peak. This is reflected by the sums over $\DDd(n)^{\circ a, \ast b}$ and $\DDd(n)^{\circ a-1, \ast b}$
	
	A similar relation occurs for the bistatistics $(\area, \pbounce)$ and $(\area, \bounce)$. 
\end{remark}

\subsection{Labelled decorated Dyck paths}

The definitions in this section first appeared in \cite{Haglund-Remmel-Wilson-2015}.

\medskip

Combining the objects of the previous two sections we define some sets of \emph{labelled decorated Dyck paths} whose objects are labelled Dyck paths with some kind of decorations. 

The set $\LDD^{\ast k}(n)$ consists of labelled Dyck paths with $k$ decorations on the rises (or falls by the correspondence explained in Remark~\ref{rmk:rises-falls-correspondence}). For a path $D\in \LDD^{\ast k}(n)$, its $\dinv$ is defined to be the dinv of labelled Dyck paths and its $\area$ in the same way as for decorated Dyck paths.   

\begin{definition}
	The \emph{contractible valleys} of a labelled Dyck path $D$ are the indices
	\begin{align*}
		\Val(D) \coloneqq & \; \{2\leq i\leq n \mid a_i(D)<a_{i-1}(D)\} \\
		\cup & \; \{2\leq i \leq n \mid a_i(D)=a_{i-1}(D) \; \text{and} \; l_i(D)>l_{i-1}(D)\},
	\end{align*}
	or the vertical steps that are directly preceded by a horizontal step and if that horizontal step were to be removed and the rest of the path moved one step to the west, the labels would still be strictly increasing in columns.
\end{definition}

The set $\LDD^{\vee k}(n)$ consists of labelled Dyck paths where $k$ contractible valleys are decorated with a $\vee$. For $D\in \LDD^{\vee k}(n)$ set $\DVal(D)\subseteq \Val(D)\subseteq\{1,\dots,n\}$ to be the set of indices of the decorated contractible valleys. The $\area$ of such a path is simply the area of the underlying Dyck path. Furthermore,  we define  \[\dinv(D) \coloneqq \left( \sum_{i \not \in \DVal(D)} d_i(D) \right) - k. \] This is the statistic $\text{dinv}^-$ defined in \cite{Haglund-Remmel-Wilson-2015}, where one can find the argument for the positivity of this quantity. 

Finally, if $D$ is a labelled decorated Dyck path, we define $x^D$ in the same way as for labelled Dyck paths.

\subsection{Partially labelled Dyck paths}

All the definitions in this section first appeared in \cite{Haglund-Remmel-Wilson-2015}, except for the pmaj, which is new.

\medskip

A \emph{partially labelled Dyck path} is a Dyck path whose vertical steps are labelled with (not necessarily distinct) non-negative integers such that the labels appearing in each column are strictly increasing from bottom to top, and $0$ does not appear in the first column.

These are essentially labelled Dyck paths in which we allow $0$ to be a label. For the definition of the monomial $x^D$ associated to a partially labelled Dyck path $D$, though, we will assume $x_0 = 1$.

We also have the possibility of decorating rises, and three statistics $\area$, $\dinv$, and $\pmaj$ as for standard labelled Dyck paths. We denote by $\PLD(m,n)^{\ast k}$ the set of partially labelled Dyck paths with $m$ zero labels, $n$ non-zero labels, and $k$ decorated rises.

\subsection{Parallelogram polyominoes}
\label{section:parallelogram-polyominoes}

The statistics bounce and dinv on parallelogram polyominoes first appeared in \cite{Dukes-LeBorgne-2013} and \cite{Aval-DAdderio-Dukes-Hicks-LeBorgne-2014} respectively. The statistic area is classical. The definitions for labelled parallelogram polyominoes are new.

\begin{definition}
	\label{def:parallelogrampolyominoes}
	A \emph{parallelogram polyomino} of size $m \times n$ is a pair of lattice paths from $(0,0)$ to $(m,n)$ using only north and east steps, such that the first one (the \emph{red path}) lies always strictly above the second one (the \emph{green path}), except when they meet in the extremal points. A \emph{labelled parallelogram polyomino} is a parallelogram polyomino where the vertical steps of the first path are labelled with (not necessarily distinct) positive integers such that the labels appearing in each column are strictly increasing from bottom to top.
\end{definition}

The set of all parallelogram polyominoes of size $m \times n$ is denoted by $\PP(m,n)$, while the set of labelled ones is denoted by $\LPP(m,n)$. For $P \in \LPP(m,n)$ we set $l_i(D)$ to be the label of the $i$-th vertical step. We define the associated monomial $x^P$ as we did for the labelled Dyck paths.

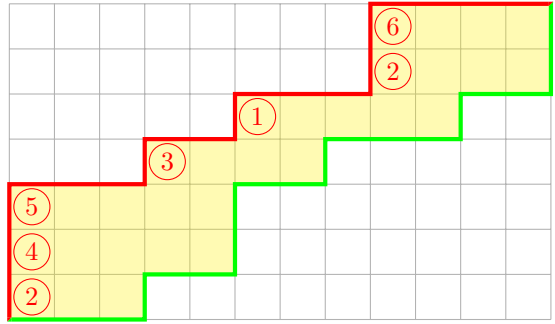
\begin{figure}[!ht]
	\centering
	\begin{tikzpicture}[scale=0.6]
	\draw[gray!60, thin] (0,0) grid (12,7);
	
	\filldraw[yellow, opacity=0.3] (0,0) -- (3,0) -- (3,1) -- (5,1) -- (5,3) -- (7,3) -- (7,4) -- (10,4) -- (10,5) -- (12,5) -- (12,7) -- (8,7) -- (8,5) -- (5,5) -- (5,4) -- (3,4) -- (3,3) -- (0,3) -- cycle;
	
	\draw[green, sharp <-sharp >, sharp angle = 45, line width=1.6pt] (0,0) -- (3,0) -- (3,1) -- (5,1) -- (5,3) -- (7,3) -- (7,4) -- (10,4) -- (10,5) -- (12,5) -- (12,7);
	
	\draw[red, sharp <-sharp >, sharp angle = -45, line width=1.6pt] (0,0) -- (0,3) -- (3,3) -- (3,4) -- (5,4) -- (5,5) -- (8,5) -- (8,7) -- (12,7);
	
	\draw[red]
	(0.5,0.5) circle(0.4 cm) node {$2$}
	(0.5,1.5) circle(0.4 cm) node {$4$}
	(0.5,2.5) circle(0.4 cm) node {$5$}
	(3.5,3.5) circle(0.4 cm) node {$3$}
	(5.5,4.5) circle(0.4 cm) node {$1$}
	(8.5,5.5) circle(0.4 cm) node {$2$}
	(8.5,6.5) circle(0.4 cm) node {$6$};
	\end{tikzpicture}
	
	\caption{A $12 \times 7$ labelled parallelogram polyomino.}
	\label{fig:LPP}
\end{figure}

Parallelogram polyominoes can be encoded using their \textit{area word}, which is an area word in the alphabet $\overline{\mathbb{N}} \setminus \{0\} \coloneqq \bar{0} < 1 < \bar{1} < 2 < \bar{2} < \dots$. It can be computed in two equivalent ways.

The first one consists of drawing a diagonal of slope $-1$ from the end of every horizontal green step, and attaching to that step the length of that diagonal (i.e. the number of squares it crosses). Then, one puts a dot in every square not crossed by any of those diagonals, and attaches to each vertical red step the number of dots in the corresponding row. Next, one bars the numbers attached to vertical red steps, and finally one reads those numbers following the diagonals of slope $-1$, reading the labels when encountering the end of its step and the red label before the green one. See Figure~\ref{fig:aw-polyominoes} for an example.

\begin{figure}[!ht]
	\centering
	\begin{tikzpicture}[scale=0.6]
	\draw[gray!60, thin] (0,0) grid (12,7);
	
	\filldraw[yellow, opacity=0.3] (0,0) -- (3,0) -- (3,1) -- (5,1) -- (5,3) -- (7,3) -- (7,4) -- (10,4) -- (10,5) -- (12,5) -- (12,7) -- (8,7) -- (8,5) -- (5,5) -- (5,4) -- (3,4) -- (3,3) -- (0,3) -- cycle;
	
	\draw[black]
	(1,0) -- (0,1)
	(2,0) -- (0,2)
	(3,0) -- (0,3)
	(4,1) -- (2,3)
	(5,1) -- (3,3)
	(6,3) -- (5,4)
	(7,3) -- (5,5)
	(8,4) -- (7,5)
	(9,4) -- (8,5)
	(10,4) -- (8,6)
	(11,5) -- (9,7)
	(12,5) -- (10,7);
	
	\filldraw[fill=black]
	(2.5,1.5) circle (2pt)
	(1.5,2.5) circle (2pt)
	(4.5,2.5) circle (2pt)
	(3.5,3.5) circle (2pt)
	(4.5,3.5) circle (2pt)
	(6.5,4.5) circle (2pt)
	(9.5,5.5) circle (2pt)
	(8.5,6.5) circle (2pt)
	(11.5,6.5) circle (2pt);
	
	\node[below] at (0.5,0) {$1$};
	\node[below] at (1.5,0) {$2$};
	\node[below] at (2.5,0) {$3$};
	\node[below] at (3.5,1) {$2$};
	\node[below] at (4.5,1) {$2$};
	\node[below] at (5.5,3) {$1$};
	\node[below] at (6.5,3) {$2$};
	\node[below] at (7.5,4) {$1$};
	\node[below] at (8.5,4) {$1$};
	\node[below] at (9.5,4) {$2$};
	\node[below] at (10.5,5) {$2$};
	\node[below] at (11.5,5) {$2$};
	
	\node[left] at (0,0.5) {$\bar{0}$};
	\node[left] at (0,1.5) {$\bar{1}$};
	\node[left] at (0,2.5) {$\bar{2}$};
	\node[left] at (3,3.5) {$\bar{2}$};
	\node[left] at (5,4.5) {$\bar{1}$};
	\node[left] at (8,5.5) {$\bar{1}$};
	\node[left] at (8,6.5) {$\bar{2}$};
	
	\draw[green, sharp <-sharp >, sharp angle = 45, line width=1.6pt] (0,0) -- (3,0) -- (3,1) -- (5,1) -- (5,3) -- (7,3) -- (7,4) -- (10,4) -- (10,5) -- (12,5) -- (12,7);
	
	\draw[red, sharp <-sharp >, sharp angle = -45, line width=1.6pt] (0,0) -- (0,3) -- (3,3) -- (3,4) -- (5,4) -- (5,5) -- (8,5) -- (8,7) -- (12,7);
	\end{tikzpicture}
	
	\caption{The construction of the area word, which is $\bar{0} 1 \bar{1} 2 \bar{2} 3 2 2 \bar{2} 1 \bar{1} 2 1 1 \bar{1} 2 \bar{2} 2 2$.}
	\label{fig:aw-polyominoes}
\end{figure}
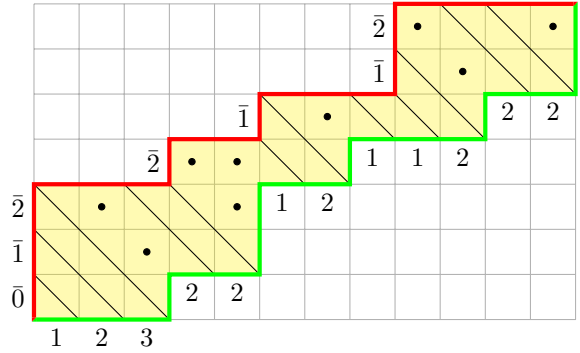

Equivalently, we can build a Dyck path of size $m+n$ from the polyomino in the following way. Running over red and green steps alternatively, we draw a north step in our Dyck path if the corresponding step in the polyomino was either a red north step or a green east step, and we draw an east step in our Dyck path otherwise. Then we take the area word of the Dyck path and replace $0$'s with $\bar{0}$'s, $1$'s with $1$'s, $2$'s with $\bar{1}$'s, $3$'s with $2$'s and so on (replacing the alphabet $\mathbb{N}$ with $\overline{\mathbb{N}} \setminus \{0\}$).

It is not hard to check that those definitions are equivalent, and that the following holds (see \cite{Aval-DAdderio-Dukes-Hicks-LeBorgne-2014}*{Section~3} for detailed proofs and examples). 

\begin{proposition}
	For $m, n \geq 1$, there is a bijective correspondence between $m \times n$ parallelogram polyominoes and area words of length $m+n$ in the alphabet $\overline{\mathbb{N}} \setminus \{0\}$ with exactly one $\bar{0}$ as first letter, exactly $m$ unbarred letters, and exactly $n$ barred letters.
\end{proposition}

For a parallelogram polyomino $P$, we denote by $a(P)$ its area word, and by $a_i(P)$ its $i$-th letter. For $a \in \overline{\mathbb{N}}$, let $\vert a \vert$ be the value of that letter (i.e. if the letter is barred, just take the underlying number).

\begin{definition}
	We define the statistic \emph{area} on $\PP(m,n)$ and $\LPP(m,n)$ as \[ \area(P) = \sum_{i=1}^{m+n} \vert a_i(P) \vert, \] which is also the number of whole squares between the two paths.
\end{definition}

Since each letter in the area word but the first $\bar{0}$ has value at least $1$, $\area(P) \geq m+n-1$, for all $P \in \PP(m,n)$. This is also obvious with the other definition, because since the two paths only meet in the extremal points, then there is at least a ``snake'' of squares from $(0,0)$ to $(m,n)$ of length $m+n-1$ that lies between the two paths.

Similarly to Dyck paths, parallelogram polyominoes have a \emph{bounce path}. To compute it, start by drawing a single horizontal step, then pursue the following algorithm: draw vertical steps until the path hits the end of a horizontal red step; then draw horizontal steps until the path hits the end of a vertical green step; repeat until it reaches $(m,n)$. Now, attach to each step of the bounce path a letter of the alphabet $\overline{\mathbb{N}}$ starting from $\bar{0}$ and going up by one in $\overline{\mathbb{N}}$ every time the path changes direction. Let us call \textit{bounce word} the sequence $b(P)$ of letters we used (see Figure~\ref{fig:bouncepath-polyominoes} for an example). Let $b_i(P)$ be the $i$-th letter in the bounce word of $P$.

\begin{definition}
	We define the statistic \emph{bounce} on $\PP(m,n)$ and $\LPP(m,n)$ as \[ \bounce(P) \coloneqq \sum_{i=1}^{m+n} \vert b_i(P) \vert. \]
\end{definition}

Since each letter in the bounce word but the first $\bar{0}$ has value at least $1$, $\bounce(P) \geq m+n-1$ for all $P \in \PP(m,n)$.

\begin{figure}[!ht]
	\centering
	\begin{tikzpicture}[scale=0.6]
	\draw[gray!60, thin] (0,0) grid (12,7);
	
	\filldraw[yellow, opacity=0.3] (0,0) -- (3,0) -- (3,1) -- (5,1) -- (5,3) -- (7,3) -- (7,4) -- (10,4) -- (10,5) -- (12,5) -- (12,7) -- (8,7) -- (8,5) -- (5,5) -- (5,4) -- (3,4) -- (3,3) -- (0,3) -- cycle;
	
	\draw[green, sharp <-sharp >, sharp angle = 45, line width=1.6pt] (0,0) -- (3,0) -- (3,1) -- (5,1) -- (5,3) -- (7,3) -- (7,4) -- (10,4) -- (10,5) -- (12,5) -- (12,7);
	\draw[red, sharp <-sharp >, sharp angle = -45, line width=1.6pt] (0,0) -- (0,3) -- (3,3) -- (3,4) -- (5,4) -- (5,5) -- (8,5) -- (8,7) -- (12,7);
	
	\draw[dashed, thick, opacity=0.6] (0,0) -- (1,0) -- (1,3) -- (5,3) -- (5,4) -- (7,4) -- (7,5) -- (10,5) -- (10,7) -- (12,7);
	
	\node[above] at (0.5,0) {$\bar{0}$};
	\node[right] at (1,0.5) {$1$};
	\node[right] at (1,1.5) {$1$};
	\node[right] at (1,2.5) {$1$};
	\node[above] at (1.5,3) {$\bar{1}$};
	\node[above] at (2.5,3) {$\bar{1}$};
	\node[above] at (3.5,3) {$\bar{1}$};
	\node[above] at (4.5,3) {$\bar{1}$};
	\node[right] at (5,3.5) {$2$};
	\node[above] at (5.5,4) {$\bar{2}$};
	\node[above] at (6.5,4) {$\bar{2}$};
	\node[right] at (7,4.5) {$3$};
	\node[above] at (7.5,5) {$\bar{3}$};
	\node[above] at (8.5,5) {$\bar{3}$};
	\node[above] at (9.5,5) {$\bar{3}$};
	\node[right] at (10,5.5) {$4$};
	\node[right] at (10,6.5) {$4$};
	\node[above] at (10.5,7) {$\bar{4}$};
	\node[above] at (11.5,7) {$\bar{4}$};	
	\end{tikzpicture}
	
	\caption{A parallelogram polyomino in which the bounce path is shown.}
	\label{fig:bouncepath-polyominoes}
\end{figure}
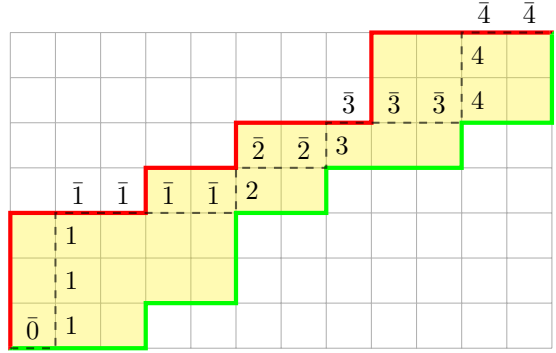

\begin{definition}
	Let $P \in \PP(m,n)$. For $1 \leq i < j \leq m+n$, we (mysteriously) say that the pair $(i,j)$ is an \textit{inversion} if $a_j(P)$ is the successor of $a_i(P)$ in the alphabet $\overline{\mathbb{N}}$. Then we define $d_i(P) \coloneqq \# \{ i < j \mid (i,j) \; \text{is an inversion} \}$. 
\end{definition}

\begin{definition}
	We define the statistic \emph{dinv} on $\PP(m,n)$ as \[ \dinv(P) \coloneqq \sum_{i=1}^{m+n} d_i(P). \]
\end{definition}

Notice that by definition of area word each letter but the first $\bar{0}$ must be the successor of a letter to its left, and this implies that, for $P \in \PP(m,n)$, $\dinv(P) \geq m+n-1$.

\begin{definition}
	Let $P \in \LPP(m,n)$. We define its \emph{parking word} as follows.
	
	Let $C_1$ be the multiset containing the labels appearing in the first column of $P$, and let $p_1(P) \coloneqq \max C_1$. For $i > 1$, at step $i$, if the $i$-th step of the green path is vertical let $C_i = C_{i-1} \setminus \{p_{i-1}(P)\}$; if the $i$-th step of the green path is horizontal let $C_i$ be the multiset obtained from $C_{i-1}$ by replacing $p_{i-1}(P)$ with a $0$, and then adding all the labels in the column of $P$ containing the $i$-th green step (which we recall being horizontal). Next, let $p_i(P) \coloneqq \max \, \{a \in C_i \mid a \leq p_{i-1}(P) \}$ if this set is non-empty, and $p_i(P) \coloneqq \max \, C_i$ otherwise. We finally define the parking word of $P$ as $p(P) \coloneqq p_1(P) \cdots p_{m+n-1}(P)$.
\end{definition}

\begin{definition}
	We define the statistic \emph{pmaj} on $\LPP(m,n)$ as \[ \pmaj(P) \coloneqq \mathsf{maj}(p_{m+n-1}(P) \cdots p_1(P)) + m + n - 1. \] 
\end{definition}

For example, the parallelogram polyomino in Figure~\ref{fig:LPP} (cf. also Figure~\ref{fig:aw-polyominoes} and Figure~\ref{fig:bouncepath-polyominoes}) has area equal to $30$, bounce equal to $41$, and dinv equal to $35$. Its parking word is $542000031000006200$ and hence its pmaj is $33$. 
\begin{remark} \label{rmk:polyo_pmaj_bounce}
	Analogously as it holds for Dyck paths, we have that $\pmaj(D) \leq \bounce(D)$, and if $l_i(D) = i$ (or, more generally, if all the labels are in strictly increasing order), then the equality holds. The converse is not true in general. 
\end{remark}

\subsection{Decorated parallelogram polyominoes}

The definitions in this section are all new.

\medskip

Once again, we want to extend the definitions we just gave to decorated objects. Let us give some definitions.

\begin{definition}
	The \emph{rises} of a parallelogram polyomino $P$ are the indices \[ \Rise(P) \coloneqq \{1 \leq i \leq m+n-1 \mid \vert a_{i}(P) \vert < \vert a_{i+1}(P) \vert \},\] which correspond in the figure to diagonals of slope $-1$ from the endpoint of a vertical red step to the endpoint of a horizontal green step.
\end{definition}

\begin{definition}
	The \emph{red peaks} of a parallelogram polyomino $P$ are the horizontal red steps that immediately follow a vertical red step.
\end{definition}

\begin{definition}
	A \emph{decorated parallelogram polyomino} is a parallelogram polyomino where certain red peaks are decorated with a symbol ${\color{red} \bullet}$ (on the point joining the peak to the vertical step preceding it) \emph{or} certain rises are decorated with a symbol $\ast$ (on the segment joining the endpoint of a vertical red step to the endpoint of a horizontal green step).
\end{definition}

We do not decorate the two sets simultaneously. By $\PP(m,n)^{\circ k}$ we denote the set of parallelogram polyominoes of size $m \times n$ with $k$ decorated red peaks, and by $\PP(m,n)^{\ast k}$ we denote the set of parallelogram polyominoes of size $m \times n$ with $k$ decorated rises. We do not decorate the leftmost red peak, nor do we decorate the first $\bar{0}1$ rise.

\begin{definition}
	Let $P \in \PP(m,n)^{\ast k}$, and let $a_1(P) \cdots a_{m+n}(P)$ be its area word. Now, let $\mathsf{DRise}(P) \subseteq \Rise(P)$ be the set of indices such that $i \in \mathsf{DRise}(P)$ if the $i$-th letter of the area word corresponds to a decorated rise. We define the area of $P$ as \[ \area(P) \coloneqq \sum_{i \not \in \mathsf{DRise(P)}} \vert a_{i}(P) \vert . \]
\end{definition}

Notice that a decoration on the first $\bar{0}1$ rise would not have changed the area. This is one of the reasons that suggest not to decorate it.

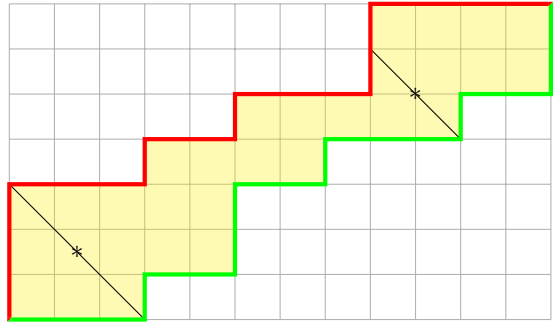
\begin{figure}[!ht]
	\centering
	\begin{tikzpicture}[scale=0.6]
	\draw[gray!60, thin] (0,0) grid (12,7);
	
	\filldraw[yellow, opacity=0.3] (0,0) -- (3,0) -- (3,1) -- (5,1) -- (5,3) -- (7,3) -- (7,4) -- (10,4) -- (10,5) -- (12,5) -- (12,7) -- (8,7) -- (8,5) -- (5,5) -- (5,4) -- (3,4) -- (3,3) -- (0,3) -- cycle;
	
	\draw
	(0,3) -- (3,0)
	(10,4) -- (8,6);
	
	\draw[green, sharp <-sharp >, sharp angle = 45, line width=1.6pt] (0,0) -- (3,0) -- (3,1) -- (5,1) -- (5,3) -- (7,3) -- (7,4) -- (10,4) -- (10,5) -- (12,5) -- (12,7);
	\draw[red, sharp <-sharp >, sharp angle = -45, line width=1.6pt] (0,0) -- (0,3) -- (3,3) -- (3,4) -- (5,4) -- (5,5) -- (8,5) -- (8,7) -- (12,7);
	
	\draw (1.5,1.5) node {$\ast$};
	\draw (9,5) node {$\ast$};
	\end{tikzpicture}
	
	\caption{A parallelogram polyomino with two decorated rises. Its area word is $\bar{0} 1 \bar{1} 2 {\color{red} \bar{2}} 3 2 2 \bar{2} 3 3 \bar{2} 1 1 {\color{red} \bar{1}} 2 \bar{2} 2 2$, where the decorated rises are highlighted in red.}
	\label{fig:decorated-rises-polyominoes}
\end{figure}

\begin{definition}
	Let $P \in \PP(m,n)^{\circ k}$, and let $b_1(P) \cdots b_{m+n}(P)$ be its bounce word. Now, let $\mathsf{DPeak}(P)$ be the set of indices such that $i \in \mathsf{DPeak}(P)$ if the $i$-th step of the bounce path is horizontal and lies in the same column as a decorated red peak. We define the bounce of $P$ as \[ \bounce(P) \coloneqq \sum_{i \not \in \mathsf{DPeak(P)}} \vert b_{i}(P) \vert . \]
\end{definition}

Notice that a decoration on the leftmost red peak would not have changed the bounce, since the label below it must be $\bar{0}$. This is one of the reasons that suggest not to decorate it.

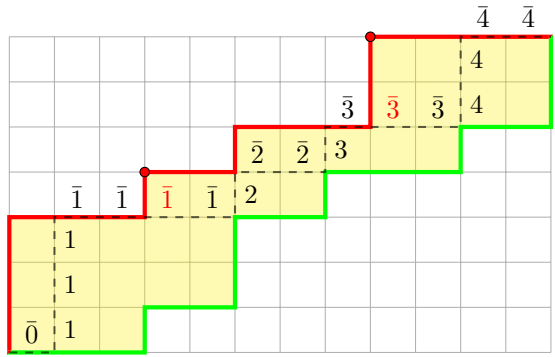
\begin{figure}[!ht]
	\centering
	\begin{tikzpicture}[scale=0.6]
	\draw[gray!60, thin] (0,0) grid (12,7);
	
	\filldraw[yellow, opacity=0.3] (0,0) -- (3,0) -- (3,1) -- (5,1) -- (5,3) -- (7,3) -- (7,4) -- (10,4) -- (10,5) -- (12,5) -- (12,7) -- (8,7) -- (8,5) -- (5,5) -- (5,4) -- (3,4) -- (3,3) -- (0,3) -- cycle;
	
	\draw[green, sharp <-sharp >, sharp angle = 45, line width=1.6pt] (0,0) -- (3,0) -- (3,1) -- (5,1) -- (5,3) -- (7,3) -- (7,4) -- (10,4) -- (10,5) -- (12,5) -- (12,7);
	\draw[red, sharp <-sharp >, sharp angle = -45, line width=1.6pt] (0,0) -- (0,3) -- (3,3) -- (3,4) -- (5,4) -- (5,5) -- (8,5) -- (8,7) -- (12,7);
	
	\draw[dashed, thick, opacity=0.6] (0,0) -- (1,0) -- (1,3) -- (5,3) -- (5,4) -- (7,4) -- (7,5) -- (10,5) -- (10,7) -- (12,7);
	
	\filldraw[fill=red]
	(3,4) circle (3pt)
	(8,7) circle (3pt);
	
	\node[above] at (0.5,0) {$\bar{0}$};
	\node[right] at (1,0.5) {$1$};
	\node[right] at (1,1.5) {$1$};
	\node[right] at (1,2.5) {$1$};
	\node[above] at (1.5,3) {$\bar{1}$};
	\node[above] at (2.5,3) {$\bar{1}$};
	\node[above, red] at (3.5,3) {$\bar{1}$};
	\node[above] at (4.5,3) {$\bar{1}$};
	\node[right] at (5,3.5) {$2$};
	\node[above] at (5.5,4) {$\bar{2}$};
	\node[above] at (6.5,4) {$\bar{2}$};
	\node[right] at (7,4.5) {$3$};
	\node[above] at (7.5,5) {$\bar{3}$};
	\node[above, red] at (8.5,5) {$\bar{3}$};
	\node[above] at (9.5,5) {$\bar{3}$};
	\node[right] at (10,5.5) {$4$};
	\node[right] at (10,6.5) {$4$};
	\node[above] at (10.5,7) {$\bar{4}$};
	\node[above] at (11.5,7) {$\bar{4}$};	
	\end{tikzpicture}
	
	\caption{A parallelogram polyomino with two decorated red peaks. The letters to ignore in its bounce word while computing the bounce are highlighted in red.}
	\label{fig:decorated-peaks-polyominoes}
\end{figure}

\subsection{Reduced polyominoes}

The definitions in this section are all new.

\medskip

We now discuss a reduced version of the parallelogram polyominoes. Even though these objects are strictly related to the original ones, it turns out that some combinatorial results that we will discuss later in the article look more natural on these reduced siblings.

Since in Definition~\ref{def:parallelogrampolyominoes} it is asked that the two paths meet each other only in the extremal points, then the red path has to begin with a vertical step and end with a horizontal step, while the green path has to begin with a horizontal step and end with a vertical step. We can adjust the definition in order to remove this restriction, and get \emph{reduced polyominoes}.

\begin{definition}
	\label{def:reducedpolyominoes}
	A \emph{reduced polyomino} of size $m \times n$ is a pair of lattice paths from $(0,0)$ to $(m,n)$ using only north and east steps, such that the first one (the \emph{red path}) lies always weakly above the second one (the \emph{green path}).
\end{definition}

The set of reduced polyominoes of size $m \times n$ is denoted by $\RP(m,n)$. We do not define a labelled version of reduced polyominoes, since the combinatorics with the labels looks more natural on the regular ones.

Reduced polyominoes are also encoded by their area word, in the alphabet $\overline{\mathbb{N}}$. It is computed using the first of the two algorithms described in Subsection~\ref{section:parallelogram-polyominoes} (notice that some diagonals can have length $0$), but adding an artificial $0$ at the beginning.

The second algorithm can also be used, but with a slight modification: in order to get a Dyck path we should artificially add a north step at the beginning and an east step at the end, and then in the resulting area word we should replace the alphabet $\mathbb{N}$ with $\overline{\mathbb{N}}$. Notice that this gives a bijective correspondence between reduced polyominoes with semi-perimeter $m+n$ and Dyck paths of size $m+n+1$.

\begin{figure}[!ht]
	\centering
	\begin{tikzpicture}[scale=0.6]
	\draw[gray!60, thin] (0,0) grid (6,11);
	
	\filldraw[yellow, opacity=0.3] (0,0) -- (1,0) -- (2,0) -- (2,1) -- (2,2) -- (2,3) -- (3,3) -- (3,4) -- (3,5) -- (4,5) -- (4,6) -- (4,7) -- (4,8) -- (5,8) -- (6,8) -- (6,9) -- (6,10) -- (6,11) -- (5,11) -- (5,10) -- (5,9) -- (4,9) -- (4,8) -- (4,7) -- (4,6) -- (3,6) -- (3,5) -- (3,4) -- (2,4) -- (1,4) -- (1,3) -- (1,2) -- (0,2) -- (0,1) -- (0,0);
	
	\draw
	(1,0) -- (0,1)
	(2,0) -- (0,2)
	(3,3) -- (2,4)
	(4,5) -- (3,6)
	(5,8) -- (4,9)
	(6,8) -- (5,9);
	
	\filldraw[fill=black]
	(1.5,1.5) circle (2pt)
	(1.5,2.5) circle (2pt)
	(1.5,3.5) circle (2pt)
	(5.5,9.5) circle (2pt)
	(5.5,10.5) circle (2pt);
	
	\node[below] at (0.5,0) {$1$};
	\node[below] at (1.5,0) {$2$};
	\node[below] at (2.5,3) {$1$};
	\node[below] at (3.5,5) {$1$};
	\node[below] at (4.5,8) {$1$};
	\node[below] at (5.5,8) {$1$};
	
	\node[left] at (0,0.5) {$\bar{0}$};
	\node[left] at (0,1.5) {$\bar{1}$};
	\node[left] at (1,2.5) {$\bar{1}$};
	\node[left] at (1,3.5) {$\bar{1}$};
	\node[left] at (3,4.5) {$\bar{0}$};
	\node[left] at (3,5.5) {$\bar{0}$};
	\node[left] at (4,6.5) {$\bar{0}$};
	\node[left] at (4,7.5) {$\bar{0}$};
	\node[left] at (4,8.5) {$\bar{0}$};
	\node[left] at (5,9.5) {$\bar{1}$};
	\node[left] at (5,10.5) {$\bar{1}$};
	
	\draw[red, sharp <-sharp >, sharp angle = -45, line width=1.6pt] (0,0) -- (0,1) -- (0,2) -- (1,2) -- (1,3) -- (1,4) -- (2,4) -- (3,4) -- (3,5) -- (3,6) -- (4,6) -- (4,7) -- (4,8) -- (4,9) -- (5,9) -- (5,10) -- (5,11) -- (6,11);
	
	\draw[green, sharp <-sharp >, sharp angle = 45, line width=1.6pt] (0,0) -- (1,0) -- (2,0) -- (2,1) -- (2,2) -- (2,3) -- (3,3) -- (3,4) -- (3,5) -- (4,5) -- (4,6) -- (4,7) -- (4,8) -- (5,8) -- (6,8) -- (6,9) -- (6,10) -- (6,11);
	
	\draw[red, dashed, sharp <-sharp >, sharp angle = -45, line width=1.6pt] (0,0) -- (0,1) -- (0,2) -- (1,2) -- (1,3) -- (1,4) -- (2,4) -- (3,4) -- (3,5) -- (3,6) -- (4,6) -- (4,7) -- (4,8) -- (4,9) -- (5,9) -- (5,10) -- (5,11) -- (6,11);
	\end{tikzpicture}
	
	\caption{A $6 \times 11$ reduced polyomino. Its area word is $0 \bar{0} 1 \bar{1} 2 \bar{1} \bar{1} 1 \bar{0} \bar{0} 1 \bar{0} \bar{0} \bar{0} 1 1 \bar{1} \bar{1}$.}
	\label{fig:aw-reduced-polyomino}
\end{figure}
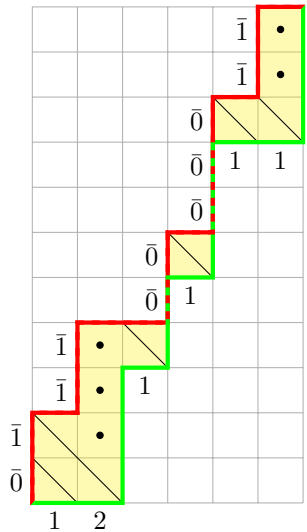

\begin{proposition}
	For $m, n \geq 0$, there is a bijective correspondence between $m \times n$ reduced polyominoes and area words of length $m+n+1$ in the alphabet $\overline{\mathbb{N}}$ starting with $0$, with exactly $m+1$ unbarred letters, and exactly $n$ barred letters.
\end{proposition}

Statistics on reduced polyominoes look mostly like the corresponding ones for parallelogram polyominoes. The \emph{area} is just, once again, the number of whole squares between the two paths. The \emph{bounce path} starts horizontally from $(0,0)$ and does not artificially turn north after one step; instead it goes east until it hits the beginning of a vertical green step, and then it bounces every time it hits the \emph{beginning} of a step of any of the two paths (notice that we did differently in the regular case, as we bounced every time the path hit the \emph{end} of a step). The labelling of the bounce path starts from $0$ and goes up by one in the alphabet $\overline{\mathbb{N}}$ every time it changes direction; of course, the \emph{bounce} is the sum of the values of all the labels appearing in the bounce path.

\begin{proposition}
	\label{normalizing map}
	There is a bijection $\mathsf{r} \colon \PP(m,n) \rightarrow \RP(n-1,m-1)$ such that $\area(P) = \area(\mathsf{r}(P)) + m+n-1$ and $\bounce(P) = \bounce(\mathsf{r}(P)) + m+n-1$.
\end{proposition}

\begin{proof}
	For $P \in \PP(m,n)$, to compute $\mathsf{r}(P)$ remove the first and the last step from both paths of $P$, then take the symmetry with respect to the line $x=y$, thus swapping the green and the red paths. It is obvious that the area decreases by $m+n-1$, and the bounce path of $\mathsf{r}(P)$ is the same as the bounce path of $P$ except that it is reflected with respect to the line $x=y$, the first two steps are deleted, and the labels are all decreased by one unit (two steps in the alphabet $\overline{\mathbb{N}}$).
\end{proof}

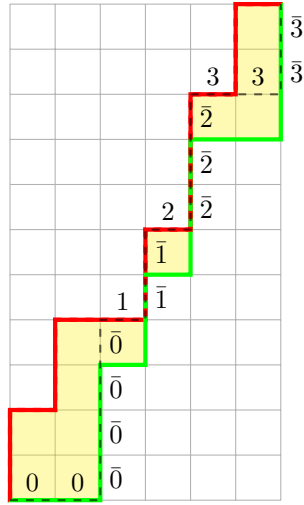
\begin{figure}[!ht]
	\centering
	\begin{tikzpicture}[scale=0.6]
	\draw[gray!60, thin] (0,0) grid (6,11);
	
	\filldraw[yellow, opacity=0.3] (0,0) -- (1,0) -- (2,0) -- (2,1) -- (2,2) -- (2,3) -- (3,3) -- (3,4) -- (3,5) -- (4,5) -- (4,6) -- (4,7) -- (4,8) -- (5,8) -- (6,8) -- (6,9) -- (6,10) -- (6,11) -- (5,11) -- (5,10) -- (5,9) -- (4,9) -- (4,8) -- (4,7) -- (4,6) -- (3,6) -- (3,5) -- (3,4) -- (2,4) -- (1,4) -- (1,3) -- (1,2) -- (0,2) -- (0,1) -- (0,0);
	
	\draw[red, sharp <-sharp >, sharp angle = -45, line width=1.6pt] (0,0) -- (0,1) -- (0,2) -- (1,2) -- (1,3) -- (1,4) -- (2,4) -- (3,4) -- (3,5) -- (3,6) -- (4,6) -- (4,7) -- (4,8) -- (4,9) -- (5,9) -- (5,10) -- (5,11) -- (6,11);
	
	\draw[green, sharp <-sharp >, sharp angle = 45, line width=1.6pt] (0,0) -- (1,0) -- (2,0) -- (2,1) -- (2,2) -- (2,3) -- (3,3) -- (3,4) -- (3,5) -- (4,5) -- (4,6) -- (4,7) -- (4,8) -- (5,8) -- (6,8) -- (6,9) -- (6,10) -- (6,11);
	
	\draw[red, sharp <-sharp >, sharp angle = -45, dashed, line width=1.6pt] (0,0) -- (0,1) -- (0,2) -- (1,2) -- (1,3) -- (1,4) -- (2,4) -- (3,4) -- (3,5) -- (3,6) -- (4,6) -- (4,7) -- (4,8) -- (4,9) -- (5,9) -- (5,10) -- (5,11) -- (6,11);
	
	\draw[dashed, opacity=0.6, thick] (0,0) -- (1,0) -- (2,0) -- (2,1) -- (2,2) -- (2,3) -- (2,4) -- (3,4) -- (3,5) -- (3,6) -- (4,6) -- (4,7) -- (4,8) -- (4,9) -- (5,9) -- (6,9) -- (6,10) -- (6,11);
	
	\node[above] at (0.5,0.0) {$0$};
	\node[above] at (1.5,0.0) {$0$};
	\node[right] at (2.0,0.5) {$\bar{0}$};
	\node[right] at (2.0,1.5) {$\bar{0}$};
	\node[right] at (2.0,2.5) {$\bar{0}$};
	\node[right] at (2.0,3.5) {$\bar{0}$};
	\node[above] at (2.5,4.0) {$1$};
	\node[right] at (3.0,4.5) {$\bar{1}$};
	\node[right] at (3.0,5.5) {$\bar{1}$};
	\node[above] at (3.5,6.0) {$2$};
	\node[right] at (4.0,6.5) {$\bar{2}$};
	\node[right] at (4.0,7.5) {$\bar{2}$};
	\node[right] at (4.0,8.5) {$\bar{2}$};
	\node[above] at (4.5,9.0) {$3$};
	\node[above] at (5.5,9.0) {$3$};
	\node[right] at (6.0,9.5) {$\bar{3}$};
	\node[right] at (6.0,10.5) {$\bar{3}$};
	\end{tikzpicture}
	
	\caption{The image of the parallelogram polyomino in Figure~\ref{fig:bouncepath-polyominoes} through the map $\mathsf{r}$. Notice the correspondence between the two bounce paths.}
	\label{fig:bounce-reduced-polyomino}
\end{figure}

The statistic \emph{dinv} is significantly different.

\begin{definition}
	Let $P \in \RP(m,n)$. For $1 \leq i < j \leq m+n$, we (legitimately) say that the pair $(i,j)$ is an \textit{inversion} if $a_i(P)$ is the successor of $a_j(P)$ in the alphabet $\overline{\mathbb{N}}$. Then we define $d_i(P) \coloneqq \# \{ i < j \mid (i,j) \; \text{is an inversion} \}$.
\end{definition}

Once again, the dinv of the polyomino is the total number of inversions.

\begin{remark}
	These three statistics on reduced polyominoes have the same distribution, up to a normalization factor $m+n-1$, of the corresponding ones on parallelogram polyominoes. We already proved it for area and bounce, and we will show in Section~\ref{sec:reduced_polyo} that the same holds for the dinv.
\end{remark}

\subsection{Decorated reduced polyominoes}

The definitions in this section are all new.

\medskip

The same normalizations obviously apply to decorated objects. The definitions for reduced polyominoes resemble the corresponding ones for Dyck paths a bit more than the ones for parallelogram polyominoes do.

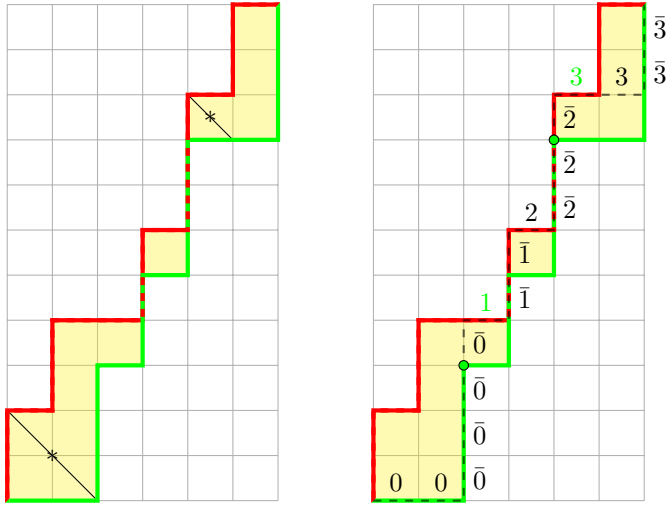
\begin{figure}[!ht]
	\centering
	\begin{minipage}{.40\textwidth}
		\centering
		\begin{tikzpicture}[scale=0.6]
		\draw[gray!60, thin] (0,0) grid (6,11);
		
		\filldraw[yellow, opacity=0.3] (0,0) -- (1,0) -- (2,0) -- (2,1) -- (2,2) -- (2,3) -- (3,3) -- (3,4) -- (3,5) -- (4,5) -- (4,6) -- (4,7) -- (4,8) -- (5,8) -- (6,8) -- (6,9) -- (6,10) -- (6,11) -- (5,11) -- (5,10) -- (5,9) -- (4,9) -- (4,8) -- (4,7) -- (4,6) -- (3,6) -- (3,5) -- (3,4) -- (2,4) -- (1,4) -- (1,3) -- (1,2) -- (0,2) -- (0,1) -- (0,0);
		
		\draw
		(2,0) -- (0,2)
		(5,8) -- (4,9);
		
		\draw (1,1) node {$\ast$};
		\draw (4.5,8.5) node {$\ast$};
		
		\draw[red, sharp <-sharp >, sharp angle = -45, line width=1.6pt] (0,0) -- (0,1) -- (0,2) -- (1,2) -- (1,3) -- (1,4) -- (2,4) -- (3,4) -- (3,5) -- (3,6) -- (4,6) -- (4,7) -- (4,8) -- (4,9) -- (5,9) -- (5,10) -- (5,11) -- (6,11);
		
		\draw[green, sharp <-sharp >, sharp angle = 45, line width=1.6pt] (0,0) -- (1,0) -- (2,0) -- (2,1) -- (2,2) -- (2,3) -- (3,3) -- (3,4) -- (3,5) -- (4,5) -- (4,6) -- (4,7) -- (4,8) -- (5,8) -- (6,8) -- (6,9) -- (6,10) -- (6,11);
		
		\draw[red, sharp <-sharp >, sharp angle = -45, dashed, line width=1.6pt] (0,0) -- (0,1) -- (0,2) -- (1,2) -- (1,3) -- (1,4) -- (2,4) -- (3,4) -- (3,5) -- (3,6) -- (4,6) -- (4,7) -- (4,8) -- (4,9) -- (5,9) -- (5,10) -- (5,11) -- (6,11);
		\end{tikzpicture}
	\end{minipage}%
	\begin{minipage}{.40\textwidth}
		\centering
		\begin{tikzpicture}[scale=0.6]
		\draw[gray!60, thin] (0,0) grid (6,11);
		
		\filldraw[yellow, opacity=0.3] (0,0) -- (1,0) -- (2,0) -- (2,1) -- (2,2) -- (2,3) -- (3,3) -- (3,4) -- (3,5) -- (4,5) -- (4,6) -- (4,7) -- (4,8) -- (5,8) -- (6,8) -- (6,9) -- (6,10) -- (6,11) -- (5,11) -- (5,10) -- (5,9) -- (4,9) -- (4,8) -- (4,7) -- (4,6) -- (3,6) -- (3,5) -- (3,4) -- (2,4) -- (1,4) -- (1,3) -- (1,2) -- (0,2) -- (0,1) -- (0,0);
		
		\draw[red, sharp <-sharp >, sharp angle = -45, line width=1.6pt] (0,0) -- (0,1) -- (0,2) -- (1,2) -- (1,3) -- (1,4) -- (2,4) -- (3,4) -- (3,5) -- (3,6) -- (4,6) -- (4,7) -- (4,8) -- (4,9) -- (5,9) -- (5,10) -- (5,11) -- (6,11);
		
		\draw[green, sharp <-sharp >, sharp angle = 45, line width=1.6pt] (0,0) -- (1,0) -- (2,0) -- (2,1) -- (2,2) -- (2,3) -- (3,3) -- (3,4) -- (3,5) -- (4,5) -- (4,6) -- (4,7) -- (4,8) -- (5,8) -- (6,8) -- (6,9) -- (6,10) -- (6,11);
		
		\draw[red, dashed, sharp <-sharp >, sharp angle = -45, line width=1.6pt] (0,0) -- (0,1) -- (0,2) -- (1,2) -- (1,3) -- (1,4) -- (2,4) -- (3,4) -- (3,5) -- (3,6) -- (4,6) -- (4,7) -- (4,8) -- (4,9) -- (5,9) -- (5,10) -- (5,11) -- (6,11);
		
		\draw[dashed, opacity=0.6, thick] (0,0) -- (1,0) -- (2,0) -- (2,1) -- (2,2) -- (2,3) -- (2,4) -- (3,4) -- (3,5) -- (3,6) -- (4,6) -- (4,7) -- (4,8) -- (4,9) -- (5,9) -- (6,9) -- (6,10) -- (6,11);
		
		\filldraw[fill=green]
		(2,3) circle (3pt)
		(4,8) circle (3pt);
		
		\node[above] at (0.5,0.0) {$0$};
		\node[above] at (1.5,0.0) {$0$};
		\node[right] at (2.0,0.5) {$\bar{0}$};
		\node[right] at (2.0,1.5) {$\bar{0}$};
		\node[right] at (2.0,2.5) {$\bar{0}$};
		\node[right] at (2.0,3.5) {$\bar{0}$};
		\node[above, green] at (2.5,4.0) {$1$};
		\node[right] at (3.0,4.5) {$\bar{1}$};
		\node[right] at (3.0,5.5) {$\bar{1}$};
		\node[above] at (3.5,6.0) {$2$};
		\node[right] at (4.0,6.5) {$\bar{2}$};
		\node[right] at (4.0,7.5) {$\bar{2}$};
		\node[right] at (4.0,8.5) {$\bar{2}$};
		\node[above, green] at (4.5,9.0) {$3$};
		\node[above] at (5.5,9.0) {$3$};
		\node[right] at (6.0,9.5) {$\bar{3}$};
		\node[right] at (6.0,10.5) {$\bar{3}$};
		\end{tikzpicture}
	\end{minipage}
	
	\caption{A reduced polyomino with two decorated rises (left) and area word $0 \bar{0} 1 \bar{1} {\color{green} 2} \bar{1} \bar{1} 1 \bar{0} \bar{0} 1 \bar{0} \bar{0} \bar{0} {\color{green} 1} 1 \bar{1} \bar{1}$, and one with two decorated green peaks (right).}
	\label{fig:decorated-reduced-polyominoes}
\end{figure}

\begin{definition}
	The \emph{rises} of a reduced polyomino $P$ are the indices \[ \Rise(P) \coloneqq \{2 \leq i \leq m+n \mid \vert a_{i}(P) \vert > \vert a_{i-1}(P) \vert \},\] which once again correspond in the figure to diagonals of slope $-1$ from the endpoint of a vertical red step to the endpoint of a horizontal green step.
\end{definition}

\begin{definition}
	The \emph{green peaks} of a reduced polyomino $P$ are the horizontal green steps that immediately follow a vertical green step.
\end{definition}

\begin{definition}
	A \emph{decorated reduced polyomino} is a reduced polyomino where certain green peaks are decorated with a symbol ${\color{green} \bullet}$ \emph{or} certain rises are decorated with a symbol $\ast$.
\end{definition}

Once again, we do not decorate the two sets simultaneously. By $\RP(m,n)^{\circ k}$ we denote the set of reduced polyominoes of size $m \times n$ with $k$ decorated green peaks, and by $\RP(m,n)^{\ast k}$ we denote the set of reduced polyominoes of size $m \times n$ with $k$ decorated rises. This time we allow decorations on any rise or green peak, since all of them give a non-zero contribution to the corresponding statistic.

The statistics area and bounce are defined in the same way we did for decorated parallelogram polyominoes, i.e. we ignore the contribution of decorated rises in the area word while computing the area, and we ignore the labels of the bounce path that lie in the same column as a decorated green peak while computing the bounce.

\subsection{Two car parking functions}

The definitions in this section first appeared in \cite{HHLRU-2005}.

\medskip

We need one more combinatorial object, namely two car parking functions. 

\begin{definition}
	A \emph{two car parking function} is a labelled Dyck path such that all the labels have value $1$ or $2$.
\end{definition}

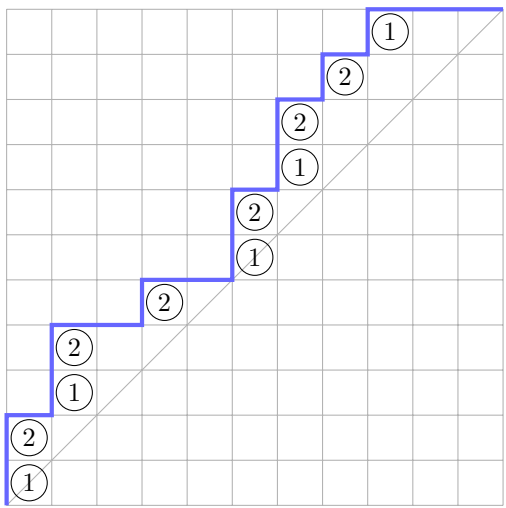
\begin{figure}[!ht]
	\begin{center}
		\begin{tikzpicture}[scale = 0.6]
		\draw[gray!60, thin] (0,0) grid (11,11);
		\draw[gray!60, thin] (0,0) -- (11,11);
		
		\draw[blue!60, line width=1.6pt] (0,0) -- (0,1) -- (0,2) -- (1,2) -- (1,3) -- (1,4) -- (2,4) -- (3,4) -- (3,5) -- (4,5) -- (5,5) -- (5,6) -- (5,7) -- (6,7) -- (6,8) -- (6,9) -- (7,9) -- (7,10) -- (8,10) -- (8,11) -- (9,11) -- (10,11) -- (11,11);
		
		\draw
		(0.5,0.5) circle(0.4 cm) node {$1$}
		(0.5,1.5) circle(0.4 cm) node {$2$}
		(1.5,2.5) circle(0.4 cm) node {$1$}
		(1.5,3.5) circle(0.4 cm) node {$2$}
		(3.5,4.5) circle(0.4 cm) node {$2$}
		(5.5,5.5) circle(0.4 cm) node {$1$}
		(5.5,6.5) circle(0.4 cm) node {$2$}
		(6.5,7.5) circle(0.4 cm) node {$1$}
		(6.5,8.5) circle(0.4 cm) node {$2$}
		(7.5,9.5) circle(0.4 cm) node {$2$}
		(8.5,10.5) circle(0.4 cm) node {$1$};
		\end{tikzpicture}
	\end{center}
	
	\caption{A two car parking function with five $1$'s and six $2$'s.}
	\label{fig:2cpf}
\end{figure}

Observe that replacing the $1$'s and the $2$'s with two decreasing sequences $n, \dots, 1$ and $m+n, \dots, n+1$ in the dinv (resp. pmaj) reading word, the dinv (resp. pmaj) is not altered. In other words, two car parking functions with $n$ labels equal to $1$ and $m$ labels equal to $2$ are essentially the same as parking functions whose dinv (resp. pmaj) reading word is in $(n, \dots, 1) \shuffle (m+n, \dots, n+1)$, since the relabelling is a bijection that preserves the bistatistic $(\dinv, \area)$ (resp. $(\area, \pmaj)$). To avoid confusion, we will refer to labelled Dyck paths with labels in $\{1,2\}$ as two car parking functions, and to labelled Dyck paths whose reading word is a shuffle of two decreasing sequences as \emph{two-shuffle parking functions}. For more implications of this fact and the definition of what it means to be a shuffle, we refer to Section \ref{sec:link_delta}.

\newpage
\subsection{Square paths} \label{sec:sq_path_defs}

In this section we introduce new statistics area, dinv and bounce on square paths, that are similar to (but different from) the ones introduced in \cite{Loehr-Warrington-square-2007}.
\begin{definition}
	A \emph{square path ending east} of size $n$ is a lattice paths going from $(0,0)$ to $(n,n)$ consisting of east or north unit steps, and ending with an east step. We denote the set of such paths by $\mathsf{SQ^E}(n)$. Similarly, we define the set $\mathsf{SQ^N}(n)$ of square paths ending north.
\end{definition}

\begin{definition}
	Given $P\in \mathsf{SQ^E}(n)$ or $P\in \mathsf{SQ^N}(n)$ , we define its \emph{area word} $w_1(P)\cdots w_n(P)$ to be the sequence of integers such that the $i$-th vertical step of $P$ lies on the diagonal $y=x+w_i(P)$. It is an area word in the alphabet $\mathbb Z$. 
\end{definition}

\begin{definition}
	A \emph{parking square path ending east} is a square path ending east $P$ whose vertical steps are labelled from bottom to top by $l_1(P), \dots, l_n(P)$ such that $l_1(P) \cdots l_n(P) \in  \mathfrak{S}_n$ and the labels in the same column are strictly increasing from bottom to top. We denote by $\mathsf{PSQ^E}(n)$ the set of such paths. The set $\PSQN(n)$ of \emph{parking square paths ending north} is defined analogously.  
\end{definition}
See Figure \ref{SQpath} for an example.

\begin{figure}[!ht]
	\begin{center}
		\begin{tikzpicture}[scale=1.2]
		\draw[gray!60,thin] (0,0) grid[step=0.5 cm](5.5,5.5);
		\draw[gray!60, thin] (2,0) to (7.5,5.5);
		\filldraw (3.5,1.5) circle(2pt);
		\draw[blue!60, line width=1.6pt] (3.5,1.5)|-(4,2.5)|-(5,4.5) |-(5.5,5.5) (0,0)|-(1,1)|-(3.5,1.5);
		\draw (.25, 0.25) node {4} circle(.2 cm)
		(0.25, 0.75) node {6} circle(.2 cm)
		(1.25, 1.25) node {11} circle(.2 cm)
		(3.75, 1.75) node {7} circle(.2 cm)
		(3.75, 2.25) node {10} circle(.2 cm)
		(4.25, 2.75) node {1} circle(.2 cm)
		(4.25, 3.25) node {2} circle(.2 cm)
		(4.25, 3.75) node {5} circle(.2 cm)
		(4.25, 4.25) node {8} circle(.2 cm)
		(5.25, 4.75) node {3} circle(.2 cm)
		(5.25, 5.25) node {9} circle(.2 cm);
		\fill[gray,opacity=.2](0,0)--(0,1)--(1,1)--(1,1.5)--(3,1.5)--(3,1)--(2.5,1)--(2.5,.5)--(2,.5)--(2,0) (3.5,2) rectangle (4,2.5) (4,2.5) rectangle (4.5, 4) (4.5,3) rectangle (5,4)(5,3.5) rectangle (5.5,5.5) (5,4.5)rectangle (4,4) (5.5,4.5) rectangle (6.5,5.5) rectangle (7,5) (5.5,4.5) rectangle (6,4);
		\draw [decorate,decoration={brace, mirror,amplitude=8pt},xshift=0pt,yshift=0pt](0,0) --(2,0);
		\draw [decorate,decoration={brace, mirror,amplitude=8pt},xshift=0pt,yshift=0pt](5.5,0)--(5.5,1.5);
		\draw (1,-.5) node{$c$} (6,.75) node {$j$};
		\draw [gray!60, thin](5.5,4.5) grid[step=.5](6.5,5.5) rectangle (7,5) (5.5,4.5) rectangle (6,4) (6,4.5) rectangle (6.5,5) ;
		\end{tikzpicture}
	\end{center}
	\caption{Example of a labelled square path}\label{SQpath}
\end{figure}
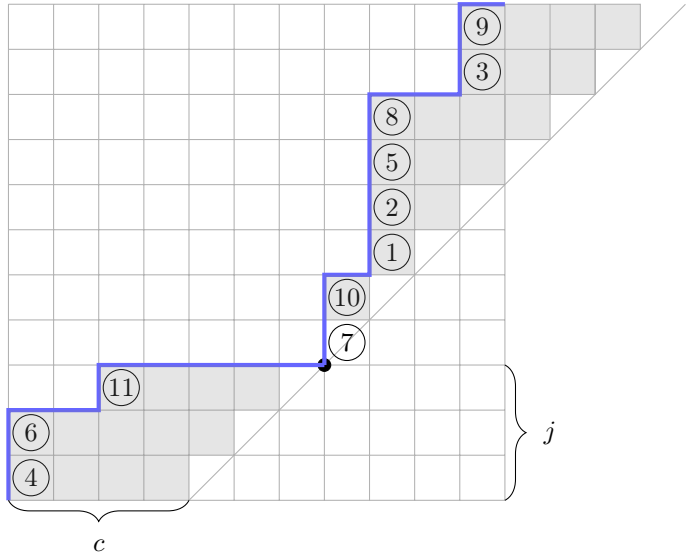

We define an area statistic on each of these sets. As usual, tha area does not depend on the labelling. However the definition for paths ending east and ending north is slightly different.
\begin{definition}\label{def: area sqe}
	Let $P\in \SQE(n)\sqcup \PSQE(n)$. Set 
	\begin{align*}
		& c \coloneqq -\min \{w_i(P)\mid i=1,2,\dots n\} \\
		& j \coloneqq \min\{i\mid w_i(P)=-c\}-1. 
	\end{align*}
	
	We define \[ \mathsf{area}(P) \coloneqq \sum_{i=1}^n w_i(P)+nc-j-c. \]
\end{definition}

For example, the path in Figure \ref{SQpath} has $c=4$, $j=3$ and area $31-3-4=24$. In this picture, we coloured grey the whole squares between the path and the line $y=x-c$: these squares are counted by $\sum_{i=1}^n w_i(P)+nc$, so that the area is equal to the number of grey squares minus $j+c$.

\begin{definition}
	For $P\in \SQN(n)\sqcup \PSQN(n) $ and $c$ and $j$ defined in the same way as in Definition \ref{def: area sqe}, we set \[\mathsf{area}(P) \coloneqq \sum_{i=1}^n w_i(P)+nc-j\]
\end{definition}

\begin{definition}
	For $P\in \PSQE(n) $ we define its \emph{diagonal inversions}, \emph{dinv} and \emph{dinv reading word} in a way analogous to how we did it for labelled Dyck paths (see Definitions \ref{def: inversion LD}, \ref{def: dinv LD} and \ref{def: dinv reading word LD}). 
\end{definition}

Indeed these definitions still make sense, even if the letters of the area word of a square path ending east might be negative. 

Next we define the dinv of unlabelled square paths.

\begin{definition}
	Let $P\in \mathsf{SQ^E}(n)\sqcup \mathsf{SQ^N}(n)$, we set
	\begin{align*}
		\dinv(P)& \coloneqq \sum_{1\leq i<j\leq n}\chi(w_i(P)=w_j(P)) \\
		& + \sum_{1\leq i<j< n}\chi(w_i(P) = w_j(P)+1) \\
		&  + \sum_{1\leq i< n}\chi(w_i(P) = w_n(P)+1 \; \text{and} \; w_n(P) \geq  0),
	\end{align*} 
	where $\chi$ is the indicator function which is defined as $\chi(\mathcal{P})=1$ if $\mathcal{P}$ is true and $\chi(\mathcal{P})=0$ otherwise.

	Notice that the last term means that the last step of a path ending north does not create secondary dinv.
\end{definition}

Finally, we define the bounce of unlabelled objects. 

\begin{definition} Let $P\in \mathsf{SQ^E}(n)\sqcup \mathsf{SQ^N}(n)$. The \emph{bounce path} of $P$ starts at the point $(c+j, j)$ and travels north until it hits the beginning of an east step when it turns east until it hits the diagonal $y=x-c$, where it turns north again, and so on. When it crosses the line $y=n$, at the point $(n+m,n)$ it stops and starts again at the point $(m,0)$. Here there is a slight difference between the definitions for paths ending east and paths ending north:
	\begin{itemize}
		\item if the path ends east, then if $m=0$ the bounce path starts traveling east at the point $(0,0)$, otherwise, it travels north; 
		\item if the path ends north, then the bounce path always travels north starting from $(m,0)$.  
	\end{itemize} 
	Then the path bounces on the \emph{end} of east steps and again on the diagonal $y=x-c$. It ends when it arrives at the point $(c+j,j)$. We then label the vertical steps of the bounce starting with $0$'s (recall that the bounce path starts at $(c+j,j)$) and adding $1$ every time the bounce path bounces. Finally, $\bounce(P)$ is defined to be the sum of the labels of the bounce path. 
\end{definition}

For an example, see Figure~\ref{squarebounceex}, whose bounce equals $15$. 

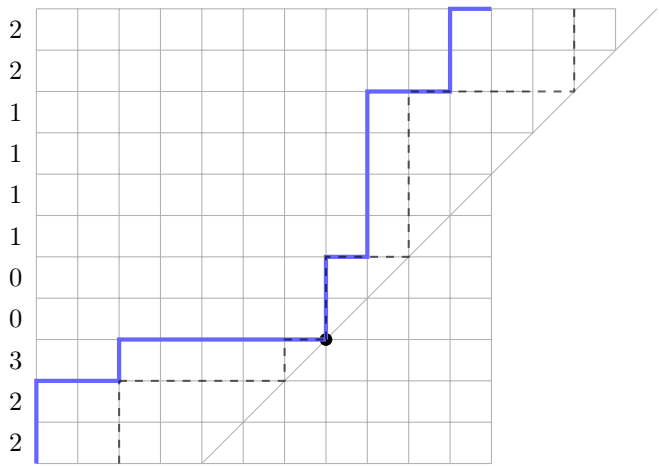
\begin{figure}[!ht]
	\begin{center}
		\begin{tikzpicture}[scale=1.1]
		\draw[gray!60,thin] (0,0) grid[step=0.5 cm](5.5,5.5) (2,0) to (7.5,5.5); 
		\filldraw (3.5,1.5) circle(2pt);
		\draw[blue!60, line width=1.6pt](3.5,1.5)|-(4,2.5)|-(5,4.5) |-(5.5,5.5) (0,0)|-(1,1)|-(3.5,1.5);
		\draw[gray!60, thin] (5.5,4.5) grid[step=.5](6.5,5.5) rectangle (7,5) (5.5,4.5) rectangle (6,4) ;
		\draw[thick, dashed, opacity=0.6] (3.5,1.5)|-(4.5,2.5)|-(6.5,4.5) to (6.5,5.5)(1,0)--(1,1)(1,1)--(3,1)|-(3.5,1.5);
		\draw (-.25, 0.25) node {2}
		(-.25, 0.75) node {2}
		(-.25, 1.25) node {3}
		(-.25, 1.75) node {0}
		(-.25, 2.25) node {0}
		(-.25, 2.75) node {1}
		(-.25, 3.25) node {1}
		(-.25, 3.75) node {1}
		(-.25, 4.25) node {1}
		(-.25, 4.75) node {2}
		(-.25, 5.25) node {2};
		\end{tikzpicture}
	\end{center}
	\caption{Example of a path in $\mathsf{SQ^E}(n)$, its bounce path and its labels (left). } \label{squarebounceex}
\end{figure}

\chapter{Conjectures}

In this chapter we state two known conjectures, the Delta conjecture and a generalization of it, both appearing for the first time in \cite{Haglund-Remmel-Wilson-2015}, and three new conjectures. We also give an overview of the state of the art of these problems. We refer to Section~\ref{section:combinatorial definitions} for the definitions of the combinatorial objects, and to Section~\ref{sec:symmetric_fcts_intro} for the definitions of symmetric function theory.

\section{The Delta conjecture} \label{sec:Delta_Conj}

Most of the results in this paper relate to the Delta conjecture, first stated in \cite{Haglund-Remmel-Wilson-2015}, which gives two combinatorial interpretations of the symmetric functions of the form $\Delta_{e_k}e_n$ where $1\leq k\leq n$. It is therefore pertinent to give the full statement of this conjecture.

%

The Delta conjecture is in fact a pair of conjectures, generally referred to as the \emph{rise version} and the \emph{valley version} of the Delta conjecture. There is, as of yet, no proof that these two versions of the Delta conjecture are equivalent. 

Their first statement in \cite{Haglund-Remmel-Wilson-2015} is given in terms of labelled Dyck paths.  
%
%

\begin{conjecture}[Delta conjecture]\label{conj: 1st Delta} For any integers $n>k\geq 0$, 
	\begin{align}
		\Delta'_{e_{n-k-1}}e_n&=\left.\sum_{D\in \LD(n) }q^{\dinv(D)}t^{\area(D)} \prod_{i\in \Rise(D)}\left(1+\frac{z}{t^{a_i(D)}} \right) x^D \right|_{z^{k}}
		\label{eq: delta conj LD rise} \\
		&=\left.\sum_{D\in \LD(n) }q^{\dinv(D)}t^{\area(D)} \prod_{i\in \Val(D)}\left(1+\frac{z}{q^{d_i(D)+1}} \right) x^D \right|_{z^{k}} \label{eq: delta conj LD valley} 
	\end{align} 
\end{conjecture}
We should immediately notice that for $k=0$ the Delta conjecture reduces to the Shuffle conjecture (see \cite{HHLRU-2005}), recently proved by Carlsson and Mellit \cite{Carlsson-Mellit-ShuffleConj-2015} (see also \cite{Haglund-Xin_Lecture-Notes}).

Again in \cite{Haglund-Remmel-Wilson-2015}, the authors reformulate Conjecture~\ref{conj: 1st Delta} in terms of \emph{decorated} labelled Dyck paths. 
\begin{proposition}[\cite{Haglund-Remmel-Wilson-2015}*{Proposition~3.1}] Let $\text{Rise}_{n,k}(x;q,t)$ and $\text{Val}_{n,k}(x;q,t)$ be equal to the right hand sides of equations \eqref{eq: delta conj LD rise} and \eqref{eq: delta conj LD valley} respectively. Then   
	\begin{align*}
		\text{Rise}_{n,k}(x;q,t)&= \sum_{D\in \LDD^{\ast k}(n)}q^{\dinv(D)}t^{\area(D)}x^D  \\ 
		\text{Val}_{n,k}(x;q,t)&= \sum_{D\in \LDD^{\vee k}(n)}q^{\dinv(D)}t^{\area(D)}x^D
	\end{align*}
\end{proposition}
\begin{remark}
	Notice that we use a slightly different notation for $\text{Rise}_{n,k}(x;q,t)$ and $\text{Val}_{n,k}(x;q,t)$ with respect to \cite{Haglund-Remmel-Wilson-2015}.
\end{remark}

\subsection{State of the art} \label{Delta_State_of_the_art}

While not much is known about the valley version of the Delta conjecture, several special cases of the rise version have been proved, so in this work \emph{we will focus only on these}. We summarise the most general progresses in the following table. 
\begin{center}
	\begin{tabular}{c|c}
		Conditions & Reference \\
		\hline
		$k=n-2$ & \cite{Haglund-Remmel-Wilson-2015} \\
		$k=0$ & \cite{Carlsson-Mellit-ShuffleConj-2015} \\
		$q=1$ & \cite{Romero-Deltaq1-2017} \\
		$q=0$ or $t=0$ & \cite{Garsia-Haglund-Remmel-Yoo-2017} \\
		$\< \cdot, e_{n-d}h_d\> $ & Section~\ref{subsec:qtSchroeder}  \\
		$\< \cdot, h_{n-d}h_d\> $ & Section~\ref{subsec:two-shuffle}
	\end{tabular}
\end{center}
Moreover, combining our results in Section~\ref{subsec:qtSchroeder} with the results in \cite{Zabrocki-4Catalan-2016}, we get a ``compositional'' refinement of the Delta conjecture in the $\< \cdot, e_{n-d}h_d\>$ case.

For more results related to the Delta conjecture, see also \cites{Remmel-Wilson-2015,Wilson-article-2017,Qiu-Remmel-Sergel-Xin-2017,Rhoades-2018,Haglund-Rhoades-Shimozono-arxiv,Haglund-Rhoades-Shimozono-Advances,Wilson-PhD-2015,Haglund-Xin_Lecture-Notes}.

\section{The generalised Delta conjecture} \label{sec:gen_Delta}

In \cite{Haglund-Remmel-Wilson-2015}, the authors state a conjecture for a combinatorial interpretation of $\Delta_{h_{m}} \Delta'_{e_{n-k-1}} e_{n}$ in terms of partially labelled Dyck paths.


%

\begin{conjecture}[\cite{Haglund-Remmel-Wilson-2015}*{Conjecture~7.4}]
	\[ \Delta_{h_{m}} \Delta'_{e_{n-k-1}} e_{n} = \sum_{D \in \PLD(m,n)^{\ast k}} q^{\dinv(D)} t^{\area(D)} x^D \]
\end{conjecture}

We notice immediately that for $m=0$ we recover precisely the Delta conjecture, that is why we call this the \emph{generalised Delta conjecture}.

\subsection{State of the art}

We already observed that for $m=0$ we recover the Delta conjecture, for which we already discussed the state of the art in Section~\ref{Delta_State_of_the_art}. For $m\geq 1$, we have the following results:
\begin{center}
	\begin{tabular}{c|c}
		Conditions & Reference \\
		\hline
		$m\geq 1$ and $\< \cdot, e_{n-d}h_d\> $ & \cite{DAdderio-Iraci-VandenWyngaerd-2018} \\
		$m\geq 1$, and $q=0$ or $t=0$ & \cite{DAdderio-Iraci-VandenWyngaerd-Deltat0}
	\end{tabular}
\end{center}

\section{Our conjecture with $\pmaj$}

We guess a similar formula, that uses the $\pmaj$.

\begin{conjecture} \label{conj:our_pmaj}
	\[ \Delta_{h_{m}} \Delta'_{e_{n-k-1}} e_{n} = \sum_{D \in \PLD(m,n)^{\ast k}} q^{\area(D)} t^{\pmaj(D)} x^D \]
\end{conjecture}

\subsection{State of the art}

Other than computer verification, in support of our conjecture, we have the following results:
\begin{center}
	\begin{tabular}{c|c}
		Conditions & Reference \\
		\hline
		$m=0$ and $\< \cdot, e_{n-d}h_d\> $ & Section~\ref{subsec:qtSchroeder}\\
		$k=n-1$ and $\< \cdot, e_{n}\>$ & Section~\ref{subsec:two-shuffle}
	\end{tabular}
\end{center}

In the last case, in fact, we have some statistics preserving bijections that allow us to rewrite the conjecture in terms of polyominoes. 

\section{Our conjecture with polyominoes}\label{conj:our_polyo}

We guess a similar conjecture, but with polyominoes.

\begin{conjecture}
	For $m,n \in \mathbb{N}$, it holds
	\[ (qt)^{m+n+1}\Delta_{h_{m}} e_{n+1} = \sum_{P \in \LPP(m+1,n+1)} q^{\area(P)} t^{\pmaj(P)} x^P. \]
\end{conjecture}

\subsection{State of the art}

Other than computer verification, in support of our conjecture, we have the following result:
\begin{center}
	\begin{tabular}{c|c}
		Conditions & Reference \\
		\hline
		$\< \cdot, e_{n+1}\> $ & \cite{Aval-DAdderio-Dukes-Hicks-LeBorgne-2014} (see also Remark~\ref{rmk:pmaj_cases})
	\end{tabular}
\end{center}

\section{Our square conjecture} \label{sec:our_square_conj}

In this section we explain how a simple combinatorial transformation translates the Delta conjecture for $\Delta_{e_{n-1}}e_n$ into a new square conjecture, similar to the one for $\nabla \omega(p_n)$ proposed in \cite{Loehr-Warrington-square-2007} and proved in \cite{Leven-2016} after the breakthrough in \cite{Carlsson-Mellit-ShuffleConj-2015}. In particular, we get a new $q,t$-square theorem out of it (see Section \ref{subsec: qt square result}).

The interest in these observations relies into a more surprising one: the two symmetric functions $\Delta_{e_{n-1}}e_n$ and $\nabla \omega(p_n)$ have the same evaluation at $t=1/q$. See Section~\ref{sec:observ_one_over_q} for a proof of this statement.

For $P \in \PSQE(n)$, set $\delta(P)$ to be its dinv reading word. 
\begin{conjecture} \label{conj:new_square}
	For every $n\geq 1$
	\begin{equation}
		\Delta_{e_{n-1}}e_n=\sum_{P\in \mathsf{PSQ^E}(n)}q^{\mathsf{dinv}(P)}t^{\mathsf{area}(P)}Q_{\mathsf{Des}(\delta(P)),n}.
	\end{equation}
\end{conjecture}

A similar conjecture can be formulated using the square paths ending north.

\subsection{State of the art} \label{sec:square_paths_map}

In support of our conjecture, we provide here its equivalence with the Delta conjecture for $\Delta_{e_{n-1}}e_n$.

Let $\LDD^{\ast k}(n)$ be the set of labelled Dyck paths with $k$ decorations on falls as in \cite{Haglund-Remmel-Wilson-2015} .

We describe a bijective map 
\begin{equation}  \label{gammamap}
	\gamma_E: \mathsf{PSQ^E}(n) \rightarrow \LDD^{\ast 0}(n) \sqcup \LDD^{\ast 1}(n) .
\end{equation}
We have $\LDD^{\ast 0}(n)\subseteq \PSQE(n)$ and so we define $\gamma_{E_{\big{\vert} \LDD^{\ast 0}(n) }}$ to be the identity map. Now take $P\in \PSQE(n)$ such that $P$ is not a Dyck path. Create a labelled Dyck path $D$ from $P$ as follows: start from $(0,0)$ by copying $P$ and its labels starting from the point $(j+c,j)$. When we arrive at the end of $P$ we add an extra horizontal step (thus creating a fall since $P$ ended with an east step). Now continue $D$ by copying $P$ and its labels starting from $(0,0)$. The path $D$ ends when we reach the point $(c+j-1, j)$ in $P$. Finally decorate the fall that was created by adding the extra horizontal step. It is easy to see how to invert $\gamma_E$. We refer to Figure \ref{gamma} for an example.  
\begin{figure}[!ht]
	\centering
	\begin{minipage}{.6\textwidth}
		\centering
		\begin{tikzpicture}[scale=.6]
		\draw[gray!60, thin] (0,0) grid (8,8) grid (11,5);
		\fill[white](8,5)--(11,5)--(11,8)--(8,5);
		\draw[white,ultra thick] (8,5)--(11,5)--(11,8);
		\draw[gray!60, thin](3,0)--(11,8) (8,8)--(8,0);
		\draw[blue!60, line width=1.6pt](0,0)|-(3,1)|-(4,2) (5,2)|-(6,5)|-(7,7)|-(8,8);
		\draw[red!90, line width=1.6pt](4,2)--(5,2) ;
		\filldraw (5,2) circle (3pt);
		\draw (.5,.5) node{2} circle (.4cm)
		(3.5,1.5) node{7}circle (.4cm)
		(5.5,2.5) node{1}circle (.4cm)
		(5.5,3.5) node{6}circle (.4cm)
		(5.5,4.5) node{8}circle (.4cm)
		(6.5,5.5) node{4}circle (.4cm)
		(6.5,6.5) node{5}circle (.4cm)
		(7.5,7.5) node{3} circle (.4cm);
		\end{tikzpicture}
	\end{minipage}%
	\begin{minipage}{.4\textwidth}
		\centering
		\begin{tikzpicture}[scale=.6]
		\draw[gray!60, thin] (0,0) grid (8,8) (0,0)--(8,8);
		\draw[ultra thick, blue!60](0,0)|-(1,3)|-(2,5)|-(4,6)|-(7,7)|-(8,8) ;
		\draw[ultra thick, sharp <-sharp >, sharp > angle = 45, red!90] (3,6)--(4,6);
		\filldraw (0,0) circle (3pt);
		\draw
		(0.5,0.5) node {1} circle (0.4cm)
		(0.5,1.5) node {6} circle (0.4cm)
		(0.5,2.5) node {8} circle (0.4cm)
		(1.5,3.5) node {4} circle (0.4cm)
		(1.5,4.5) node {5} circle (0.4cm)
		(2.5,5.5) node {3} circle (0.4cm)
		(4.5,6.5) node {2} circle (0.4cm)
		(7.5,7.5) node {7} circle (0.4cm)
		(2.5,6.5) node {$\ast$};
		\end{tikzpicture}
	\end{minipage} \caption{The map $\gamma_E$. }\label{gamma}
\end{figure}
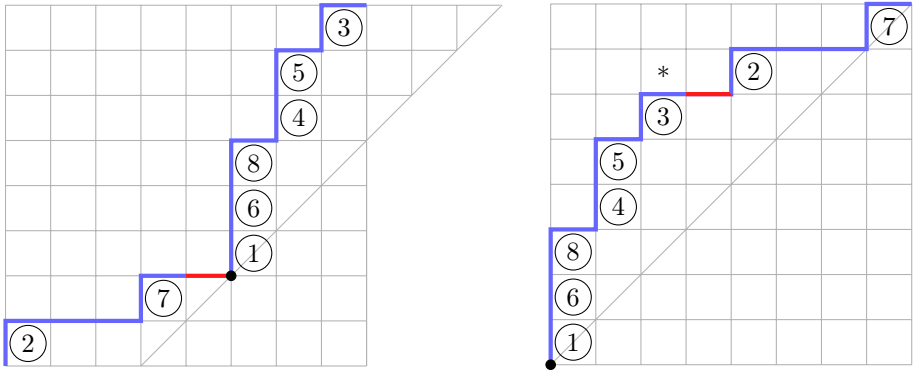

We get that for all $P \in \PSQE(n)$ \begin{align*} \area(\gamma_E(P))&=\area(P)\\ \dinv(\gamma_E(P))&=\dinv(P).\end{align*}

\begin{remark}
	Observe that a combinatorial formula for $\Delta_{e_{n-1}}e_n$ in terms of labelled Dyck paths already appears in \cite[Theorem~7.10]{Bergeron-Garsia-Segel-Xin-Rational} (the formula is now proved thanks to the results in \cite{Carlsson-Mellit-ShuffleConj-2015}):
	\[\Delta_{e_{n-1}}e_n = \sum_{D \in \LD(n)}[\mathsf{ret}(D)]_{1/t}t^{\area(D)}q^{\dinv(D)}x^D  \]
	where
	\[\mathsf{ret}(D):=\min(\{i-1:a_i(D)=0,i>1\}\cup\{n\}).\]
	
	To these authors, it is unclear how to relate this formula to Conjecture~\ref{conj:new_square}.
\end{remark}

\chapter{Our results} 

In this section we give an account of most of our results. Also for this chapter, we refer to Section~\ref{section:combinatorial definitions} for the definitions of the combinatorial objects, and to Section~\ref{sec:symmetric_fcts_intro} for the definitions of symmetric function theory. Detailed proofs of all the statements can be found in Part \ref{part: proofs}.

\section{A decorated $q,t$-Schröder}
\label{subsec:qtSchroeder}

In \cite{Haglund-Remmel-Wilson-2015} (half of Problem 8.1) and \cite{Wilson-PhD-2015}*{Conjecture~5.2.1.1}, the authors asked for a proof of the decorated version of the famous $q,t$-Schr\"{o}der theorem of Haglund \cite{Haglund-Schroeder-2004}. This is a combinatorial interpretation of the formula
\begin{align}
	\label{eq:qtSchroederLHS}
	\< \Delta_{e_{a+b-k-1}}'e_{a+b}, e_a h_{b} \>
\end{align}
as a sum of weights $q^{\dinv(D)}t^{\area(D)}$ of decorated Dyck paths. We solve this problem, by providing two extra combinatorial interpretations of that same polynomial.

\begin{theorem} \label{thm: resultsection: decoqtSchroeder}
	For $a,b,k\in \mathbb{N}$, $a\geq 1$, $a+b\geq k+1$, we have
	\begin{align}
		\label{eq:qtSchroder2.1} \< \Delta_{e_{a+b-k-1}}'e_{a+b},e_a h_{b} \> & = \sum_{D\in \DD(a+b)^{\circ b, \ast k} } q^{\dinv(D)}t^{\area(D)} \\
		\label{eq:qtSchroder2.2} & =\sum_{D\in \DD(a+b)^{\circ b, \ast k} } q^{\area(D)}t^{\pbounce(D)} \\
		\label{eq:qtSchroder2.3} & =\sum_{D\in \DD(a+b)^{\circ b, \ast k} } q^{\area(D)}t^{\bounce(D)}.
	\end{align}
\end{theorem}
See Section~\ref{sec:deco_qt_Schroeder} for a proof of this theorem.
\begin{remark}
	Observe that Remark~\ref{rmk:bounce_pmaj} together with the discussion in Section~\ref{sec:link_delta} shows that equations \eqref{eq:qtSchroder2.1} and \eqref{eq:qtSchroder2.2} prove the case $m=0$ and $\<\cdot, e_{n-d}h_d\>$ of our Conjecture~\ref{conj:our_pmaj}.
\end{remark}

Notice that the equality \eqref{eq:qtSchroder2.1} is indeed the case $\< \cdot , e_a h_b \>$ of the Delta conjecture, once we ``standardize'' suitably our decorated Dyck paths (see Section~\ref{sec:link_delta} for more details). 

In the previous theorem, to make the connection with Schr\"{o}der paths explicit, the decorated peaks can be thought of as diagonal steps.

The proof of the equations  \eqref{eq:qtSchroder2.1} and \eqref{eq:qtSchroder2.2} relies on showing that that both sides satisfy the same recursion and initial conditions. We state here the combinatorial recursion for the $q,t$-enumerator of $(\dinv,\area)$:

Let $a,b,n,k\in \mathbb{N}$, with $1 \leq k \leq n$. Let $\DDd(n \backslash k )^{\circ a, \ast b}$ be the subset of $\DDd(n)^{\circ a, \ast b}$ whose elements have an area word containing exactly $k$ zeros, and set
\begin{align*}
	\DDd_{q,t}(n)^{\circ a, \ast b} & :=\sum_{D \in \DDd(n)^{\circ a, \ast b}}q^{\dinv(D)}t^{\area(D)} \\
	\DDd_{q,t}(n\backslash k)^{\circ a, \ast b} & :=\sum_{D \in \DDd(n\backslash k)^{\circ a, \ast b}}q^{\dinv(D)}t^{\area(D)}.
\end{align*}

Since $\DDd(n)^{\circ a, \ast b} = \bigsqcup_{k=1}^n \DDd(n \backslash k)^{\circ a, \ast b}$, we have \[ \DDd_{q,t}(n)^{\circ a, \ast b} = \sum_{k=1}^n \DDd_{q,t}(n \backslash k )^{\circ a, \ast b}, \]
and we already observed in Remark~\ref{rem: combinatorial pieri rule } that
\[\sum_{D\in \DD(n)^{\circ a, \ast b} } q^{\dinv(D)}t^{\area(D)}=\DDd_{q,t}(n)^{\circ a, \ast b}+\DDd_{q,t}(n)^{\circ a-1, \ast b}.\]

We have a recursive formula for $\DDd_{q,t}(n \backslash k )^{\circ a, \ast b}$  (see Section~\ref{sec:combinatorial_recursions_ddyck} for a proof).

\begin{theorem}
	Let $a,b,n,k \in \mathbb{N}$, with $1 \leq k \leq n$. Then \[ \DDd_{q,t}(n \backslash n )^{\circ b, \ast a} = \delta_{a,0} q^{\binom{n-b}{2}}\qbinom{n-1}{b}_q, \]	and for $1 \leq k \leq n-1$
	\begin{align*}
		\DDd_{q,t}(n \backslash k)^{\circ b, \ast a} & = t^{n-k-a} \sum_{s=0}^{b} \sum_{i=0}^{a} \sum_{h=1}^{n-k} q^{\binom{k-s}{2}+\binom{i+1}{2}} \qbinom{k}{s}_q \qbinom{h-1}{i}_q  \\
		\times \qbinom{h+k-s-i-1}{h}_q & \left( \DDd_{q,t}(n-k \backslash h)^{\circ b-s, \ast a-i} + \DDd_{q,t}(n-k \backslash h )^{\circ b-s, \ast a-i-1}  \right),
	\end{align*}
	with initial conditions \[\DDd_{q,t}(0 \backslash k)^{\circ b, \ast a} = \DDd_{q,t}(n \backslash 0)^{\circ b, \ast a} = 0. \]
\end{theorem}

%

To prove equation \eqref{eq:qtSchroder2.3} one more element is needed to switch the decorations on peak and rises. We achieve this with a new combinatorial bijection.

\begin{theorem} \label{thm: resultsection: psi_map}
	There exists a bijective map $$\psi: \DDp(n)^{\circ a, \ast b} \rightarrow  \DDp(n)^{\circ b, \ast a}$$ such that for all $D\in \DD(n)^{\circ a, \ast b}$ \begin{align*}
		\area(D) & = \area(\psi(D)) \\
		\pbounce(D) & = \pbounce(\psi(D)).
	\end{align*}
\end{theorem} 
See Section~\ref{sec:psi_map} for a proof of this theorem.

We also provide a purely combinatorial way of proving the equality of the right hand sides of equations \eqref{eq:qtSchroder2.1} and \eqref{eq:qtSchroder2.3} by proving that Halglund's zeta map (also known as \emph{sweep map}) switches the bistatistics and the decorations on peaks and rises. 

\begin{theorem}
	There exists a bijective map \[ \zeta: \DDd(n)^{\circ a, \ast b} \rightarrow \DDb(n)^{\circ b, \ast a} \] such that for all $D\in \DDd(n)^{\circ a, \ast b}$ \begin{align*}
		\area(D) & = \bounce(\zeta(D)) \\
		\dinv(D) & = \area(\zeta(D)).
	\end{align*}
\end{theorem}
See Section~\ref{sec:zeta_map} for a proof of this theorem.

\section{A decorated $q,t$-Narayana}
\label{subsec:two-shuffle}

In \cite{Aval-DAdderio-Dukes-Hicks-LeBorgne-2014} the authors give a combinatorial interpretation of the $q,t$-Narayana $\< \Delta_{h_m} e_{n+1}, e_{n+1} \>$ as a sum of weights $q^{\area(P)} t^{\bounce(P)}$ (equivalently $q^{\dinv(P)} t^{\area(P)}$) of parallelogram polyominoes. We provide a combinatorial interpretation of the more general $\< \Delta_{h_m} e_{n+1}, s_{k+1,1^{n-k}} \>$ in terms of decorated parallelogram polyominoes (using the equivalent reduced version).

\begin{theorem}
	For $m, n, k \in \mathbb{N}$, $0 \leq k \leq m$, we have
	
	\begin{align}
		\< \Delta_{h_m} e_{n+1}, s_{k+1,1^{n-k}} \> & = \sum_{P \in \RP(m,n)^{\ast k}} q^{\dinv(P)} t^{\area(P)} \\
		\label{eq:RPswitch}	& = \sum_{P \in \RP(m,n)^{\circ k}} q^{\area(P)} t^{\bounce(P)}
	\end{align}
\end{theorem}
See Section~\ref{sec:2shuffle_delta} for a proof of this theorem. 

\begin{remark} \label{rmk:pmaj_cases}
	Observe that the case $k=0$ of this results combined with Proposition~\ref{normalizing map} gives essentially the main result in \cite{Aval-DAdderio-Dukes-Hicks-LeBorgne-2014}. This same case, combined with Theorem~\ref{th:pmaj}, proves the case $k=n-1$ and $\<\cdot, e_{n}\>$ of our Conjecture~\ref{conj:our_pmaj}. Moreover, that same case, combined with Remark~\ref{rmk:polyo_pmaj_bounce} and a discussion of the shuffles for labelled parallelogram polyominoes similar to the one in Section~\ref{sec:link_delta}, proves also the case $\<\cdot, e_{n+1}\>$ of our Conjecture~\ref{conj:our_polyo}. 
\end{remark}

The proof of the previous theorem uses a recursive formula, that we are going to state here on the combinatorial side for the $q,t$-enumerator of $(\dinv,\area)$:

Let $\RP(m\backslash r,n)^{\ast k}$ be the subset of $\RP(m,n)^{\ast k}$ whose area word has exactly $r$ $0$'s, and set
\[\RP_{q,t}(m\backslash r,n)^{\ast k}:= \sum_{P \in \RP(m,n)^{\ast k}} q^{\dinv(P)} t^{\area(P)}.\]
Since $\RP(m,n)^{\ast k}=\bigsqcup_{k=1}^{m+1} \RP(m\backslash r,n)^{\ast k}$ we have $\RP_{q,t}(m,n)^{\ast k}=\sum_{k=1}^{m+1} \RP_{q,t}(m\backslash r,n)^{\ast k}$.

We have a recursive formula for $\RP_{q,t}(m\backslash r,n)^{\ast k}$ (see Section~\ref{sec:reduced_polyo} for a proof).
\begin{theorem}
	
	For $m \geq 0$, $n \geq 0$, $k \geq 0$, and $1 \leq r \leq m+1$, the polynomials $\RP_{q,t}(m \backslash r, n)^{\ast k}$ satisfy the recursion
	\begin{align*}
		\RP_{q,t}(m \backslash r, n)^{\ast k} = & \sum_{s=1}^{n} t^{m+n+1-r-s-k} \qbinom{r+s-1}{s}_q \sum_{h=0}^{k} q^{\binom{h}{2}} \qbinom{s}{h}_q \\ \times & \sum_{u=1}^{m-r+1} \qbinom{s+u-h-1}{s-1}_q \RP_{q,t}(m-r \, \backslash \, u, \, n-s)^{\ast \, k-h}
	\end{align*}
	
	with initial conditions \[ \RP_{q,t}(m \backslash m+1, n)^{\ast 0} = \qbinom{m+n}{m}_q \] and $\RP_{q,t}(m \backslash r, 0)^{\ast k} = \delta_{k,0} \delta_{r,m+1}$.
\end{theorem}

In \cite{Aval-DAdderio-Dukes-Hicks-LeBorgne-2014}*{Section~4}, the authors give a bijection $\zeta \colon \PP(n,m) \rightarrow \PP(m,n)$ swapping $(\area, \bounce)$ and $(\dinv, \area)$. In fact, the same map has a stronger property.

\begin{theorem}
	For $m \geq 1$, $n \geq 1$ and $k \geq 0$, the bijection $\zeta \colon \PP(n,m) \rightarrow \PP(m,n)$ in \cite{Aval-DAdderio-Dukes-Hicks-LeBorgne-2014}*{Theorem~4.1} extends to a bijection \[ \zeta \colon \PP(n , m)^{\circ k} \rightarrow \PP(m, n )^{\ast k} \] mapping the bistatistic $(\area, \bounce)$ to $(\dinv, \area)$.
\end{theorem}
See Section~\ref{sec:original_polyo} for a proof of this theorem.

Essentially the same map translates into the setting of reduced parallelogram polyominoes.
\begin{theorem}
	For $m \geq 0$, $n \geq 0$ and $k \geq 0$, there is a bijection \[\bar{\zeta} \colon \RP(m , n)^{\circ k} \rightarrow \RP(m , n)^{\ast k}\] mapping the bistatistic $(\area, \bounce)$ to $(\dinv, \area)$.
\end{theorem}
See Section~\ref{sec:reduced_polyo} for a proof of this theorem.

Notice that this last theorem explains the equality \eqref{eq:RPswitch}.

We also found an interesting relation with the two car parking functions.

\begin{theorem}
	For $m \geq 0$, $n \geq 0$ and $k \geq 0$, there exists a bijection \[\phi \colon \RP(m , n)^{\ast k} \rightarrow \PF^2(m , n)^{\ast k}\] such that $(\dinv(P),\area(P)) = (\dinv(\phi(P)), \area(\phi(P)))$ for all $P\in \RP(m, n)^{\ast k}$.
\end{theorem}
See Section~\ref{sec:2cars_PF} for a proof of this theorem.

This theorem (suitably interpreted) suggested the following identity.
\begin{theorem}
	Let $m,n,k\in \mathbb{N}$, $m\geq 0$, $n\geq 0$ and $n\geq k\geq 0$. Then
	\begin{align}
		\< \Delta_{h_m} e_{n+1}, s_{k+1,1^{n-k}} \> =  \< \Delta_{e_{m+n-k-1}}' e_{m+n}, h_m h_n \>.
	\end{align}
	In particular
	\begin{align}
		\< \Delta_{h_m} e_{n+1}, s_{k+1,1^{n-k}} \> =  \< \Delta_{h_{n}} e_{m+1}, s_{k+1,1^{m-k}} \>.
	\end{align}
\end{theorem}
See Section~\ref{sec:2more_thrms} for a proof of this theorem.

All the previous theorems, combined, lead to the following theorem.
\begin{theorem}\label{thm: two shuffle}
	For $m, n, k \in \mathbb{N}$, $0 \leq k \leq m$, we have
	\begin{equation}
		\< \Delta_{e_{m+n-k-1}}' e_{m+n}, h_m h_n \> = \sum_{D \in \PF^2(m,n)^{\ast k}} q^{\dinv(D)} t^{\area(D)} .
	\end{equation}
\end{theorem}
See Section~\ref{sec:2shuffle_delta} for a proof of this theorem.

The previous theorem is indeed the case $\< \cdot , h_m h_n \>$ of the Delta conjecture, once we ``standardize'' the two car parking functions (see the next section for more details). This proves the remaining half of Problem~8.1 in \cite{Haglund-Remmel-Wilson-2015}, by proving Conjecture 5.2.2.1 in \cite{Wilson-PhD-2015} (the other half of Problem~8.1 is solved in Section~\ref{subsec:qtSchroeder}). See Section~\ref{sec:reduced_polyo} for the recursion giving in fact a refinement of this theorem.

\section{Links with the Delta conjecture} \label{sec:link_delta}

In this section, we explain how the results of the previous two sections are special cases of the Delta conjecture. In order to do this we will use the theory of \emph{shuffles} from \cite{HHLRU-2005} (see also \cite{Haglund-Book-2008}*{Chapter~6}).

We recall that the rise version of the Delta conjecture predicts that 
\begin{equation}\label{eq: delta rise}
	\Delta'_{e_{n-k-1}}e_n= \sum_{D\in \LDD^{\ast k}(n)}q^{\dinv(D)}t^{\area(D)}x^D. 
\end{equation}

The labels of a labelled Dyck path are not required to be distinct. 

\begin{definition}
	A \emph{parking decorated Dyck path} is a labelled decorated Dyck path  whose labels are exactly the numbers from $1$ to $n$.  
\end{definition}

We will describe a way of relabelling a labelled decorated Dyck path into a parking decorated Dyck path, without changing the dinv and area statistics.  

Recall that, given a labelled Dyck path $D$, its dinv reading word is obtained by reading the labels of the path along the diagonals, from left to right, starting with the main diagonal.

\begin{definition} The \emph{reverse dinv reading word} is the word obtained by reading the dinv reading word from right to left. For $D\in \LDD^{\ast a}(n)$ we denote by $r(D)$ its reverse dinv reading word. 
\end{definition}

For example, the reverse dinv reading word of the left path in Figure~\ref{fig: standardisation} is $35131322$.

\begin{definition}
	Given an ordered set of distinct integers $(a_1,...,a_l)$ with $1\leq a_i\leq n$ and $\sigma\in \mathfrak{S}_n$, we say that $\sigma$ is an  \emph{$(a_1,...,a_l)$-shuffle} if the relative order of the $a_i$ is preserved in $\sigma$. In other words, $a_i$ occurs before $a_{i+1}$ in $\sigma_1\dots \sigma_n$. Given a composition $\mu$ of $n$, we say that $\sigma\in \mathfrak{S}_n$ is a \emph{$\mu$-shuffle} if it is a shuffle of $(1,2,\dots,\mu_1)$, $(\mu_1+1, \dots, \mu_1+\mu_2)$, $\dots$. 
\end{definition}

For example, if $\mu=(2,2)$, then the $\mu$-shuffles are 
\begin{center}
	\begin{tabular}{cccccc}
		1234&1324& 3124&3142&3412&1342.	
	\end{tabular}
\end{center}

\begin{definition}
	Let  $D\in \LDD^{\ast a}(n)$ be such that its labels consist of $\mu_1$ $1$'s, $\mu_2$ $2$'s and so forth. As such, $\mu=(\mu_1,\mu_2,\dots)$ is a weak composition of $n$. We define  the \emph{standardization} $s(D)$ of $D$ by keeping the same Dyck path and replacing the $1$'s by $1,2, \dots \mu_1 $, the $2$'s by $\mu_1+1, \dots , \mu_1+\mu_2$, and so on, such that the reverse dinv reading word of $s(D)$ is a $\mu$-shuffle. 
\end{definition}

See Figure~\ref{fig: standardisation} for an example. It is not hard   to see that $\dinv(\sigma(D))=\dinv(D)$. And of course, since the area does not depend on the labelling, we have also $\area(\sigma(D))=\area(D)$. 

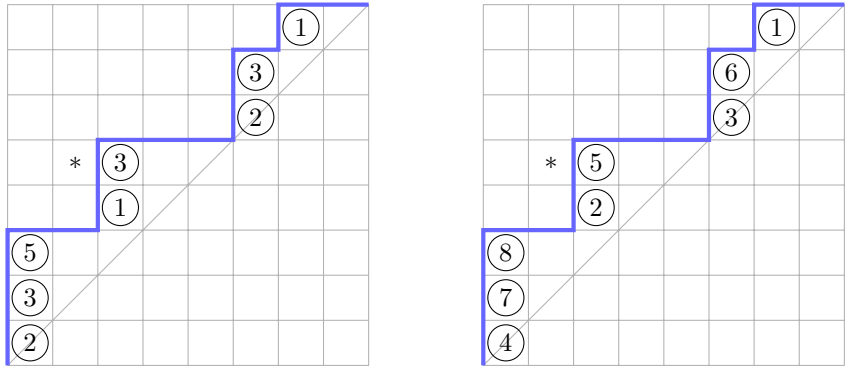
\begin{figure}[!ht]
	\centering
	\begin{minipage}{.5 \textwidth}
		\centering
		\begin{tikzpicture}[scale=0.6]
		\draw[gray!60, thin] (0,0) grid (8,8) (0,0) -- (8,8);
		\draw[blue!60, line width = 1.6pt] (0,0)|-(2,3)|-(5,5)|-(6,7)|-(8,8);
		
		\draw
		(0.5,0.5) circle(0.4 cm) node {$2$}
		(0.5,1.5) circle(0.4 cm) node {$3$}
		(0.5,2.5) circle(0.4 cm) node {$5$}
		(2.5,3.5) circle(0.4 cm) node {$1$}
		(2.5,4.5) circle(0.4 cm) node {$3$}
		(5.5,5.5) circle(0.4 cm) node {$2$}
		(5.5,6.5) circle(0.4 cm) node {$3$}
		(6.5,7.5) circle(0.4 cm) node {$1$};
		\draw (1.5,4.5) node {$\ast$};
		\end{tikzpicture}
	\end{minipage}%
	\begin{minipage}{.5 \textwidth}
		\centering
		\begin{tikzpicture}[scale=0.6]
		\draw[gray!60, thin] (0,0) grid (8,8) (0,0) -- (8,8);
		\draw[blue!60, line width = 1.6pt] (0,0)|-(2,3)|-(5,5)|-(6,7)|-(8,8);
		
		\draw
		(0.5,0.5) circle(0.4 cm) node {$4$}
		(0.5,1.5) circle(0.4 cm) node {$7$}
		(0.5,2.5) circle(0.4 cm) node {$8$}
		(2.5,3.5) circle(0.4 cm) node {$2$}
		(2.5,4.5) circle(0.4 cm) node {$5$}
		(5.5,5.5) circle(0.4 cm) node {$3$}
		(5.5,6.5) circle(0.4 cm) node {$6$}
		(6.5,7.5) circle(0.4 cm) node {$1$};
		\draw (1.5,4.5) node {$\ast$};
		\end{tikzpicture}
	\end{minipage}
	\caption{Standardization}
	\label{fig: standardisation}
\end{figure}

\begin{proposition}\label{prop: scal h}
	The rise version of the Delta conjecture (equation (\ref{eq: delta rise})) is equivalent to \[ \<\Delta'_{e_{n-k-1}}e_n,h_{\mu}\>=\sum_{\substack{D\in \LDD(n)^{\ast k }\\ r(D) \text{  is a } \mu-\text{shuffle}}}q^{\dinv(D)}t^{\area(D)} \] for all partitions $\mu$ of $n$.  
\end{proposition}
\begin{proof}
	Take $\mu=(\mu_1,\mu_2,\dots)$ a partition of $n$. The coefficient of $x_1^{\mu_1}x_2^{\mu_2}\dots $ on the right hand side of (\ref{eq: delta rise}) is the sum of $q^{\dinv(D)}t^{\area(D)}$, where $D$ are all the paths in $\LDD(n)^{\ast k }$ whose labels are $\mu_1$ $1$'s, $\mu_2$ $2$'s, etc. Replacing each of these paths by its standardization, this coefficient is exactly equal to \[\sum_{\substack{D\in \LDD(n)^{\ast k }\\ r(D) \text{  is a } \mu-\text{shuffle}}}q^{\dinv(D)}t^{\area(D)},\] indeed no two differently labelled (decorated) Dyck paths with the same set of labels have the same standardization. Since $\Delta'_{e_{n-k-1}}e_n$ is symmetric it follows that the rise version of the Delta conjecture holds if and only if \[\Delta'_{e_{n-k-1}}{e_n}_{\big |_{ m_\mu}}=\sum_{\substack{D\in \LDD(n)^{\ast k }\\ r(D) \text{  is a } \mu-\text{shuffle}}}q^{\dinv(D)}t^{\area(D)},\] for all $\mu$ partition of $n$ and so the proposition follows because \[\Delta'_{e_{n-k-1}}{e_n}_{\big |_{ m_\mu}}= \< \Delta'_{e_{n-k-1}} e_n, h_\mu\>.\]
\end{proof}

\begin{corollary}
	The first equation in Theorem~\ref{thm: two shuffle} is the case $\<\cdot,h_mh_n\>$ of the Delta conjecture.  
\end{corollary}
\begin{proof}
	The result referred to is \[ \< \Delta_{e_{m+n-k-1}}' e_{m+n}, h_m h_n \>  = \sum_{D \in \PF^2(m,n)^{\ast k}} q^{\dinv(D)} t^{\area(D)}. \] So setting $\mu=(m,n)$ and doing the change of variables $n\rightarrow n+m$ in Proposition~\ref{prop: scal h} yields 
	\[ \<\Delta_{e_{m+n-k-1}}'e_{m+n},h_{m}h_{n}\>=\sum_{\substack{D\in \LDD(n)^{\ast k }\\ r(D) \text{  is a } (m,n)-\text{shuffle}}}q^{\dinv(D)}t^{\area(D)}, \] because $h_{(m,n)}=h_mh_n$ by definition. The right hand sides of these equations coincide because standardizing the paths in the former yields the latter without changing the statistics.
\end{proof}

\begin{definition}
	Let $\mu$ and $\nu$ be partitions such that $|\mu|+|\nu|=n$, then $\sigma\in \mathfrak{S}_n$ is a \emph{$\mu,\nu$-shuffle} if it is a shuffle of the decreasing sequences $(\nu_1,\dots, 2,1)$, $(\nu_1+\nu_2, \dots, \nu_1),\dots\,$ and the increasing sequences $(|\nu|+1,\dots, |\nu|+\mu_1)$, $(|\nu|+\mu_1+1, \dots, |\nu| +\mu_1+\mu_2),\dots$.   
\end{definition}
To link our $q,t$-Schröder result to the Delta conjecture we use \cite{Haglund-Book-2008}*{Theorem 6.10}:

\begin{theorem}
	\label{thm: general eh}
	Let $f[Z]$ be a symmetric function of homogeneous degree $n$ and $c(\sigma)$ a family of constants. If \[\<f,h_\mu\>=\sum_{\substack{\sigma\in \mathfrak{S}_n \\ \sigma \text{ is a $\mu$-shuffle}}}c(\sigma)\] holds for all compositions $\mu$ of $n$ then
	\[\<f,e_\nu h_\mu\>=\sum_{\substack{\sigma\in \mathfrak{S}_n \\ \sigma \text{ is a $\mu,\nu$-shuffle}}}c(\sigma)\] holds for all pairs of compositions $\mu$, $\nu$.   
\end{theorem}
We refer to \cite{Haglund-Book-2008} for the proof. 

\begin{corollary}
	The equation \eqref{eq:qtSchroder2.1} in Theorem~\ref{thm: resultsection: decoqtSchroeder} is the case $\<\cdot , e_ah_b\>$ of the Delta conjecture. 
\end{corollary}

\begin{proof} Taking $f=\Delta'_{e_{n-k-1}}e_n$ and, 
	\[c(\sigma)=\sum_{\substack{\D\in \LDD^{\ast k}(n)\\ r(D)=\sigma }}q^{\dinv(D)}t^{\area(D)},\] 
	the hypothesis of Theorem~\ref{thm: general eh} is satisfied by Proposition~\ref{prop: scal h}. Therefore, its conclusion holds and we have  
	\begin{align*}
		\<\Delta'_{e_{n-k-1}}e_n, e_\nu h_\mu \>&=\sum_{\substack{\sigma\in \mathfrak{S}_n\\ \sigma \text{ is a $\mu,\nu$-shuffle}} }\sum_{\substack{\D\in \LDD^{\ast k}(n)\\ r(D)=\sigma }}q^{\dinv(D)}t^{\area(D)}\\
		& =\sum_{\substack{\D\in \LDD^{\ast k}(n)\\ r(D) \text{ is a $\mu,\nu$-shuffle} }}q^{\dinv(D)}t^{\area(D)}.
	\end{align*}
	Taking $\mu=(a)$, $\nu=(b)$ and making the change of variable $n\rightarrow a+b$ we obtain 
	\[
	\< \Delta_{e_{a+b-k-1}}'e_{a+b},e_a h_{b} \>=\sum_{\substack{\D\in \LDD^{\ast k}(a+b)\\ r(D) \text{ is a $(a),(b)$-shuffle} }}q^{\dinv(D)}t^{\area(D)}.
	\]
	So we have to show that the right hand side of this equation equals the right hand side of the first result in Theorem~\ref{thm: resultsection: decoqtSchroeder}:
	\[\sum_{\substack{\D\in \LDD^{\ast k}(a+b)\\ r(D) \text{ is a $(a),(b)$-shuffle} }}q^{\dinv(D)}t^{\area(D)} = \sum_{D\in \DD(a+b)^{\circ b, \ast k} } q^{\dinv(D)}t^{\area(D)}.\]
	
	In other words if $S \coloneqq \{D\in \LDD^{\ast k}(a+b)\mid r(D) \text{ is an $(a),(b)$-shuffle}\}$ we need a bijection $\alpha \colon S \rightarrow \DD(a+b)^{\circ b,\ast k}$ that preserves the statistics dinv and area.
	
	Let $D\in S$, define the underlying the Dyck path of $\alpha(D)$ to be the same as underlying the Dyck path of $D$. Decorate the same rises in $\alpha(D)$ as in $D$. It follows that the area is preserved by $\alpha$. Since $r(D)$ is an $(a),(b)$-shuffle, the $b$ largest labels of $D$ occur in decreasing order in $r(D)$. Since the labels must be increasing in columns, it follows from the definition of $r(D)$ that these $b$ biggest labels must be labelling peaks. Decorate these $b$ peaks and forget about the labels: we obtain a path $\alpha(D)$ in $\DD(a+b)^{\circ b, \ast k}$.

	For example, the path $D$ on the left in Figure~\ref{fig: alpha} has $r(D)=65748321$ which is a shuffle of $54321$ and $678$. We decorate the three peaks labelled by $6,7$ and $8$ and forget about the labels to obtain the path on the right. 
	
	It follows almost directly from the definitions of the dinv on these two sets that for any path $D\in S$ we have $\dinv(\alpha(D))=\dinv(D)$.
	
	It is easy to see that this map is invertible. 
	
	\begin{figure}[ht!]
		\centering
		\begin{minipage}{.5 \textwidth}
			\centering
			\begin{tikzpicture}[scale=.6]
			\draw[gray!60, thin] (0,0) grid (8,8) (0,0) -- (8,8);
			\draw[blue!60, line width = 1.6pt] (0,0)|-(1,3)|-(3,4)|-(4,6)|-(8,8);
			
			\draw
			(0.5,0.5) circle(0.4 cm) node {$1$}
			(0.5,1.5) circle(0.4 cm) node {$2$}
			(0.5,2.5) circle(0.4 cm) node {$8$}
			(1.5,3.5) circle(0.4 cm) node {$4$}
			(3.5,4.5) circle(0.4 cm) node {$3$}
			(3.5,5.5) circle(0.4 cm) node {$7$}
			(4.5,6.5) circle(0.4 cm) node {$5$}
			(4.5,7.5) circle(0.4 cm) node {$6$};
			
			\draw 
			(-.5,1.5) node {$\ast$}
			(3.5,7.5) node {$\ast$};
			\end{tikzpicture}
		\end{minipage}%
		\begin{minipage}{.5 \textwidth}
			\centering
			\begin{tikzpicture}[scale=.6]
			\draw[gray!60, thin] (0,0) grid (8,8) (0,0) -- (8,8);
			
			\draw[blue!60, line width = 1.6pt] (0,0)|-(1,3)|-(3,4)|-(4,6)|-(8,8);
			
			\draw 
			(-.5,1.5) node {$\ast$}
			(3.5,7.5) node {$\ast$};
			
			\filldraw[fill=blue!60]
			(0,3) circle (2.5 pt)
			(3,6) circle (2.5 pt)
			(4,8) circle (2.5 pt);
			\end{tikzpicture}
		\end{minipage}
		\caption{The map $\alpha: S\rightarrow \DD^{\circ b, \ast a}(a+b)$}\label{fig: alpha}
	\end{figure} 
	
\end{proof}

%
%
%
%

\section{A symmetry result}
In \cite{Zabrocki-4Catalan-2016}, Zabrocki established a consequence of the Delta conjecture, the so called \emph{$4$-variable Catalan conjecture}. In fact, this conjecture, stated in \cite{Haglund-Remmel-Wilson-2015}, makes other predictions that have not been explained in \cite{Zabrocki-4Catalan-2016}. We managed to prove one of them:
\begin{theorem} \label{thm:resultsection:symmetry}
	For $n>k+\ell$, the following expression is symmetric in $k$ and $\ell$:
	\begin{equation}
		\langle \Delta_{h_k}\Delta_{e_{n-k-\ell-1}}' e_{n-k},e_{n-k}\rangle.
	\end{equation}
\end{theorem}
We provide both an algebraic and a combinatorial proof of this fact. The combinatorial proof (see Section~\ref{sec:deco_qt_Schroeder}) relies on the $\psi$-map of Theorem \ref{thm: resultsection: psi_map}. The algebraic one (see Section~\ref{sec:algebraic_proof_symmetry}) uses the following theorem.

\begin{theorem} \label{thm:resultsection:schroeder_identity}
	For all $a,b,k\in \mathbb{N}$, with $a\geq 1$, $b\geq 1$ and $1\leq k\leq a$, we have
	\begin{equation} \label{eq:Schroeder_identityr}
		\langle \Delta_{e_a}'\Delta_{e_{a+b-k-1}}'e_{a+b},h_{a+b}\rangle = \langle \Delta_{h_k}\Delta_{e_{a-k}}' e_{a+b-k},e_{a+b-k}\rangle.
	\end{equation}
\end{theorem}
See Section~\ref{sec:proof_identity_symmetry} for a proof of this theorem.
%
%

\section{A new $q,t$-square}
\label{subsec: qt square result}

Combining the results in Section \ref{subsec:qtSchroeder} with the discussion in Section~\ref{sec:our_square_conj}, we get the following interpretation of $\<\Delta_{e_{n-1}}e_n, e_n\>$ in terms of square paths. 

\begin{theorem} We have
	\begin{align}
		\langle \Delta_{e_{n-1}}e_n,e_n \rangle& =\sum_{P\in \mathsf{SQ^E}(n)}q^{\area(P)}t^{\bounce(P)} = \sum_{P\in \mathsf{SQ^E}(n)}q^{\dinv(P)} t^{\area(P)} \\ & =\sum_{P\in \mathsf{SQ^N}(n)}q^{\area(P)}t^{\bounce(P)} =\sum_{P\in \mathsf{SQ^N}(n)}q^{\dinv(P)} t^{\area(P)}. 
	\end{align}
\end{theorem}
See Section~\ref{sec:new_qt_square} for a proof of this theorem.

We also make the interesting observation that there is a connection between this new $q,t$-square and the older one proved in \cite{Loehr-Warrington-square-2007} as a combinatorial interpretation of $ \langle \Delta_{e_{n}}\omega (p_n),e_n \rangle$.

\begin{proposition}
	\label{prop:coincr}
	We have
	\begin{equation}
		\langle \Delta_{e_{n-1}}e_n,e_n \rangle\big{|}_{t=1/q} = \langle \Delta_{e_{n}}\omega (p_n),e_n \rangle\big{|}_{t=1/q}=q^{-\binom{n}{2}} \frac{1}{1+q^n} \qbinom{2n}{n}_q.
	\end{equation}
\end{proposition}
See Section~\ref{sec:observ_one_over_q} for a proof of this proposition.

In fact in Section~\ref{sec:observ_one_over_q} we will suggest that this link between the two $q,t$-square theorems could be just an instance of a more general phenomenon.

\section{Symmetric functions identities}

Other than the ones mentioned in previous sections, we would like to highlight here some of the most important new symmetric functions identities that we prove in this paper in order to get the mentioned results.

The most important (and technical) one is the following summation formula. This is a generalization of a theorem implicit in a work of Haglund (e.g. see equation (2.38) in \cite{Haglund-Schroeder-2004}).

\begin{theorem}
	For $m,k\geq 1$ and $\ell\geq 0$, we have
	
	\begin{align*}
		& \sum_{\gamma\vdash m}\frac{\widetilde{H}_\gamma[X]}{w_\gamma} h_k[(1-t)B_\gamma]e_\ell[B_\gamma] \\
		= & \sum_{j=0}^{\ell} t^{\ell-j}\sum_{s=0}^{k}q^{\binom{s}{2}} \qbinom{s+j}{s}_q \qbinom{k+j-1}{s+j-1}_q h_{s+j} \left[ \frac{X}{1-q} \right] h_{\ell-j} \left[ \frac{X}{M} \right] e_{m-s-\ell} \left[ \frac{X}{M} \right] .
	\end{align*}
\end{theorem}
See Section~\ref{sec:sum_formula} for the proof of this theorem.

The following theorem is crucial in order to get the result in Section~\ref{subsec:two-shuffle}.
\begin{theorem}
	Let $m,n,k\in \mathbb{N}$, $m\geq 0$, $n\geq 0$ and $m\geq k\geq 0$. Then
	\begin{align}
		\sum_{r=1}^{m-k+1} t^{m-k-r+1} \< \Delta_{h_{m-k-r+1}} \Delta_{e_k} e_n \left[ X \frac{1-q^r}{1-q} \right], e_n \> = \< \Delta_{h_m} e_{n+1}, s_{k+1,1^{n-k}} \>.
	\end{align}
\end{theorem}
See Section~\ref{sec:2more_thrms} for a proof of this theorem.

The following result shows a surprisingly tight relation between the square conjecture in \cite{Loehr-Warrington-square-2007} and our Conjecture~\ref{conj:new_square}.
\begin{theorem} \label{thm:square_cnjs_at_1_qopenp}
	We have
	\begin{equation}
		{\Delta_{e_n} \omega(p_n)}_{\big{|}_{t=1/q}}={\Delta_{e_{n-1}}e_n}_{\big{|}_{t=1/q}}.
	\end{equation}
\end{theorem}
See Section~\ref{sec:observ_one_over_q} for a proof of this theorem.

The following result suggests that Proposition~\ref{prop:coincr} is not just an isolated coincidence.
\begin{proposition} 
	For any $f\in \Lambda^{(k)}$ we have
	\begin{equation}
		{\langle\Delta_f e_n,e_{n-1}e_1\rangle}_{\big{|}_{t=1/q}}={\langle\Delta_f \omega(p_n),e_{n}\rangle}_{\big{|}_{t=1/q}}.
	\end{equation}
\end{proposition}
See Section~\ref{sec:observ_one_over_q} for a proof of this proposition.

\section{A few open problems}

Other than the open questions in \cite{Haglund-Remmel-Wilson-2015} that we did not answer, our work leaves open a few natural combinatorial questions. We refer to Section~\ref{sec:combinatorial_recursions_ddyck} for the missing definitions.

For example, consider the identity
\begin{equation}
	\DDd_{q,t}(n \backslash k )^{\circ a, \ast b} = \sum_{j=0}^{n-k} \DDp_{q,t}(n\backslash k+j \backslash j )^{\circ a, \ast b},
\end{equation}
which is obtained by combining Theorem~\ref{thm:qt_enumerators_formulae} with Theorem~\ref{thm:rel_F_H}. There should be a bijective explanation of this identity. Notice that, summing over $k$, this would lead to a bijective explanation of the identity
\begin{equation} \DDd_{q,t}(n)^{\circ a, \ast b}= \DDp_{q,t}(n)^{\circ a, \ast b},
\end{equation}
coming from Theorem~\ref{thm:decoqtSchroeder}.

We observe here that for $b=0$ these explanations are provided by the bijection given in \cite{Egge-Haglund-Killpatrick-Kremer-2003}, so it is conceivable to look for an extension of that bijection to objects with decorated rises.

Another natural problem is the following: assuming Conjecture~\ref{conj:new_square} and recalling the results in \cite{Leven-2016}, find a combinatorial explanation of Theorem~\ref{thm:square_cnjs_at_1_qopenp}.

\part{Proofs}\label{part: proofs}

\chapter{Symmetric functions}

In this chapter we provide most of our results on symmetric functions. We start with a section that, together with Section~\ref{sec:symmetric_fcts_intro}, indicates most of the notations and the results that we will use in the rest of the paper.

\section{Basic identities}

We integrate here Section~\ref{sec:symmetric_fcts_intro} with more notations, and we introduce some basic identities that we will use all along this article. Again, the main references that we will use for symmetric functions are \cite{Macdonald-Book-1995} and \cite{Stanley-Book-1999}. 

As we already said, we will make extensive use of the \emph{plethystic notation}.

We have for example the addition formulas
\begin{align}
	p_k[X+Y]=p_k[X]+p_k[Y]\quad \text{ and } \quad p_k[X-Y]=p_k[X]-p_k[Y],
\end{align}
and
\begin{align}
	\label{eq:e_h_sum_alphabets}
	e_n[X+Y]=\sum_{i=0}^ne_{n-i}[X]e_i[Y]\quad \text{ and } \quad  h_n[X+Y]=\sum_{i=0}^nh_{n-i}[X]h_i[Y].
\end{align}
Notice in particular that $p_k[-X]$ equals $-p_k[X]$ and not $(-1)^kp_k[X]$. As the latter sort of negative sign can be also useful, it is customary to use the notation $\epsilon$ to express it: we will have $p_k[\epsilon X] = (-1)^k p_k[X]$, so that, in general,
\begin{align}
	\label{eq:minusepsilon}
	f[-\epsilon X] = \omega f[X]
\end{align}
for any symmetric function $f$, where $\omega$ is the fundamental algebraic involution which sends $e_k$ to $h_k$, $s_{\lambda}$ to $s_{\lambda'}$ and $p_k$ to $(-1)^{k-1}p_k$.

With our definition of the Hall scalar product on symmetric functions, we have the orthogonality
\begin{align}
	\< p_{\lambda}, p_{\mu} \> = z_{\mu} \chi(\lambda=\mu)
\end{align}
which defines the integers $z_{\mu}$, where $\chi(\mathcal{P})=1$ if the statement $\mathcal{P}$ is true, and $\chi(\mathcal{P})=0$ otherwise.

Recall the \emph{Cauchy identities}
\begin{align}
	\label{eq:Cauchy_identities}
	e_n[XY] = \sum_{\lambda\vdash n} s_{\lambda}[X] s_{\lambda'}[Y] \quad \text{ and } \quad h_n[XY] = \sum_{\lambda\vdash n} s_{\lambda}[X] s_{\lambda}[Y].
\end{align}
With the symbol ``$\perp$'' we denote the operation of taking the adjoint of an operator with respect to the Hall scalar product, i.e.
\begin{align}
	\< f^\perp g, h \> = \< g, fh \> \quad \text{ for all } f, g, h \in \Lambda.
\end{align}
We introduce also the operator
\begin{align}
	\tau_z f[X] \coloneqq f[X+z] \qquad \text{ for all } f[X] \in \Lambda,
\end{align}
so for example
\begin{align}
	\tau_{-\epsilon} f[X] = f[X-\epsilon] \qquad \text{ for all } f[X] \in \Lambda.
\end{align}
The operator $\tau_z$ can be computed using the following formula (see \cite{Garsia-Haiman-Tesler-Explicit-1999}*{Theorem~1.1} for a proof):
\begin{equation}
	\tau_z =\sum_{r\geq 0}z^rh_r^\perp.
\end{equation}

We record here also the following well-known identity (see \cite{Garsia-Hicks-Stout-2011}*{Lemma~2.1} for a proof).

For all $\mu\vdash n$ we have 
\begin{equation} \label{eq:s_plethystic_eval}
	s_{\mu}[1-u]=\left\{\begin{array}{ll}
		(-u)^r(1-u) & \text{ if }\mu=(n-r,1^r)\text{ for some }r\in \{0,1,2,\dots,n-1\},\\
		0 & \text{ otherwise}.
	\end{array}\right.
\end{equation}

We refer also to \cite{Haglund-Book-2008} for more informations on this topic.

\subsection{Macdonald symmetric function toolkit}

With the notation of Section~\ref{sec:symmetric_fcts_intro}, we set $M \coloneqq (1-q)(1-t)$, and we define for every partition $\mu$
\begin{align}
	B_{\mu} & \coloneqq B_{\mu}(q,t)=\sum_{c\in \mu}q^{a_{\mu}'(c)}t^{l_{\mu}'(c)} \\
	D_{\mu} & \coloneqq MB_{\mu}(q,t)-1 \\
	T_{\mu} & \coloneqq T_{\mu}(q,t)=\prod_{c\in \mu}q^{a_{\mu}'(c)}t^{l_{\mu}'(c)} \\
	\Pi_{\mu} & \coloneqq \Pi_{\mu}(q,t)=\prod_{c\in \mu/(1)}(1-q^{a_{\mu}'(c)}t^{l_{\mu}'(c)}) \\
	w_{\mu} & \coloneqq w_{\mu}(q,t)=\prod_{c\in \mu} (q^{a_{\mu}(c)} - t^{l_{\mu}(c) + 1}) (t^{l_{\mu}(c)} - q^{a_{\mu}(c) + 1}).
\end{align}
Notice that
\begin{align}
	\label{eq:Bmu_Tmu}
	B_{\mu} = e_1[B_{\mu}]\quad \text{ and } \quad T_{\mu}=e_{|\mu|}[B_{\mu}].
\end{align}

It is useful to introduce the so called \emph{star scalar product} on $\Lambda$ given by
\[ \< p_{\lambda},p_{\mu} \>_*=(-1)^{|\mu|-|\lambda|}\prod_{i=1}^{\ell(\mu)}(1-q^{\mu_i})(1-t^{\mu_i}) z_{\mu}\chi(\lambda=\mu). \]

For every symmetric function $f[X]$ and $g[X]$ we have (see \cite{Garsia-Haiman-Tesler-Explicit-1999}*{Proposition~1.8})
\begin{equation}
	\langle f,g\rangle_*= \langle \omega \phi f,g\rangle=\langle \phi \omega f,g\rangle
\end{equation}
where 
\begin{equation}
	\phi f[X] \coloneqq f[MX]\qquad \text{ for all } f[X]\in \Lambda.
\end{equation}
For every symmetric function $f[X]$ we set
\begin{equation}
	f^*=f^*[X] \coloneqq \phi^{-1}f[X]=f\left[\frac{X}{M}\right].
\end{equation}
Then for all symmetric functions $f,g,h$ we have
\begin{equation} \label{eq:hperp_estar_adjoint}
	\langle h^\perp f,g\rangle_*=\langle h^\perp f,\omega \phi g\rangle=\langle f,h\omega \phi g\rangle=\langle f,\omega \phi ( (\omega h)^* \cdot g)\rangle =\langle f,  (\omega h)^* \cdot g \rangle_*,
\end{equation}
so the operator $h^\perp$ is the adjoint of the multiplication by $(\omega h)^*$ with respect to the star scalar product.

It turns out that the Macdonald polynomials are orthogonal with respect to the star scalar product: more precisely
\begin{align}
	\label{eq:H_orthogonality}
	\< \widetilde{H}_{\lambda},\widetilde{H}_{\mu}\>_*=w_{\mu}(q,t)\chi(\lambda=\mu).
\end{align}
These orthogonality relations give the following Cauchy identities
\begin{align}
	\label{eq:Mac_Cauchy}
	e_n \left[ \frac{XY}{M} \right] = \sum_{\mu \vdash n} \frac{ \widetilde{H}_{\mu} [X] \widetilde{H}_\mu [Y]}{w_\mu} \quad \text{ for all } n.
\end{align}
%

We will use the following form of \emph{Macdonald-Koornwinder reciprocity} (see \cite{Macdonald-Book-1995}*{p.~332} or \cite{Garsia-Haiman-Tesler-Explicit-1999}): for all partitions $\alpha$ and $\beta$
\begin{align}
	\label{eq:Macdonald_reciprocity}
	\frac{\widetilde{H}_{\alpha}[MB_{\beta}]}{\Pi_{\alpha}} = \frac{\widetilde{H}_{\beta}[MB_{\alpha}]}{\Pi_{\beta}}.
\end{align}
One of the most important identities in this theory is the following one (see \cite{Garsia-Haiman-Tesler-Explicit-1999}*{Theorem~I.2}): for every symmetric function $f[X]$ and every partition $\mu$, we have
\begin{align}
	\label{eq:glenn_formula}
	\< f[X], \widetilde{H}_\mu[X+1]\>_*= \left.\nabla^{-1}\tau_{-\epsilon} f[X]\right|_{X=D_\mu}.
\end{align}

\subsection{Pieri rules and summation formulae}

For a given $k\geq 1$, we define the Pieri coefficients $c_{\mu \nu}^{(k)}$ and $d_{\mu \nu}^{(k)}$ by setting
\begin{align}
	\label{eq:def_cmunu} h_{k}^\perp \widetilde{H}_{\mu}[X] & =\sum_{\nu \subset_k \mu} c_{\mu \nu}^{(k)} \widetilde{H}_{\nu}[X], \\
	\label{eq:def_dmunu} e_{k}\left[\frac{X}{M}\right] \widetilde{H}_{\nu}[X] & = \sum_{\mu \supset_k \nu} d_{\mu \nu}^{(k)} \widetilde{H}_{\mu}[X].
\end{align}
The following identity is \cite{Bergeron-Haiman-2013}*{Proposition~5}, written in the notation of \cite{Garsia-Haglund-Xin-Zabrocki-Pieri-2016}, which is coherent with ours:
\begin{align}
	\label{eq:cmunu_recursion}
	c_{\mu \nu}^{(k+1)} = \frac{1}{B_{\mu/\nu}} \sum_{\nu\subset_1 \alpha\subset_k\mu} c_{\mu \alpha}^{(k)} c_{\alpha \nu}^{(1)} \frac{T_{\alpha}}{T_{\nu}} \quad \text{ with } \quad B_{\mu/\nu} \coloneqq B_{\mu} - B_{\nu},
\end{align}
where $\nu\subset_k \mu$ means that $\nu$ is contained in $\mu$ (as Ferrers diagrams) and $\mu/\nu$ has $k$ lattice cells, and the symbol $\mu \supset_k \nu$ is analogously defined. It follows from \eqref{eq:H_orthogonality} that
\begin{align}
	\label{eq:rel_cmunu_dmunu}
	c_{\mu \nu}^{(k)} = \frac{w_{\mu}}{w_{\nu}}d_{\mu \nu}^{(k)}.
\end{align}
For every $m\in \mathbb{N}$ with $m\geq 1$, and for every $\gamma\vdash m$, we have 
\begin{align*}
	B_{\gamma} & = e_{1}[B_{\gamma}]\\
	\text{(using \eqref{eq:Mac_hook_coeff})}& = \langle \widetilde{H}_{\gamma}, e_1h_{m-1}\rangle \\
	& = \langle \widetilde{H}_{\gamma}, h_1h_{m-1}\rangle \\
	& = \langle h_{1}^{\perp} \widetilde{H}_{\gamma}, h_{m-1}\rangle\\
	\text{(using \eqref{eq:def_dmunu})} & = \sum_{\delta\subset_1 \gamma}c_{\gamma \delta}^{(1)}\langle \widetilde{H}_{\delta}, h_{m-1}\rangle\\ 
	\text{(using \eqref{eq:Mac_hook_coeff})} & = \sum_{\delta\subset_1 \gamma}c_{\gamma \delta}^{(1)},
\end{align*}	
so we get the well-known summation formula
\begin{equation} \label{eq:Pieri_sum1}
	B_{\gamma}=\sum_{\delta\subset_1 \gamma}c_{\gamma \delta}^{(1)}.
\end{equation}

\subsection{$q$-notation}

Using the notation in Section~\ref{sec:symmetric_fcts_intro}, recall the well-known recursion
\begin{align}
	\label{eq:qbin_recursion}
	\qbinom{n}{k}_q = q^k \qbinom{n-1}{k}_q + \qbinom{n-1}{k-1}_q = \qbinom{n-1}{k}_q + q^{n-k} \qbinom{n-1}{k-1}_q.
\end{align}

Recall also the standard notation for the $q$-\emph{rising factorial}
\begin{align}
	(a;q)_s \coloneqq (1 - a)(1 - qa)(1 - q^2 a) \cdots (1 - q^{s-1} a).
\end{align}
It is well-known (cf. \cite{Stanley-Book-1999}*{Theorem~7.21.2}) that (recall that $h_0 = 1$)
\begin{align}
	\label{eq:h_q_binomial}
	h_k[[n]_q] = \frac{(q^{n};q)_k}{(q;q)_k} = \qbinom{n+k-1}{k}_q \quad \text{ for } n \geq 1 \text{ and } k \geq 0,
\end{align}
and (recall that $e_0 = 1$)
\begin{align} \label{eq:e_q_binomial}
	e_k[[n]_q] = q^{\binom{k}{2}} \qbinom{n}{k}_q \quad \text{ for all } n, k \geq 0.
\end{align}
Also (cf. \cite{Stanley-Book-1999}*{Corollary~7.21.3})
\begin{align}
	\label{eq:h_q_prspec}
	h_k\left[\frac{1}{1-q}\right] = \frac{1}{(q;q)_k} = \prod_{i=1}^k \frac{1}{1-q^i} \quad \text{ for } k \geq 0,
\end{align}
and
\begin{align}
	\label{eq:e_q_prspec}
	e_k\left[\frac{1}{1-q}\right] = \frac{q^{\binom{k-1}{2}}}{(q;q)_k} = q^{\binom{k-1}{2}} \prod_{i=1}^k \frac{1}{1-q^i} \quad \text{ for } k \geq 0.
\end{align}

\subsection{Useful identities}

In this section we collect some results from the literature that we are going to use later in the text.

The symmetric functions $E_{n,k}$ were introduced in \cite{Garsia-Haglund-qtCatalan-2002} by means of the following expansion:
\begin{align}
	\label{eq:def_Enk}
	e_n \left[ X \frac{1-z}{1-q} \right] = \sum_{k=1}^n \frac{(z;q)_k}{(q;q)_k} E_{n,k}.
\end{align}
Notice that setting $z=q^j$ in \eqref{eq:def_Enk} we get
\begin{align}
	\label{eq:en_q_sum_Enk}
	e_n \left[ X \frac{1-q^j}{1-q} \right] = \sum_{k=1}^n \frac{(q^j;q)_k}{(q;q)_k} E_{n,k} = \sum_{k=1}^n \qbinom{k+j-1}{k}_q E_{n,k} .
\end{align}
In particular, for $j=1$, we get
\begin{align}
	\label{eq:en_sum_Enk}
	e_n = E_{n,1} + E_{n,2} + \cdots +E_{n,n}.
\end{align}
The following identity is proved in \cite{Garsia-Haglund-qtCatalan-2002}*{Proposition~2.2}:
\begin{align}
	\label{eq:garsia_haglund_eval}
	\widetilde{H}_\mu[(1-t)(1-q^j)] = (1-q^j) \Pi_\mu h_j[(1-t)B_\mu].
\end{align}
So, using \eqref{eq:Mac_Cauchy} with $Y = [j]_q = \frac{1-q^j}{1-q}$, we get
\begin{align} \label{eq:qn_q_Macexp}
	e_n \left[ X \frac{1-q^j}{1-q} \right] & = \sum_{\mu \vdash n} \frac{\widetilde{H}_\mu[X] \widetilde{H}_\mu [(1-t)(1-q^j)] }{w_\mu} \\
	\notag & = (1-q^j) \sum_{\mu\vdash n} \frac{ \Pi_\mu \widetilde{H}_\mu[X] h_j[(1-t)B_\mu]}{w_\mu}.
\end{align}

For $\mu \vdash n$, Macdonald proved (see \cite{Macdonald-Book-1995}*{p.~362}) that
\begin{align}
	\label{eq:Mac_hook_coeff_ss}
	\< \widetilde{H}_{\mu}, s_{(n-r,1^r)} \> = e_r[B_{\mu}-1],
\end{align}
so that, since by Pieri rule $e_r h_{n-r} = s_{(n-r,1^r)} + s_{(n-r+1,1^{r-1})}$,
\begin{align}
	\label{eq:Mac_hook_coeff}
	\< \widetilde{H}_{\mu}, e_r h_{n-r} \> = e_r[B_{\mu}].
\end{align}
The following well-known identity is an easy consequence of \eqref{eq:Mac_hook_coeff}.

\begin{lemma}
	\label{lem:Mac_hook_coeff}
	For any symmetric function $f\in \Lambda^{(n)}$,
	\begin{align}
		\label{eq:lem_e_h_Delta}
		\< \Delta_{e_{d}} f, h_n \> = \< f, e_d h_{n-d} \>.
	\end{align}
\end{lemma}

\begin{proof}
	Checking it on the Macdonald basis elements $\widetilde{H}_\mu\in \Lambda^{(n)}$, using \eqref{eq:Mac_hook_coeff}, we get
	\[ \< \Delta_{e_{d}} \widetilde{H}_\mu, h_n \> = e_d[B_\mu] \< \widetilde{H}_\mu, h_n \> = e_d[B_\mu] = \< \widetilde{H}_\mu, e_d h_{n-d} \>. \]
\end{proof}

The following lemma is due to Haglund. 
\begin{lemma}[\cite{Haglund-Schroeder-2004}*{Corollary~2}]
	\label{lem:Haglund}
	For positive integers $d,n$ and any symmetric function $f\in \Lambda^{(n)}$,
	\begin{align} \label{eq:Haglund_Lemma}
		\< \Delta_{e_{d-1}} e_n, f \> = \< \Delta_{\omega f} e_d, h_d \>.
	\end{align}
\end{lemma}

The following theorem is due to Haglund.

\begin{theorem}[\cite{Haglund-Schroeder-2004}*{Theorem~2.11}]
	\label{thm:Haglund_formula}
	For $n,k\in \mathbb{N}$ with $1\leq k\leq n$, \[ \< \nabla E_{n,k},h_n\>= \delta_{n,k}, \]
	where $\delta_{n,k}$ is the Kronecker delta function, i.e. $\delta_{n,k}=1$ if $n=k$ and $\delta_{n,k}=0$ if $n\neq k$.
	
	In addition, if $m>0$ and $\lambda \vdash m$,
	\begin{align}
		\< \Delta_{s_\lambda} \nabla E_{n,k}, h_n \> = t^{n-k} \< \Delta_{h_{n-k}} e_m \left[ X \frac{1-q^k}{1-q} \right], s_{\lambda'} \>,
	\end{align}
	or equivalently
	\begin{align}
		\< \Delta_{s_\lambda} \nabla E_{n,k}, h_n\> = t^{n-k} \sum_{\mu \vdash m} \frac{(1-q^k) h_k[(1-t)B_\mu] h_{n-k}[B_\mu] \Pi_{\mu} \widetilde{K}_{\lambda' \mu}}{w_{\mu}} .
	\end{align}
\end{theorem}

We need another theorem of Haglund: the following is essentially \cite{Haglund-Schroeder-2004}*{Theorem~2.5}.
\begin{theorem}
	For $k,n\in \mathbb{N}$ with $1\leq k\leq n$, 
	\begin{align}
		\label{eq:Haglund_nablaEnk}
		\nabla E_{n,k} = t^{n-k}(1-q^k) \mathbf{\Pi} h_k \left[ \frac{X}{1-q} \right] h_{n-k} \left[ \frac{X}{M} \right],
	\end{align}
	where $\mathbf{\Pi}$ is the invertible linear operator defined by
	\begin{equation}
		\mathbf{\Pi} \widetilde{H}_\mu[X] = \Pi_\mu \widetilde{H}_\mu[X] \qquad \text{ for all } \mu.
	\end{equation}
\end{theorem}

The following expansions are well-known, and they can be deduced from the Cauchy identities \eqref{eq:Mac_Cauchy}.

\begin{proposition} 
	For $n\in \mathbb{N}$ we have
	\begin{align}
		\label{eq:en_expansion}
		e_n[X] = e_n \left[ \frac{XM}{M} \right] = \sum_{\mu \vdash n} \frac{M B_\mu \Pi_{\mu} \widetilde{H}_\mu[X]}{w_\mu}.
	\end{align}
	Moreover, for all $k\in \mathbb{N}$ with $0\leq k\leq n$, we have
	\begin{align}
		\label{eq:e_h_expansion}
		h_k \left[ \frac{X}{M} \right] e_{n-k} \left[ \frac{X}{M} \right] = \sum_{\mu \vdash n} \frac{e_k[B_\mu] \widetilde{H}_\mu[X]}{w_\mu}.
	\end{align}
\end{proposition}



\section{A summation formula} \label{sec:sum_formula}

This section is the most technically demanding part of the paper and it is devoted to prove the following theorem.
\begin{theorem}
	For $m,k\geq 1$ and $\ell\geq 0$, we have
	\begin{equation} \label{eq:mastereq}
		\sum_{\gamma\vdash m}\frac{\widetilde{H}_\gamma[X]}{w_\gamma} h_k[(1-t)B_\gamma]e_\ell[B_\gamma]=\qquad \qquad \qquad \qquad  \qquad \qquad  \qquad \qquad 
	\end{equation}
	\begin{align*} 
		=\sum_{j=0}^{\ell} t^{\ell-j}\sum_{s=0}^{k}q^{\binom{s}{2}} \begin{bmatrix}
			s+j\\
			s
		\end{bmatrix}_q \begin{bmatrix}
			k+j-1\\
			s+j-1
		\end{bmatrix}_qh_{s+j}\left[\frac{X}{1-q}\right] h_{\ell-j}\left[\frac{X}{M}\right] e_{m-s-\ell}\left[\frac{X}{M}\right].
	\end{align*}
\end{theorem}

\subsection{Two special cases}

We start by establishing two special cases of this formula: the cases $\ell=0$ and $\ell=m$.

The case $\ell=0$ is implicit in the work of Haglund (e.g. see equation (2.38) in \cite{Haglund-Schroeder-2004}). Though we don't need this special case in our argument, we outline here its proof (referring to \cite{Haglund-Schroeder-2004} for some details) since it shows the strategy that we will use also in the general case. 

\begin{theorem}[Haglund] \label{thm:Haglund_summation}
	For $m,k\geq 1$ we have
	\begin{equation} \label{eq:Haglund_summation}
		\sum_{\mu\vdash m} \frac{\widetilde{H}_\mu[X]}{w_\mu}h_k[(1-t)B_\mu] = \sum_{s=1}^{m}
		q^{\binom{s}{2}} 
		\begin{bmatrix}
			k-1\\
			s-1\\
		\end{bmatrix}_q h_s \left[ \frac{X}{1-q}\right] e_{m-s}\left[ \frac{X}{M}\right].
	\end{equation}	
\end{theorem}
\begin{proof}
	Let us set
	\begin{equation}
		g[X] \coloneqq \sum_{\mu\vdash m} \frac{ \widetilde{H}_\mu[X]}{w_\mu}h_k[(1-t)B_{\mu}].
	\end{equation}
	The idea is to use the expansion 
	\begin{equation}
		\sum_{\mu\vdash m} \frac{ \widetilde{H}_\mu[X]}{w_\mu}h_k[(1-t)B_{\mu}]=g[X]= \sum_{\mu\vdash m} \frac{ \widetilde{H}_\mu[X]}{w_\mu} \< g[X],\widetilde{H}_\mu[X]\>_*.
	\end{equation}
	We suppose that $g[X]$ is of the form \[ g[X]=\sum_{r\geq 0}e_r^* f[X] \] for some $f[X]\in \Lambda$, so that
	\begin{align*}
		\< g[X],\widetilde{H}_\mu[X]\>_* & = \< \sum_{r\geq 0}e_r^*f[X],\widetilde{H}_\mu[X]\>_*\\ 
		\text{(using \eqref{eq:hperp_estar_adjoint})}& = \< f[X], \sum_{r\geq 0}h_r^\perp \widetilde{H}_\mu[X]\>_*\\ 
		& = \< f[X], \widetilde{H}_\mu[X+1]\>_*.
	\end{align*}
	Then we can use the identity \eqref{eq:glenn_formula} to get
	\begin{equation}
		\< g[X],\widetilde{H}_\mu[X]\>_* = \< f[X], \widetilde{H}_\mu[X+1]\>_*=  \left.\nabla^{-1}\tau_{-\epsilon} f[X]\right|_{X=D_\mu}.
	\end{equation}
	So we look for an $f[X]$ such that
	\begin{equation}
		\left.\nabla^{-1}\tau_{-\epsilon}f[X]\right|_{X=D_\mu}=h_k[(1-t)B_{\mu}].
	\end{equation}
	Clearly, we can use
	\begin{equation}
		\nabla^{-1}\tau_{-\epsilon}f[X]  =  h_{k}\left[\frac{X+1}{1-q}\right],
	\end{equation}
	so that
	\begin{equation}
		f[X]  =  \tau_{\epsilon}\nabla h_{k}\left[\frac{X+1}{1-q}\right].
	\end{equation}
	This is precisely equation (2.76) in \cite{Haglund-Schroeder-2004}. 
	
	Now notice that
	\begin{equation}
		\sum_{\mu\vdash m} \frac{\widetilde{H}_\mu[X]}{w_\mu}\< g[X],\widetilde{H}_\mu[X]\>_*=(g[X])_m,
	\end{equation}
	where $(g[X])_d$ denotes the homogeneous component of $g[X]$ of degree $d$.
	
	Equation (2.83) in \cite{Haglund-Schroeder-2004} shows that for all $d\geq 0$
	\begin{equation}
		(f[X])_d  = h_{d}\left[\frac{X}{1-q}\right]q^{\binom{d}{2}} \begin{bmatrix}
			k-1\\
			k-d\\
		\end{bmatrix}_q,
	\end{equation}
	so that
	\begin{align*}
		\sum_{\mu\vdash m} \frac{ \widetilde{H}_\mu[X]}{w_\mu}h_k[(1-t)B_{\mu}] & = \sum_{\mu\vdash m} \frac{ \widetilde{H}_\mu[X]}{w_\mu} \< g[X],\widetilde{H}_\mu[X]\>_*\\
		& =(g[X])_m\\
		& =\sum_{d=0}^{m}e_{m-d}\left[\frac{X}{M}\right](f[X])_d\\
		& =\sum_{d=0}^{m}e_{m-d}\left[\frac{X}{M}\right]h_{d}\left[\frac{X}{1-q}\right]q^{\binom{d}{2}} \begin{bmatrix}
			k-1\\
			k-d\\
		\end{bmatrix}_q\\
		& =\sum_{d=0}^{m}q^{\binom{d}{2}} \begin{bmatrix}
			k-1\\
			d-1\\
		\end{bmatrix}_qh_{d}\left[\frac{X}{1-q}\right]e_{m-d}\left[\frac{X}{M}\right],
	\end{align*}
	as we wanted.
\end{proof}

We now prove the case $\ell=m$ (recall \eqref{eq:Bmu_Tmu}), as we need it to establish the general case.
\begin{theorem}
	For $m\geq 1$ and $k\geq 1$ we have
	\begin{align} \label{eq:basic_summation_2}
		\sum_{\mu\vdash m} \frac{T_\mu \widetilde{H}_\mu[X]}{w_\mu}h_k[(1-t)B_\mu]  & =  \sum_{j=1}^m \begin{bmatrix}
			k+j-1\\
			j-1
		\end{bmatrix}_q h_j\left[\frac{X}{1-q}\right] h_{m-j}\left[\frac{tX}{M}\right] .
	\end{align}	
\end{theorem}
\begin{proof}
	Using \eqref{eq:en_q_sum_Enk}, we have
	\begin{equation}
		\nabla e_m\left[X\frac{1-q^k}{1-q}\right]=\sum_{j=1}^m \begin{bmatrix}
			k+j-1\\
			j
		\end{bmatrix}_q\nabla E_{mj} .
	\end{equation}
	Also, Haglund showed in equation (7.86) of \cite{Haglund-Book-2008} that
	\begin{equation}
		\nabla E_{mj}=t^{m-j}(1-q^j)\mathbf{\Pi} \left(h_{m-j}\left[\frac{X}{M}\right] h_{j}\left[\frac{X}{1-q}\right]\right),
	\end{equation}
	where $\mathbf{\Pi}$ is the linear operator defined by
	\begin{equation}
		\mathbf{\Pi} \widetilde{H}_\mu[X]=\Pi_\mu \widetilde{H}_\mu[X]\qquad \text{for all }\mu.
	\end{equation}
	So we get
	\begin{align} \label{eq:enq_aux1}
		\notag \nabla e_n\left[X\frac{1-q^k}{1-q}\right] & =\sum_{k=1}^n \begin{bmatrix}
			k+j-1\\
			j
		\end{bmatrix}_qt^{m-j}(1-q^j)\mathbf{\Pi} \left(h_{m-j}\left[\frac{X}{M}\right] h_{j}\left[\frac{X}{1-q}\right]\right)\\ 
		& =(1-q^k)\sum_{k=1}^n \begin{bmatrix}
			k+j-1\\
			j-1
		\end{bmatrix}_qt^{m-j}\mathbf{\Pi} \left(h_{m-j}\left[\frac{X}{M}\right] h_{j}\left[\frac{X}{1-q}\right]\right)  ,
	\end{align}
	where we used the obvious
	\begin{equation}
		(1-q^j)\begin{bmatrix}
			k+j-1\\
			j
		\end{bmatrix}_q=(1-q^k)\begin{bmatrix}
			k+j-1\\
			j-1
		\end{bmatrix}_q.
	\end{equation}
	
	But, using \eqref{eq:qn_q_Macexp}, we also have
	\begin{align} \label{eq:enq_aux2}
		\notag \nabla e_m\left[X\frac{1-q^k}{1-q}\right] & =(1-q^k)\sum_{\mu\vdash m}\frac{\Pi_\mu T_\mu \widetilde{H}_\mu[X]h_k[(1-t)B_\mu]}{w_\mu}\\
		& =(1-q^k)\mathbf{\Pi}\sum_{\mu\vdash m}\frac{T_\mu \widetilde{H}_\mu[X]h_k[(1-t)B_\mu]}{w_\mu}.
	\end{align}
	Comparing \eqref{eq:enq_aux1} and \eqref{eq:enq_aux2}, and noticing that the operator $\mathbf{\Pi}$ is invertible, we get the result.
\end{proof}

\subsection{A key identity}

Now we want to prove the following key identity.
\begin{proposition} \label{prop:key_identity}
	For $j\geq 0$ and $i\geq 1$ we have
	\begin{equation} \label{eq:key_identity}
		\nabla \left( h_i\left[\frac{X}{1-q}\right] e_j\left[\frac{X}{M}\right] \right) = 
		\sum_{s =1}^{i+j} t^{i+j-s}q^{\binom{i}{2}} \begin{bmatrix}
			s\\
			i
		\end{bmatrix}_q  h_s\left[\frac{X}{1-q}\right] h_{i+j-s}\left[\frac{X}{M}\right].
	\end{equation}
\end{proposition}

In order to prove this proposition, we need a proposition and a lemma.

The following proposition is proved in \cite{Garsia-Hicks-Stout-2011}*{Proposition~2.6}.
\begin{proposition} 
	For $i\geq 1$ and $j\geq 0$ we have
	\begin{equation} \label{eq:GHS_2_6}
		h_i\left[\frac{X}{1-q}\right] e_j\left[\frac{X}{M}\right]  = \sum_{\mu\vdash i+j} \frac{\widetilde{H}_\mu[X]}{w_\mu} \sum_{r=1}^i \begin{bmatrix}
			i-1\\
			r-1
		\end{bmatrix}_q q^{\binom{r}{2}+r-ir}(-1)^{i-r}h_r[(1-t)B_\mu].
	\end{equation}
\end{proposition}

The following elementary lemma is proved in the appendix of the present article.
\begin{lemma} \label{lem:elementary1}
	For $i\geq 1$, $a\geq -i$ and $s\geq 0$ we have
	\begin{equation} \label{eq:first_qlemma}
		\sum_{r=1}^{i}\begin{bmatrix}
			i-1\\
			r-1
		\end{bmatrix}_q \begin{bmatrix}
			r+s+a-1\\
			s-1
		\end{bmatrix}_q q^{\binom{r}{2}+r-ir}(-1)^{i-r} =q^{\binom{i}{2}+(i-1)a} \begin{bmatrix}
			s+a\\
			i+a
		\end{bmatrix}_q .
	\end{equation}
\end{lemma}

\begin{proof}[Proof of Proposition~\ref{prop:key_identity}]
	We have
	\begin{equation*}
		\nabla \left( h_i\left[\frac{X}{1-q}\right] e_j\left[\frac{X}{M}\right]\right)= \qquad\qquad\qquad\qquad \qquad\qquad
	\end{equation*}
	\begin{align*}
		\text{(using \eqref{eq:GHS_2_6})}	& = \nabla \sum_{\mu\vdash i+j} \frac{\widetilde{H}_\mu[X]}{w_\mu} \sum_{r=1}^i \begin{bmatrix}
			i-1\\
			r-1
		\end{bmatrix}_q q^{\binom{r}{2}+r-ir}(-1)^{i-r}h_r[(1-t)B_\mu]\\
		& = \sum_{r=1}^i \begin{bmatrix}
			i-1\\
			r-1
		\end{bmatrix}_q q^{\binom{r}{2}+r-ir}(-1)^{i-r} \sum_{\mu\vdash i+j} \frac{T_\mu \widetilde{H}_\mu[X]}{w_\mu} h_r[(1-t)B_\mu]\\
		\text{(using \eqref{eq:basic_summation_2})}	& = \sum_{r=1}^i \begin{bmatrix}
			i-1\\
			r-1
		\end{bmatrix}_q q^{\binom{r}{2}+r-ir}(-1)^{i-r}   \sum_{s =1}^{i+j} \begin{bmatrix}
			r+s-1\\
			s-1
		\end{bmatrix}_q h_s\left[\frac{X}{1-q}\right] h_{i+j-s}\left[\frac{tX}{M}\right]  \\
		& = \sum_{s =1}^{i+j} \left(\sum_{r=1}^i \begin{bmatrix}
			i-1\\
			r-1
		\end{bmatrix}_q q^{\binom{r}{2}+r-ir}(-1)^{i-r}   \begin{bmatrix}
			r+s-1\\
			s-1
		\end{bmatrix}_q \right) h_s\left[\frac{X}{1-q}\right] h_{i+j-s}\left[\frac{tX}{M}\right]  \\
		\text{(using \eqref{eq:first_qlemma})}	& = \sum_{s =1}^{i+j} q^{\binom{i}{2}} \begin{bmatrix}
			s\\
			i
		\end{bmatrix}_q  h_s\left[\frac{X}{1-q}\right] h_{i+j-s}\left[\frac{tX}{M}\right] \\
		& = \sum_{s =1}^{i+j} t^{i+j-s}q^{\binom{i}{2}} \begin{bmatrix}
			s\\
			i
		\end{bmatrix}_q  h_s\left[\frac{X}{1-q}\right] h_{i+j-s}\left[\frac{X}{M}\right]  .
	\end{align*}
\end{proof}

\subsection{The main argument}

The main argument follows the main strategy of the proof of Theorem~\ref{thm:Haglund_summation}: we look for an $f[X]\in \Lambda$ such that
\begin{equation}
	\left.\nabla^{-1}\tau_{-\epsilon}f[X]\right|_{X=D_\mu}=h_k[(1-t)B_{\mu}]e_\ell[B_\mu],
\end{equation}
so that
\begin{equation}
	\sum_{\mu\vdash m} \frac{ \widetilde{H}_\mu[MB_\gamma]}{w_\mu}h_k[(1-t)B_{\mu}]e_{\ell}[B_\mu]=\sum_{d=0}^{m}e_{m-d}\left[\frac{X}{M}\right](f[X])_d.
\end{equation}

Clearly, we can use
\begin{align*}
	\nabla^{-1} \tau_{-\epsilon} f[X] & =  h_{k}\left[\frac{X+1}{1-q}\right] e_{\ell}\left[\frac{X+1}{M}\right]\\
	\text{(using \eqref{eq:e_h_sum_alphabets})}& = \sum_{i=0}^{k} \sum_{j=0}^{\ell}h_{k-i}\left[\frac{1}{1-q}\right]e_{\ell-j}\left[\frac{1}{M}\right]h_{i}\left[\frac{X}{1-q}\right] e_{j}\left[\frac{X}{M}\right]
\end{align*}
so
\begin{align*}
	\tau_{-\epsilon}f[X] & = \sum_{i=0}^{k} \sum_{j=0}^{\ell}h_{k-i}\left[\frac{1}{1-q}\right]e_{\ell-j}\left[\frac{1}{M}\right]\nabla \left( h_{i}\left[\frac{X}{1-q}\right] e_{j}\left[\frac{X}{M}\right]\right)\\
	\text{(using \eqref{eq:key_identity})}& =  \sum_{j=0}^{\ell}h_{k}\left[\frac{1}{1-q}\right]e_{\ell-j}\left[\frac{1}{M}\right]
	h_{j}\left[\frac{X}{M}\right]+\\
	& + \sum_{i=1}^{k} \sum_{j=0}^{\ell}h_{k-i}\left[\frac{1}{1-q}\right]e_{\ell-j}\left[\frac{1}{M}\right]
	\sum_{s =1}^{i+j}t^{i+j-s} q^{\binom{i}{2}} \begin{bmatrix}
		s\\
		i
	\end{bmatrix}_q  h_s\left[\frac{X}{1-q}\right] h_{i+j-s}\left[\frac{X}{M}\right]\\
	\text{(using \eqref{eq:e_h_sum_alphabets})} & =  h_{k}\left[\frac{1}{1-q}\right]  h_{\ell}\left[\frac{X-\epsilon}{M}\right]+\\
	& + \sum_{i=1}^{k} \sum_{j=0}^{\ell}h_{k-i}\left[\frac{1}{1-q}\right]e_{\ell-j}\left[\frac{1}{M}\right]
	\sum_{s =1}^{i+j}t^{i+j-s} q^{\binom{i}{2}} \begin{bmatrix}
		s\\
		i
	\end{bmatrix}_q  h_s\left[\frac{X}{1-q}\right] h_{i+j-s}\left[\frac{X}{M}\right]
\end{align*}
and therefore
\begin{align*}
	f[X]
	& =  h_{k}\left[\frac{1}{1-q}\right]  h_{\ell}\left[\frac{X+\epsilon -\epsilon}{M}\right]+\\
	& + \sum_{i=1}^{k} \sum_{j=0}^{\ell}h_{k-i}\left[\frac{1}{1-q}\right]e_{\ell-j}\left[\frac{1}{M}\right]
	\sum_{s =1}^{i+j}t^{i+j-s} q^{\binom{i}{2}} \begin{bmatrix}
		s\\
		i
	\end{bmatrix}_q  h_s\left[\frac{X+\epsilon}{1-q}\right] h_{i+j-s}\left[\frac{X+\epsilon}{M}\right]\\
	& =  h_{k}\left[\frac{1}{1-q}\right]  h_{\ell}\left[\frac{X}{M}\right]+ \sum_{i=1}^{k} \sum_{j=0}^{\ell}h_{k-i}\left[\frac{1}{1-q}\right]e_{\ell-j}\left[\frac{1}{M}\right]
	\sum_{s =1}^{i+j}t^{i+j-s} q^{\binom{i}{2}} \begin{bmatrix}
		s\\
		i
	\end{bmatrix}_q\\
	& \qquad \times \sum_{b=0}^s(-1)^{s-b} h_{s-b}\left[\frac{1}{1-q}\right] \sum_{c=0}^{i+j-s}(-1)^{i+j-s-c} h_{i+j-s-c}\left[\frac{1}{M}\right] h_b\left[\frac{X}{1-q}\right] h_{c}\left[\frac{X}{M}\right],
\end{align*}
where in the last equality we used \eqref{eq:e_h_sum_alphabets}.

From this formula, it is easy to extract the homogeneous component of $f[X]$ of degree $d\geq 0$:
\begin{align*}
	(f[X])_d & = \delta_{d,\ell} h_{k}\left[\frac{1}{1-q}\right]  h_{\ell}\left[\frac{X}{M}\right]\\
	& + \sum_{b=0}^{k+\ell} 
	\sum_{i=1}^{k} 
	\sum_{j=0}^{\ell}
	\sum_{s =1}^{i+j-d+b} 
	h_{k-i}\left[\frac{1}{1-q}\right]e_{\ell-j}\left[\frac{1}{M}\right]\\
	& \qquad \times t^{i+j-s} q^{\binom{i}{2}} \begin{bmatrix}
		s\\
		i
	\end{bmatrix}_q h_{s-b}\left[\frac{1}{1-q}\right] (-1)^{i+j-d} h_{i+j-s-d+b}\left[\frac{1}{M}\right] h_b\left[\frac{X}{1-q}\right] h_{d-b}\left[\frac{X}{M}\right],
\end{align*}
where $\delta_{d,\ell}$ is the Kronecker delta function, i.e. $\delta_{d,\ell}=1$ if $d=\ell$ and $\delta_{d,\ell}=0$ if $d\neq \ell$.

Notice that here and later we use the convention that $e_r=h_r=0$ for $r<0$.

Finally
\begin{align}
	& \hspace{-0.5cm} \sum_{d=0}^{k+\ell}e_{m-d}\left[\frac{X}{M}\right](f[X])_d= \\
	& = e_{m-\ell}\left[\frac{X}{M}\right]h_k\left[\frac{1}{1-q}\right] h_{\ell}\left[\frac{X}{M}\right]+\\
	\notag & + \sum_{d=0}^{k+\ell}\sum_{b=0}^{k+\ell} 
	\sum_{i=1}^{k} 
	\sum_{j=0}^{\ell}
	\sum_{s =1}^{i+j-d+b} 
	h_{k-i}\left[\frac{1}{1-q}\right]e_{\ell-j}\left[\frac{1}{M}\right]\\
	\notag & \qquad \times t^{i+j-s} q^{\binom{i}{2}} \begin{bmatrix}
		s\\
		i
	\end{bmatrix}_q h_{s-b}\left[\frac{1}{1-q}\right] (-1)^{i+j-d} h_{i+j-s-d+b}\left[\frac{1}{M}\right] \\
	\notag & \qquad \times e_{m-d}\left[\frac{X}{M}\right]h_b\left[\frac{X}{1-q}\right] h_{d-b}\left[\frac{X}{M}\right].
\end{align}
If we set
\begin{align*}
	g(\ell,b,d,k)&  \coloneqq  
	\sum_{i=1}^{k} 
	\sum_{j=0}^{\ell}
	\sum_{s =1}^{i+j-d+b} 
	h_{k-i}\left[\frac{1}{1-q}\right]e_{\ell-j}\left[\frac{1}{M}\right]\\
	& \qquad \times t^{i+j-s} q^{\binom{i}{2}} \begin{bmatrix}
		s\\
		i
	\end{bmatrix}_q h_{s-b}\left[\frac{1}{1-q}\right] (-1)^{i+j-d} h_{i+j-s-d+b}\left[\frac{1}{M}\right],
\end{align*}
then we proved the identity
\begin{align} \label{eq:crucial_sum_lemma}
	\notag	& \hspace{-0.5cm} \sum_{\mu\vdash m} \frac{ \widetilde{H}_\mu[X]}{w_\mu}h_k[(1-t)B_{\mu}]e_\ell[B_\mu]=\\
	& = h_k\left[\frac{1}{1-q}\right] e_{m-\ell}\left[\frac{X}{M}\right] h_{\ell}\left[\frac{X}{M}\right]+ \\
	\notag	& + \sum_{d=0}^{k+\ell}\sum_{b=0}^{k+\ell} g(\ell,b,d,k) e_{m-d}\left[\frac{X}{M}\right]h_b\left[\frac{X}{1-q}\right] h_{d-b}\left[\frac{X}{M}\right] .
\end{align} 

Now we want to simplify the coefficients $g(\ell,b,d,k)$. 

\subsection{Simplifying coefficients}

Notice that because of the factor $h_{d-b}[X/M]$ in \eqref{eq:crucial_sum_lemma}, we can set $g(\ell,b,d,k)=0$ for $d<b$.

Using \eqref{eq:minusepsilon}, we have
\begin{align*}
	g(\ell,b,d,k)&  \coloneqq  
	\sum_{i=1}^{k} 
	\sum_{j=0}^{\ell}
	\sum_{s =1}^{i+j-d+b} 
	h_{k-i}\left[\frac{1}{1-q}\right]e_{\ell-j}\left[\frac{1}{M}\right]\\
	& \qquad \times t^{i+j-s} q^{\binom{i}{2}} \begin{bmatrix}
		s\\
		i
	\end{bmatrix}_q h_{s-b}\left[\frac{1}{1-q}\right] (-1)^{i+j-d} h_{i+j-s-d+b}\left[\frac{1}{M}\right] \\
	& = \sum_{i=1}^{k} 
	\sum_{s =1}^{i+\ell-d+b} 
	\sum_{j=\max(0,s-i+d-b)}^{\ell}
	h_{k-i}\left[\frac{1}{1-q}\right]e_{\ell-j}\left[\frac{1}{M}\right]\\
	& \qquad \times t^{i+j-s} q^{\binom{i}{2}} \begin{bmatrix}
		s\\
		i
	\end{bmatrix}_q h_{s-b}\left[\frac{1}{1-q}\right] (-1)^{i+j-d} h_{i+j-s-d+b}\left[\frac{1}{M}\right] \\
	& = \sum_{i=1}^{k} 
	\sum_{s =1}^{i+\ell-d+b} 
	h_{k-i}\left[\frac{1}{1-q}\right]  q^{\binom{i}{2}} \begin{bmatrix}
		s\\
		i
	\end{bmatrix}_q e_{s-b}\left[-\frac{1}{1-q}\right]  \\
	& \qquad \times t^{d-b}\sum_{j=\max(0,s-i+d-b)}^{\ell}
	e_{\ell-j}\left[\frac{1}{M}\right] e_{i+j-s-d+b}\left[\frac{-t}{M}\right].
\end{align*}
Observe that in the last formula, because of the factor $\qbinom{s}{i}_q$, the terms with $s<i$ all vanish. Since $g(\ell,b,d,k)=0$ for $d<b$, the nonzero terms in the last formula can only occur for $s-i+d-b\geq 0$. Therefore we can replace the last formula by
\begin{align*}
	& \hspace{-0.5cm} \sum_{i=1}^{k} 
	\sum_{s =1}^{i+\ell-d+b} 
	h_{k-i}\left[\frac{1}{1-q}\right]  q^{\binom{i}{2}} \begin{bmatrix}
		s\\
		i
	\end{bmatrix}_q e_{s-b}\left[-\frac{1}{1-q}\right] \\
	& \qquad \times t^{d-b}\sum_{j= s-i+d-b}^{\ell}
	e_{\ell-j}\left[\frac{1}{M}\right] e_{i+j-s-d+b}\left[\frac{-t}{M}\right] .
\end{align*}

Using \eqref{eq:e_h_sum_alphabets} this expression becomes
\begin{align*}
	& \hspace{-0.5cm} t^{d-b} \sum_{i=1}^{k} 
	\sum_{s =1}^{i+\ell-d+b} 
	h_{k-i}\left[\frac{1}{1-q}\right]  q^{\binom{i}{2}} \begin{bmatrix}
		s\\
		i
	\end{bmatrix}_q e_{s-b}\left[-\frac{1}{1-q}\right] 
	e_{i+\ell-s-d+b}\left[\frac{1}{1-q}\right].
\end{align*}
Observe that since the factor $\qbinom{s}{i}_q$ vanishes for $s<i$, and we always have $i\geq 1$, the last formula does not change if we start its internal sum with $s=0$. Therefore

\begin{align*}
	& \hspace{-0.5cm} t^{d-b} \sum_{i=1}^{k} 
	\sum_{s =1}^{i+\ell-d+b} 
	h_{k-i}\left[\frac{1}{1-q}\right]  q^{\binom{i}{2}} \begin{bmatrix}
		s\\
		i
	\end{bmatrix}_q e_{s-b}\left[-\frac{1}{1-q}\right] 
	e_{i+\ell-s-d+b}\left[\frac{1}{1-q}\right]=\\
	& = t^{d-b} \sum_{i=1}^{k} 
	\sum_{s =0}^{i+\ell-d+b} 
	h_{k-i}\left[\frac{1}{1-q}\right]  q^{\binom{i}{2}} \begin{bmatrix}
		s\\
		i
	\end{bmatrix}_q e_{s-b}\left[-\frac{1}{1-q}\right] 
	e_{i+\ell-s-d+b}\left[\frac{1}{1-q}\right]\\
	& = t^{d-b} \sum_{i=0}^{k} 
	\sum_{s =0}^{i+\ell-d+b} 
	h_{k-i}\left[\frac{1}{1-q}\right]  q^{\binom{i}{2}} \begin{bmatrix}
		s\\
		i
	\end{bmatrix}_q e_{s-b}\left[-\frac{1}{1-q}\right] 
	e_{i+\ell-s-d+b}\left[\frac{1}{1-q}\right] \\
	& \qquad - t^{d-b} h_{k}\left[\frac{1}{1-q}\right] \sum_{s=0}^{\ell-d+b}e_{s-b}\left[-\frac{1}{1-q}\right] 
	e_{\ell-s-d+b}\left[\frac{1}{1-q}\right]\\
	& = t^{d-b} \sum_{i=0}^{k} 
	\sum_{s =0}^{i+\ell-d+b} 
	h_{k-i}\left[\frac{1}{1-q}\right]  q^{\binom{i}{2}} \begin{bmatrix}
		s\\
		i
	\end{bmatrix}_q e_{s-b}\left[-\frac{1}{1-q}\right] 
	e_{i+\ell-s-d+b}\left[\frac{1}{1-q}\right] \\
	& \qquad - t^{d-b} h_{k}\left[\frac{1}{1-q}\right] e_{\ell-d}\left[\frac{1}{1-q}-\frac{1}{1-q}\right]\\
	& = t^{d-b} \sum_{i=0}^{k} 
	\sum_{s =0}^{i+\ell-d+b} 
	h_{k-i}\left[\frac{1}{1-q}\right]  q^{\binom{i}{2}} \begin{bmatrix}
		s\\
		i
	\end{bmatrix}_q e_{s-b}\left[-\frac{1}{1-q}\right] 
	e_{i+\ell-s-d+b}\left[\frac{1}{1-q}\right] \\
	& \qquad -\delta_{\ell,d} t^{d-b} h_{k}\left[\frac{1}{1-q}\right]
\end{align*}
where we used the usual \eqref{eq:e_h_sum_alphabets}.

At this point we need the following elementary lemma, which we prove in the appendix of the present article.
\begin{lemma} \label{lem:elementary2}
	For $\ell\geq 0$, $d\geq b\geq 0$ and $k\geq 1$ we have
	\begin{align} \label{eq:qlemma3}
		\notag	& \hspace{-0.5cm}  q^{\binom{d-\ell}{2}} \qbinom{b}{d-\ell}_q\qbinom{k+\ell-d+b-1}{k-d+\ell}_q=\\
		& = \sum_{i=0}^{k} 
		\sum_{s =0}^{i+\ell-d+b} 
		h_{k-i}\left[\frac{1}{1-q}\right]  q^{\binom{i}{2}} \begin{bmatrix}
			s\\
			i
		\end{bmatrix}_q e_{s-b}\left[-\frac{1}{1-q}\right] 
		e_{i+\ell-s-d+b}\left[\frac{1}{1-q}\right].
	\end{align}
\end{lemma}
So, using Lemma~\ref{lem:elementary2}, we finally proved
\begin{align*}
	g(\ell,b,d,k)	 & =- \delta_{d,\ell}t^{d-b}h_k\left[\frac{1}{1-q}\right]+  t^{d-b} q^{\binom{d-\ell}{2}}\qbinom{b}{d-\ell}_q \qbinom{k+\ell-d+b-1}{k-d+\ell}_q.
\end{align*}
Replacing this in \eqref{eq:crucial_sum_lemma} gives
\begin{align*} 
	& \hspace{-0.5cm} \sum_{\mu\vdash m} \frac{ \widetilde{H}_\mu[X]}{w_\mu}h_k[(1-t)B_{\mu}]e_\ell[B_\mu]=\\
	& = h_k\left[\frac{1}{1-q}\right] e_{m-\ell}\left[\frac{X}{M}\right] h_{\ell}\left[\frac{X}{M}\right]+ \sum_{d=0}^{k+\ell}\sum_{b=0}^{k+\ell} g(\ell,b,d,k) e_{m-d}\left[\frac{X}{M}\right]h_b\left[\frac{X}{1-q}\right] h_{d-b}\left[\frac{X}{M}\right] \\
	& = h_k\left[\frac{1}{1-q}\right] e_{m-\ell}\left[\frac{X}{M}\right] h_{\ell}\left[\frac{X}{M}\right]\\
	& - \sum_{d=0}^{k+\ell}\sum_{b=0}^{k+\ell}\delta_{d,\ell}t^{d-b}h_k\left[\frac{1}{1-q}\right] e_{m-d}\left[\frac{X}{M}\right]h_b\left[\frac{X}{1-q}\right] h_{d-b}\left[\frac{X}{M}\right]\\
	& + \sum_{d=0}^{k+\ell}\sum_{b=0}^{k+\ell}t^{d-b} q^{\binom{d-\ell}{2}}\qbinom{b}{d-\ell}_q \qbinom{k+\ell-d+b-1}{k-d+\ell}_q e_{m-d}\left[\frac{X}{M}\right]h_b\left[\frac{X}{1-q}\right] h_{d-b}\left[\frac{X}{M}\right] \\
	& = h_k\left[\frac{1}{1-q}\right] e_{m-\ell}\left[\frac{X}{M}\right] h_{\ell}\left[\frac{X}{M}\right]\\
	& - \sum_{b=0}^{k+\ell} h_k\left[\frac{1}{1-q}\right] e_{m-\ell}\left[\frac{X}{M}\right]h_b\left[\frac{X}{1-q}\right] h_{\ell-b}\left[\frac{tX}{M}\right]\\
	& + \sum_{d=0}^{k+\ell}\sum_{b=0}^{k+\ell}t^{d-b} q^{\binom{d-\ell}{2}}\qbinom{b}{d-\ell}_q \qbinom{k+\ell-d+b-1}{k-d+\ell}_q e_{m-d}\left[\frac{X}{M}\right]h_b\left[\frac{X}{1-q}\right] h_{d-b}\left[\frac{X}{M}\right] \\
	& = h_k\left[\frac{1}{1-q}\right] e_{m-\ell}\left[\frac{X}{M}\right] h_{\ell}\left[\frac{X}{M}\right]\\
	& - h_k\left[\frac{1}{1-q}\right] e_{m-\ell}\left[\frac{X}{M}\right] h_{\ell}\left[X\left(\frac{1}{1-q} + \frac{t}{M}\right)\right]\\
	& + \sum_{d=0}^{k+\ell}\sum_{b=0}^{k+\ell}t^{d-b} q^{\binom{d-\ell}{2}}\qbinom{b}{d-\ell}_q \qbinom{k+\ell-d+b-1}{k-d+\ell}_q e_{m-d}\left[\frac{X}{M}\right]h_b\left[\frac{X}{1-q}\right] h_{d-b}\left[\frac{X}{M}\right] \\
	& =  \sum_{d=0}^{k+\ell}\sum_{b=0}^{k+\ell}t^{d-b} q^{\binom{d-\ell}{2}}\qbinom{b}{d-\ell}_q \qbinom{k+\ell-d+b-1}{k-d+\ell}_q e_{m-d}\left[\frac{X}{M}\right]h_b\left[\frac{X}{1-q}\right] h_{d-b}\left[\frac{X}{M}\right] 
\end{align*} 
where we used again \eqref{eq:e_h_sum_alphabets}. Now we get \eqref{eq:mastereq} simply by changing the indices of summation: $d\mapsto s+\ell$, $b\mapsto s+j$.

\section{Three families of plethystic formulae}

We introduce here three (pairs of) families of polynomials, which turn out to be the pieces needed for our recursions.

\medskip

\emph{By convention, all our polynomials are equal to $0$ if any of their indices is negative.}

\medskip

We set 
\begin{equation}
	\widetilde{F}_{n,k}^{(d,\ell)} \coloneqq \< \Delta_{h_{\ell}} \Delta_{e_{n-\ell-d-1}}'E_{n-\ell,k},e_{n-\ell}\>.
\end{equation}
and
\begin{align} \label{eq:rel_F_Ftilde}
	\notag F_{n,k}^{(d,\ell)} &  \coloneqq \< \Delta_{h_{\ell}} \Delta_{e_{n-\ell-d}}E_{n-\ell,k},e_{n-\ell}\>\\
	\text{(using \eqref{eq:deltaprime})}& =\widetilde{F}_{n,k}^{(d,\ell)}+\widetilde{F}_{n,k}^{(d-1,\ell)}.
\end{align}
Observe that, using Theorem~\ref{thm:Haglund_formula} and Lemma~\ref{lem:Mac_hook_coeff}, we have
\begin{align}
	\notag F_{n,k}^{(d,\ell)} & =\< \Delta_{h_{\ell}} \Delta_{e_{n-d-\ell}}E_{n-\ell,k},e_{n-\ell}\>\\ 
	& =t^{n-\ell-k}\left\< \Delta_{h_{n-\ell-k}} e_{n-d}\left[X\frac{1-q^k}{1-q}\right] ,e_{\ell}h_{n-d-\ell}\right\> \\
	\notag & =t^{n-\ell-k}\left\< \Delta_{h_{n-\ell-k}}\Delta_{e_{\ell}} e_{n-d}\left[X\frac{1-q^k}{1-q}\right] ,h_{n-d}\right\> .
\end{align}
Also, using \eqref{eq:rel_F_Ftilde}, we have 
\begin{align} \label{eq:rel_Ftilde_F}
	\notag \widetilde{F}_{n,k}^{(d,\ell)} & =\< \Delta_{h_{\ell}} \Delta_{e_{n-d-\ell-1}}'E_{n-\ell,k},e_{n-\ell}\>\\ 
	& =\sum_{u\geq 1}(-1)^{u-1}\< \Delta_{h_{\ell}} \Delta_{e_{n-(d+u)-\ell}}E_{n-\ell,k},e_{n-\ell}\> \\
	\notag & =\sum_{u\geq 1}(-1)^{u-1} F_{n,k}^{(d+u,\ell)}.
\end{align}

We set 
\begin{equation}
	\widetilde{G}_{n,k}^{(d,\ell)} \coloneqq t^{n-\ell}q^{\binom{k}{2}}\left\<  \Delta_{e_{n-\ell}}\Delta_{e_{n-d-1}}' e_{n}\left[X\frac{1-q^k}{1-q}\right] ,h_{n}\right\>
\end{equation}
and
\begin{align} \label{eq:rel_G_Gtilde}
	\notag G_{n,k}^{(d,\ell)} & \coloneqq  t^{n-\ell}q^{\binom{k}{2}}\left\<  \Delta_{e_{n-\ell}}\Delta_{e_{n-d}} e_{n}\left[X\frac{1-q^k}{1-q}\right] ,h_{n}\right\>\\
	\text{(using \eqref{eq:deltaprime})} & =\widetilde{G}_{n,k}^{(d,\ell)}+\widetilde{G}_{n,k}^{(d-1,\ell)}.
\end{align}
Observe that, using Lemma~\ref{lem:Mac_hook_coeff}, we have
\begin{align*}
	G_{n,k}^{(d,\ell)} &  \coloneqq  t^{n-\ell}q^{\binom{k}{2}}\left\<  \Delta_{e_{n-\ell}}\Delta_{e_{n-d}} e_{n}\left[X\frac{1-q^k}{1-q}\right] ,h_{n}\right\>\\ 
	& = t^{n-\ell}q^{\binom{k}{2}}\left\<  \Delta_{e_{n-\ell}} e_{n}\left[X\frac{1-q^k}{1-q}\right] ,e_{n-d}h_{d}\right\> .
\end{align*}
Also, using \eqref{eq:rel_G_Gtilde}, we have
\begin{align} \label{eq:rel_Gtilde_G}
	\notag \widetilde{G}_{n,k}^{(d,\ell)} & = t^{n-\ell}q^{\binom{k}{2}}\left\<  \Delta_{e_{n-\ell}}\Delta_{e_{n-d-1}}' e_{n}\left[X\frac{1-q^k}{1-q}\right] ,h_{n}\right\>\\ 
	& =\sum_{u\geq 1}(-1)^{u-1}t^{n-\ell}q^{\binom{k}{2}}\left\<  \Delta_{e_{n-\ell}}\Delta_{e_{n-(d+u)}} e_{n}\left[X\frac{1-q^k}{1-q}\right] ,h_{n}\right\>\\
	\notag  & =\sum_{u\geq 1}(-1)^{u-1}G_{n,k}^{(d+u,\ell)}.
\end{align}

We set
\begin{align}
	\overline{H}_{n,n,i}^{(d,\ell)} &  \coloneqq \delta_{i,\ell} q^{\binom{n-\ell-d}{2}}\begin{bmatrix}
		n-1\\
		\ell
	\end{bmatrix}_q\begin{bmatrix}
		n-\ell-1\\
		d
	\end{bmatrix}_q
\end{align}
and for $1\leq k<n$
\begin{align} \label{eq:rel_Htilde_Ftilde}
	\notag \overline{H}_{n,k,i}^{(d,\ell)} &  \coloneqq   t^{n-k}\sum_{s=0}^{k-i}\sum_{h=0}^{n-k}q^{\binom{s}{2}} \begin{bmatrix}
		k-i\\
		s
	\end{bmatrix}_q \begin{bmatrix}
		k-1\\
		i
	\end{bmatrix}_q \begin{bmatrix}
		s+i-1+h\\
		h
	\end{bmatrix}_q \widetilde{F}_{n-k,h}^{(d-k+i+s,\ell-i)}.
\end{align}
Similarly
\begin{align}
	H_{n,n,i}^{(d,\ell)} &  \coloneqq \overline{H}_{n,n,i}^{(d,\ell)}+\overline{H}_{n,n,i}^{(d-1,\ell)}=\delta_{i,\ell} q^{\binom{n-\ell-d}{2}}\begin{bmatrix}
		n-1\\
		\ell
	\end{bmatrix}_q\begin{bmatrix}
		n-\ell\\
		d
	\end{bmatrix}_q
\end{align}
and for $1\leq k<n$
\begin{align}   \label{eq:rel_H_F}
	H_{n,k,i}^{(d,\ell)} &  \coloneqq  \overline{H}_{n,k,i}^{(d,\ell)}+\overline{H}_{n,k,i}^{(d-1,\ell)}\\
	\text{(using \eqref{eq:rel_F_Ftilde})}& =  t^{n-k}\sum_{s=0}^{k-i}\sum_{h=0}^{n-k}q^{\binom{s}{2}} \begin{bmatrix}
		k-i\\
		s
	\end{bmatrix}_q \begin{bmatrix}
		k-1\\
		i
	\end{bmatrix}_q \begin{bmatrix}
		s+i-1+h\\
		h
	\end{bmatrix}_q F_{n-k,h}^{(d-k+i+s,\ell-i)}.
\end{align}
Notice that, using the definitions, we have
\begin{align} \label{eq:rel_Htilde_H}
	\overline{H}_{n,k,i}^{(d,\ell)} & =\sum_{u\geq 1}(-1)^{u-1}H_{n,k,i}^{(d+u,\ell)}.
\end{align}

\subsection{Recursive relations among the families}

In this section we establish some recursive relations among our families.

\begin{theorem}
	For $k,\ell,d\geq 0$, $n\geq k+\ell$ and $n\geq d$, we have
	\begin{align} \label{eq:Fnn_formula}
		F_{n,n}^{(d,\ell)} & = \delta_{\ell,0} q^{\binom{n-d}{2}}\begin{bmatrix}
			n\\
			d
		\end{bmatrix}_q  ,
	\end{align}
	and for $n>k$
	\begin{align} \label{eq:rel_F_G}
		F_{n,k}^{(d,\ell)} & =\sum_{s=1}^{\min(k,n-d)} 
		\begin{bmatrix}
			k\\
			s\\
		\end{bmatrix}_q G_{n-k,s}^{(d-k+s,\ell)}.
	\end{align}
	Similarly
	\begin{align} \label{eq:Ftildenn_formula}
		\widetilde{F}_{n,n}^{(d,\ell)} & = \delta_{\ell,0} q^{\binom{n-d}{2}}\begin{bmatrix}
			n-1\\
			d
		\end{bmatrix}_q  ,
	\end{align}
	and for $n>k$
	\begin{align} \label{eq:rel_Ftilde_Gtilde}
		\widetilde{F}_{n,k}^{(d,\ell)} & =\sum_{s=1}^{\min(k,n-d)} 
		\begin{bmatrix}
			k\\
			s\\
		\end{bmatrix}_q \widetilde{G}_{n-k,s}^{(d-k+s,\ell)}.
	\end{align}
\end{theorem}
\begin{proof}
	Notice that
	\begin{equation}
		F_{n,n}^{(d,\ell)}=\< \Delta_{h_{\ell}} \Delta_{e_{n-\ell-d}}E_{n-\ell,n},e_{n-\ell}\>
	\end{equation}
	so clearly it is equal to $0$ when $\ell\neq 0$. For $\ell=0$ we get
	\begin{align}
		F_{n,n}^{(d,0)} & =\< \Delta_{e_{n-d}}E_{n ,n},e_{n }\>\\
		\text{(using \eqref{eq:Mac_hook_coeff})} & =\< \Delta_{e_{n-d}}\nabla E_{n ,n},h_{n }\>.
	\end{align}
	So using \cite{Haglund-Schroeder-2004}*{Proposition~2.3}, we get
	\begin{align}
		F_{n,n}^{(d,0)} & =e_{n-d}[[n]_q]\\
		\text{(using \eqref{eq:e_q_binomial})} & = q^{\binom{n-d}{2}} \qbinom{n}{d}_q.
	\end{align}
	If we set
	\begin{equation}
		f(n,d) \coloneqq q^{\binom{n-d}{2}} \qbinom{n-1}{d}_q,
	\end{equation}
	then
	\begin{align}
		f(n,d)+f(n,d-1)& =q^{\binom{n-d}{2}}\begin{bmatrix}
			n-1\\
			d
		\end{bmatrix}_q+q^{\binom{n-d+1}{2}}\begin{bmatrix}
			n-1\\
			d-1
		\end{bmatrix}_q\\
		& =q^{\binom{n-d}{2}}\left(\begin{bmatrix}
			n-1\\
			d
		\end{bmatrix}_q+q^{n-d}\begin{bmatrix}
			n-1\\
			d-1
		\end{bmatrix}_q\right)\\
		& = q^{\binom{n-d}{2}}\begin{bmatrix}
			n\\
			d
		\end{bmatrix}_q,
	\end{align}
	so that
	\begin{align}
		q^{\binom{n-d}{2}}\begin{bmatrix}
			n-1\\
			d
		\end{bmatrix}_q & =f(n,d)\\
		& =\sum_{u\geq 1}(-1)^{u-1}q^{\binom{n-(d+u)}{2}}\begin{bmatrix}
			n\\
			d+u
		\end{bmatrix}_q\\
		& =\sum_{u\geq 1}(-1)^{u-1}F_{n,n}^{(d+u,0)}.
	\end{align}
	
	Therefore \eqref{eq:Ftildenn_formula} follows from \eqref{eq:Fnn_formula} using \eqref{eq:rel_Ftilde_F}. 
	
	Similarly, \eqref{eq:rel_Ftilde_Gtilde} follows from \eqref{eq:rel_F_G} using \eqref{eq:rel_Ftilde_F} and \eqref{eq:rel_Gtilde_G}. Hence it remains to prove only \eqref{eq:rel_F_G}.
	
	Using \eqref{eq:qn_q_Macexp} and Lemma~\ref{lem:Mac_hook_coeff}, we have
	\begin{align*}
		F_{n,k}^{(d,\ell)} & = t^{n-\ell-k}\left\< \Delta_{h_{n-\ell-k}}\Delta_{e_{\ell}} e_{n-d}\left[X\frac{1-q^k}{1-q}\right] ,h_{n-d}\right\> \\
		& =t^{n-\ell-k} (1-q^k)\sum_{\gamma\vdash n-d}\frac{\Pi_\gamma}{w_\gamma} h_k[(1-t)B_\gamma] h_{n-\ell-k}[B_\gamma]e_{\ell}[B_\gamma]\\
		\text{(using \eqref{eq:e_h_expansion})}& =t^{n-\ell-k} (1-q^k)\sum_{\gamma\vdash n-d}\frac{\Pi_\gamma}{w_\gamma} h_k[(1-t)B_\gamma] \sum_{\nu\vdash n-k} e_{n-k-\ell}[B_\nu]\frac{\widetilde{H}_\nu[MB_\gamma]}{w_\nu} \\
		\text{(using \eqref{eq:Macdonald_reciprocity})}& =t^{n-\ell-k} (1-q^k)\sum_{\nu\vdash n-k} e_{n-k-\ell}[B_\nu] \frac{\Pi_\nu}{w_\nu}\sum_{\gamma\vdash n-d} h_k[(1-t)B_\gamma]  \frac{\widetilde{H}_\gamma[MB_\nu]}{w_\gamma}\\
		\text{(using \eqref{eq:Haglund_summation})}& =t^{n-\ell-k} (1-q^k)\sum_{\nu\vdash n-k} e_{n-k-\ell}[B_\nu] \frac{\Pi_\nu}{w_\nu}\\
		& \qquad \times \sum_{s=1}^{\min(k,n-d)}
		q^{\binom{s}{2}} 
		\begin{bmatrix}
			k-1\\
			s-1\\
		\end{bmatrix}_q e_{n-d-s}[B_\nu]h_s[(1-t)B_\nu]\\
		& =t^{n-\ell-k}  \sum_{s=1}^{\min(k,n-d)}
		q^{\binom{s}{2}} 
		\begin{bmatrix}
			k\\
			s\\
		\end{bmatrix}_q\\
		& \qquad \times (1-q^s)\sum_{\nu\vdash n-k}  \frac{\Pi_\nu}{w_\nu} e_{n-k-\ell}[B_\nu] e_{n-d-s}[B_\nu]h_s[(1-t)B_\nu] \\
		\text{(using \eqref{eq:qn_q_Macexp})}& =t^{n-\ell-k}  \sum_{s=1}^{\min(k,n-d)}
		q^{\binom{s}{2}} 
		\begin{bmatrix}
			k\\
			s\\
		\end{bmatrix}_q \left\<  \Delta_{e_{n-k-\ell}}\Delta_{e_{n-d-s}} e_{n-k}\left[X\frac{1-q^s}{1-q}\right] ,h_{n-k}\right\> \\
		& =  \sum_{s=1}^{\min(k,n-d)} 
		\begin{bmatrix}
			k\\
			s\\
		\end{bmatrix}_q   G_{n-k,s}^{(d-k+s,\ell)} .
	\end{align*}	
\end{proof}

\begin{theorem} 
	For $k,\ell,d\geq 0$, $n\geq k+\ell$ and $n\geq d$, we have
	\begin{align} \label{eq:rel_G_F}
		G_{n,k}^{(d,\ell)}& =\sum_{h=1}^{n} \sum_{s=0}^{\ell} t^{n-\ell}q^{\binom{k}{2}+\binom{s}{2}}\begin{bmatrix}
			k\\
			s\\
		\end{bmatrix}_q \begin{bmatrix}
			h+k-s-1\\
			h-s\\
		\end{bmatrix}_q  F_{n,h}^{(d,\ell-s)}\\
		\notag	& =\sum_{h=1}^{n}\sum_{s=0}^{\ell}t^{n-\ell}q^{\binom{k}{2}+\binom{s+1}{2}}\begin{bmatrix}
			h-1\\
			s\\
		\end{bmatrix}_q \begin{bmatrix}
			h+k-s-1\\
			h\\
		\end{bmatrix}_q(F_{n,h}^{(d,\ell-s)}+F_{n,h}^{(d,\ell-s-1)}),
	\end{align}
	and
	\begin{align} \label{eq:rel_Gtilde_Ftilde}
		\widetilde{G}_{n,k}^{(d,\ell)}& =\sum_{h=1}^{n} \sum_{s=0}^{\ell} t^{n-\ell}q^{\binom{k}{2}+\binom{s}{2}}\begin{bmatrix}
			k\\
			s\\
		\end{bmatrix}_q \begin{bmatrix}
			h+k-s-1\\
			h-s\\
		\end{bmatrix}_q  \widetilde{F}_{n,h}^{(d,\ell-s)}\\
		\notag	& =\sum_{h=1}^{n}\sum_{s=0}^{\ell}t^{n-\ell}q^{\binom{k}{2}+\binom{s+1}{2}}\begin{bmatrix}
			h-1\\
			s\\
		\end{bmatrix}_q \begin{bmatrix}
			h+k-s-1\\
			h\\
		\end{bmatrix}_q(\widetilde{F}_{n,h}^{(d,\ell-s)}+\widetilde{F}_{n,h}^{(d,\ell-s-1)}).
	\end{align}
\end{theorem}
\begin{proof}
	Clearly equations \eqref{eq:rel_Gtilde_Ftilde} follow from equations \eqref{eq:rel_G_F} using \eqref{eq:rel_F_Ftilde} and \eqref{eq:rel_G_Gtilde}. Hence we will prove only equations \eqref{eq:rel_G_F}.
	
	Using \eqref{eq:qn_q_Macexp} and \eqref{eq:Mac_hook_coeff}, we have
	\begin{align*}
		G_{n,k}^{(d,\ell)} & =t^{n-\ell}q^{\binom{k}{2}}\left\<  \Delta_{e_{n-\ell}}\Delta_{e_{n-d}} e_{n}\left[X\frac{1-q^k}{1-q}\right] ,h_{n}\right\>	 \\
		&= t^{n-\ell}q^{\binom{k}{2}}(1-q^k)\sum_{\nu\vdash n} \frac{\Pi_\nu}{w_\nu}h_k[(1-t)B_\nu]e_{n-\ell}[B_\nu]e_{n-d}[B_\nu] \\
		\text{(using \eqref{eq:e_h_expansion})}&= t^{n-\ell}q^{\binom{k}{2}}(1-q^k)\sum_{\nu\vdash n} \frac{\Pi_\nu}{w_\nu}h_k[(1-t)B_\nu]e_{n-\ell}[B_\nu]\sum_{\gamma\vdash n-d}\frac{\widetilde{H}_\gamma[MB_\nu]}{w_\gamma}  \\
		\text{(using \eqref{eq:Macdonald_reciprocity})}&= t^{n-\ell}q^{\binom{k}{2}}\sum_{\gamma\vdash n-d} \frac{\Pi_\gamma}{w_\gamma} (1-q^k)\sum_{\nu\vdash n} \frac{\widetilde{H}_\nu[MB_\gamma]}{w_\nu}h_k[(1-t)B_\nu]e_{n-\ell}[B_\nu]\\
		\text{(using \eqref{eq:mastereq})} & =t^{n-\ell} \sum_{\gamma\vdash n-d} \frac{\Pi_\gamma}{w_\gamma} (1-q^k)\sum_{j=0}^{n-\ell} t^{n-\ell-j}\sum_{s=0}^{\ell}q^{\binom{k}{2}+\binom{s}{2}} \begin{bmatrix}
			s+j\\
			s
		\end{bmatrix}_q \begin{bmatrix}
			k+j-1\\
			s+j-1
		\end{bmatrix}_q\\
		& \qquad \times h_{s+j}\left[(1-t)B_\gamma\right] h_{n-\ell-j}\left[B_\gamma\right] e_{\ell-s}\left[B_\gamma\right]\\
		& =t^{n-\ell}\sum_{\gamma\vdash n-d} \frac{\Pi_\gamma}{w_\gamma} \sum_{j=0}^{n-\ell} t^{n-\ell-j}\sum_{s=0}^{\ell}q^{\binom{k}{2}+\binom{s}{2}} \begin{bmatrix}
			k\\
			s
		\end{bmatrix}_q \begin{bmatrix}
			k+j-1\\
			j
		\end{bmatrix}_q\\
		& \qquad \times (1-q^{s+j}) h_{s+j}\left[(1-t)B_\gamma\right] h_{n-\ell-j}\left[B_\gamma\right] e_{\ell-s}\left[B_\gamma\right],
	\end{align*}
	where in the last equality we used the following elementary lemma, whose proof is in the appendix.
	\begin{lemma} \label{lem:elementary3}
		For all $k,s,j\geq 0$
		\begin{equation}
			\begin{bmatrix}
				k\\
				s
			\end{bmatrix}_q \begin{bmatrix}
				k+j-1\\
				j
			\end{bmatrix}_q (1-q^{s+j}) =(1-q^k) \begin{bmatrix}
				s+j\\
				s
			\end{bmatrix}_q \begin{bmatrix}
				k+j-1\\
				k-s
			\end{bmatrix}_q .
		\end{equation}
	\end{lemma}
	Rearranging the terms, and using \eqref{eq:qn_q_Macexp} and Lemma~\ref{lem:Mac_hook_coeff}, we get	
	\begin{align*}
		G_{n,k}^{(d,\ell)}&= \sum_{j=0}^{n-\ell} \sum_{s=0}^{\ell}t^{n-\ell}q^{\binom{k}{2}+\binom{s}{2}} \begin{bmatrix}
			k\\
			s
		\end{bmatrix}_q \begin{bmatrix}
			k+j-1\\
			j
		\end{bmatrix}_qt^{n-\ell-j}\\
		& \qquad \times \sum_{\gamma\vdash n-d} \frac{\Pi_\gamma}{w_\gamma}(1-q^{s+j}) h_{s+j}\left[(1-t)B_\gamma\right] h_{n-\ell-j}\left[B_\gamma\right] e_{\ell-s}\left[B_\gamma\right]\\
		& =\sum_{j=0}^{n-\ell} \sum_{s=0}^{\ell}t^{n-\ell} q^{\binom{k}{2}+\binom{s}{2}}\begin{bmatrix}
			k\\
			s\\
		\end{bmatrix}_q \begin{bmatrix}
			k+j-1\\
			j\\
		\end{bmatrix}_q  t^{n-\ell-j}\\
		& \qquad \times \left\< \Delta_{h_{n-\ell-j}} \Delta_{e_{\ell-s}} e_{n-d}\left[X[s+j]_q\right],h_{n-d}\right\>\\
		& =\sum_{j=0}^{n-\ell} \sum_{s=0}^{\ell}t^{n-\ell} q^{\binom{k}{2}+\binom{s}{2}}\begin{bmatrix}
			k\\
			s\\
		\end{bmatrix}_q \begin{bmatrix}
			k+j-1\\
			j\\
		\end{bmatrix}_q  F_{n,s+j}^{(d,\ell-s)}\\
		& =\sum_{h=1}^{n} \sum_{s=0}^{\ell} t^{n-\ell}q^{\binom{k}{2}+\binom{s}{2}}\begin{bmatrix}
			k\\
			s\\
		\end{bmatrix}_q \begin{bmatrix}
			h+k-s-1\\
			h-s\\
		\end{bmatrix}_q  F_{n,h}^{(d,\ell-s)}.
	\end{align*}
	For the last equality in \eqref{eq:rel_G_F} we can use Lemma~\ref{lem:elementary4}: we have
	\begin{align*}
		G_{n,k}^{(d,\ell)}& =\sum_{h=1}^{n} \sum_{s=0}^{\ell} t^{n-\ell}q^{\binom{k}{2}+\binom{s}{2}}\begin{bmatrix}
			k\\
			s\\
		\end{bmatrix}_q \begin{bmatrix}
			h+k-s-1\\
			h-s\\
		\end{bmatrix}_q  F_{n,h}^{(d,\ell-s)}\qquad \qquad \qquad \qquad 
	\end{align*}
	\begin{align*}
		& =\sum_{h=1}^{n} \sum_{s=0}^{\ell}t^{n-\ell}q^{\binom{k}{2}} \left(q^{\binom{s+1}{2}}\begin{bmatrix}
			h-1\\
			s\\
		\end{bmatrix}_q \begin{bmatrix}
			h+k-s-1\\
			h\\
		\end{bmatrix}_q +q^{\binom{s}{2}}\begin{bmatrix}
			h-1\\
			s-1\\
		\end{bmatrix}_q \begin{bmatrix}
			h+k-s\\
			h\\
		\end{bmatrix}_q\right)\\
		& \qquad \times F_{n,h}^{(d,\ell-s)}\\
		& =\sum_{h=1}^{n}\sum_{s=0}^{\ell}t^{n-\ell}q^{\binom{k}{2}+\binom{s+1}{2}}\begin{bmatrix}
			h-1\\
			s\\
		\end{bmatrix}_q \begin{bmatrix}
			h+k-s-1\\
			h\\
		\end{bmatrix}_q(F_{n,h}^{(d,\ell-s)}+F_{n,h}^{(d,\ell-s-1)}),
	\end{align*}
	as we wanted.
\end{proof}
\begin{theorem} \label{thm:rel_F_H}
	For $k,\ell,d\geq 0$, $n\geq k+\ell$ and $n\geq d$, we have
	\begin{equation} \label{eq:rel_F_and_H}
		F_{n,k}^{(d,\ell)}= \sum_{j=0}^{\min(n-k,\ell)}H_{n,k+j,j}^{(d,\ell)},
	\end{equation}
	and
	\begin{equation} \label{eq:rel_Ftilde_and_Hbar}
		\widetilde{F}_{n,k}^{(d,\ell)}= \sum_{j=0}^{\min(n-k,\ell)}\overline{H}_{n,k+j,j}^{(d,\ell)}.
	\end{equation}
\end{theorem}
\begin{proof}
	Clearly \eqref{eq:rel_Ftilde_and_Hbar} follows from \eqref{eq:rel_F_and_H}  using \eqref{eq:rel_F_Ftilde} and \eqref{eq:rel_H_F}. Hence we will prove only \eqref{eq:rel_F_and_H}.
	
	For $n=k+\ell$, using \eqref{eq:qn_q_Macexp} and Lemma~\ref{lem:Mac_hook_coeff}, we have
	\begin{align*}
		F_{n,k}^{(d,\ell)} & =\< \Delta_{h_{\ell}} \Delta_{e_{n-\ell-d}}E_{n-\ell,k},e_{n-\ell}\>\\
		& =t^{n-\ell-k}\left\< \Delta_{h_{n-\ell-k}}\Delta_{e_{\ell}} e_{n-d}\left[X\frac{1-q^k}{1-q}\right] ,h_{n-d}\right\> \\
		& = t^{n-\ell-k}(1-q^k)\sum_{\gamma\vdash n-d}\frac{\Pi_\gamma}{w_\gamma}h_k[(1-t)B_\gamma]h_{n-\ell-k}[B_\gamma]e_\ell[B_\gamma]\\
		& = (1-q^k)\sum_{\gamma\vdash n-d}\frac{\Pi_\gamma}{w_\gamma}h_k[(1-t)B_\gamma]e_\ell[B_\gamma]\\
		\text{(using \eqref{eq:garsia_haglund_eval})}& = \sum_{\gamma\vdash n-d}\frac{\widetilde{H}_\gamma[M[k]_q]}{w_\gamma}e_\ell[B_\gamma]\\
		\text{(using \eqref{eq:e_h_expansion})}& = h_\ell[[k]_q]e_{n-\ell-d}[[k]]_q\\
		\text{(using \eqref{eq:h_q_binomial})}& = \begin{bmatrix}
			k+\ell-1\\
			\ell
		\end{bmatrix}_qe_{k-d}[[k]]_q\\
		\text{(using \eqref{eq:e_q_binomial})}& = q^{\binom{k-d}{2}}\begin{bmatrix}
			k\\
			d
		\end{bmatrix}_q\begin{bmatrix}
			k+\ell-1\\
			\ell
		\end{bmatrix}_q.
	\end{align*}
	On the other hand, using \eqref{eq:en_q_sum_Enk},
	\begin{align*}
		\sum_{j=0}^{\min(n-k,\ell)}H_{n,k+j,j}^{(d,\ell)}& = \sum_{j=0}^{\min(n-k,\ell)} t^{n-k-j}\sum_{s=0}^{k}q^{\binom{s}{2}} \begin{bmatrix}
			k\\
			s
		\end{bmatrix}_q \begin{bmatrix}
			k+j-1\\
			j
		\end{bmatrix}_q \qquad \qquad \qquad \\
		& \qquad \times \left\< \Delta_{h_{\ell-j}} \Delta_{e_{n-d-s-\ell}}e_{n-k-\ell}\left[X\frac{1-q^{s+j}}{1-q}\right],e_{n-k-\ell} \right\>\\
		& = \sum_{j=0}^{\ell} t^{n-k-j}\sum_{s=0}^{k}q^{\binom{s}{2}} \begin{bmatrix}
			k\\
			s
		\end{bmatrix}_q \begin{bmatrix}
			k+j-1\\
			j
		\end{bmatrix}_q\\
		& \qquad \times \left\< \Delta_{h_{\ell-j}} \Delta_{e_{k-d-s}}e_{0}\left[X\frac{1-q^{s+j}}{1-q}\right],e_{0} \right\>\\
		& = q^{\binom{k-d}{2}}\begin{bmatrix}
			k\\
			d
		\end{bmatrix}_q\begin{bmatrix}
			k+\ell-1\\
			\ell
		\end{bmatrix}_q,
	\end{align*}
	which is the term $s=k-d$ of the internal sum. This proves the case $n=k+\ell$.
	
	For $n>k+\ell$, using \eqref{eq:qn_q_Macexp} and Lemma~\ref{lem:Mac_hook_coeff}, we have
	\begin{align*}
		F_{n,k}^{(d,\ell)} & =\< \Delta_{h_{\ell}} \Delta_{e_{n-\ell-d}}E_{n-\ell,k},e_{n-\ell}\>\\
		& =t^{n-\ell-k}\left\< \Delta_{h_{n-\ell-k}}\Delta_{e_{\ell}} e_{n-d}\left[X\frac{1-q^k}{1-q}\right] ,h_{n-d}\right\> \\
		& = t^{n-\ell-k}(1-q^k)\sum_{\gamma\vdash n-d}\frac{\Pi_\gamma}{w_\gamma}h_k[(1-t)B_\gamma]h_{n-\ell-k}[B_\gamma]e_\ell[B_\gamma]\\
		\text{(using \eqref{eq:e_h_expansion})}& = t^{n-\ell-k}(1-q^k)\sum_{\gamma\vdash n-d}\frac{\Pi_\gamma}{w_\gamma}h_k[(1-t)B_\gamma]\sum_{\mu\vdash n-\ell-k}T_\mu\frac{\widetilde{H}_\mu[MB_\gamma]}{w_\mu} e_\ell[B_\gamma]
	\end{align*}
	\begin{align*}
		\text{(using \eqref{eq:Macdonald_reciprocity})}& = t^{n-\ell-k}\sum_{\mu\vdash n-\ell-k} T_\mu\frac{\Pi_\mu}{w_\mu} (1-q^k)\sum_{\gamma\vdash n-d}\frac{\widetilde{H}_\gamma[MB_\mu]}{w_\gamma} h_k[(1-t)B_\gamma]e_\ell[B_\gamma].
	\end{align*}
	On the other hand, using \eqref{eq:en_q_sum_Enk},
	\begin{align*}
		\sum_{j=0}^{\min(n-k,\ell)}H_{n,k+j,j}^{(d,\ell)}& = \sum_{j=0}^{\min(n-k,\ell)} t^{n-k-j}\sum_{s=0}^{k}q^{\binom{s}{2}} \begin{bmatrix}
			k\\
			s
		\end{bmatrix}_q \begin{bmatrix}
			k+j-1\\
			j
		\end{bmatrix}_q\\
		& \qquad \times \left\< \Delta_{h_{\ell-j}} \Delta_{e_{n-d-s-\ell}}e_{n-k-\ell}\left[X\frac{1-q^{s+j}}{1-q}\right],e_{n-k-\ell} \right\>\\
		\text{(using \eqref{eq:qn_q_Macexp})}& = t^{n-k-\ell} \sum_{j=0}^{\min(n-k,\ell)} t^{\ell-j}\sum_{s=0}^{k}q^{\binom{s}{2}} \begin{bmatrix}
			k\\
			s
		\end{bmatrix}_q \begin{bmatrix}
			k+j-1\\
			j
		\end{bmatrix}_q\\
		& \qquad \times (1-q^{s+j}) \sum_{\mu\vdash n-\ell-k} \frac{\Pi_\mu}{w_\mu}h_{s+j}[(1-t)B_\mu]T_\mu h_{\ell-j}[B_\mu] e_{n-d-s-\ell}[B_\mu] \\
		\qquad 	& = t^{n-k-\ell} \sum_{\mu\vdash n-\ell-k} T_\mu \frac{\Pi_\mu}{w_\mu} \sum_{j=0}^{\min(n-k,\ell)} t^{\ell-j}\sum_{s=0}^{k}q^{\binom{s}{2}} \begin{bmatrix}
			k\\
			s
		\end{bmatrix}_q \begin{bmatrix}
			k+j-1\\
			j
		\end{bmatrix}_q\\
		& \qquad \times (1-q^{s+j})  h_{s+j}[(1-t)B_\mu] h_{\ell-j}[B_\mu] e_{n-d-s-\ell}[B_\mu]
	\end{align*}
	\begin{align*}
		\qquad \quad & = t^{n-k-\ell} \sum_{\mu\vdash n-\ell-k} T_\mu \frac{\Pi_\mu}{w_\mu}(1-q^k) \sum_{j=0}^{\min(n-k,\ell)} t^{\ell-j}\sum_{s=0}^{k}q^{\binom{s}{2}} \begin{bmatrix}
			s+j\\
			s
		\end{bmatrix}_q \begin{bmatrix}
			k+j-1\\
			k-s
		\end{bmatrix}_q\\
		& \qquad \times   h_{s+j}[(1-t)B_\mu] h_{\ell-j}[B_\mu] e_{n-d-s-\ell}[B_\mu],
	\end{align*}
	where in the last equality we used Lemma~\ref{lem:elementary3}.
	
	Now comparing the result of the last two computations and using \eqref{eq:mastereq}, we complete our proof.
\end{proof}

\subsection{Two recursions for $F_{n,k}^{(d,\ell)}$ and $\widetilde{F}_{n,k}^{(d,\ell)}$}

We are now able to establish two recursions for $F_{n,k}^{(d,\ell)}$ and $\widetilde{F}_{n,k}^{(d,\ell)}$.

\begin{theorem} \label{thm:reco1_F_and_Ftilde}
	For $k,\ell,d\geq 0$, $n\geq k+\ell$ and $n\geq d$, the $F_{n,k}^{(d,\ell)}$ satisfy the following recursion: for $n\geq 1$
	\begin{align}
		F_{n,n}^{(d,\ell)} & =\delta_{\ell,0}  q^{\binom{n-d}{2}}\begin{bmatrix}
			n\\
			d
		\end{bmatrix}_q  ,
	\end{align}
	and, for $n\geq 1$ and $1\leq k<n$,
	\begin{align} \label{eq:reco1_F_F}
		F_{n,k}^{(d,\ell)} & = \sum_{s=0}^{k}\sum_{i=0}^{\ell}q^{\binom{s}{2}+\binom{i+1}{2}}t^{n-k-\ell}\begin{bmatrix}
			k\\
			s
		\end{bmatrix}_q\sum_{h=1}^{n-k} \begin{bmatrix}
			h-1\\
			i
		\end{bmatrix}_q \\
		\notag & \quad \times \begin{bmatrix}
			h+s-i-1\\
			h
		\end{bmatrix}_q (F_{n-k,h}^{(d-k+s,\ell-i)}+F_{n-k,h}^{(d-k+s,\ell-i-1)}),
	\end{align}
	with initial conditions
	\begin{align}
		F_{0,k}^{(d,\ell)} & = \delta_{k,0}\delta_{\ell,0}\delta_{d,0}, \qquad F_{n,0}^{(d,\ell)} = \delta_{n,0}\delta_{\ell,0}\delta_{d,0}.
	\end{align}
	
	Similarly, the $\widetilde{F}_{n,k}^{(d,\ell)}$ satisfy the following recursion: for $n\geq 1$
	\begin{align}
		\widetilde{F}_{n,n}^{(d,\ell)} & = \delta_{\ell,0}  q^{\binom{n-d}{2}}\begin{bmatrix}
			n-1\\
			d
		\end{bmatrix}_q  ,
	\end{align}
	and, for $n\geq 1$ and $1\leq k<n$,
	\begin{align} \label{eq:reco1_Ftilde_Ftilde}
		\widetilde{F}_{n,k}^{(d,\ell)} & = \sum_{s=0}^{k}\sum_{i=0}^{\ell}q^{\binom{s}{2}+\binom{i+1}{2}}t^{n-k-\ell}\begin{bmatrix}
			k\\
			s
		\end{bmatrix}_q\sum_{h=1}^{n-k} \begin{bmatrix}
			h-1\\
			i
		\end{bmatrix}_q \\
		\notag & \quad \times \begin{bmatrix}
			h+s-i-1\\
			h
		\end{bmatrix}_q (\widetilde{F}_{n-k,h}^{(d-k+s,\ell-i)}+\widetilde{F}_{n-k,h}^{(d-k+s,\ell-i-1)}),
	\end{align}
	with initial conditions
	\begin{align}
		\widetilde{F}_{0,k}^{(d,\ell)} & =  \widetilde{F}_{n,0}^{(d,\ell)} =0.
	\end{align}
\end{theorem}
\begin{proof}
	Combining \eqref{eq:rel_F_G} and \eqref{eq:rel_G_F} we immediately get \eqref{eq:reco1_F_F}. Similarly, combining \eqref{eq:rel_Ftilde_Gtilde} and \eqref{eq:rel_Gtilde_Ftilde} we immediately get \eqref{eq:reco1_Ftilde_Ftilde}. The remaining initial conditions are straightforward to check.
\end{proof}

\begin{theorem} \label{thm:dvw}
	For $k,\ell,d\geq 0$, $n\geq k+\ell$ and $n\geq d$, the $F_{n,k}^{(d,\ell)}$ satisfy the following recursion: for $n\geq 1$
	\begin{align}
		F_{n,n}^{(d,\ell)} & =\delta_{\ell,0} q^{\binom{n-d}{2}}\begin{bmatrix}
			n\\
			d
		\end{bmatrix}_q,
	\end{align}
	\begin{align*}
		F_{n,k}^{(d,\ell)}
		& = \delta_{n,k+\ell}\; q^{\binom{k-d}{2}}\begin{bmatrix}
			n-1\\
			\ell
		\end{bmatrix}_q\begin{bmatrix}
			k\\
			d
		\end{bmatrix}_q\\
		& \quad + \sum_{j=0}^{\min(n-k,\ell)}t^{n-k-j}\sum_{s=0}^{k}\sum_{h=1}^{n-k-j}q^{\binom{s}{2}} \begin{bmatrix}
			k\\
			s
		\end{bmatrix}_q \begin{bmatrix}
			k+j-1\\
			j
		\end{bmatrix}_q \begin{bmatrix}
			s+j-1+h\\
			h
		\end{bmatrix}_q F_{n-k-j,h}^{(d-k+s,\ell-j)}.
	\end{align*}
	with initial conditions
	\begin{align}
		F_{0,k}^{(d,\ell)} & = \delta_{k,0}\delta_{\ell,0}\delta_{d,0}, \qquad F_{n,0}^{(d,\ell)} = \delta_{n,0}\delta_{\ell,0}\delta_{d,0}.
	\end{align}
	Similarly, the $\widetilde{F}_{n,k}^{(d,\ell)}$ satisfy the following recursion: for $n\geq 1$
	\begin{align}
		\widetilde{F}_{n,n}^{(d,\ell)} & = \delta_{\ell,0} q^{\binom{n-d}{2}}\begin{bmatrix}
			n-1\\
			d
		\end{bmatrix}_q   ,
	\end{align}
	and, for $n\geq 1$ and $1\leq k<n$, 
	\begin{align*}
		\widetilde{F}_{n,k}^{(d,\ell)} & = \delta_{n,k+\ell} q^{\binom{k-d}{2}}\begin{bmatrix}
			n-1\\
			\ell
		\end{bmatrix}_q\begin{bmatrix}
			k-1\\
			d
		\end{bmatrix}_q\\
		& \quad + \sum_{j=0}^{\min(n-k,\ell)}t^{n-k-j}\sum_{s=0}^{k}\sum_{h=1}^{n-k-j}q^{\binom{s}{2}} \begin{bmatrix}
			k\\
			s
		\end{bmatrix}_q \begin{bmatrix}
			k+j-1\\
			j
		\end{bmatrix}_q \begin{bmatrix}
			s+j-1+h\\
			h
		\end{bmatrix}_q \widetilde{F}_{n-k-j,h}^{(d-k+s,\ell-j)},
	\end{align*}
	with initial conditions
	\begin{align}
		\widetilde{F}_{0,k}^{(d,\ell)} & = \widetilde{F}_{n,0}^{(d,\ell)} =0.
	\end{align}
\end{theorem}
\begin{proof}
	Combining Theorem~\ref{thm:rel_F_H} and the definitions of the polynomials $\overline{H}_{n,k,i}^{(d,\ell)}$ and $H_{n,k,i}^{(d,\ell)}$, the result follows.
\end{proof}

\subsection{A recursion for $G_{n,k}^{(d,\ell)}$ and $\widetilde{G}_{n,k}^{(d,\ell)}$}

In this section we establish a recursion for $G_{n,k}^{(d,\ell)}$ and $\widetilde{G}_{n,k}^{(d,\ell)}$.

\begin{theorem} \label{thm:recoG_and_Gtilde}
	For $k,\ell,d\geq 0$, $n\geq k+\ell$ and $n\geq d$, the $G_{n,k}^{(d,\ell)}$ satisfy the following recursion: 
	\begin{align} \label{eq:reco_G}
		G_{n,k}^{(d,\ell)} & = t^{n-\ell}q^{\binom{k}{2}+\binom{\ell}{2}}\begin{bmatrix}
			k\\
			\ell
		\end{bmatrix}_q \begin{bmatrix}
			n+k-\ell-1\\
			n-\ell
		\end{bmatrix}_q q^{\binom{n-d}{2}}\begin{bmatrix}
			n\\
			d
		\end{bmatrix}_q\\
		& \quad +\sum_{h=1}^{n-1}\sum_{s=0}^{\ell}\sum_{j=1}^{\min(n-d,h)}t^{n-\ell}q^{\binom{k}{2}+\binom{s}{2}}\begin{bmatrix}
			k\\
			s
		\end{bmatrix}_q \begin{bmatrix}
			h\\
			j
		\end{bmatrix}_q \\
		\notag & \quad \times \begin{bmatrix}
			h+k-s-1\\
			h-s
		\end{bmatrix}_q G_{n-h,j}^{(d-h+j,\ell-s)},
	\end{align}
	with initial conditions
	\begin{equation}
		G_{0,k}^{(d,\ell)}=\delta_{d,0}\delta_{\ell,0}q^{\binom{k}{2}}.
	\end{equation}
	
	Similarly
	\begin{align} \label{eq:reco_Gtilde}
		\widetilde{G}_{n,k}^{(d,\ell)} & = t^{n-\ell}q^{\binom{k}{2}+\binom{\ell}{2}}\begin{bmatrix}
			k\\
			\ell
		\end{bmatrix}_q \begin{bmatrix}
			n+k-\ell-1\\
			n-\ell
		\end{bmatrix}_q q^{\binom{n-d}{2}}\begin{bmatrix}
			n-1\\
			d
		\end{bmatrix}_q\\
		& \quad +\sum_{h=1}^{n-1}\sum_{s=0}^{\ell}\sum_{j=1}^{\min(n-d,h)}t^{n-\ell}q^{\binom{k}{2}+\binom{s}{2}}\begin{bmatrix}
			k\\
			s
		\end{bmatrix}_q \begin{bmatrix}
			h\\
			j
		\end{bmatrix}_q \\
		\notag & \quad \times \begin{bmatrix}
			h+k-s-1\\
			h-s
		\end{bmatrix}_q \widetilde{G}_{n-h,j}^{(d-h+j,\ell-s)},
	\end{align}
	with initial conditions
	\begin{equation}
		\widetilde{G}_{0,k}^{(d,\ell)}=0.
	\end{equation}
\end{theorem}
\begin{proof}
	Using \eqref{eq:rel_G_F}, we have
	\begin{align}
		G_{n,k}^{(d,\ell)}& =\sum_{h=1}^{n} \sum_{s=0}^{\ell} t^{n-\ell}q^{\binom{k}{2}+\binom{s}{2}}\begin{bmatrix}
			k\\
			s\\
		\end{bmatrix}_q \begin{bmatrix}
			h+k-s-1\\
			h-s\\
		\end{bmatrix}_q  F_{n,h}^{(d,\ell-s)}\\
		& = \sum_{s=0}^{\ell} t^{n-\ell}q^{\binom{k}{2}+\binom{s}{2}}\begin{bmatrix}
			k\\
			s\\
		\end{bmatrix}_q \begin{bmatrix}
			n+k-s-1\\
			n-s\\
		\end{bmatrix}_q  F_{n,n}^{(d,\ell-s)}\\
		& \quad + \sum_{h=1}^{n-1} \sum_{s=0}^{\ell} t^{n-\ell}q^{\binom{k}{2}+\binom{s}{2}}\begin{bmatrix}
			k\\
			s\\
		\end{bmatrix}_q \begin{bmatrix}
			h+k-s-1\\
			h-s\\
		\end{bmatrix}_q  F_{n,h}^{(d,\ell-s)}\\
		\text{(using \eqref{eq:Fnn_formula})}& = t^{n-\ell}q^{\binom{k}{2}+\binom{\ell}{2}}\begin{bmatrix}
			k\\
			\ell
		\end{bmatrix}_q \begin{bmatrix}
			n+k-\ell-1\\
			n-\ell
		\end{bmatrix}_q q^{\binom{n-d}{2}}\begin{bmatrix}
			n\\
			d
		\end{bmatrix}_q\\
		& \quad + \sum_{h=1}^{n-1} \sum_{s=0}^{\ell} t^{n-\ell}q^{\binom{k}{2}+\binom{s}{2}}\begin{bmatrix}
			k\\
			s\\
		\end{bmatrix}_q \begin{bmatrix}
			h+k-s-1\\
			h-s\\
		\end{bmatrix}_q  F_{n,h}^{(d,\ell-s)}
	\end{align}
	\begin{align}
		\text{(using \eqref{eq:rel_F_G})}& = t^{n-\ell}q^{\binom{k}{2}+\binom{\ell}{2}}\begin{bmatrix}
			k\\
			\ell
		\end{bmatrix}_q \begin{bmatrix}
			n+k-\ell-1\\
			n-\ell
		\end{bmatrix}_q q^{\binom{n-d}{2}}\begin{bmatrix}
			n\\
			d
		\end{bmatrix}_q\\
		&  \quad +\sum_{h=1}^{n-1}\sum_{s=0}^{\ell}\sum_{j=1}^{\min(n-d,h)}t^{n-\ell}q^{\binom{k}{2}+\binom{s}{2}}\begin{bmatrix}
			k\\
			s
		\end{bmatrix}_q \begin{bmatrix}
			h\\
			j
		\end{bmatrix}_q \\
		\notag & \quad \times \begin{bmatrix}
			h+k-s-1\\
			h-s
		\end{bmatrix}_q \widetilde{G}_{n-h,j}^{(d-h+j,\ell-s)},
	\end{align}
	as we wanted.
	
	For \eqref{eq:reco_Gtilde} we get the same argument using \eqref{eq:rel_Gtilde_Ftilde}, \eqref{eq:Ftildenn_formula} and then \eqref{eq:rel_Ftilde_Gtilde}.
	
	The initial conditions are easy to check: since $\Delta_{e_{-r}}=\Delta_{e_{-r}}'=0$ for $r>0$ we have
	\begin{align*}
		G_{0,k}^{(d,\ell)} & =t^{0-\ell}q^{\binom{k}{2}}\left\<  \Delta_{e_{0-\ell}}\Delta_{e_{0-d}} e_{0}\left[X\frac{1-q^k}{1-q}\right] ,h_{0}\right\>=\delta_{\ell,0}\delta_{d,0}q^{\binom{k}{2}}
	\end{align*}
	and
	\begin{align*}
		\widetilde{G}_{0,k}^{(d,\ell)} & =t^{0-\ell}q^{\binom{k}{2}}\left\<  \Delta_{e_{0-\ell}}\Delta_{e_{0-d-1}}' e_{0}\left[X\frac{1-q^k}{1-q}\right] ,h_{0}\right\>=0.
	\end{align*}
	
	This completes the proof.
\end{proof}

\subsection{A recursion for $H_{n,k,i}^{(d,\ell)}$ and $\overline{H}_{n,k,i}^{(d,\ell)}$}

In this section we establish a recursion for $H_{n,k,i}^{(d,\ell)}$ and $\overline{H}_{n,k,i}^{(d,\ell)}$.

\begin{theorem} \label{thm:reco_H_and_Htilde}
	For $k,\ell,d\geq 0$, $n\geq k+\ell$ and $n\geq d$, the $H_{n,k,i}^{(d,\ell)}$ satisfy the following recursion:
	\begin{align}
		H_{n,n,i}^{(d,\ell)} & =\delta_{i,\ell} q^{\binom{n-\ell-d}{2}}\begin{bmatrix}
			n-1\\
			\ell
		\end{bmatrix}_q\begin{bmatrix}
			n-\ell\\
			d
		\end{bmatrix}_q,
	\end{align}
	and for $1\leq k<n$
	\begin{align} \label{eq:reco_H}
		H_{n,k,i}^{(d,\ell)} & = t^{n-k}\sum_{s=0}^{k-i}\sum_{h=0}^{n-k}q^{\binom{s}{2}} \begin{bmatrix}
			k-i\\
			s
		\end{bmatrix}_q \begin{bmatrix}
			k-1\\
			i
		\end{bmatrix}_q \begin{bmatrix}
			s+i-1+h\\
			h
		\end{bmatrix}_q \sum_{j=0}^{\min(n-k-h,\ell-i)}H_{n-k,h+j,j}^{(d-k+i+s,\ell-i)},
	\end{align}
	with initial conditions
	\begin{equation}
		H_{0,k,i}^{(d,\ell)}=H_{n,0,i}^{(d,\ell)}=0.
	\end{equation}
	Similarly
	\begin{align}
		\overline{H}_{n,n,i}^{(d,\ell)} & =\delta_{i,\ell} q^{\binom{n-\ell-d}{2}}\begin{bmatrix}
			n-1\\
			\ell
		\end{bmatrix}_q\begin{bmatrix}
			n-\ell-1\\
			d
		\end{bmatrix}_q,
	\end{align}
	and for $1\leq k<n$
	\begin{align} \label{eq:reco_Hbar}
		\overline{H}_{n,k,i}^{(d,\ell)} & = t^{n-k}\sum_{s=0}^{k-i}\sum_{h=0}^{n-k}q^{\binom{s}{2}} \begin{bmatrix}
			k-i\\
			s
		\end{bmatrix}_q \begin{bmatrix}
			k-1\\
			i
		\end{bmatrix}_q \begin{bmatrix}
			s+i-1+h\\
			h
		\end{bmatrix}_q \sum_{j=0}^{\min(n-k-h,\ell-i)}\overline{H}_{n-k,h+j,j}^{(d-k+i+s,\ell-i)},
	\end{align}
	with initial conditions
	\begin{equation}
		\overline{H}_{0,k,i}^{(d,\ell)}=\overline{H}_{n,0,i}^{(d,\ell)}=0.
	\end{equation}
\end{theorem}
\begin{proof}
	Clearly \eqref{eq:reco_Hbar} follows from \eqref{eq:reco_H} using \eqref{eq:rel_Htilde_H}. Hence we will only prove \eqref{eq:reco_H}.
	
	To prove \eqref{eq:reco_H}, just combine the definition \eqref{eq:rel_Htilde_H} of $H_{n,k,i}^{(d,\ell)}$ and \eqref{eq:rel_F_and_H}.
	
	The initial conditions are easy to check.
\end{proof}
\begin{remark} \label{rem:rewriting_reco_H}
	Using some change of variables, we can rewrite the recursions in Theorem~\ref{thm:reco_H_and_Htilde} in the following way:
	\begin{align*} 
		\notag \overline{H}_{n,k,i}^{(d,\ell)} & = t^{n-k}\sum_{s=0}^{k-i}\sum_{h=0}^{n-k}q^{\binom{s}{2}} \begin{bmatrix}
			k-i\\
			s
		\end{bmatrix}_q \begin{bmatrix}
			k-1\\
			i
		\end{bmatrix}_q \begin{bmatrix}
			s+i-1+h\\
			h
		\end{bmatrix}_q \sum_{j=0}^{\min(n-k-h,\ell-i)}\overline{H}_{n-k,h+j,j}^{(d-k+i+s,\ell-i)}\\
		\notag & = t^{n-k}\sum_{s=0}^{k-i}\sum_{j=0}^{\ell-i}\sum_{f=j}^{n-k+j}q^{\binom{s}{2}} \begin{bmatrix}
			k-i\\
			s
		\end{bmatrix}_q \begin{bmatrix}
			k-1\\
			i
		\end{bmatrix}_q \begin{bmatrix}
			s+i-1+f-j\\
			f-j
		\end{bmatrix}_q \overline{H}_{n-k,f,j}^{(d-k+i+s,\ell-i)}\\
		& = t^{n-k}\sum_{p=i}^{k}\sum_{j=0}^{\ell-i}\sum_{f=j}^{n-k+j}q^{\binom{p-i}{2}} \begin{bmatrix}
			k-i\\
			p-i
		\end{bmatrix}_q \begin{bmatrix}
			k-1\\
			i
		\end{bmatrix}_q \begin{bmatrix}
			p-1+f-j\\
			f-j
		\end{bmatrix}_q \overline{H}_{n-k,f,j}^{(d-k+p,\ell-i)}.
	\end{align*}
	Similarly for $H_{n,k,i}^{(d,\ell)}$.
\end{remark}

\section{Another symmetric function identity}
The goal of this section is to prove the following theorem and deduce some of its consequences.

\begin{theorem} \label{thm:schroeder_identity}
	For all $a,b,k\in \mathbb{N}$, with $a\geq 1$, $b\geq 1$ and $1\leq k\leq a$, we have
	\begin{equation} \label{eq:Schroeder_identity}
		\< \Delta_{e_a}'\Delta_{e_{a+b-k-1}}'e_{a+b},h_{a+b}\> = \< \Delta_{h_k}\Delta_{e_{a-k}}' e_{a+b-k},e_{a+b-k}\>.
	\end{equation}
\end{theorem}

\subsection{Proof of Theorem~\ref{thm:schroeder_identity}} \label{sec:proof_identity_symmetry}

We start working on the right hand side of \eqref{eq:Schroeder_identity}: using \eqref{eq:en_expansion}, we have
\begin{align*}
	& \hspace{-0.5cm} \Delta_{h_k}\Delta_{e_{a-k}}' e_{a+b-k}=\\
	& = \sum_{\gamma\vdash a+b-k} h_k[B_\gamma] e_{a-k}[B_\gamma-1] \widetilde{H}_{\gamma}[X] \frac{MB_{\gamma}\Pi_{\gamma}}{w_{\gamma}}\\
	\text{(using \eqref{eq:e_h_sum_alphabets})} & = \sum_{\gamma\vdash a+b-k} h_k[B_\gamma] \sum_{i=0}^{a-k}(-1)^ie_{a-k-i}[B_\gamma] \widetilde{H}_{\gamma}[X] \frac{MB_{\gamma}\Pi_{\gamma}}{w_{\gamma}}\\
	& = \sum_{i=0}^{a-k}(-1)^i\sum_{\gamma\vdash a+b-k} h_k\left[\frac{MB_\gamma}{M}\right] e_{a-k-i}\left[\frac{MB_\gamma}{M}\right] \widetilde{H}_{\gamma}[X] \frac{MB_{\gamma}\Pi_{\gamma}}{w_{\gamma}}\\
	\text{(using \eqref{eq:e_h_expansion})} & = \sum_{i=0}^{a-k}(-1)^i\sum_{\gamma\vdash a+b-k} \sum_{\mu\vdash a-i}\frac{e_k[B_{\mu}]\widetilde{H}_{\mu}[MB_{\gamma}]}{w_{\mu}} \cdot  \widetilde{H}_{\gamma}[X] \frac{MB_{\gamma}\Pi_{\gamma}}{w_{\gamma}}\\
	\text{($e_k[B_{\mu}]\! =\! 0$ for $k\! >\! |\mu|$)} & = \sum_{i=0}^{a}(-1)^i\sum_{\gamma\vdash a+b-k} \sum_{\mu\vdash a-i}\frac{e_k[B_{\mu}]\widetilde{H}_{\mu}[MB_{\gamma}]}{w_{\mu}} \cdot  \widetilde{H}_{\gamma}[X] \frac{MB_{\gamma}\Pi_{\gamma}}{w_{\gamma}}\\
	\text{(using \eqref{eq:Macdonald_reciprocity})} & = \sum_{i=0}^{a}(-1)^i\sum_{\gamma\vdash a+b-k} \sum_{\mu\vdash a-i}\frac{e_k[B_{\mu}]}{w_{\mu}}\frac{\widetilde{H}_{\gamma}[MB_{\mu}]\Pi_{\mu}}{\Pi_{\gamma}} \cdot  \widetilde{H}_{\gamma}[X] \frac{MB_{\gamma}\Pi_{\gamma}}{w_{\gamma}}\\
	& = \sum_{i=0}^{a}(-1)^i\sum_{\gamma\vdash a+b-k} \sum_{\mu\vdash a-i} e_k\left[\frac{MB_\mu}{M}\right] \widetilde{H}_{\gamma}[MB_{\mu}] \cdot  \widetilde{H}_{\gamma}[X] \frac{\Pi_{\mu} MB_{\gamma}}{w_{\mu}w_{\gamma}}\\
	\text{(using \eqref{eq:def_dmunu})} & = \sum_{i=0}^{a}(-1)^i\sum_{\gamma\vdash a+b-k} \sum_{\mu\vdash a-i} \sum_{\beta \supset_k \gamma} d_{\beta \gamma}^{(k)} \widetilde{H}_{\beta}[MB_{\mu}] \cdot  \widetilde{H}_{\gamma}[X] \frac{\Pi_{\mu} MB_{\gamma}}{w_{\mu}w_{\gamma}} \\
	\text{(using \eqref{eq:rel_cmunu_dmunu})} & = \sum_{i=0}^{a}(-1)^i\sum_{\gamma\vdash a+b-k} \sum_{\mu\vdash a-i} \sum_{\beta \supset_k \gamma} c_{\beta \gamma}^{(k)}\frac{w_{\gamma}}{w_{\beta}} \widetilde{H}_{\beta}[MB_{\mu}] \cdot  \widetilde{H}_{\gamma}[X] \frac{\Pi_{\mu} MB_{\gamma}}{w_{\mu}w_{\gamma}}\\
	& =\sum_{i=0}^{a}(-1)^i \sum_{\mu\vdash a-i}  \sum_{\beta\vdash a+b} \sum_{\gamma \subset_k \beta} c_{\beta \gamma}^{(k)}B_{\gamma} \widetilde{H}_{\gamma}[X] \cdot \widetilde{H}_{\beta}[MB_{\mu}] \frac{\Pi_{\mu} M}{w_{\mu}w_{\beta}}\\
	\text{(using \eqref{eq:Macdonald_reciprocity})} & =  \sum_{i=0}^{a}(-1)^i \sum_{\beta\vdash a+b} \sum_{\gamma \subset_k \beta} c_{\beta \gamma}^{(k)}B_{\gamma} \widetilde{H}_{\gamma}[X] \cdot \sum_{\mu\vdash a-i}  \frac{\widetilde{H}_{\mu}[MB_{\beta}]}{w_{\mu}}\frac{\Pi_{\beta} M}{w_{\beta}} \\
	\text{(using \eqref{eq:e_h_expansion})}  & =  \sum_{\beta\vdash a+b} \sum_{\gamma \subset_k \beta} c_{\beta \gamma}^{(k)}B_{\gamma} \widetilde{H}_{\gamma}[X] \cdot \sum_{i=0}^{a}(-1)^ie_{a-i}[B_{\beta}]\frac{\Pi_{\beta} M}{w_{\beta}}\\
	\text{(using \eqref{eq:e_h_sum_alphabets})} & =  \sum_{\beta\vdash a+b} \sum_{\gamma \subset_k \beta} c_{\beta \gamma}^{(k)}B_{\gamma} \widetilde{H}_{\gamma}[X] \cdot e_{a}[B_{\beta}-1]\frac{\Pi_{\beta} M}{w_{\beta}}.
\end{align*}
Taking the scalar product with $e_{a+b-k}$ we get
\begin{align*}
	\< \Delta_{h_k}\Delta_{e_{a-k}}' e_{a+b-k} ,e_{a+b-k}\> &  =  \sum_{\beta\vdash a+b} \sum_{\gamma \subset_k \beta} c_{\beta \gamma}^{(k)}B_{\gamma} \< \widetilde{H}_{\gamma}[X],e_{a+b-k}\> \cdot e_a[B_{\beta}-1]\frac{\Pi_{\beta} M}{w_{\beta}}\\
	\text{(using \eqref{eq:Mac_hook_coeff})} &  =  \sum_{\beta\vdash a+b} \sum_{\gamma \subset_k \beta} c_{\beta \gamma}^{(k)}B_{\gamma} e_{a+b-k}[B_{\gamma}] \cdot e_a[B_{\beta}-1]\frac{\Pi_{\beta} M}{w_{\beta}}\\
	\text{(using \eqref{eq:Bmu_Tmu})} &  =  \sum_{\beta\vdash a+b} \sum_{\gamma \subset_k \beta} c_{\beta \gamma}^{(k)}B_{\gamma} T_{\gamma} \cdot e_a[B_{\beta}-1]\frac{\Pi_{\beta} M}{w_{\beta}}.
\end{align*}
For the left hand side we have
\begin{align*}
	\< \Delta_{e_a}'\Delta_{e_{a+b-k-1}}'e_{a+b},h_{a+b}\> 
	&  =  \sum_{\beta\vdash a+b} e_{a+b-k-1}[B_{\beta}-1]B_{\beta} \< \widetilde{H}_{\beta}[X],h_{a+b}\> \cdot e_a[B_{\beta}-1]\frac{\Pi_{\beta} M}{w_{\beta}}\\
	\text{(using \eqref{eq:Bmu_Tmu})} &  =  \sum_{\beta\vdash a+b} e_{a+b-k-1}[B_{\beta}-1]B_{\beta}   \cdot e_a[B_{\beta}-1]\frac{\Pi_{\beta} M}{w_{\beta}}.
\end{align*}
So, in order to conclude the proof of \eqref{eq:Schroeder_identity}, it will be enough to prove the following lemma.
\begin{lemma} 
	For every $n,k\in \mathbb{N}$, with $n> k\geq 1$, and for every $\beta\vdash n$, we have
	\begin{equation} \label{eq:lem_eBcmunu}
		e_{n-k-1}[B_{\beta}-1]B_{\beta}=\sum_{\gamma \subset_k \beta} c_{\beta \gamma}^{(k)}B_{\gamma} T_{\gamma} .
	\end{equation}
\end{lemma}
\begin{proof}
	For $\gamma\vdash n-k$, applying \eqref{eq:Pieri_sum1}, we have 
	\begin{equation}
		B_{\gamma}=\sum_{\delta\subset_1 \gamma}c_{\gamma \delta}^{(1)}.
	\end{equation}
	
	Therefore
	\begin{align*}
		\sum_{\gamma \subset_k \beta} c_{\beta \gamma}^{(k)}B_{\gamma} T_{\gamma} & = \sum_{\gamma \subset_k \beta} c_{\beta \gamma}^{(k)}\sum_{\delta\subset_1 \gamma}c_{\gamma \delta}^{(1)} T_{\gamma}\\
		& = \sum_{\delta\subset_{k+1} \beta}\sum_{\delta\subset_1 \gamma \subset_k \beta} c_{\beta \gamma}^{(k)}c_{\gamma \delta}^{(1)} T_{\gamma}\\
		& = \sum_{\delta\subset_{k+1} \beta}T_{\delta}B_{\beta/\delta}\frac{1}{B_{\beta/\delta}}\sum_{\delta\subset_1 \gamma \subset_k \beta} c_{\beta \gamma}^{(k)}c_{\gamma \delta}^{(1)} \frac{T_{\gamma}}{T_{\delta}}\\
		\text{(using \eqref{eq:cmunu_recursion})}& = \sum_{\delta\subset_{k+1} \beta}T_{\delta}B_{\beta/\delta} c_{\beta \delta}^{(k+1)}\\
		& = B_{\beta}\sum_{\delta\subset_{k+1} \beta} c_{\beta \delta}^{(k+1)}T_{\delta} -\sum_{\delta\subset_{k+1} \beta} c_{\beta \delta}^{(k+1)}B_{\delta}T_{\delta},
	\end{align*}
	which gives
	\begin{equation}
		\sum_{\gamma \subset_k \beta} c_{\beta \gamma}^{(k)}B_{\gamma} T_{\gamma}=-\sum_{\delta\subset_{k+1} \beta} c_{\beta \delta}^{(k+1)}B_{\delta}T_{\delta} +B_{\beta}\sum_{\delta\subset_{k+1} \beta} c_{\beta \delta}^{(k+1)}T_{\delta}.
	\end{equation}
	Observe that by \eqref{eq:Bmu_Tmu} we have
	\begin{align*}
		\sum_{\delta\subset_{k+1} \beta} c_{\beta \delta}^{(k+1)}T_{\delta} 
		& = \sum_{\delta\subset_{k+1} \beta} c_{\beta \delta}^{(k+1)}e_{n-k-1}[B_{\delta}]\\
		\text{(using \eqref{eq:Mac_hook_coeff})}& = \sum_{\delta\subset_{k+1} \beta} c_{\beta \delta}^{(k+1)}\< \widetilde{H}_{\delta},e_{n-k-1}\> \\
		\text{(using \eqref{eq:def_cmunu})}& =\< h_{k+1}^{\perp}\widetilde{H}_{\beta},e_{n-k-1}\>\\
		& =\< \widetilde{H}_{\beta},e_{n-k-1}h_{k+1}\>\\
		\text{(using \eqref{eq:Mac_hook_coeff})}  & = e_{n-k-1}[B_{\beta}],
	\end{align*}
	so
	\begin{equation} \label{eq:inductive_step}
		\sum_{\gamma \subset_k \beta} c_{\beta \gamma}^{(k)}B_{\gamma} T_{\gamma}=-\sum_{\delta\subset_{k+1} \beta} c_{\beta \delta}^{(k+1)}B_{\delta}T_{\delta} +B_{\beta}e_{n-k-1}[B_{\beta}].
	\end{equation}
	Therefore we can argue by induction on $n-k$: for $n-k=1$, we have $k=n-1$, so
	\begin{equation}
		\sum_{\gamma \subset_{n-1} \beta} c_{\beta \gamma}^{(n-1)}B_{\gamma} T_{\gamma}=c_{\beta (1)}^{(n-1)} B_{(1)}T_{(1)}=c_{\beta (1)}^{(n-1)},
	\end{equation}
	but
	\begin{align*}
		c_{\beta (1)}^{(n-1)} & =c_{\beta (1)}^{(n-1)}\< e_{1},e_1\>\\
		& =c_{\beta (1)}^{(n-1)}\< \widetilde{H}_{(1)},e_1\>\\
		& = \< h_{n-1}^{\perp}\widetilde{H}_{\beta},e_1\>\\
		& = \< \widetilde{H}_{\beta},e_1 h_{n-1}\>\\
		\text{(using \eqref{eq:Mac_hook_coeff})} & = e_1[B_{\beta}]= B_{\beta} = e_0[B_{\beta}-1]B_{\beta}.
	\end{align*}
	If $n-k\geq 2$, by induction
	\begin{equation}
		\sum_{\delta\subset_{k+1} \beta} c_{\beta \delta}^{(k+1)}B_{\delta}T_{\delta}=e_{n-k-2}[B_{\beta}-1]B_{\beta},
	\end{equation}
	so, using \eqref{eq:inductive_step}, we have
	\begin{align*}
		\sum_{\gamma \subset_k \beta} c_{\beta \gamma}^{(k)}B_{\gamma} T_{\gamma} & =-\sum_{\delta\subset_{k+1} \beta} c_{\beta \delta}^{(k+1)}B_{\delta}T_{\delta} +B_{\beta}e_{n-k-1}[B_{\beta}]\\
		& = -e_{n-k-2}[B_{\beta}-1]B_{\beta}+B_{\beta}e_{n-k-1}[B_{\beta}]\\
		& = B_{\beta}(e_{n-k-1}[B_{\beta}]-e_{n-k-2}[B_{\beta}-1])\\
		& = B_{\beta}e_{n-k-1}[B_{\beta}-1].
	\end{align*}
	This concludes the proof of the lemma.
\end{proof}
We already observed that the lemma finishes the proof of Theorem~\ref{thm:schroeder_identity}.

\subsection{Some consequences of Theorem~\ref{thm:schroeder_identity}}\label{sec:algebraic_proof_symmetry}

We deduce here a few consequences of Theorem~\ref{thm:schroeder_identity}.

First of all, we have the following corollary.
\begin{corollary}
	\begin{equation} \label{eq:Schroeder_corol}
		\< \Delta_{e_a}\Delta_{e_{a+b-k-1}}'e_{a+b},h_{a+b}\> = \< \Delta_{h_k}\Delta_{e_{a-k}} e_{a+b-k},e_{a+b-k}\>.
	\end{equation}
\end{corollary}
\begin{proof}
	The same identity \eqref{eq:Schroeder_identity} with $a$ and $b$ replaced by $a-1$ and $b+1$ respectively, gives 
	\begin{equation} \label{eq:Schroeder_id2}
		\< \Delta_{e_{a-1}}'\Delta_{e_{a+b-k-1}}'e_{a+b},h_{a+b}\> = \< \Delta_{h_k}\Delta_{e_{a-k-1}}' e_{a+b-k},e_{a+b-k}\>.
	\end{equation}
	Adding this one to \eqref{eq:Schroeder_identity} and using $\Delta_{e_m}=\Delta_{e_m}'+\Delta_{e_{m-1}}'$ on $\Lambda^{(n)}$ with $1\leq m\leq n$, we get the result.
\end{proof}
As an application of \eqref{eq:Schroeder_identity}, we give a first algebraic proof of the symmetry \eqref{eq:HRW_symmetry} predicted by Conjecture~7.1 in \cite{Haglund-Remmel-Wilson-2015}, which is left open in \cite{Zabrocki-4Catalan-2016} as Conjecture~16. A combinatorial proof will be provided in Section~\ref{sec:deco_qt_Schroeder}.
\begin{theorem} \label{thm:symmetry}
	For $n>k+\ell$, the following expression is symmetric in $k$ and $\ell$:
	\begin{equation}
		\< \Delta_{h_k}\Delta_{e_{n-k-\ell-1}}' e_{n-k},e_{n-k}\>.
	\end{equation}
\end{theorem}
\begin{proof}[Algebraic proof of Theorem~\ref{thm:symmetry}]
	Using \eqref{eq:Schroeder_identity}, with $n=a+b$ and $\ell=b-1$, we have
	\begin{equation}
		\< \Delta_{h_k}\Delta_{e_{n-k-\ell-1}}' e_{n-k},e_{n-k}\>=\< \Delta_{e_{n-\ell-1}}'\Delta_{e_{n-k-1}}'e_{n},h_{n}\> ,
	\end{equation}
	which is obviously symmetric in $k$ and $\ell$.
\end{proof}
\begin{proposition} \label{prop:SF_sums_F_H}
	For $n,d,\ell\geq 0$, $n\geq k+\ell$ and $n\geq d$, we have
	\begin{equation} \label{eq:Fnk_sum}
		\< \Delta_{e_{n-\ell-1}}'e_{n},e_{n-d}h_{d}\>=\sum_{k=1}^n F_{n,k}^{(d,\ell)}=\sum_{k=1}^n \sum_{i=0}^{k-1} H_{n,k,i}^{(d,\ell)},
	\end{equation}
	and
	\begin{equation} \label{eq:Ftildenk_sum}
		\< \Delta_{e_{n-\ell-1}}'e_{n},s_{d+1,1^{n-d-1}}\>=\sum_{k=1}^n \widetilde{F}_{n,k}^{(d,\ell)}=\sum_{k=1}^n \sum_{i=0}^{k-1} \overline{H}_{n,k,i}^{(d,\ell)}.
	\end{equation}
\end{proposition}
\begin{proof}
	To prove the first equality in \eqref{eq:Fnk_sum}, observe that
	\begin{align*}
		\sum_{k=1}^n\notag F_{n,k}^{(d,\ell)} & =\< \Delta_{h_{\ell}} \Delta_{e_{n-d-\ell}}\sum_{k=1}^{n-\ell}E_{n-\ell,k},e_{n-\ell}\>\\ 
		\text{(using \eqref{eq:en_sum_Enk})}	& =\< \Delta_{h_{\ell}} \Delta_{e_{n-d-\ell}}e_{n-\ell},e_{n-\ell}\> \\
		\text{(using \eqref{eq:Schroeder_corol})}& =\< \Delta_{e_{n-d}}\Delta_{e_{n-\ell-1}}'e_{n},h_{n}\>\\
		\text{(using \eqref{eq:Mac_hook_coeff})}& =\< \Delta_{e_{n-\ell-1}}'e_{n},e_{n-d}h_{d}\>.
	\end{align*}
	The second equality follows simply from \eqref{eq:rel_F_and_H}.
	
	A similar computation using \eqref{eq:Schroeder_identity}, \eqref{eq:Mac_hook_coeff_ss} and \eqref{eq:rel_Ftilde_and_Hbar} gives \eqref{eq:Ftildenk_sum}.
\end{proof}

\section{Two theorems and a corollary}\label{sec:2more_thrms}

We prove two theorems and a corollary on symmetric functions.

\begin{theorem}
	\label{thm:magic_equality}
	Let $m,n,k\in \mathbb{N}$, $m\geq 0$, $n\geq 0$ and $n\geq k\geq 0$. Then
	\begin{align}
		\label{eq:connection_id}
		\< \Delta_{h_m} e_{n+1}, s_{k+1,1^{n-k}} \> =  \< \Delta_{e_{m+n-k-1}}' e_{m+n}, h_m h_n \>.
	\end{align}
	In particular
	\begin{align}
		\label{eq:symmetry_id}
		\< \Delta_{h_m} e_{n+1}, s_{k+1,1^{n-k}} \> =  \< \Delta_{h_{n}} e_{m+1}, s_{k+1,1^{m-k}} \>.
	\end{align}
\end{theorem}
\begin{proof}
	Using \eqref{eq:Mac_hook_coeff}, we have
	\begin{align*}
		\< \Delta_{h_m} e_{n+1}, h_{k+1}e_{n-k} \>  & = \< \Delta_{h_m e_{n-k}} e_{n+1}, h_{n+1} \> \\
		\text{(using \eqref{eq:Haglund_Lemma})}& = \< \Delta_{e_n} e_{m+n-k}, e_m h_{n-k}\> \\
		\text{(using \eqref{eq:Mac_hook_coeff})}& = \< \Delta_{e_ne_m} e_{m+n-k},  h_{m+n-k}\> \\
		\text{(using \eqref{eq:Haglund_Lemma})} & = \< \Delta_{e_{m+n-k-1}} e_{m+n},  h_{m}h_n\> .
	\end{align*}
	Now using this identity, \eqref{eq:deltaprime} and the classical
	\begin{align}
		h_{k+1}e_{n-k}=s_{k+2,1^{n-k-1}}+s_{k+1,1^{n-k}},
	\end{align}
	we get
	\begin{align*}
		\< \Delta_{h_m} e_{n+1}, s_{k+1,1^{n-k}} \> & = \sum_{j\geq 0}(-1)^{j} \< \Delta_{h_m} e_{n+1}, h_{k+1+j}e_{n-k-j}\>\\
		& = \sum_{j\geq 0}(-1)^{j} \< \Delta_{e_{m+n-k-1-j}} e_{m+n},  h_{m}h_n\> \\
		& =  \< \Delta_{e_{m+n-k-1}}' e_{m+n},  h_{m}h_n\>.
	\end{align*}
	This proves \eqref{eq:connection_id}. Now \eqref{eq:symmetry_id} follows immediately from the fact that the right hand side of \eqref{eq:connection_id} is obviously symmetric in $m$ and $n$.
\end{proof}
%
%

\begin{theorem}
	\label{thm:SF_sum}
	Let $m,n,k\in \mathbb{N}$, $m\geq 0$, $n\geq 0$ and $m\geq k\geq 0$. Then
	\begin{align}
		\sum_{r=1}^{m-k+1} t^{m-k-r+1} \< \Delta_{h_{m-k-r+1}} \Delta_{e_k} e_n \left[ X \frac{1-q^r}{1-q} \right], e_n \> = \< \Delta_{h_m} e_{n+1}, s_{k+1,1^{n-k}} \>.
	\end{align}
\end{theorem}

\begin{proof}
	Using \eqref{eq:qn_q_Macexp}, we have
	\begin{align*}
		& \hspace{-0.5cm}\sum_{r=1}^{m-k+1} t^{m-k-r+1} \< \Delta_{h_{m-k-r+1}} \Delta_{e_k} e_n \left[ X \frac{1-q^r}{1-q} \right], e_n \>= \\
		= & \sum_{r=1}^{m-k+1} t^{m-k-r+1} \sum_{\mu\vdash n} (1-q^r) h_r[(1-t)B_\mu] \frac{\Pi_\mu}{w_\mu} h_{m-r-k+1}[B_\mu] e_k[B_\mu] T_\mu \\
		= & \sum_{\mu \vdash n} \left( \sum_{r=1}^{m-k+1} t^{m-k-r+1} (1-q^r) h_r[(1-t)B_\mu] h_{m-r-k+1}[B_\mu] \right) \frac{\Pi_\mu}{w_\mu} e_k[B_\mu] T_\mu \\
		\text{(using \eqref{eq:Haglund_nablaEnk})} = & \sum_{\mu \vdash n} \left( \sum_{r=1}^{m-k+1} \left. \mathbf{\Pi}^{-1} \nabla E_{m-k+1,r}[X] \right|_{X=MB_\mu} \right) \frac{\Pi_\mu}{w_\mu} e_k[B_\mu] T_\mu \\
		\text{(using \eqref{eq:en_sum_Enk})} = & \sum_{\mu \vdash n} \left( \left. \mathbf{\Pi}^{-1} \nabla e_{m-k+1}[X] \right|_{X=MB_\mu} \right) \frac{\Pi_\mu}{w_\mu} e_k[B_\mu] T_\mu \\
		\text{(using \eqref{eq:en_expansion})} = & \sum_{\mu \vdash n} \sum_{\gamma\vdash m-k+1} T_\gamma \frac{MB_\gamma}{w_\gamma} \widetilde{H}_\gamma[MB_\mu] \frac{\Pi_\mu}{w_\mu} e_k[B_\mu] T_\mu \\
	\end{align*}
	\begin{align*}
		= & \sum_{\gamma\vdash m-k+1} \sum_{\mu \vdash n} T_\mu M \frac{\Pi_\mu}{w_\mu} T_\gamma \frac{B_\gamma}{w_\gamma} e_k \left[ \frac{MB_\mu}{M} \right] \widetilde{H}_\gamma[MB_\mu]\qquad \qquad \qquad  \\
		= & \sum_{\gamma \vdash m-k+1} \sum_{\mu \vdash n} T_\mu M \frac{\Pi_\mu}{w_\mu} T_\gamma \frac{B_\gamma}{w_\gamma} \sum_{\alpha \supset_k \gamma} d_{\alpha \gamma}^{(k)}  \widetilde{H}_\alpha[MB_\mu] \\
		\text{(using \eqref{eq:Macdonald_reciprocity})} = & \sum_{\gamma \vdash m-k+1} \sum_{\alpha \supset_k \gamma} M \Pi_{\alpha} \sum_{\mu \vdash n} T_\mu \frac{\widetilde{H}_\mu[MB_\alpha] }{w_\mu} T_\gamma B_\gamma \frac{d_{\alpha \gamma}^{(k)}}{w_\gamma} \\
		\text{(using \eqref{eq:e_h_expansion})} = & \sum_{\gamma \vdash m-k+1} \sum_{\alpha \supset_k \gamma} M \Pi_{\alpha} h_n \left[ \frac{MB_\alpha}{M} \right] T_\gamma B_\gamma \frac{d_{\alpha \gamma}^{(k)}}{w_\gamma} \\
		\text{(using \eqref{eq:rel_cmunu_dmunu})} = & \sum_{\alpha \vdash m+1} \sum_{\gamma \subset_k \alpha} M \Pi_{\alpha} h_n \left[ B_\alpha \right] T_\gamma B_\gamma \frac{c_{\alpha \gamma}^{(k)}}{w_\alpha} \\
		= & \sum_{\alpha \vdash m+1} M \frac{\Pi_{\alpha}}{w_\alpha} h_n \left[ B_\alpha \right] \sum_{\gamma \subset_k \alpha} T_\gamma B_\gamma c_{\alpha \gamma}^{(k)} \\
		\text{(using \eqref{eq:lem_eBcmunu})} = & \sum_{\alpha \vdash m+1} M \frac{\Pi_{\alpha}}{w_\alpha} h_n \left[ B_\alpha \right] e_{m-k}[B_\alpha -1] B_\alpha \\
		\text{(using \eqref{eq:Mac_hook_coeff_ss})} = & \< \Delta_{h_n} e_{m+1}, s_{k+1,1^{m-k}} \> \\
		\text{(using \eqref{eq:symmetry_id})} = & \< \Delta_{h_m} e_{n+1}, s_{k+1,1^{n-k}} \>
	\end{align*}
\end{proof}

The following corollary is an immediate consequence of these results.

\begin{corollary} \label{cor:SF_sum}
	Let $m,n,k\in \mathbb{N}$, $m\geq 0$, $n\geq 0$ and $m\geq k\geq 0$. Then
	\[ \sum_{r=1}^{m-k+1} t^{m-k-r+1} \< \Delta_{h_{m-k-r+1}} \Delta_{e_k} e_n \left[ X \frac{1-q^r}{1-q} \right], e_n \> = \< \Delta_{e_{m+n-k-1}}' e_{m+n}, h_m h_n \>. \]
\end{corollary}

\section{$\Delta_f\omega(p_n)$ at $t=1/q$}

The goal of this section is to prove the following theorem.
\begin{theorem} \label{thm:omegapn}
	For any $f\in \Lambda^{(k)}$ we have
	\begin{equation}
		{\Delta_f \omega(p_n)}_{\big{|}_{t=1/q}}=\frac{[n]_qf[[n]_q]}{[k]_qq^{k(n-1)}}e_n[X[k]_q].
	\end{equation}
\end{theorem}
\begin{proof}
	By the definition of plethystic notation and Murnaghan-Nakayama rule, we have
	\begin{align*}
		\omega(p_n)/(1-q^n) & =(-1)^{n-1}p_n/(1-q^n)\\
		& = ((-1)^{n-1}p_n)[X/(1-q)]\\
		& = \left(\sum_{a=1}^n(-1)^{a-1}s_{(a,1^{n-a})}\right)[X/(1-q)]\\
		& = \sum_{a=1}^n(-1)^{a-1}s_{(a,1^{n-a})}[X/(1-q)]
	\end{align*}
	It is well-known (see \cite{Garsia-Haiman-qLagrange-1996}) that
	\begin{equation}
		\widetilde{H}_{\mu}(X;q,1/q)=Cs_{\mu}\left[\frac{X}{1-q}\right]
	\end{equation}
	for all $\mu\vdash n$ and some constant $C$ (depending on $\mu$), so
	\begin{align*}
		{\Delta_f \omega(p_n)/(1-q^n)}_{\big{|}_{t=1/q}} & = \sum_{a=1}^n(-1)^{a-1}\Delta_f s_{(a,1^{n-a})}[X/(1-q)]\\
		& = \sum_{a=1}^n(-1)^{a-1}f[B_{(a,1^{n-a})}(q,1/q)] s_{(a,1^{n-a})}[X/(1-q)].
	\end{align*}
	But notice that
	\begin{equation}
		B_{(a,1^{n-a})}(q,1/q)=q^{-(n-a)}[n]_q,
	\end{equation}
	so that
	\begin{align*}
		{\Delta_f \omega(p_n)/(1-q^n)}_{\big{|}_{t=1/q}} & = \sum_{a=1}^n(-1)^{a-1}f[B_{(a,1^{n-a})}(q,1/q)] s_{(a,1^{n-a})}[X/(1-q)]\\
		& = \sum_{a=1}^n(-1)^{a-1}f[q^{-(n-a)}[n]_q] s_{(a,1^{n-a})}[X/(1-q)]\\
		& = \sum_{a=1}^n(-1)^{a-1}q^{-k(n-a)}f[[n]_q] s_{(a,1^{n-a})}[X/(1-q)]\\
		& =  f[[n]_q] q^{-k(n-1)}\sum_{a=1}^n(-1)^{a-1}q^{k(a-1)} s_{(a,1^{n-a})}[X/(1-q)]\\
		& =  \frac{f[[n]_q]}{(1-q^k) q^{k(n-1)}}\sum_{a=1}^n(-q^k)^{a-1} (1-q^k) s_{(a,1^{n-a})}[X/(1-q)]
	\end{align*}
	\begin{align*}
		\text{(using \eqref{eq:s_plethystic_eval})} & =  \frac{f[[n]_q]}{(1-q^k) q^{k(n-1)}}\sum_{a=1}^ns_{(n-a+1,1^{a-1})}[1-q^k] s_{(a,1^{n-a})}[X/(1-q)]\\
		\text{(using \eqref{eq:s_plethystic_eval})}  & = \frac{f[[n]_q]}{(1-q^k) q^{k(n-1)}}\sum_{\mu\vdash n} s_{\mu'}[1-q^k] s_{\mu}[X/(1-q)]\\
		\text{(using \eqref{eq:Cauchy_identities})}& = \frac{f[[n]_q]}{(1-q^k) q^{k(n-1)}}e_n[X[k]_q].
	\end{align*}
	This immediately implies our theorem.
\end{proof}

\chapter{Combinatorics of decorated Dyck paths} \label{chapter:dec_Dyck}

In this chapter we collect some results about decorated Dyck paths.

\section{Haglund's $\zeta$ (sweep) map}\label{sec:zeta_map}

The bounce behaves well under Haglund's $\zeta$ map.

\begin{theorem}
	\label{thm:zeta_map}
	There exists a bijective map \[ \zeta: \DDd(n)^{\circ a, \ast b} \rightarrow \DDb(n)^{\circ b, \ast a} \] such that for all $D\in \DDd(n)^{\circ a, \ast b}$ \begin{align*}
		\area(D) & = \bounce(\zeta(D)) \\
		\dinv(D) & = \area(\zeta(D)).
	\end{align*}
\end{theorem}

\begin{proof}
	We remind the reader of the notations: 
	\begin{center}
		\begin{itemize}
			\item $ \DDd(n)^{\circ a, \ast b} \subseteq \DD(n)^{\circ a, \ast b}$  is the subset where the \emph{rightmost highest peak}, i.e.\ the rightmost peak that is the furthest removed from the main diagonal, is never decorated, denoted as such because decorating this peak does not alter the dinv statistic;
			\item $\DDb(n)^{\circ a, \ast b} \subseteq \DD(n)^{\circ a, \ast b}$ is the subset where the \emph{first peak}, i.e.\ the peak in the leftmost column, is never decorated, denoted as such because decorating this peak does not alter the bounce statistic.
		\end{itemize}
	\end{center}

	The map is Haglund's $\zeta$ map on undecorated Dyck paths (\cite[Theorem 3.15]{Haglund-Book-2008}), whose definition we know recall; we just have to specify what happens to the decorations. 
	
	Take $D\in\DDd(n)^{\circ a, \ast b}$ and rearrange its area word in ascending order. This new word, call it $u$, will be the bounce word of $\zeta(D)$. We construct $\zeta(D)$ as follows. First we draw the bounce path corresponding to $u$. The first vertical stretch and last horizontal stretch of $\zeta(D)$ are fixed by this path. For the section of the path  between the peaks of the bounce path we apply the following procedure: place a pen on the top of the $i$-th peak (counting from the bottom row) of the bounce path and scan the area word of $D$ from left to right. Every time we encounter a letter equal to $i-1$ we draw an east step and when we encounter a letter equal to $i$ we draw a north step. By construction of the bounce path, we end up with our pen on top of the $(i+1)$-th peak of the bounce path. Note that in an area word a letter equal to $i\neq 0$ cannot appear unless it is preceded by a letter equal to $i-1$. This means that starting from the $i$-th peak, we always start with a horizontal step which explains why $u$ is indeed the bounce word of $\zeta(D)$. It is well-known (cf. \cite[Theorem 3.15]{Haglund-Book-2008}) that, ignoring the decorations, this map is a bijection from $\DDd(n)^{\circ a, \ast b}$ into $\DDb(n)^{\circ b, \ast a}$ sending the bistatistic $(\dinv,\area)$ into $(\area,\bounce)$.
	
	Let us now deal with the decorated rises. Let $j$ be such that the $j$-th vertical step of $D$ is a decorated rise. It follows that  $a_{j}(D) = a_{j-1}(D)+1$. Take $i$ such that $a_{j-1}(D)=i-1$. While scanning the area word to construct the path between the $i$-th and $(i+1)$-th peak of the bounce path, we will encounter $a_{j-1}(D)=i-1$, directly followed by $a_{j}(D)=i$. This will correspond to a horizontal step followed by a vertical step in $\zeta(D)$, sometimes called a \emph{valley}. Decorate the first peak following this valley. This peak must be in the section of the path between the $i$-th and $(i+1)$-th peak of the bounce path so the label of the vertical step of the bounce path contained in the same row must be equal to $i$. It now follows from the definition of the statistics that $\area$ gets sent into $\bounce$. Note that the first peak never gets decorated in this way, since it is not preceded by a valley.

	Finally, we determine where to send the decorated peaks. The most natural thing to do is to send decorated peaks into decorated falls. Let $j$ be the index of a decorated peak. It follows that $a_{j+1}(D) \leq a_j(D)$. Take once more $i$ such that $a_{j}(D)=i-1$. When scanning the area word to construct the path between the $i$-th and $(i+1)$-th peak of the bounce path, we will encounter $a_j(D)=i-1$ which corresponds to a horizontal step, which we call $h$. Since $a_{j+1}(D) \leq i-1$ we must, by definition of an area word, encounter a letter equal to $i-1$ before we might encounter another letter equal to $i$. It follows that our scanning continues with the encounter of a letter equal to $i-1$ or that we encounter no more letters equal to $i-1$ or $i$. In the first case it is clear that $h$ is followed by a horizontal step. In the second case, if the $i$-th step of the bounce path is not the last one, then $h$ is also followed by a horizontal step since the next portion of the path must start with a horizontal step. Only when the $i$-th peak is the last peak of the bounce path and $a_j(D)$ is the last occurrence of $i-1$ in the area word, $h$ is not followed by a horizontal step. But this means that the $j$-th vertical step was the rightmost highest peak of $D$ and so it was not supposed to be decorated as $D \in \DDd(n)^{\circ a, \ast b}$.

	Thus, $h$ is always a fall, which we decorate. So in the area of $\zeta(D)$ the squares directly under $h$ and above the main diagonal are not counted. We distinguish two types of squares under $h$, those that are under the bounce path of $\zeta(D)$ and those above it. The number of squares of the first kind equals the number of letters equal to $i-1$ that occur in the area word after its $j$-th letter. The number of squares of the second kind is equal to the number of letters equal to $i$ that occur in the area word before its $j$-th letter. Since $j\in \mathsf{Peak}(D)$ the pairs formed by $j$ and the indices of these letters are exactly the pairs that are not counted for the dinv. The pairs that do contribute to the dinv correspond to the squares under undecorated horizontal steps. We conclude that dinv gets sent into area. 
	
	One readily checks that this map is invertible: the bounce path of $\zeta(D)$ provides the letters of the area word of $D$, and the path the interlacing of the letters. The decorations of $D$ can be obtained applying the arguments above backwards.  
	
	\begin{figure*}[!ht]
		\begin{minipage}{.5\textwidth}
			\centering
			\begin{tikzpicture}[scale=.6]
			\draw
			(8.5,0.5) node {0}
			(8.5,1.5) node {1}
			(8.5,2.5) node {2}
			(8.5,3.5) node {2}
			(8.5,4.5) node {3}
			(8.5,5.5) node {1}
			(8.5,6.5) node {0}
			(8.5,7.5) node {1};
			
			\fill[pattern=north west lines, pattern color=gray](0,1) rectangle (1,2) (1,4) rectangle (4,5)(6,7) rectangle (7,8);
			\draw[gray!60, thin](0,0) grid (8,8) (0,0)--(8,8);
			\draw[blue!60, line width=1.6pt](0,0)--(0,3)--(1,3)--(1,5)--(4,5)|-(6,6)|-(8,8);
			
			\draw
			(-0.5,1.5) node {$\ast$}
			(0.5,4.5) node {$\ast$}
			(5.5,7.5) node {$\ast$};
			
			\filldraw[fill=blue!60]
			(4,6) circle (3 pt)
			(6,8) circle (3 pt);

			\end{tikzpicture}
			\caption{An element $D\in \DDd(8)^{\circ 2, \ast 3}$ and its area word.}
		\end{minipage}%
		\begin{minipage}{.5 \textwidth}
			\centering
			\begin{tikzpicture}[scale=0.6]
			\draw[gray!60, thin](0,0) grid (8,8) (0,0)--(8,8);
			
			\fill[pattern=north west lines, pattern color=gray] (3,4) rectangle (4,7) (4,5) rectangle (5,7); 
			\draw[blue!60, line width=1.6pt] (0,0)|-(1,2)|-(2,4)|-(3,5)|-(7,7)|-(8,8);
			
			\draw[dashed, opacity=0.6, ultra thick] (0,0)|-(2,2)|-(5,5)|-(7,7)|-(8,8);
			
			\draw
			(8.5,0.5) node {0}
			(8.5,1.5) node {0}
			(8.5,2.5) node {1}
			(8.5,3.5) node {1}
			(8.5,4.5) node {1}
			(8.5,5.5) node {2}
			(8.5,6.5) node {2}
			(8.5,7.5) node {3};
			
			\draw
			(3.5,7.5) node {$\ast$}
			(4.5,7.5) node {$\ast$};
			
			\filldraw[fill=blue!60]
			(1,4) circle (3 pt)
			(2,5) circle (3 pt)
			(7,8) circle (3 pt);

			\end{tikzpicture}
			\caption{The image $\zeta(D)\in \DDb(8)^{\circ 3, \ast 2}$ and its bounce word.}
			
		\end{minipage}
		
	\end{figure*}
	
\end{proof}
\section{The $\psi$ map exchanging peaks and falls} \label{sec:psi_map}

We define a map that explains an interesting symmetry.

\begin{theorem} \label{thm:psi_map}
	There exists a bijective map $$\psi: \DDp(n)^{\circ a, \ast b} \rightarrow  \DDp(n)^{\circ b, \ast a}$$ such that for all $D\in \DDp(n)^{\circ a, \ast b}$ \begin{align*}
		\area(D) & = \mathsf{area} (\psi(D)) \\
		\pbounce(D) & = \pbounce(\psi(D)).
	\end{align*}
\end{theorem}

\begin{proof}Let us first look at a simpler map, $\psi_0$, that transforms one decorated fall into a decorated peak, conserving the said statistics. Call the endpoint of the decorated fall $x$. We travel southward from $x$ until we hit the main diagonal then travel west until we hit the endpoint of a north step of the path. Call this point $y$. Now we modify the portion of the path between $x$ and $y$ by deleting the decorated east step and by adding an east step right after the point $y$. See Figure~\ref{fig:psi0}. Since $y$ was the endpoint of a north step and we added an east step after $y$, we have created a peak which we decorate. 
	
	\begin{figure}[!ht]
		\centering
		\begin{minipage}{.5 \textwidth}
			\centering
			\begin{tikzpicture}
			\draw
			(2,4) node[anchor=south west] {$x$}
			(0,2) node[anchor=north west] {$y$};
			
			\draw[gray!60, thin] grid (0,0) grid (5,5) (0,0)--(5,5);
			\fill[pattern=north west lines, pattern color=gray](1,2) rectangle (2,4);
			\draw[blue!60, line width=1.6pt](0,0)--(0,3)-|(1,4)--(3,4)|-(5,5);
			\draw[red!90, line width=1.6pt, sharp <-sharp >, sharp < angle = 45] (1,4) -- (2,4);
			
			\draw(1.5,4.5) node {$\ast$};
			\draw[thick, dashed] (2,4)--(2,2)--(0,2);
			\filldraw (2,4) circle (1.5pt) (0,2) circle (1.5pt);
			\end{tikzpicture}
		\end{minipage}%
		\begin{minipage}{.5\textwidth}
			\centering
			\begin{tikzpicture}
			\draw
			(2,4) node[anchor=south west] {$x$}
			(0,2) node[anchor=north west] {$y$};
			
			\draw[gray!60, thin] grid (0,0) grid (5,5) (0,0)--(5,5);
			\draw[blue!60, line width=1.6pt](0,0)--(0,2)--(1,2)|-(2,3)|-(3,4)|-(5,5);				
			\draw[red!90, line width=1.6pt, sharp <-sharp >, sharp < angle = 45, sharp > angle = 45] (0,2) -- (1,2);
			
			\filldraw[fill=blue!60] (0,2) circle (3 pt);
			\filldraw (2,4) circle (1.5pt) (0,2) circle (1.5pt);
			\end{tikzpicture}
		\end{minipage}
		\caption{The map $\psi_0$.} \label{fig:psi0}
	\end{figure}
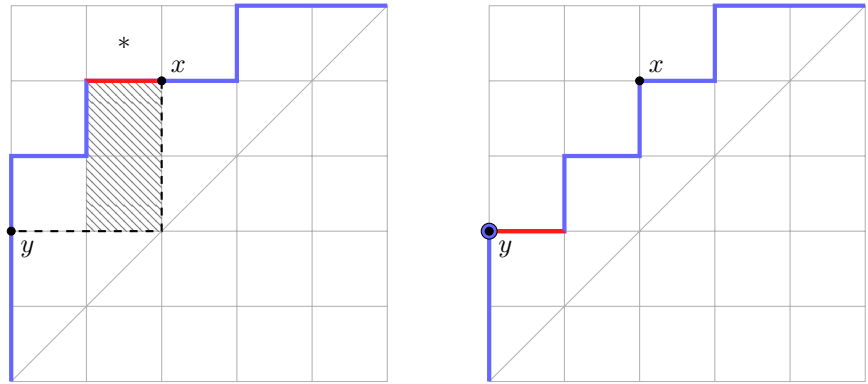
	
	Since $x$ is the endpoint of a fall, it is never on the main diagonal. It follows that the image is still a Dyck path. Furthermore, the area has not been altered since the number of squares beneath the path are exactly diminished by the number of squares under the decorated fall. Finally, recalling the definition of $\pbounce$, it is easy to see that the pbounce path is not altered, so that this statistic is not altered either. Clearly, $\psi_0$ is invertible. 
	
	Notice that $\psi_0^{-1}$ would not work with a decoration on the peak in the top row: the added horizontal step would be the last of the path, so not a fall. This explains why $\psi$ is defined on $\DDp(n)^{\circ a, \ast b} $. 
	
	The general map $\psi$ relies on $\psi_0$ as follows: apply $\psi_0$ to all the decorated falls and $\psi_0^{-1}$ to all the decorated peaks, one by one. However, we must be careful, because the order in which we transform the decorations matters. We use the following order:
	\begin{enumerate}
		\item apply $\psi_0$, from left to right, to all decorated falls. We end up with a path all of whose decorations are on peaks;
		\item apply $\psi_0^{-1}$, from top to bottom, to all the decorated peaks that did not come from decorated falls in step (1). 
	\end{enumerate}
	
	See Figure~\ref{fig:psi_definition} for an example.
	
	Notice that if we do not respect these rules, then the area of the final path might not be the same. Indeed $\psi_0$ requires pushing a portion of the path to the right. If there happens to be a decorated fall in this portion, then one decreases the number of squares under this step and so increases the overall area. Similarly, when applying $\psi_0^{-1}$, a portion of the path gets pushed to the left, so when there is a decorated fall in this portion the overall area decreases. These two situations are always avoided by respecting the given order. Since every step of $\psi$ preserves $\area$ and $\pbounce$, so does $\psi$. 
	
	It is clear that $\psi^2: \DDp(n)^{\circ a, \ast b} \to \DDp(n)^{\circ a, \ast b}$ is the identity map, in other words $\psi$ is an involution, so it is bijective. 
	
	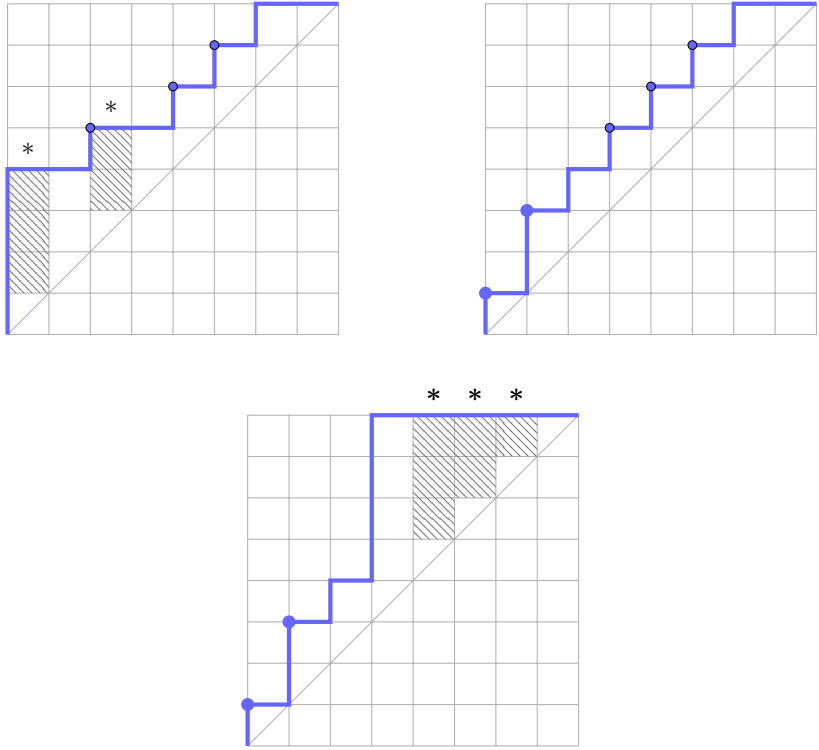
\begin{figure}[!ht]
		\centering
		\begin{minipage}{.5 \textwidth}
			\centering
			\begin{tikzpicture}[scale=0.55]
			\draw[gray!60, thin] (0,0) grid (8,8) (0,0)--(8,8);
			\fill[pattern=north west lines, pattern color=gray] (0,1) rectangle (1,4) (2,3) rectangle (3,5);
			\draw[blue!60, line width=1.6pt] (0,0) -- (0,4) -- (2,4) -- (2,5) -- (4,5) -- (4,6) -- (5,6) -- (5,7) --(6,7) -- (6,8) -- (8,8);
			
			\draw
			(0.5,4.5) node {$\ast$} 
			(2.5,5.5) node {$\ast$};
			
			\filldraw[fill=blue!60]
			(2,5) circle (3 pt)
			(4,6) circle (3 pt)
			(5,7) circle (3 pt);
			\end{tikzpicture}	
		\end{minipage}%
		\begin{minipage}{.5 \textwidth}
			\centering
			\begin{tikzpicture}[scale=0.55]
			\draw[gray!60, thin](0,0) grid (8,8) (0,0)--(8,8);
			\draw[blue!60, line width=1.6pt](0,0)|-(1,1)--(1,3)--(2,3)|-(3,4)|-(4,5)|-(5,6)|-(6,7)|-(8,8);
			
			\filldraw[blue!60, ultra thick]
			(0,1) circle (3 pt)
			(1,3) circle (3 pt);
			
			\filldraw[fill=blue!60]
			(3,5) circle (3 pt) 
			(5,7) circle (3 pt)
			(4,6) circle (3 pt);
			\end{tikzpicture}		
		\end{minipage}
		\vspace{.5cm}
		
		\begin{minipage}{\textwidth}
			\centering
			\begin{tikzpicture}[scale=0.55]
			\draw[gray!60, thin](0,0) grid (8,8) (0,0)--(8,8);
			\fill[pattern=north west lines, pattern color=gray] (5,5) |- (6,6) |- (7,7) |- (8,8) -- (4,8) -- (4,5) -- (5,5);
			\draw[blue!60, line width=1.6pt](0,0)|-(1,1)--(1,3)--(2,3)|-(3,4)|-(8,8);
			
			\filldraw[blue!60, ultra thick]
			(0,1) circle (3 pt)
			(1,3) circle (3 pt);
			\draw
			(4.5,8.5) node {$\boldsymbol{\ast}$}
			(5.5,8.5) node {$\boldsymbol{\ast}$}
			(6.5,8.5) node {$\boldsymbol{\ast}$};
			\end{tikzpicture}	
		\end{minipage}
		\caption{The two steps in the definition of the map $\psi$}
		\label{fig:psi_definition}
	\end{figure}
\end{proof}

\section{Combinatorial recursions}
\label{sec:combinatorial_recursions_ddyck}

Let us introduce some notation.

\begin{align*}
	&\DDd_{q,t}(n)^{\circ a, \ast b} \coloneqq \sum_{D\in \DDd(n)^{\circ a, \ast b}} q^{\dinv(D)} t^{\area(D)} \\
	&\DDb_{q,t}(n)^{\circ a, \ast b} \coloneqq \sum_{D\in \DDb(n)^{\circ a, \ast b}} q^{\area(D)} t^{\bounce(D)} \\ 
	&\DDp_{q,t}(n)^{\circ a, \ast b} \coloneqq \sum_{D\in \DDp(n)^{\circ a, \ast b}} q^{\area(D)} t^{\pbounce(D)}.
\end{align*}

In general, when considering a subset of one of the sets $\DDd(n)^{\circ a, \ast b}, \DDb(n)^{\circ a, \ast b}$ and $ \DDp(n)^{\circ a, \ast b}$, and subscript $q,t$ appears, it denotes the $q,t$-polynomial that is the sum over this subset with the relevant bistatisitc.

In this section, we study the polynomials of this kind by breaking them into pieces and finding a recursive way to construct these pieces. 

\subsection{bounce}

Let $a,b,n,k\in \mathbb{N}$, with $1 \leq k \leq n$. We introduce two new sets 

\[ \DDb(n\backslash k \backslash 0 )^{\circ a, \ast b} \subseteq \DDb(n \backslash k )^{\circ a, \ast b} \subseteq \DDb(n)^{\circ a, \ast b} \]

\begin{itemize}
	\item a path in $ \DDb(n)^{\circ a, \ast b}$ is a path in $ \DDb(n \backslash k)^{\circ a, \ast b}$ if its bounce word contains exactly $k$ zeros; 
	\item a path in $ \DDb(n \backslash k)^{\circ a, \ast b}$ is a path in $ \DDb(n \backslash k \backslash 0 )^{\circ a, \ast b}$ if its first $k$ columns do not contain any decorated falls.   
\end{itemize}

Since $\DDb(n)^{\circ a, \ast b} = \bigsqcup_{k=1}^n \DDb(n \backslash k)^{\circ a, \ast b}$, we have \[ \DDb_{q,t}(n)^{\circ a, \ast b} = \sum_{k=1}^n \DDb_{q,t}(n \backslash k )^{\circ a, \ast b} \]

We prove two recursive formulas for $\DDb_{q,t}(n \backslash k )^{\circ a, \ast b}$ and $\DDb_{q,t}(n \backslash k \backslash 0)^{\circ a, \ast b}$.

\begin{theorem}
	\label{thm:reco_first_bounce_S_R}
	Let $a,b,n,k\in \mathbb{N}$, $1 \leq k \leq n$. Then 
	\begin{align} 
		\DDb_{q,t}(n\backslash n)^{\circ a, \ast b} & = \delta_{a,0} q^{\binom{n-b}{2}}\qbinom{n-1}{b}_q  ,
	\end{align}
	and for $k < n$
	\begin{equation}
		\DDb_{q,t}(n\backslash k )^{\circ a, \ast b}=\sum_{s=0}^{b} \qbinom{k}{s}_q  \DDb_{q,t}(n-s \backslash k-s \backslash 0 )^{\circ a, \ast b-s}.
	\end{equation}
\end{theorem}

\begin{proof}
	First of all, consider the case $k=n$. The paths in $\DDb(n \backslash n)^{\circ a, \ast b}$ are the paths that consist of $n$ north steps followed by $n$ east steps and so the bounce must be zero. There can be no decorations of the first and only peak of these paths so $a$ must equal zero. Of the $n$ consecutive east steps, $b$ are decorated falls. The last east step must be undecorated since it is not a fall. We choose an interlacing between the remaining $n-1-b$ undecorated and $b$ decorated east steps. In each interlacing the contributions of the undecorated east steps are distinct, between $1$ and $n$. It follows that the area is $q$-counted by \[\delta_{a,0}q^{n-1-b}q^{\binom{n-1-b}{2}} \qbinom{n-1}{n-1-b}_q= \delta_{a,0}q^{\binom{n-b}{2}} \qbinom{n-1}{b}_q.\]

	Now consider $k<n$. Take a path $P \in\DDb(n-s \backslash k-s \backslash 0 )^{\circ a, \ast b-s}$. The first $k-s$ columns of $P$ do not contain any decorated fall. We will add $s$ falls to these first columns and prepend $s$ vertical steps to $P$, thus obtaining a  path $P'$ in $\DDb(n \backslash k )^{\circ a, \ast b}$ (as in Figure \ref{fig R and S}). Going through all the possible values of $s$, we obtain in this way all the paths of $\DDb(n \backslash k )^{\circ a, \ast b}$. The bounce remains unchanged since one obtains the bounce word of $P'$ by prepending $s$ zeros to the bounce word of $P$, and the decorated peaks stay in the same position. Choose an interlacing between the $k-s$ first horizontal steps of $P$ and the $s$ falls that are to be added. Reading the interlacing left to right, we prepend the decorated falls to the following undecorated horizontal steps. We may also end the interlacing with a decorated fall since $P$ has $k-s$ zeros in its bounce word, so the portion of $P$ in the first $k-s$ columns is followed by a horizontal step. The squares under the decorated falls do not contribute to the area but every time an undecorated horizontal step precedes a decorated fall in the interlacing, one square of area is added. This explains the factor $\qbinom{k-s+s}{s}_q = \qbinom{k}{s}_q$ and proves the theorem.	
\end{proof}

\begin{figure}[!ht]
	\centering
	\begin{minipage}{.35\textwidth}
		\centering
		\begin{tikzpicture}[scale=0.6]
		\draw[gray!60, thin] (3,3) grid (10,10) (3,3)--(10,10);
		\draw[dashed, opacity=0.6, ultra thick] (3,3)|-(6,6)|-(9,9)|-(10,10);
		\draw[blue!60, line width=1.6pt] (3,3)|-(4,6)|-(5,7)|-(6,8)--(6,9)--(8,9)|-(10,10);
		\filldraw[fill=blue!60] (4,7) circle (3 pt);
		\draw(6.5,9.5) node {$\ast$};
		\fill[gray, opacity=.4] (3,4) rectangle (4,6)(4,5) rectangle (5,7)(5,6) rectangle (6,8);
		\end{tikzpicture}
	\end{minipage}%
	\begin{minipage}{.65 \textwidth}
		\centering	
		\begin{tikzpicture}[scale=0.6]
		\draw[gray!60] (0,0) grid (10,10) (0,0)--(10,10);
		\fill[pattern=north west lines, pattern color=gray](0,1) rectangle (1,6)(4,5) rectangle (5,9)rectangle (6,6);
		
		\draw[thick, opacity=0.7] (2.9,2.9) rectangle (10.1,10.1); 
		\draw[blue!60, sharp <-sharp >, sharp > angle = -45, line width=1.6pt] (6,9)--(8,9)|-(10,10) (0,3)--(0,6) (1,6)-|(2,7) -|(3,8)-|(4,9);
		\draw[red!90, sharp <-sharp >, sharp < angle = 45, line width=1.6pt] (0,0)--(0,3) (0,6)--(1,6) (4,9)--(6,9);
		
		\draw
		(0.5,6.5) node {$\ast$}
		(4.5,9.5) node {$\ast$}
		(5.5,9.5) node {$\ast$}
		(6.5,9.5) node {$\ast$};
		
		\filldraw[fill=blue!60] (2,7) circle (3 pt);
		
		\fill[gray, opacity=0.4]
		(1,2) rectangle (2,4)
		(2,3) rectangle (3,5)
		(3,4) rectangle (4,6);
		\draw[dashed, opacity=0.6, ultra thick] (0,0) |- (6,6) |- (9,9) |- (10,10);
		\end{tikzpicture}
	\end{minipage}
	
	\caption{Illustration of the proof of Theorem~\ref{thm:reco_first_bounce_S_R}.} \label{fig R and S}
\end{figure}
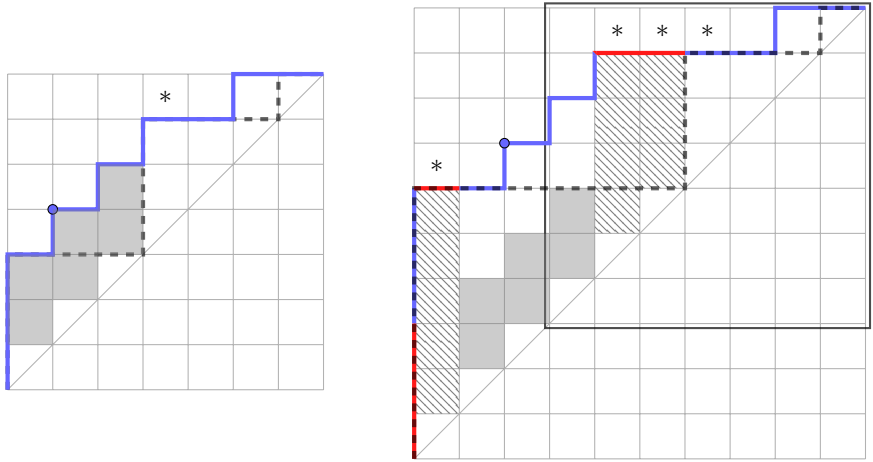

\begin{theorem}
	\label{thm:reco_first_bounce_R_S}
	Let $a,b,n,k \in \mathbb{N}$, with $1 \leq k \leq n$. Then \[ \DDb_{q,t}(n \backslash n\backslash 0)^{\circ a, \ast b} = \delta_{0,a} \delta_{0,b}
	q^{\binom{n}{2}}, \]
	and for $1 \leq k \leq n-1$ 
	\begin{align*}
		\DDb_{q,t}(n \backslash k \backslash 0)^{\circ a, \ast b}  & = t^{n-k-a} \sum_{i=0}^a q^{\binom{k}{2}+\binom{i+1}{2}} \sum_{h=1}^{n-k} \qbinom{h-1}{i}_q \qbinom{h+k-i-1}{h}_q \\
		& \times \left( \DDb_{q,t}(n-k \backslash h)^{\circ a-i, \ast b} + \DDb_{q,t}(n-k \backslash h)^{\circ a-i-1, \ast b} \right).
	\end{align*}
\end{theorem}

\begin{proof}
	First suppose that $k=n$, that is, the bounce word contains $k$ zeros so the path must go north for $n$ steps and then east for $n$ steps. It follows that the bounce equals zero.  It is clear that for $\DDb(n \backslash n \backslash 0 )^{\circ a, \ast b} $ to be nonempty, $a$ and $b$ have to equal $0$ since we cannot decorate the first (and only) peak and the first $k=n$ columns may not contain a decorated fall. So there is only one path in $ \DDb(n \backslash n \backslash 0)^{\circ a, \ast b}$ and its area is equal to $\binom{n}{2}$. 
	
	Now suppose that $k<n$. We explain a procedure that uniquely describes an element of $\DDb(n \backslash k \backslash 0)^{\circ a, \ast b}$.
	
	Start with an element $D$ of $\DDb(n-k \backslash h)^{\circ a-i, \ast b} $ for $1 \leq h \leq n-k$. We will extend this path to obtain an element $D'$ of $\DDb(n \backslash k \backslash 0 )^{\circ a, \ast b}$, keeping track of the statistics during the process. We refer to Figure~\ref{bounce recursion} for a visual aid. Consider an $n \times n$ grid. Place $D$ in the top right corner of this grid such that $D$ goes from $(k,k)$ to $(n,n)$. Since the number of $0$'s in the bounce word of $D$ equals $h$ we know that $D$ starts with $h$ vertical steps, followed by a horizontal step.  Delete these $h$ vertical steps from $D$. To obtain a path $D'$ of size $n$ we need to complete what remains of $D$ with a path from $(0,0)$ to $(k,k+h)$ which we will call $P$. Note that $D'$ must have $k$ zeros in its bounce word, so $P$ must start with $k$ vertical steps followed by a horizontal step. Hence the path that is still to be determined goes from $(1,k)$ to $(k,k+h)$. Note that we have $h \geq 1$ since $k < n$.
	
	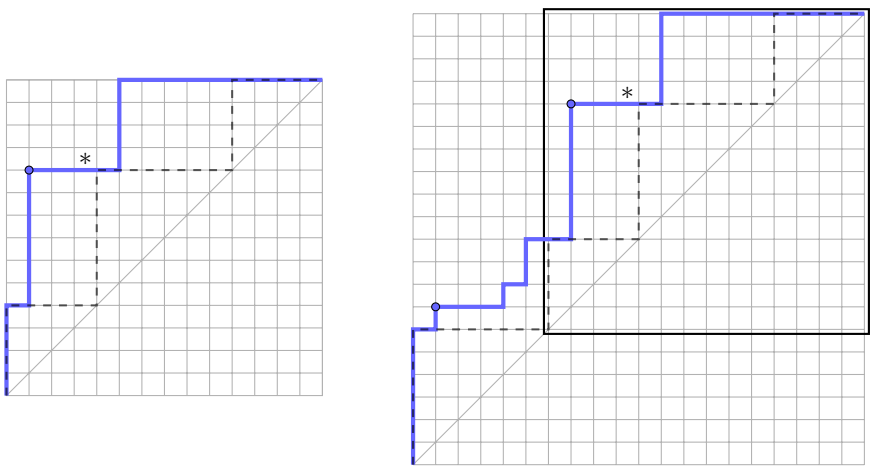
\begin{figure}[!ht]
		\centering
		\begin{minipage}{0.4 \textwidth}
			\centering
			\begin{tikzpicture}[scale=0.6]
			\draw[gray!60, thin] (3,3) rectangle (10,10);
			\draw[gray!60, thin, step=0.5] (3,3) grid (10,10) (3,3) -- (10, 10);
			\draw[blue!60, line width=1.6pt](3,3) |-(3.5,5) |-(5.5,8)|-(10,10);
			\draw[thick, dashed, opacity=0.6](3,3)|-(5,5)|-(8,8)|-(10,10);
			\filldraw[fill=blue!60] (3.5,8) circle (2.5 pt);
			\draw (4.75, 8.25) node {$\ast$};
			\end{tikzpicture}
			
		\end{minipage}%
		\begin{minipage}{0.6 \textwidth}
			\centering
			\begin{tikzpicture}[scale=0.6]
			\draw[gray!60, thin, step=0.5] (0,0) grid (10,10) (0,0) -- (10, 10);
			\draw[thick] (2.9,2.9) rectangle (10.1,10.1); 
			\draw[blue!60, line width=1.6pt](0,0) |-(.5,3)|-(2,3.5)|-(2.5,4)|- (3.5,5) |-(5.5,8)|-(10,10);	
			\draw[thick, dashed, opacity=0.6](0,0) |-(3,3)|-(5,5)|-(8,8)|-(10,10);			
			\filldraw[fill=blue!60] (3.5,8) circle (2.5 pt) (0.5,3.5) circle (2.5 pt);
			\draw (4.75, 8.25) node {$\ast$};
			\end{tikzpicture}
		\end{minipage}
		\caption{Theorem~\ref{thm:reco_first_bounce_R_S}: $D$ on the left and $D'$ on the right.}
		\label{bounce recursion}
	\end{figure}
	
	At this point we can already determine the bounce of $D'$. We have \[ \bounce(D')=\bounce(D)+(n-k-a). \] Indeed its bounce word is formed by $k$ $0$'s followed by the letters of the bounce word of $D$, all augmented by $1$. So $n-k$ letters are augmented by one, but $a$ of them will not contribute to the bounce as they will correspond to decorated peaks (all the decorated peaks are above the line $y=k$, as the first peak cannot be decorated). This explains the factor $t^{n-k-a}$.
	
	The area of $D'$ will be the sum of the area of $D$ with the area beneath the path $P$. We construct $P$ in two steps. We construct the path from $(k,k+h)$ to $(1,k)$ keeping track of the area underneath it and above the main diagonal. We set $i$ to be the number of decorated peaks in $P$. The different steps are illustrated in Figure \ref{creating P}.
	
	\begin{enumerate}
		\item We start by creating $i$ decorated peaks. Consider a peak as a vertical step followed by a horizontal step. \emph{We require that the top row does not contain a decorated peak.} So choose an interlacing of $i$ peaks and $h-i-1$ vertical steps. The area that lies above the bounce path and under this portion of the path is reflected by the terms \[ q^{\binom{i+1}{2}}\qbinom{h-i-1+i}{i}_q=q^{\binom{i+1}{2}}\qbinom{h-1}{i}_q. \]	
		
		\item Now, we have placed all the vertical steps and $i$ horizontal steps. So there are $k-i-1$ horizontal steps left to place (indeed, the first $k$ vertical steps of $P$ must be followed by a horizontal step). We choose an interlacing between these $k-i-1$ horizontal and $h$ vertical steps. The area above the bounce path that is created is counted by \[ \qbinom{h+k-i-1}{h}_q. \] 
	\end{enumerate}
	Finally, the area under the bounce path is equal to $q^{\binom{k}{2}}$.
	
	Depending on the choices made in the construction of $P$ we have now obtained a unique path in $\DDb(n \backslash k \backslash 0)^{\circ a, \ast b}$, and kept track of its statistics. However, not all the paths of $\DDb(n \backslash k \backslash 0)^{\circ a, \ast b}$ are obtained by this procedure. Indeed, in step (1), we required that the top row of $P$ does not contain a decorated peak. We did this in order to be able to differentiate between the situation when the horizontal step following the decorated peak is part of $P$ and the situation where it isn't. We solve this by adding another term, now starting from a path in $\DDb(n-k \backslash h)^{\circ a-i-1, \ast b}  $ and considering the last row of $P$ to be always decorated (as opposed to never decorated before). Now the argument for the area is exactly the same, as a decorated peak does not influence the area. For the bounce, the argument does not change either: $D$ has $a-i-1$ decorated peaks, we place $i$ decorated peaks in step (1) and consider the last row of $P$ to contain a decorated peak, which totals to $a$ decorated peaks and so the bounce still goes up by $n-k-a$.
	\begin{figure}
		\centering
		\begin{minipage}{0.5\textwidth}
			\centering
			\begin{tikzpicture}[scale=0.5]
			\draw[gray!60, thin] (0,0) grid (8,6);
			
			\filldraw[gray, opacity=.25](5,0)--(8,0)--(8,3)--(7,3)--(7,2)--(6,2)--(6,1)--(5,1);
			\draw[gray!60, thin] (0,-8) grid (8,0);
			\filldraw[white](-.2,-8.2)--(8.2,0.2)--(8.2,-8.2)--(-.2,-8.2);
			\draw[gray!60, thin] (0,-8)--(8,0);
			
			\draw[blue!60, line width=1.6pt] (8,6)--(8,5)--(7,5)--(7,3)--(6,3)--(6,1)--(5,1)--(5,0) (0,-8)|-(1,0);
			\draw[dashed, opacity=0.6, thick] (0,-8) -- (0,0) -| (8,6);
			\draw (9.4,3) node {$h$} (4,7.5) node {$k$} (9.4,-4) node {$k$};
			
			\filldraw[fill=blue!60]
			(5,1) circle (3 pt)
			(6,3) circle (3 pt)
			(7,5) circle (3 pt);
			
			\draw[decorate,decoration={brace, mirror,amplitude=15pt},xshift=0pt,yshift=0pt] (8,0)--(8,6);
			\draw[decorate,decoration={brace,amplitude=15pt},xshift=0pt,yshift=0pt] (0,6)--(8,6);
			\draw[decorate,decoration={brace,amplitude=15pt},xshift=0pt,yshift=0pt] (8,0)--(8,-8);
			\end{tikzpicture}  
		\end{minipage}%
		\begin{minipage}{.5 \textwidth}
			\centering
			\begin{tikzpicture}[scale=0.5]
			\draw[gray!60, thin] (0,0) grid (8,6);
			\draw [gray!60 ,thin] (0,-8) grid (8,0);
			\filldraw[gray, opacity=0.3];
			
			\draw (9.4,3) node {$h$} (4,7.5) node {$k$} (9.4,-4) node {$k$};
			\draw[blue!60, line width=1.6pt] (0,-8)|-(1,0);
			
			\draw[blue!60, line width=1.6pt] (8,6) -- (8,5) -- (7,5) -- (7,4) -- (6,4) -- (6,3) -- (5,3) (3,3)|-(2,1)(1,1)--(1,0);
			\draw[red!90, sharp <-sharp >, sharp < angle = 45, sharp > angle = 45, line width=1.6pt] (6,4)--(7,4);
			\draw[red!90, line width=1.6pt] (1,1)--(2,1) (3,3)--(5,3);
			
			\draw[dashed, opacity=0.6, thick] (0,-8)--(0,0)-|(8,6);
			
			\filldraw[fill=blue!60]
			(1,1) circle (3 pt)
			(3,3) circle (3 pt)
			(7,5) circle (3 pt);
			
			\filldraw[white](-.2,-8.2)--(8.2,0.2)--(8.2,-8.2)--(-.2,-8.2);
			\draw[gray!60, thin] (0,-8)--(8,0);
			\draw[decorate,decoration={brace, mirror,amplitude=15pt},xshift=0pt,yshift=0pt](8,0)--(8,6);
			\draw[decorate,decoration={brace,amplitude=15pt},xshift=0pt,yshift=0pt](8,0)--(8,-8);
			\draw[decorate,decoration={brace,amplitude=15pt},xshift=0pt,yshift=0pt](0,6)--(8,6);
			\end{tikzpicture} 
		\end{minipage} 
		\caption{Creating $P$, step 1 and 2}\label{creating P}
	\end{figure}
	
	Summing over all possible values of $h$ and $i$, we obtain the result.
\end{proof}

Let $a,b,n,k \in \mathbb{N}$, with $1 \leq k \leq n$. We introduce two new sets \[ \DDd(n\backslash k \backslash 0 )^{\circ a, \ast b} \subseteq \DDd(n \backslash k )^{\circ a, \ast b}\subseteq \DDd(n)^{\circ a, \ast b} \]

\begin{itemize}
	\item a path in $ \DDd(n)^{\circ a, \ast b}$ is a path in $\DDd(n\backslash k )^{\circ a, \ast b}$ if its area word contains exactly $k$ zeros;   
	\item a path in $\DDd(n\backslash k )^{\circ a, \ast b}$ is a path in $ \DDd(n\backslash k \backslash 0 )^{\circ a, \ast b}$ if the peaks whose starting point lies on the main diagonal are not decorated. 
\end{itemize}

Since $\DDd(n)^{\circ a, \ast b} = \bigsqcup_{k=1}^n  \DDd(n \backslash k)^{\circ a, \ast b}$, we have \[ \DDd_{q,t}(n)^{\circ a, \ast b} = \sum_{k=1}^n \DDd_{q,t}(n \backslash k)^{\circ a, \ast b}. \] Using the $\zeta$ map, we can deduce the following results for $\DDd_{q,t}(n \backslash k)^{\circ a, \ast b}$ and $\DDd_{q,t}(n \backslash k \backslash 0)^{\circ a, \ast b}$.

The next proposition follows immediately from the construction of $\zeta$.

\begin{proposition}
	\label{prop:zeta_k}
	Let $a,b,n,k \in \mathbb{N}$, with $n \geq k \geq 1$. Then the $\zeta$ map defined in the proof of Theorem~\ref{thm:zeta_map}, suitably restricted, gives a bijection from $\DDd(n \backslash k)^{\circ a, \ast b} $ to $\DDb(n \backslash k)^{\circ b, \ast a} $ and from $\DDd(n \backslash k \backslash 0)^{\circ a, \ast b} $ to $\DDb(n \backslash k \backslash 0)^{\circ b, \ast a} $.
\end{proposition}

The following corollary follows from Proposition~\ref{prop:zeta_k} and Theorem~\ref{thm:zeta_map}.

\begin{corollary}
	\label{cor:zeta_map}
	Let $a,b,n,k \in \mathbb{N}$, with $n \geq k \geq 1$. Then
	\begin{align*}
		\DDd_{q,t}(n \backslash k \backslash 0)^{\circ a, \ast b} & = \DDb_{q,t}(n \backslash k \backslash 0)^{\circ b, \ast a} \\
		\DDd_{q,t}(n \backslash k)^{\circ a, \ast b} & = \DDb_{q,t}(n \backslash k)^{\circ b, \ast a}.
	\end{align*}
	In particular $\DDd_{q,t}(n \backslash k \backslash 0)^{\circ a, \ast b}$ and $\DDd_{q,t}(n \backslash k)^{\circ a, \ast b}$ satisfy the same recursion as $\DDb_{q,t}(n \backslash k \backslash 0)^{\circ b, \ast a}$ and $\DDb_{q,t}(n \backslash k)^{\circ b, \ast a}$ do in Theorem~\ref{thm:reco_first_bounce_R_S}.
\end{corollary}

We can now write down a recursion for the polynomials $\DDb_{q,t}(n \backslash k)^{\circ a, \ast b}$.

\begin{theorem}
	\label{thm:reco1_first_bounce}
	Let $a,b,n,k \in \mathbb{N}$, with $1 \leq k \leq n$. Then \[ \DDb_{q,t}(n \backslash n )^{\circ a, \ast b} = \delta_{a,0} q^{\binom{n-b}{2}}\qbinom{n-1}{b}_q, \]	and for $1 \leq k \leq n-1$
	\begin{align*}
		\DDb_{q,t}(n \backslash k)^{\circ a, \ast b} & = t^{n-k-a} \sum_{s=0}^{b} \sum_{i=0}^{a} \sum_{h=1}^{n-k} q^{\binom{k-s}{2}+\binom{i+1}{2}} \qbinom{k}{s}_q \qbinom{h-1}{i}_q  \\
		\times \qbinom{h+k-s-i-1}{h}_q & \left( \DDb_{q,t}(n-k \backslash h)^{\circ a-i, \ast b-s} + \DDb_{q,t}(n-k \backslash h )^{\circ a-i-1, \ast b-s}  \right),
	\end{align*}
	with initial conditions \[\DDb_{q,t}(0 \backslash k)^{\circ a, \ast b} = \DDb_{q,t}(n \backslash 0)^{\circ a, \ast b} = 0. \]
\end{theorem}

\begin{proof}
	The initial conditions are easy to check. The recursive step follows immediately by combining Theorem~\ref{thm:reco_first_bounce_S_R} and Theorem~\ref{thm:reco_first_bounce_R_S}.
\end{proof}

The following corollary follows immediately from Theorem~\ref{thm:reco1_first_bounce} and Proposition~\ref{prop:zeta_k}. This recursion for $\DDd_{q,t}(n \backslash k )^{\circ a, \ast b} $  is essentially due to Wilson (cf. \cite{Wilson-PhD-2015}*{Proposition~5.3.1.1}).

\begin{corollary}[Wilson]
	The polynomials $\DDd_{q,t}(n \backslash k)^{\circ b, \ast a} $ satisfy the same recursion as the polynomials $\DDb_{q,t}(n \backslash k)^{\circ a, \ast b}$ in Theorem~\ref{thm:reco1_first_bounce}.
\end{corollary}

We can now write down a recursion for the polynomials $\DDb_{q,t}(n \backslash k \backslash 0)^{\circ a, \ast b}$.

\begin{theorem}
	\label{thm:reco_R_first_bounce}
	Let $a,b,n,k\in \mathbb{N}$, with $1 \leq k \leq n$. Then
	\begin{align*}
		\DDb_{q,t}(n\backslash k \backslash 0)^{\circ a, \ast b} & = 
		t^{n-k-a}
		\sum_{i=0}^a   q^{\binom{k}{2}+\binom{i+1}{2}}
		\sum_{h=1}^{n-k} \sum_{s=0}^{b} \qbinom{h}{s}_q
		\qbinom{h-1}{i}_q 
		\qbinom{h+k-i-1}{h}_q \\
		& \quad \times ( \DDb_{q,t}(n-k-s \backslash h-s \backslash 0)^{\circ a-i, \ast b-s}  \\
		& \quad \quad + \DDb_{q,t}(n-k-s \backslash h-s \backslash 0)^{\circ a-i-1, \ast b-s} ) \\
		& =
		t^{n-k-a}
		q^{\binom{k}{2}+\binom{a}{2}}
		\qbinom{k}{a}_q 
		\qbinom{n-a-1}{n-k-a}_q q^{\binom{n-k-b}{2}} \qbinom{n-k-1}{b}_q \\
		& \quad +
		t^{n-k-a}
		\sum_{i=0}^a q^{\binom{k}{2}+\binom{i}{2}}
		\sum_{h=1}^{n-k-1} \sum_{s=0}^{b} \qbinom{h}{s}_q
		\qbinom{k}{i}_q 
		\qbinom{h+k-i-1}{h-i}_q \\
		& \quad \times \DDb_{q,t}(n-k-s \backslash h-s \backslash 0)^{\circ a-i, \ast b-s} 
	\end{align*}
	with initial conditions \[ \DDb_{q,t}(n \backslash n \backslash 0)^{\circ a, \ast b}  = \delta_{0,a} \delta_{0,b} q^{\binom{n}{2}}. \]
\end{theorem}

\begin{proof}
	For the first equality, just combine Theorem~\ref{thm:reco_first_bounce_R_S} with Theorem~\ref{thm:reco_first_bounce_S_R}.
	
	For the second equality, rearrange the terms and use the following elementary lemma, whose proof is in the appendix of this article. 
	\begin{lemma}
		\label{lem:elementary4}
		For $h\geq 1$ and $k,s\geq 0$
		\[ q^{s} \qbinom{h-1}{s}_q \qbinom{h+k-s-1}{h}_q + \qbinom{h-1}{s-1}_q \qbinom{h+k-s}{h}_q = \qbinom{k}{s}_q \qbinom{h+k-s-1}{h-s}_q. \]
	\end{lemma}
	
	The thesis follows.
\end{proof}

The following corollary follows immediately from Theorem~\ref{thm:reco_R_first_bounce} and Proposition~\ref{prop:zeta_k}. 

\begin{corollary}
	The polynomials $\DDd_{q,t}(n\backslash k \backslash 0 )^{\circ b, \ast a}$ satisfy the same recursion as the polynomials $\DDb_{q,t}(n\backslash k \backslash 0 )^{\circ a, \ast b}$ in Theorem~\ref{thm:reco_R_first_bounce}.
\end{corollary}

\subsection{pbounce} 

Let $a,b,n,k,i\in \mathbb{N}$, with $1 \leq i \leq k \leq n$. We introduce two sets \[ \DDp(n \backslash k \backslash i)^{\circ a, \ast b}\subseteq \DDp(n \backslash k)^{\circ a, \ast b} \subseteq \DDp(n)^{\circ a, \ast b} \]

\begin{itemize}
	\item a path in $\DDp(n)^{\circ a, \ast b} $ is a path in $\DDp(n \backslash k )^{\circ a, \ast b}$ if its pbounce word contains exactly $k$ zeros;   
	\item a path in $\DDp(n \backslash k)^{\circ a, \ast b}$ is a path in $\DDp(n \backslash k \backslash i)^{\circ a, \ast b}$ if among the $k$ first vertical steps of $D$, exactly $i$ are decorated rises. 
\end{itemize}

Observe that $ \DDp(n)^{\circ a, \ast b} = \bigsqcup_{k=1}^{n}\DDp(n \backslash k)^{\circ a, \ast b}$ and \[ \DDp(n\backslash k )^{\circ a, \ast b} = \bigsqcup_{i=0}^{b} \DDp(n \backslash k \backslash i)^{\circ a, \ast b}, \] so that

\begin{align*}
	\DDp_{q,t}(n\backslash k )^{\circ a, \ast b}  & = \sum_{i=0}^{b} \DDp_{q,t}(n\backslash k \backslash i)^{\circ a, \ast b}, \; \text{ and } \\ \DDp_{q,t}(n)^{\circ a, \ast b}  & = \sum_{k=1}^n \DDp_{q,t}(n \backslash k)^{\circ a, \ast b}  = \sum_{k=1}^n \sum_{i=0}^{b} \DDp_{q,t}(n \backslash k \backslash i)^{\circ a, \ast b}. 
\end{align*}

The following proposition is an immediate consequence of the definition of the map $\psi$ in Theorem~\ref{thm:psi_map}.
\begin{proposition}
	\label{prop:psi_k}
	The map $\psi$ defined in the proof of Theorem~\ref{thm:psi_map} restricts to a bijection between $\DDp(n \backslash k)^{\circ a, \ast b}$ and $\DDp(n \backslash k)^{\circ b, \ast a}$.
\end{proposition}

The next corollary follows immediately from Proposition~\ref{prop:psi_k} and Theorem~\ref{thm:psi_map}

\begin{corollary}
	\label{cor:comb_symmetry}
	Let $a,b,n,k \in \mathbb{N}$, with $1 \leq k \leq n$. We have \[ \DDp_{q,t}(n\backslash k )^{\circ a, \ast b}  = \DDp_{q,t}(n\backslash k )^{\circ b, \ast a}, \] so that \[ \DDp_{q,t}(n)^{\circ a, \ast b} = \DDp_{q,t}(n)^{\circ b, \ast a}. \]
\end{corollary}

The following theorem provides a recursion for the polynomials $\DDp_{q,t}(n \backslash k \backslash i)^{\circ a, \ast b}$.

\begin{theorem}
	\label{thm:reco_old_bounce}
	Let $a,b,n,k,i\in \mathbb N$ with $1 \leq k \leq n$. Then \[ \DDp_{q,t}(n \backslash n \backslash i)^{\circ a, \ast b} = \delta_{i,b} q^{\binom{n-a-b}{2}} \qbinom{n-b-1}{a}_q \qbinom{n-1}{b}_q, \] and for $k < n$
	\begin{align*}
		\DDp_{q,t}(n \backslash k \backslash i)^{\circ a, \ast b} = & t^{n-k} \sum_{p=1}^k \sum_{j=0}^{n-k} \sum_{f=0}^{b-i} q^{\binom{p-i}{2}} \qbinom{k-i}{p-i}_q \qbinom{k-1}{i}_q \qbinom{p-1+j-f}{j-f}_q  \\
		& \times \DDp_{q,t}(n-k \backslash j \backslash f)^{\circ a-k+p, \ast b-i} 
	\end{align*}
	with the initial conditions \[ \DDp_{q,t}(n\backslash 0 \backslash i )^{\circ a, \ast b} = \DDp_{q,t}(0 \backslash k \backslash i )^{\circ a, \ast b} = 0.
	\]
\end{theorem}

\begin{proof}
	In this proof, when we refer to a \emph{peak}, we mean the vertical step \emph{and} the horizontal step that follows it, as this simplifies the argument.
	
	We describe a procedure to obtain a unique path in $\DDp(n \backslash k \backslash i)^{\circ a, \ast b}$, keeping track of the statistics $\pbounce$ and $\area$. See Figure~\ref{fig:second_bounce_recursion} for a visual aid.
	
	We start with the case $k<n$. Consider a path $D$ in $\DDp(n-k \backslash j \backslash f )^{\circ a-k+p, \ast b-i} $. Call $P$ the portion of $D$ that starts at its lower left corner and ends at the endpoint of its $j$-th vertical step. Since the pbounce word of $D$ contains $j$ zeros, $P$ consists of vertical steps and decorated peaks and is followed by a horizontal step.  Furthermore, $P$ has $f$ decorated rises. We extend $D$ to a path of size $n$. Place $D$ in the top right corner of an $n \times n$ grid.
	
	\begin{enumerate}
		\item Let $p$ be an integer between $1$ and $k$. We add $p$ horizontal steps to $P$ as follows. Start by deleting the $f$ rises of $P$, but keeping track of where they were; they will be put back later. Next, prepend one horizontal step to the path. Then choose an interlacing between the $j-f$ vertical steps of the path and the $p-1$ remaining horizontal steps. Each time a vertical step precedes a horizontal step the area under $P$ and above the line $y=k$ goes up by 1. So the terms of 
		\[ \qbinom{j-f+p-1}{j-f}_q \] represent the added area in this zone. Conclude by putting back the $f$ decorated rises. This does not modify the area as the squares in the rows following decorated rises do not contribute to the area. The resulting $P$ starts at $(k-p,k)$ and ends at the same point as before. 
		\item Next, we construct a path from $(0,0)$ to $ (k-p,k-i)$. We do this by choosing a sequence of $k-p$ decorated peaks (recall what peak means here) and $p-i$ vertical steps. The squares east of this path and west of the main diagonal contributing to the area can be seen as of two kinds: those east and those west of the line $x=k-p$. Clearly there are $\binom{p-i}{2}$ squares of the first kind. The number of squares of the second kind depends on the interlacing of the peaks and vertical steps: every time that a vertical step precedes a peak, a square of area is created. This explains the factor \[ \qbinom{k-p+p-i}{p-i}_q = \qbinom{k-i}{p-i}_q. \]
		\item Finally, we add $i$ decorated rises. The path from $(0,0)$ to $(k-p,k-i)$ that we constructed contains $k-p$ peaks and $p-i$ vertical steps. The lowest of them cannot be a decorated rise, since the first step of a Dyck path is never a rise. So we choose an interlacing between $k-i-1$ vertical steps of the existing path and $i$ decorated rises. The decorated rises are inserted directly after the existing vertical step in the interlacing. This operation does not add area west of the line $x=k-p$ but does add one square of area east of this line every time a decorated rise precedes a vertical step that is not a decorated rise. Thus, we have a term \[ \qbinom{k-i-1+i}{i}_q=\qbinom{k-1}{i}_q. \]
	\end{enumerate}
	Depending on the choice of the mentioned interlacings, we obtain a unique path $D'$ of size $n$. Its pbounce word contains $k$ zeros (indeed, the first horizontal step that is not part of a decorated peak occurs after the $k$-th row), and $i$ of the first $k$ vertical steps are decorated rises. So $D' \in \DDp(n \backslash k \backslash i)^{\circ a, \ast b}$. We kept track of the area during the construction and it is not difficult to see that $\pbounce(D')=\pbounce(D)+n-k$. Indeed the pbounce word of $D'$ is $k$ zeros followed by the pbounce word of $D$ whose letters are all increased by 1. 	Summing over the possible values of $p$, $f$ and $j$, we obtain the right equation.
	
	Now consider the case when $k=n$. Since there is only one ``bounce'', it is clear that $i=b$, hence the factor $\delta_{i,b}$. Since there are no horizontal steps that are not part of peaks, step (1) is not necessary. In step (2), the last step of the interlacing is not allowed to be a decorated peak since it would be the last peak of the path. We must thus choose an interlacing between $n-a-b-1$ vertical steps and the $a$ peaks. The area this creates between the path and the line $x=b$ is $q$-counted by $\qbinom{n-b-1}{a}_q$ and the area between $x=b$ and the main diagonal equals ${n-a-b\choose 2}$. Finally, in step (3) we add the $b$ decorated rises and interlace the $n-b$ vertical steps and the $b$ decorated rises. The same argument as the general case applies and the area is $q$-counted by $\qbinom{n-1}{b}_q$.
	
	The initial conditions of this recursion are easy to check.  
	
	\begin{figure}[!ht]
		\centering
		\begin{tikzpicture}[scale=0.5]
		\draw[gray!60, thin] (0,0)--(20,20) (0,0) grid (20,20);
		
		\filldraw[gray,pattern=north west lines, pattern color=gray] (1,4) rectangle (4,5) (2,8) rectangle (8,9) (2,9) rectangle (9,10) (5,14) rectangle (14,15) (11,17) rectangle (17,18);
		
		\draw[blue!60, line width=1.6pt] (0,0)|-(1,1)--(1,6)--(2,6)|-(3,11)--(3,13)--(5,13)--(5,15)--(9,15)--(9,16)--(11,16)|-(13,18)|-(18,19)|-(20,20);
		
		\draw
		(0.5,4.5) node {$\ast$}
		(1.5,8.5) node {$\ast$}
		(1.5,9.5) node {$\ast$}
		(4.5,14.5) node {$\ast$}
		(11.5,14.5) node {$\ast$}
		(12.5,17.5) node {$\ast$}
		(10.5,17.5) node {$\ast$};
		
		\draw[ultra thick, dashed, opacity=0.6] (2,2)|-(11,11); 
		\draw[red!90, sharp <-sharp >, sharp > angle = -45, line width=1.6pt] (11,11)--(11,13)--(12,13)--(12,16)--(13,16)--(13,19);
		\draw[thick] (10.8,10.8) rectangle (20.2,20.2);
		
		\filldraw[fill=blue!60]
		(0,1) circle (3 pt)
		(1,6) circle (3 pt)
		(3,13) circle (3 pt)
		(9,16) circle (3 pt)
		(12,16) circle (3 pt)
		(18,20) circle (3 pt)
		(11,13) circle (3 pt);
		
		\draw (0.6, 11.4) node {$(k-p,k)$} (3.5,1.5) node {$(k-p,k-p)$} (11.5,10.4) node{$(k,k)$} ;
		\end{tikzpicture}
		\caption{The recursion for $(\area, \pbounce)$.}
		\label{fig:second_bounce_recursion}
	\end{figure}
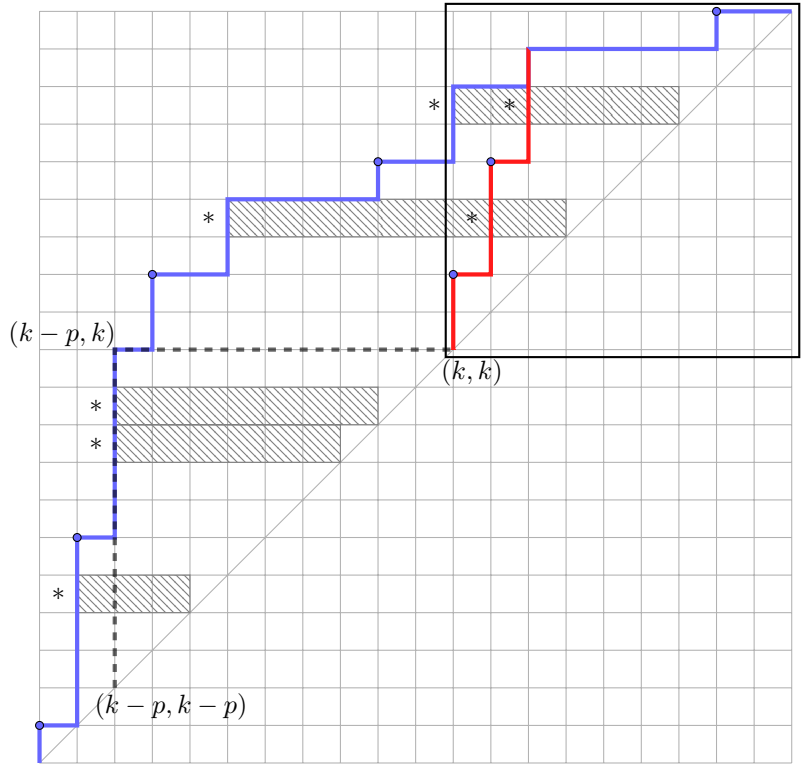
\end{proof}

\chapter{Combinatorics of polyominoes} \label{chapter:proofs-polyominoes}

In this chapter we discuss the combinatorics of parallelogram polyominoes, including a reduced version of these objects, decorations, labelling, and their relations to the combinatorics of Dyck paths.

\section{Parallelogram polyominoes} \label{sec:original_polyo}

In order to prove results about parallelogram polyominoes, it is convenient to refine our sets according to some extra parameters. Let us fix some notation.
\begin{align}
	&\PP(m,n) && \coloneqq \quad \{ P \mid P \; \text{is a $m \times n$ parallelogram polyomino} \} \\
	&\PP(m \backslash r, n) && \coloneqq \quad \{ P \in \PP(m,n) \mid P \; \text{has $r$ $1$'s in its area word} \} \label{dinvarea1} \\
	&\PP(m \backslash r, n)^{\ast k} && \coloneqq \quad \{ P \in \PP(m,n)^{\ast k} \mid  P \; \text{has $r$ $1$'s in its area word} \} \label{dinvarea2} \\
	&\PP(m, n \backslash s) && \coloneqq \quad \{ P \in \PP(m,n) \mid P \; \text{has $s$ $1$'s in its bounce word} \} \label{areabounce1} \\
	&\PP(m, n \backslash s)^{\circ k} && \coloneqq \quad \{  P \in \PP(m,n)^{\circ k} \mid P \; \text{has $s$ $1$'s in its bounce word}  \} \label{areabounce2}
\end{align}

In \cite{Aval-DAdderio-Dukes-Hicks-LeBorgne-2014}*{Section~4}, the authors give a bijection $\zeta \colon \PP(n,m) \rightarrow \PP(m,n)$ swapping $(\area, \bounce)$ and $(\dinv, \area)$. We recall here its definition.

We refer to Example~\ref{ex: zeta-pol} and Figure~\ref{fig:zetamap}. Pick a parallelogram polyomino and draw its bounce path; then, project the labels of the bounce path on both the red and the green path, i.e. label the horizontal steps (of both the red and the green path) with the label of the step of the bounce path in the same column, and the vertical steps (of both the red and the green path) with the label of the step of the bounce path in the same row.

Now, build the area word of the image as follows: pick the first bounce point on the red path, and write down the $\bar{0}$ and the $1$'s as they appear going downwards along the red path (in this case, the relative order will always be with the $\bar{0}$ first, and all the $1$'s next). Then, go to the first bounce point on the green path, and insert the $\bar{1}$'s after the correct number of $1$'s, in the same relative order in which they appear going downwards to the previous bounce point. If a letter is decorated, keep the decoration. Now, move to the second bounce point on the red path, and repeat.

As proved in \cite{Aval-DAdderio-Dukes-Hicks-LeBorgne-2014}*{Section~4}, the result will be the area word of a $n \times m$ parallelogram polyomino. It is also proved that $\area$ is mapped to $\dinv$, since the squares of the starting parallelogram polyomino correspond to the inversions in the image.

\begin{example}[The $\zeta$ map.]\label{ex: zeta-pol}
	
	First of all, we project the labels of the bounce path of the polyomino in Figure~\ref{fig:zetamap} on both the red and the green path. For clarity, we only show the labels up to $\bar{1}$ in the picture. The first bounce point on the red path is $(1,3)$: starting from there and going downwards along the red path, we read $\bar{0} 1 1 1$, which gives us the relative order of $\bar{0}$'s and $1$'s in the area word of the image.
	
	Now we move to the first bounce point on the green path, which is $(5,3)$: starting from there and going downwards along the green path, we read $1 1 \bar{1} {\color{red} \bar{1}} 1 \bar{1} \bar{1}$, which gives us the relative order of $1$'s and $\bar{1}$'s in the area word of the image. Combining it with the previous step, we get $\bar{0} 1 1 \bar{1} {\color{red} \bar{1}} 1 \bar{1} \bar{1}$. The red ${\color{red} \bar{1}}$ is the image of the label corresponding to the first decorated red peak, and thus it will be part of a decorated rise: the reason will be clear after one more step.
	
	Next, we move to the second bounce point on the red path, which is $(5,4)$: starting from there and going downwards along the red path, we read $\bar{1} {\color{red} \bar{1}} 2 \bar{1} \bar{1}$, (the $2$ is not shown in the image due to lack of space) which gives us the relative order of $\bar{1}$'s and $2$'s in the area word of the image. Combining it with the previous step, we get $\bar{0} 1 1 \bar{1} {\color{red} \bar{1}} 2 1 \bar{1} \bar{1}$. Since the red ${\color{red} \bar{1}}$ corresponded to a (decorated) peak, when reading the projected labels going downwards it must be followed by a $2$, hence it must be a rise (which we decorate).
	
	Iterating the procedure until the last point of the bounce path, we get the parallelogram polyomino with area word $\bar{0} 1 1 \bar{1} {\color{red} \bar{1}} 2 \bar{2} \bar{2} 3 \bar{3} {\color{red} \bar{3}} 4 4 \bar{4} \bar{4} \bar{3} 1 \bar{1} \bar{1}$.
	
	\begin{figure}[!ht]
		\begin{center}
			\begin{tikzpicture}[scale=0.6]
			\draw[gray!60, thin] (0,0) grid (12,7);
			\filldraw[yellow, opacity=0.3] (0,0) -- (3,0) -- (3,1) -- (5,1) -- (5,3) -- (7,3) -- (7,4) -- (10,4) -- (10,5) -- (12,5) -- (12,7) -- (8,7) -- (8,5) -- (5,5) -- (5,4) -- (3,4) -- (3,3) -- (0,3) -- cycle;
			
			\draw[green, sharp <-sharp >, sharp angle = 45, line width=1.6pt] (0,0) -- (3,0) -- (3,1) -- (5,1) -- (5,3) -- (7,3) -- (7,4) -- (10,4) -- (10,5) -- (12,5) -- (12,7);
			\draw[red, sharp <-sharp >, sharp angle = -45, line width=1.6pt] (0,0) -- (0,3) -- (3,3) -- (3,4) -- (5,4) -- (5,5) -- (8,5) -- (8,7) -- (12,7);
			
			\filldraw[fill=red]
			(3,4) circle (3pt)
			(8,7) circle (3pt);
			
			\draw[dashed, opacity=0.6, thick] (0,0) -- (1,0) -- (1,3) -- (5,3) -- (5,4) -- (7,4) -- (7,5) -- (10,5) -- (10,7) -- (12,7);
			
			\node[above] at (0.5,0) {$\bar{0}$};
			\node[right] at (1,0.5) {$1$};
			\node[right] at (1,1.5) {$1$};
			\node[right] at (1,2.5) {$1$};
			\node[above] at (1.5,3) {$\bar{1}$};
			\node[above] at (2.5,3) {$\bar{1}$};
			\node[red, above] at (3.5,3) {$\bar{1}$};
			\node[above] at (4.5,3) {$\bar{1}$};
			\node[right] at (5,3.5) {$2$};
			\node[above] at (5.5,4) {$\bar{2}$};
			\node[above] at (6.5,4) {$\bar{2}$};
			\node[right] at (7,4.5) {$3$};
			\node[above] at (7.5,5) {$\bar{3}$};
			\node[red, above] at (8.5,5) {$\bar{3}$};
			\node[above] at (9.5,5) {$\bar{3}$};
			\node[right] at (10,5.5) {$4$};
			\node[right] at (10,6.5) {$4$};
			\node[above] at (10.5,7) {$\bar{4}$};
			\node[above] at (11.5,7) {$\bar{4}$};
			
			\node[gray, left] at (0,0.5) {$1$};
			\node[gray, left] at (0,1.5) {$1$};
			\node[gray, left] at (0,2.5) {$1$};
			\node[gray, above] at (0.5,3) {$\bar{0}$};
			
			\node[gray, below] at (1.5,0) {$\bar{1}$};
			\node[gray, below] at (2.5,0) {$\bar{1}$};
			\node[gray, left] at (3,0.5) {$1$};
			\node[pink, below] at (3.5,1) {$\bar{1}$};
			\node[gray, below] at (4.5,1) {$\bar{1}$};
			\node[gray, left] at (5,1.5) {$1$};
			\node[gray, left] at (5,2.5) {$1$};				
			
			\node[pink, above] at (3.5,4) {$\bar{1}$};
			\node[gray, above] at (4.5,4) {$\bar{1}$};
			\end{tikzpicture}
		\end{center}
		
		\caption{The first three steps needed to compute the $\zeta$ map.}
		\label{fig:zetamap}
	\end{figure}
\end{example}

In fact, the map $\zeta$ has a stronger property.

\begin{theorem}
	\label{th:zetamap}
	For $m \geq 1$, $n \geq 1$, $k \geq 0$, and $1 \leq r \leq m$, the bijection $\zeta \colon \PP(n,m) \rightarrow \PP(m,n)$ in \cite{Aval-DAdderio-Dukes-Hicks-LeBorgne-2014}*{Theorem~4.1} extends to a bijection \[ \zeta \colon  \PP(n, m \backslash r)^{\circ k}\rightarrow  \PP(m \backslash r, n)^{\ast k}\] mapping the bistatistic $(\area, \bounce)$ to $(\dinv, \area)$.
\end{theorem}

\begin{proof}
	We already recalled the definition of the $\zeta$ map, and the fact that it sends $\area$ into $\dinv$. We also know that $\bounce$ would go into $\area$ if there were no decorations. So we need to understand what happens to the decorations. 
	
	Notice that red peaks are mapped into rises, because when reading the red path top to bottom, one reads the horizontal step of the peak first, which corresponds to a barred letter, and the vertical step of the peak next, which corresponds to the next unbarred letter. Moreover, the decoration is kept on a letter with the same value. This implies that $\bounce$ is mapped to $\area$.
	
	Furthermore, by construction one has that the number of $1$'s in the bounce word is equal to the number of $1$'s in the area word of the image polyomino, since that area word is just an anagram of the bounce word of the starting polyomino. This shows that an element of $\PP(n, m \backslash r)^{\circ k}$ is actually sent into an element of $\PP(m \backslash r, n)^{\ast k}$, as we wanted.
\end{proof}

We now want to compute some $q,t$-enumerators for our sets. Let us define

\[ \PP_{q,t}(m,n) \coloneqq \sum_{P\in \PP(m,n)} q^{\area(P)} t^{\bounce(P)} = \sum_{P\in \PP(m,n)} q^{\dinv(P)} t^{\area(P)} \]

which are proved to be equal in \cite{Aval-DAdderio-Dukes-Hicks-LeBorgne-2014}*{Section~6}. We analogously denote with a subscript $q,t$ the $q,t$-enumerator for the other sets previously defined, using the bistatistic $(\dinv, \area)$ for \eqref{dinvarea1} and \eqref{dinvarea2}, and the bistatistic $(\area, \bounce)$ for \eqref{areabounce1} and \eqref{areabounce2}. We are now ready to state some recursions.

\begin{theorem}
	\label{th:dinvrecursion}
	For $m \geq 1$, $n \geq 1$, $k \geq 0$, and $1 \leq r \leq m$, the polynomials $\PP_{q,t}(m \backslash r,n)^{\ast k}$ satisfy the recursion
	\begin{align*}
		\PP_{q,t}(m \backslash r, n)^{\ast k} = t^{m+n-k-1} & \sum_{s=1}^{n-1} q^{r+s} \qbinom{r+s-1}{s}_q \sum_{h=0}^{k} q^{\binom{h}{2}} \qbinom{s}{h}_q \\
		\times & \sum_{u=1}^{m-r} \qbinom{s+u-h-1}{s-1}_q \PP_{q,t}(m-r \, \backslash \, u, \, n-s)^{\ast \, k-h}
	\end{align*}
	
	with initial conditions \[ \PP_{q,t}(m \backslash m,n)^{\ast 0} = (qt)^{m+n-1} \qbinom{m+n-2}{n-1}_q \] and $\PP_{q,t}(m \backslash r,1)^{\ast k} = \delta_{r,m}\delta_{k,0} (qt)^{m}$.
\end{theorem}

\begin{proof}
	The initial conditions are immediate to check. 
	
	Let us look at the recursive step: let $P$ be a polyomino with $k$ decorated rises, $r$ be the number of $1$'s in its area word, $s$ be the number of $\bar{1}$'s, $h$ be the number of $\bar{1}$'s with a decoration, and $u$ be the number of $2$'s. We remove from the area word of $P$ all the $1$'s and $\bar{1}$'s, and then we lower by $1$ all the other entries but $\bar{0}$. In this way we get the area word of a polyomino in $\PP(m-r \, \backslash \, u, \, n-s)^{\ast \, k-h}$.
	
	We have to check that our coefficients take care of the statistics. The factor $t^{m+n-k-1}$ takes care of the fact that every non-decorated letter but $\bar{0}$ is now contributing one unit less towards the area. The factor $q^r$ takes care of the lost dinv between the $\bar{0}$ and the $1$'s, and the factor $q^s \qbinom{r+s-1}{s}_q$ takes care of the lost dinv between the $1$'s and the $\bar{1}$'s. The factor $q^{\binom{h}{2}} \qbinom{s}{h}_q$ takes care of the lost dinv between the $\bar{1}$'s and the $2$'s that follow a decorated rise: ignoring the inversion formed with the leftmost $\bar{1}$ (which will be replaced by the inversion with the $\bar{0}$ once we lower the value from $2$ to $1$), each of these $2$'s gives a contribution between $0$ and $s-1$, and since they all form rises each of these $h$ $2$'s is directly preceded by a $\bar{1}$, hence the contributions must be all different. Next, the factor $\qbinom{s+u-h-1}{s-1}_q$ takes care of the dinv between the $s-1$ $\bar{1}$'s but the first, and the $u-h$ $2$'s that are not part of a decorated rise. The inversions formed with the leftmost $\bar{1}$ are once again replaced by the inversions with the $\bar{0}$ once we lower the value from $2$ to $1$, so their contribution is preserved.
	
	The factor $\PP_{q,t}(m-r \, \backslash \, u, \, n-s)^{\ast \, k-h}$ takes care of the dinv and the area of the remaining polyomino, and this concludes the proof.
\end{proof}

\begin{theorem}
	\label{th:bouncerecursion}
	
	For $m \geq 1$, $n \geq 1$, $k \geq 0$, and $1 \leq s \leq n$, the polynomials $\PP_{q,t}(m, n \backslash s)^{\circ k}$ satisfy the recursion
	\begin{align*}
		\PP_{q,t}(m, n \backslash s)^{\circ k} = t^{m+n-k-1} & \sum_{r=1}^{m-1} q^{r+s} \qbinom{r+s-1}{r}_q \sum_{h=0}^{k} q^{\binom{h}{2}} \qbinom{r}{h}_q \\ \times & \sum_{v=1}^{n-s} \qbinom{r+v-h-1}{r-1}_q \PP_{q,t}(m-r, n-s \, \backslash \, v)^{\circ \, k-h}
	\end{align*}
	
	with initial conditions \[ \PP_{q,t}(m, n \backslash n)^{\circ 0} = (qt)^{m+n-1} \qbinom{m+n-2}{m-1}_q \] and $\PP_{q,t}(1, n \backslash s)^{\circ k} = \delta_{s,n}\delta_{k,0} (qt)^{n}$.
\end{theorem}

\begin{proof}
	The initial conditions are immediate to check. 
	
	Let us look at the recursive step: let $P$ be a polyomino with $k$ decorated red peaks, $s$ be the number of $1$'s in its bounce word (i.e. the length of the first vertical stretch of the bounce path), $r$ be the number of $\bar{1}$'s in its bounce word (i.e. the length of the first horizontal stretch of the bounce path, ignoring the first step), $h$ be the number of decorations in the first $r+1$ columns, and $v$ be the number of $2$'s in the bounce path (i.e. the length of the second vertical stretch of the bounce path). We cut the intersection of the original polyomino with the rectangle going from $(r,s)$ to $(m,n)$ (the orange rectangle in Figure~\ref{fig:recursionareabounce}): this gives us a polyomino in $\PP_{q,t}(m-r, n-s \, \backslash \, v)^{\circ \, k-h}$.
	
	We have to check that our coefficients take care of the area and the bounce that we lose in the cutting process. Observe that the bounce path of the new polyomino is the same as the bounce path of the old one, except that it starts with the last $\bar{1}$ and its labels are all one unit smaller. The factor $t^{m+n-k-1}$ takes care of the fact that every non-decorated letter in the bounce word but $\bar{0}$ is now contributing for one unit less, since it was either lowered by $1$ or deleted.
	
	The factor $q^s$ takes care of the area of the rectangle delimited by the steps of the bounce path labelled by $\bar{0}$ or $1$, and the first $s+1$ steps of the red path (highlighted in lime in Figure~\ref{fig:recursionareabounce}).
	
	The factor $q^r \qbinom{r+s-1}{r}_q$ takes care of the area of the region delimited by the steps of the bounce path labelled by $1$ or $\bar{1}$, and the first $s+r+1$ steps of the green path (highlighted in blue in Figure~\ref{fig:recursionareabounce}).
	
	The factor $q^{\binom{h}{2}} \qbinom{r}{h}_q$ takes care of the area of the rows corresponding to red peaks in the region delimited by the steps of the bounce path labelled by $\bar{1}$ or $2$, and the steps of the red path from the $s+2$-th to the $r+s+v-1$-th (highlighted in pink in Figure~\ref{fig:recursionareabounce}), where $q^{\binom{h}{2}}$ is needed to take into account the fact that if there is a peak of the red path, then there must be a vertical step followed by a horizontal step.
	
	Next, the factor $\qbinom{r+v-h-1}{r-1}_q$ takes care of the area of the remaining rows in the same region.
	
	Finally, the contribution to area and bounce given by the rest of the polyomino is given by $\PP_{q,t}(m-r, n-s \, \backslash \, v)^{\circ \, k-h}$, and this concludes the proof.
\end{proof}

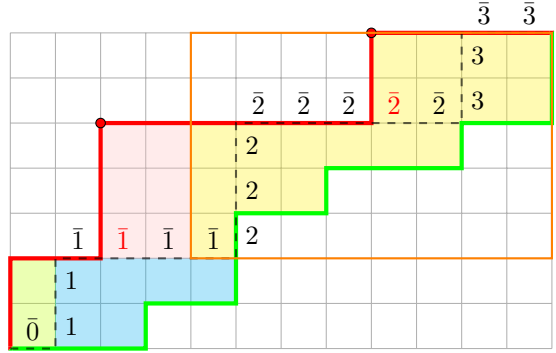
\begin{figure}[!ht]
	\begin{center}
		\begin{tikzpicture}[scale=0.6]
		\draw[gray!60, thin] (0,0) grid (12,7);
		\filldraw[yellow, opacity=0.3] (4,2) -- (5,2) -- (5,3) -- (7,3) -- (7,4) -- (10,4) -- (10,5) -- (12,5) -- (12,7) -- (8,7) -- (8,5) -- (4,5) -- cycle;
		
		\filldraw[lime, opacity=0.3] (0,0) -- (1,0) -- (1,2) -- (0,2) -- cycle;
		\filldraw[cyan, opacity=0.3] (1,0) -- (3,0) -- (3,1) -- (5,1) -- (5,2) -- (1,2) -- cycle;
		\filldraw[pink, opacity=0.3] (2,2) -- (4,2) -- (4,5) -- (2,5) -- cycle;
		
		\draw[green, sharp <-sharp >, sharp angle = 45, line width=1.6pt] (0,0) -- (3,0) -- (3,1) -- (5,1) -- (5,3) -- (7,3) -- (7,4) -- (10,4) -- (10,5) -- (12,5) -- (12,7);
		\draw[red, sharp <-sharp >, sharp angle = -45, line width=1.6pt] (0,0) -- (0,2) -- (2,2) -- (2,5) -- (5,5) -- (8,5) -- (8,7) -- (12,7);
		
		\filldraw[fill=red]
		(2,5) circle (3pt)
		(8,7) circle (3pt);
		
		\draw[dashed, opacity=0.6, thick] (0,0) -- (1,0) -- (1,2) -- (5,2) -- (5,5) -- (10,5) -- (10,7) -- (12,7);
		
		\node[above] at (0.5,0) {$\bar{0}$};
		\node[right] at (1,0.5) {$1$};
		\node[right] at (1,1.5) {$1$};
		\node[above] at (1.5,2) {$\bar{1}$};
		\node[red, above] at (2.5,2) {$\bar{1}$};
		\node[above] at (3.5,2) {$\bar{1}$};
		\node[above] at (4.5,2) {$\bar{1}$};
		\node[right] at (5,2.5) {$2$};
		\node[right] at (5,3.5) {$2$};
		\node[right] at (5,4.5) {$2$};
		\node[above] at (5.5,5) {$\bar{2}$};
		\node[above] at (6.5,5) {$\bar{2}$};
		\node[above] at (7.5,5) {$\bar{2}$};
		\node[red, above] at (8.5,5) {$\bar{2}$};
		\node[above] at (9.5,5) {$\bar{2}$};
		\node[right] at (10,5.5) {$3$};
		\node[right] at (10,6.5) {$3$};
		\node[above] at (10.5,7) {$\bar{3}$};
		\node[above] at (11.5,7) {$\bar{3}$};
		
		\draw[orange, thick] (4,2) rectangle (12,7);		
		\end{tikzpicture}
	\end{center}
	
	\caption{The first step of the recursion for $(\area, \bounce)$.}
	\label{fig:recursionareabounce}
\end{figure}

The recursion is the same one given for $(\dinv, \area)$ switching $m$ and $n$, so we have the following result, which can be also derived from Theorem~\ref{th:zetamap}.

\begin{corollary}
	$\PP_{q,t}(m \backslash r, n)^{\ast k} = \PP_{q,t}(n, m \backslash r)^{\circ k}$.
\end{corollary}

\section{Reduced polyominoes} \label{sec:reduced_polyo}

In analogy with the parallelogram polyominoes case, we have a bijection that maps $(\area, \bounce)$ to $(\dinv, \area)$. This time, it does not swap height and width. Let us fix some notation again.
\begin{align}
	&\RP(m,n) && \coloneqq \; \; \{ P \mid P \; \text{is a $m \times n$ reduced polyomino} \} \\
	&\RP(m \backslash r, n)^{\ast 0} && \coloneqq \; \; \{ P \in \RP(m,n) \mid P \; \text{has $r$ $0$'s in its area word} \} \label{mr} \\
	&\RP(m \backslash r, n)^{\ast k} && \coloneqq \; \; \{ P \in \RP(m,n)^{\ast k} \mid P \; \text{has $r$ $0$'s in its area word}  \} \\
	&\RP(m \backslash r, n)^{\circ 0} && \coloneqq \; \; \{ P \in \RP(m,n) \mid P \; \text{has $r\!-\!1$ $0$'s in its bounce word} \} \label{ns} \\
	&\RP(m \backslash r, n)^{\circ k} && \coloneqq \; \; \{ P \in \RP(m,n)^{\circ k} \mid P \; \text{has $r\!-\!1$ $0$'s in its bounce word}  \}
\end{align}

Notice that we are including the first, artificial $0$ in \eqref{mr} and that we replaced $r$ with $r-1$ in \eqref{ns}. In particular, $r$ can be equal to $m+1$.

\begin{theorem}
	\label{th:redzetamap}
	For $m \geq 0$, $n \geq 0$, $k \geq 0$, and $1 \leq r \leq m+1$, there is a bijection $\bar{\zeta} \colon \RP(m \backslash r, n)^{\circ k} \to \RP(m \backslash r, n)^{\ast k}$ mapping the bistatistic $(\area, \bounce)$ to $(\dinv, \area)$.
\end{theorem}

\begin{proof}
	The bijection is essentially the same as the one in Theorem~\ref{th:zetamap}: the only difference is that we have to read the interlacing going \textit{upwards} along the paths (i.e.\ bottom to top, left to right), and that we have to add the artificial $0$ at the beginning of the area word of the image. See Figure~\ref{fig:redzeta} for an example. 
	
	Notice that the interlacing of the letters gives an area word because of the modified bouncing rule, since the first step after a bouncing point on the green path must be vertical and the first step after a bouncing point on the red path must be horizontal. Observe also that the change from ``reading downwards'' in the original polyominoes to ``reading upwards'' in the reduced ones also agrees (as it should) with the difference in the definition of inversions (and hence of dinv) in the two settings. Finally, notice that, due to the way we label the bounce path in the reduced case, height and width do not swap under the bijection. The proof is otherwise identical. 
\end{proof}

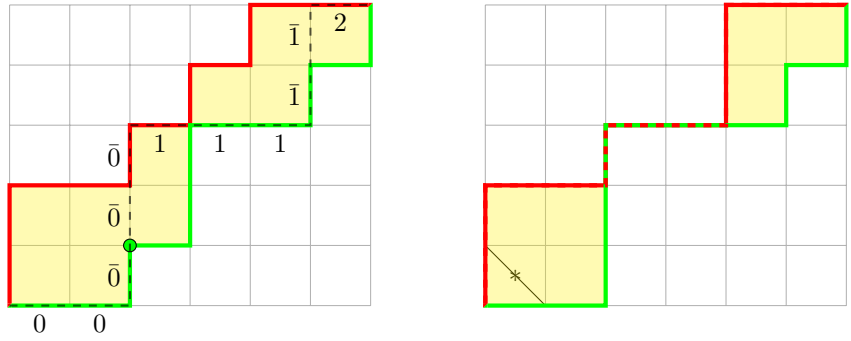
\begin{figure}[!ht]
	\begin{minipage}{0.5\textwidth}
		\begin{center}
			\begin{tikzpicture}[scale = 0.8]
			\draw[step=1.0, gray, opacity=0.6,thin] (0,0) grid (6,5);
			
			\filldraw[yellow, opacity=0.3] (0,0) -- (1,0) -- (2,0) -- (2,1) -- (3,1) -- (3,2) -- (3,3) -- (4,3) -- (5,3) -- (5,4) -- (6,4) -- (6,5) -- (5,5) -- (4,5) -- (4,4) -- (3,4) -- (3,3) -- (2,3) -- (2,2) -- (1,2) -- (0,2) -- (0,1) -- (0,0);
			
			\draw[red, sharp <-sharp >, sharp angle = -45, line width=1.6pt] (0,0) -- (0,1) -- (0,2) -- (1,2) -- (2,2) -- (2,3) -- (3,3) -- (3,4) -- (4,4) -- (4,5) -- (5,5) -- (6,5);
			
			\draw[green, sharp <-sharp >, sharp angle = 45, line width=1.6pt] (0,0) -- (1,0) -- (2,0) -- (2,1) -- (3,1) -- (3,2) -- (3,3) -- (4,3) -- (5,3) -- (5,4) -- (6,4) -- (6,5);
			
			\filldraw[fill=green] (2,1) circle (3pt);
			
			\draw[dashed, opacity=0.6, thick] (0,0) -- (2,0) -- (2,3) -- (5,3) -- (5,5) -- (6,5);
			
			\node[below] at (0.5,0) {$0$};
			\node[below] at (1.5,0) {$0$};
			\node[left] at (2,0.5) {$\bar{0}$};
			\node[left] at (2,1.5) {$\bar{0}$};				
			\node[left] at (2,2.5) {$\bar{0}$};
			\node[below] at (2.5,3) {$1$};
			\node[below] at (3.5,3) {$1$};
			\node[below] at (4.5,3) {$1$};
			\node[left] at (5,3.5) {$\bar{1}$};
			\node[left] at (5,4.5) {$\bar{1}$};
			\node[below] at (5.5,5) {$2$};
			\end{tikzpicture}
		\end{center}
	\end{minipage}%
	\begin{minipage}{0.5\textwidth}
		\begin{center}
			\begin{tikzpicture}[scale = 0.8]
			\draw[step=1.0, gray, opacity=0.6,thin] (0,0) grid (6,5);
			
			\draw (1,0) -- (0,1);
			\node at (0.5, 0.5) {$\ast$};
			
			\filldraw[yellow, opacity=0.3] (0,0) -- (1,0) -- (2,0) -- (2,1) -- (2,2) -- (2,3) -- (3,3) -- (4,3) -- (5,3) -- (5,4) -- (6,4) -- (6,5) -- (5,5) -- (4,5) -- (4,4) -- (4,3) -- (3,3) -- (2,3) -- (2,2) -- (1,2) -- (0,2) -- (0,1) -- (0,0);
			
			\draw[red, sharp <-sharp >, sharp angle = -45, line width=1.6pt] (0,0) -- (0,1) -- (0,2) -- (1,2) -- (2,2) -- (2,3) -- (3,3) -- (4,3) -- (4,4) -- (4,5) -- (5,5) -- (6,5);
			
			\draw[green, sharp <-sharp >, sharp angle = 45, line width=1.6pt] (0,0) -- (1,0) -- (2,0) -- (2,1) -- (2,2) -- (2,3) -- (3,3) -- (4,3) -- (5,3) -- (5,4) -- (6,4) -- (6,5);
			
			\draw[red, dashed, sharp <-sharp >, sharp angle = -45, line width=1.6pt] (0,0) -- (0,1) -- (0,2) -- (1,2) -- (2,2) -- (2,3) -- (3,3) -- (4,3) -- (4,4) -- (4,5) -- (5,5) -- (6,5);
			
			\node[below, white] at (0.5,0) {$0$};
			\end{tikzpicture}
		\end{center}
	\end{minipage}
	\caption{A reduced polyomino (left) and its image via the $\bar{\zeta}$ map (right). The reading word of the image is $0 \bar{0} {\color{green} 1} \bar{1} 2 \bar{0} 0 0 \bar{0} 1 \bar{1} 1$.}
	\label{fig:redzeta}
\end{figure}

Let us define the $q,t$-enumerators for those new sets analogously as we did for decorated parallelogram polyominoes. From Proposition~\ref{normalizing map} and the fact that $\PP_{q,t}(m,n)$ is symmetric in $m$ and $n$, we immediately deduce that \[ \PP_{q,t}(m,n) = (qt)^{m+n-1} \cdot \RP_{q,t}(m-1,n-1) \] where the $q,t$-enumerator for $\RP(m-1,n-1)$ is built using $(\area,\bounce)$. This implies that the latter is also symmetric in $m$ and $n$; combining this fact with Theorem~\ref{th:redzetamap}, we deduce that we could have equivalently used $(\dinv,\area)$ instead.

\begin{proposition}
	\label{prop:areabounce}
	$\PP_{q,t}(n, m \backslash r)^{\circ k} = (qt)^{m+n-1} \cdot \RP_{q,t}(m-1 \, \backslash \, r, \, n-1)^{\circ k}$.
\end{proposition}

\begin{proof}
	Applying the map $\mathsf{r}$ defined in Proposition~\ref{normalizing map} we get that both statistics decrease by $m+n-1$, and the $1$'s in the bounce word are mapped to $0$'s, except the first one that is deleted. The thesis follows.
\end{proof}

\begin{proposition}
	\label{prop:dinvarea}
	For $m \geq 1$, $n \geq 1$, $k \geq 0$, and $0 \leq r \leq m$, we have \[ \PP_{q,t}(m \backslash r, n)^{\ast k} = (qt)^{m+n-1} \cdot \RP_{q,t}(m-1 \, \backslash \, r, \, n-1)^{\ast k}. \]
\end{proposition}

\begin{proof}
	Apply $\zeta$ to the left hand side and $\bar{\zeta}$ to the right hand side, and then use Proposition~\ref{prop:areabounce}.
\end{proof}

We also have recursions for these sets.

\begin{theorem}
	\label{th:reddinvrecursion}
	
	For $m \geq 0$, $n \geq 0$, $k \geq 0$, and $1 \leq r \leq m+1$, the polynomials $\RP_{q,t}(m \backslash r, n)^{\ast k}$ satisfy the recursion
	\begin{align*}
		\RP_{q,t}(m \backslash r, n)^{\ast k} = & \sum_{s=1}^{n} t^{m+n+1-r-s-k} \qbinom{r+s-1}{s}_q \sum_{h=0}^{k} q^{\binom{h}{2}} \qbinom{s}{h}_q \\ \times & \sum_{u=1}^{m-r+1} \qbinom{s+u-h-1}{s-1}_q \RP_{q,t}(m-r \, \backslash \, u, \, n-s)^{\ast \, k-h}
	\end{align*}
	
	with initial conditions \[ \RP_{q,t}(m \backslash m+1, n)^{\ast 0} = \qbinom{m+n}{m}_q \] and $\RP_{q,t}(m \backslash r, 0)^{\ast k} = \delta_{k,0} \delta_{r,m+1}$.
\end{theorem}

\begin{theorem}
	\label{th:redbouncerecursion}
	For $m \geq 0$, $n \geq 0$, $k \geq 0$, and $1 \leq r \leq m+1$, the polynomials $\RP_{q,t}(m\backslash r,n )^{\circ k}$ satisfy the recursion
	\begin{align*}
		\RP_{q,t}(m \backslash r, n)^{\circ k} = & \sum_{s=1}^{n} t^{m+n+1-r-s-k} \qbinom{r+s-1}{s}_q \sum_{h=0}^{k} q^{\binom{h}{2}} \qbinom{s}{h}_q \\ \times & \sum_{u=1}^{m-r+1} \qbinom{s+u-h-1}{s-1}_q \RP_{q,t}(m-r \, \backslash \, u, \, n-s)^{\circ \, k-h}
	\end{align*}
	
	with initial conditions \[ \RP_{q,t}(m \backslash m+1, n)^{\circ 0} = \qbinom{m+n}{m}_q \] and $\RP_{q,t}(m \backslash r, 0)^{\circ k} = \delta_{k,0} \delta_{r,m+1}$.
\end{theorem}

Both of these can be proved either directly, with the same argument used in Theorem~\ref{th:dinvrecursion} and Theorem~\ref{th:bouncerecursion}, or from the statements of these theorems using Proposition~\ref{prop:dinvarea} and Proposition~\ref{prop:areabounce}.

\section{Two car parking functions}\label{sec:2cars_PF}

In order to understand how polyominoes are linked with the Delta conjecture, we need to pass through two car parking functions. Let us fix some notation.
\begin{align}
	&\PF^2(m,n) \! \! \! && \coloneqq \; \{ D \in \LD(m\!+\!n) \mid D \; \text{has $n$ $1$'s and $m$ $2$'s labels} \} \\
	&\PF^2(m \backslash r, n) \! \! \! && \coloneqq \; \{ D \in \PF^2(m,n) \mid D \; \text{has $r\!-\!1$ $2$'s on the main diag.} \} \\
	&\PF^2(m,n)^{\ast k} \! \! \! && \coloneqq \; \{ D \in \LDD(m\!+\!n)^{\circ 0, \ast k} \mid D \; \text{has $n$ $1$'s and $m$ $2$'s labels} \} \\
	&\PF^2(m \backslash r, n)^{\ast k} \! \! \! && \coloneqq \; \{ D \in \PF^2(m,n)^{\ast k} \mid D \; \text{has $r\!-\!1$ $2$'s on the main diag.} \}
\end{align}

And as usual we add a subscript $q,t$ to denote the relevant $q,t$-enumerators. We have the following result.

\begin{theorem}
	\label{th:polyominoesto2cpf}
	For $m \geq 0$, $n \geq 0$, $k \geq 0$, and $1 \leq r \leq m+1$, there exists a bijection $\phi \colon \RP(m \backslash r, n)^{\ast k} \rightarrow \PF^2(m \backslash r, n)^{\ast k}$ such that $(\dinv(P),\area(P)) = (\dinv(\phi(P)), \area(\phi(P)))$ for all $P\in \RP(m \backslash r, n)^{\ast k}$.
\end{theorem}

\begin{proof}
	Given the area word of a reduced polyomino $P$, remove the initial artificial $0$. Construct the area word of $\phi(P)$ as follows: its $i$-th letter is the value of the $i$-th letter of the area word of $P$. Next, label $\phi(P)$ in the following way: if the $i$-th letter of the area word of $P$ is barred then the $i$-th label of $\phi(D)$ is $1$; if it is not barred then the label is $2$. Keep the decorations as they are. Both the statistics are trivially preserved (the area word is the same, and the inversions are also the same). See Figure~\ref{fig:hbounce} for an example.
\end{proof}


We can immediately derive the following corollary.

\begin{corollary}
	\label{cor:redpolto2cpf}
	For $m \geq 0$, $n \geq 0$, $k \geq 0$, and $1 \leq r \leq m+1$, we have \[ \RP_{q,t}(m \backslash r, n)^{\ast k} = \PF^2_{q,t}(m \backslash r, n)^{\ast k}. \]
\end{corollary}

As we have seen, this map preserves the bistatistic $(\dinv, \area)$ in a trivial way. The non-trivial result is the following.

\begin{theorem}
	\label{th:pmaj}
	Let $P \in \RP(m,n)$. Then $\bounce(P) = \pmaj(\phi(P))$.
\end{theorem}

\begin{proof}
	It is enough to show that the bouncing points coincide with the points in which two consecutive letters of the parking word are different: more precisely, the bounce path bounces on the green path (i.e. stops going East and starts going North) whenever we have a $2$ followed by a $1$, and it bounces on the red path (i.e. stops going North and starts going East) whenever we have a $1$ followed by a $2$. If this holds, in fact, the value of the $i$-th letter in the bounce word (disregarding bars) coincides with the number of descents that follow $p_i(\phi(P))$ in the reverse parking word. The sum of the letters of the bounce word gives the bounce, and the sum of the number of descents to the right of each letter of the reverse parking word is exactly its major index, which is $\pmaj(\phi(P))$ by definition.	
	
	It is easy to see that the diagonals of slope $-1$ in $P$ correspond to the columns in $\phi(P)$, in the sense that we have a $1$ (resp. a $2$) in column $i$ if and only if we have a vertical red step (resp. horizontal green step) between the diagonals $x+y = i-1$ and $x+y = i$. This is a consequence of the way the area word of a reduced parallelogram polyomino is computed (see Figure~\ref{fig:decorated-reduced-polyominoes}).
	
	Hence, while writing down the parking word $p(P)$ (we will omit the dependence on $P$ from now on) according to the $\pmaj$ algorithm, we write down $2$'s until we reach the first column with no $2$'s (that could be the actual first one, in which case we write no $2$'s). This means that, on the polyomino, we hit the first diagonal with no horizontal green steps (hence with a vertical green step), so the bounce path is changing direction for the first time. Suppose that we wrote down $r-1$ $2$'s.
	
	Now we want to prove that, if we are writing down a $1$ in $p$, then the bounce path is going upwards. This is trivially true at step $r$ because the bounce path just hit the beginning of a vertical green step. Then, at step $r+i$, if $p_r = p_{r+1} = \dots = p_{r+i-1} = 1$, then $p_{r+i} = 1$ if and only if there are at least $i$ $1$'s in the first $r+i$ columns. This means that there are at least $i$ vertical red steps between the diagonals $x+y=0$ and $x+y=r+i$, which implies that the $r+i$-th red step is at least at height $i$, hence the bounce path could not have hit any horizontal red step before its $r+i$-th step, therefore the $r+i$-th step of the bounce path is vertical.
	
	Then, suppose that we wrote down $r-1$ $2$'s, and the next $2$ in $p$ is at position $r+s$ (which will conventionally be $m+n+1$ if there are no more $2$'s). This means that there is no $1$ in column $r+s$, and that there are exactly $s$ $1$'s in the first $r+s-1$ columns. In the polyomino, it means that there are exactly $s$ vertical red steps between the diagonals $x+y = 0$ and $x+y = r+s-1$, that there is a horizontal red step between the diagonals $x+y = r+s-1$ and $x+y = r+s$. Since it must be exactly at height $s$, it must be the step from $(r-1,s)$ to $(r,s)$, and the bounce path hits it after $r+s-1$ steps (see Figure~\ref{fig:hbounce}).
	
	We can now iterate the argument swapping the roles of $1$'s and $2$'s. Suppose that the next $1$ in $p$ is at position $r+s+u$ (which will conventionally be $m+n+1$ if there are no more $1$'s). This means that there is no $2$ in column $r+s+u$, and that there are exactly $r+u-1$ $2$'s in the first $r+s+u-1$ columns. In the polyomino, it means that there are exactly $r+u-1$ horizontal green steps between the diagonals $x+y = 0$ and $x+y = r+s+u-1$, that there is a vertical green step between the diagonals $x+y = r+s+u-1$ and $x+y = r+s+u$, and since there are $r+u-1$ horizontal green steps before it, it must be the step from $(r+u-1,s)$ to $(r+u-1,s+1)$. Since the bounce path was travelling East at height $s$, it hits it after $r+s+u-1$ steps (see Figure~\ref{fig:hbounce} again).
	
	Iterating the whole argument, we show that the bouncing points coincide with the points in which two consecutive letters of the parking word are different, in the way we wanted. The thesis follows.

\end{proof}


\begin{figure}[!ht]
	\begin{minipage}{0.5\textwidth}
		\begin{center}
			\begin{tikzpicture}[scale = 0.5]
			\draw[gray!60, thin] (0,0) grid (11,11);
			\draw[gray!60, thin] (0,0) -- (11,11);
			
			\draw[blue!60, line width=1.6pt] (0,0) -- (0,1) -- (0,2) -- (1,2) -- (1,3) -- (1,4) -- (2,4) -- (3,4) -- (3,5) -- (4,5) -- (5,5) -- (5,6) -- (5,7) -- (6,7) -- (6,8) -- (6,9) -- (7,9) -- (7,10) -- (8,10) -- (8,11) -- (9,11) -- (10,11) -- (11,11);
			
			\draw
			(0.5,0.5) circle(0.4 cm) node {$1$}
			(0.5,1.5) circle(0.4 cm) node {$2$}
			(1.5,2.5) circle(0.4 cm) node {$1$}
			(1.5,3.5) circle(0.4 cm) node {$2$}
			(3.5,4.5) circle(0.4 cm) node {$2$}
			(5.5,5.5) circle(0.4 cm) node {$1$}
			(5.5,6.5) circle(0.4 cm) node {$2$}
			(6.5,7.5) circle(0.4 cm) node {$1$}
			(6.5,8.5) circle(0.4 cm) node {$2$}
			(7.5,9.5) circle(0.4 cm) node {$2$}
			(8.5,10.5) circle(0.4 cm) node {$1$};
			\end{tikzpicture}
		\end{center}
	\end{minipage}%
	\begin{minipage}{0.5\textwidth}
		\begin{center}
			\begin{tikzpicture}[scale = 0.8]
			\draw[step=1.0, gray, opacity=0.6,thin] (0,0) grid (6,5);
			
			\filldraw[yellow, opacity=0.3] (0,0) -- (1,0) -- (2,0) -- (2,1) -- (3,1) -- (3,2) -- (4,2) -- (5,2) -- (6,2) -- (6,3) -- (6,4) -- (6,5) -- (5,5) -- (4,5) -- (4,4) -- (3,4) -- (3,3) -- (3,2) -- (2,2) -- (1,2) -- (0,2) -- (0,1) -- (0,0);
			
			\draw[red, sharp <-sharp >, sharp angle = -45, line width=1.6pt] (0,0) -- (0,1) -- (0,2) -- (1,2) -- (2,2) -- (3,2) -- (3,3) -- (3,4) -- (4,4) -- (4,5) -- (5,5) -- (6,5);
			
			\draw[green, sharp <-sharp >, sharp angle = 45, line width=1.6pt] (0,0) -- (1,0) -- (2,0) -- (2,1) -- (3,1) -- (3,2) -- (4,2) -- (5,2) -- (6,2) -- (6,3) -- (6,4) -- (6,5);
			
			\draw[dashed, opacity=0.6, thick] (0,0) -- (2,0) -- (2,2) -- (6,2) -- (6,5);
			
			\node[below] at (0.5,0) {$0$};
			\node[below] at (1.5,0) {$0$};
			\node[left] at (2,0.5) {$\bar{0}$};
			\node[left] at (2,1.5) {$\bar{0}$};
			\node[below] at (2.5,2) {$1$};
			\node[below] at (3.5,2) {$1$};
			\node[below] at (4.5,2) {$1$};
			\node[below] at (5.5,2) {$1$};
			\node[left] at (6,2.5) {$\bar{1}$};
			\node[left] at (6,3.5) {$\bar{1}$};
			\node[left] at (6,4.5) {$\bar{1}$};
			\end{tikzpicture}
		\end{center}
	\end{minipage}
	\caption{A two car parking function (left) and its corresponding polyomino (right). Notice that the bounce word of the polyomino $0 0 \bar{0} \bar{0} 1 1 1 1 \bar{1} \bar{1} \bar{1}$ follows the same pattern as the parking word $22112222111$ of the two car parking function.}
	\label{fig:hbounce}
\end{figure}
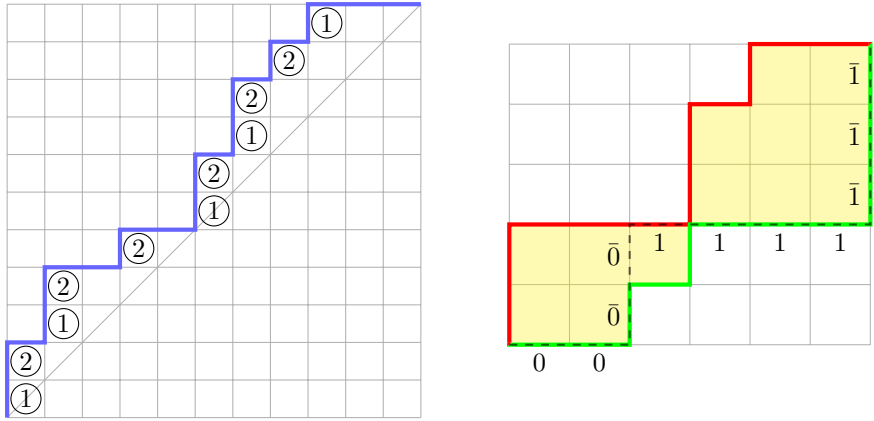

In \cite{Wilson-PhD-2015}, the author gives a recursion for rise-decorated two car parking functions. Of course we have the following.

\begin{theorem}[\cite{Wilson-PhD-2015}*{Proposition~5.3.3.1}]
	\label{th:recursion2cpf}
	For $m \geq 0$, $n \geq 0$, $k \geq 0$, and $1 \leq r \leq m+1$, the polynomials $\PF^2_{q,t}(m \backslash r, n)^{\ast k}$ satisfy the recursion
	\begin{align*}
		\PF^2_{q,t}(m \backslash r, n)^{\ast k} = & \sum_{s=1}^{n} t^{m+n+1-r-s-k} \qbinom{r+s-1}{s}_q \sum_{h=0}^{k} q^{\binom{h}{2}} \qbinom{s}{h}_q \\ \times & \sum_{u=1}^{m-r+1} \qbinom{s+u-h-1}{s-1}_q \PF^2_{q,t}(m-r \, \backslash \, u, \, n-s)^{\ast \, k-h}
	\end{align*}
	
	with initial conditions \[ \PF^2_{q,t}(m \backslash m+1, n)^{\ast 0} = \qbinom{m+n}{m}_q \] and $\PF^2_{q,t}(m \backslash r, 0)^{\ast k} = \delta_{k,0} \delta_{r,m+1}$.
\end{theorem}

This is exactly the same recursion held in Theorem~\ref{th:reddinvrecursion}, and it can be easily derived from that using Theorem~\ref{th:polyominoesto2cpf}. 

Let us recall the notation used in \cite{Wilson-PhD-2015}*{Section~5.2.2} for two car parking functions. We have

\begin{align*}
	RT_{n,m,k}^{(r-1)} & \coloneqq \PF^2(m \backslash r, n)^{\ast k} \\
	RT_{n,m,k}^{(s,r-1,i)} & \coloneqq \{ P \in RT_{n,m,k}^{(r-1)} \mid P \text{ has $s$ $2$'s and } \\ & \qquad \text{ $i$ decorated rises on the main diagonal } \} \\
	\overline{RT}_{n,m,k}^{(s,r-1,i)} & \coloneqq \{ P \in RT_{n,m,k}^{(s,r-1,i)} \mid \text{ the first rise in $P$ is not decorated } \}
\end{align*}

where the author appends $(q,t)$ (e.g. $RT_{n,m,k}^{(r-1)}(q,t)$) to denote the $q,t$-enumerator with respect to the bistatistic $(\dinv, \area)$ (as we do with the $_{q,t}$ subscript).

The following theorem is due to Wilson.

\begin{theorem}[\cite{Wilson-PhD-2015}*{Proposition~5.3.3.1}]
	\begin{align*}
		\overline{RT}_{a,b,k}^{(r,s,i)}(q,t) & = q^{\binom{i+1}{2}} t^{a-r+b-s-k} \qbinom{r+s}{r}_q \\
		& \times \sum_{h=1}^{b-s} \qbinom{h-1}{i}_q \qbinom{h+r-i-1}{h}_q RT_{a-r,b-s-1,k-i}^{(h-1)}(q,t).
	\end{align*}
\end{theorem}

We now want to explicitly check that the recursion in \cite{Wilson-PhD-2015}*{Proposition~5.3.3.1} for $RT_{n,m,k}^{(r-1)}(q,t)$ and the one in Theorem~\ref{th:recursion2cpf} for $\PF^2_{q,t}(m \backslash r, n)^{\ast k}$ are equivalent, since the statements are apparently different.

\begin{proposition}
	$\PF^2_{q,t}(m \backslash r, n)^{\ast k}$ and $RT_{n,m,k}^{(r-1)}(q,t)$ satisfy the same recursion.
\end{proposition}

\begin{proof}	
	First of all, by definition we have \[ RT_{a,b,k}^{(s)}(q,t) = \sum_{r=1}^{a} \sum_{i=0}^{k} RT_{a,b,k}^{(r,s,i)}(q,t) \] and, since any parking function with $k$ decorated rises has either $k$ or $k-1$ decorated rises which are not the first one, it holds that \[ RT_{a,b,k}^{(r,s,i)}(q,t) = \overline{RT}_{a,b,k}^{(r,s,i)}(q,t) + t^{-1} \overline{RT}_{a,b,k-1}^{(r,s,i-1)}(q,t) \]
	
	where $t^{-1}$ takes into account the fact that we are adding a decoration to the first rise. We can thus rewrite the recursion as
	\begin{align*}
		& \hspace{-0.5cm} \overline{RT}_{a,b,k}^{(r,s,i)}(q,t) + t^{-1} \overline{RT}_{a,b,k-1}^{(r,s,i-1)}(q,t)= \\
		& = q^{\binom{i+1}{2}} t^{a-r+b-s-k} \qbinom{r+s}{r}_q \sum_{h=1}^{b-s} \qbinom{h-1}{i}_q \qbinom{h+r-i-1}{h}_q RT_{a-r,b-s-1,k-i}^{(h-1)}(q,t) \\
		& + q^{\binom{i}{2}} t^{a-r+b-s-k} \qbinom{r+s}{r}_q \sum_{h=1}^{b-s} \qbinom{h-1}{i-1}_q \qbinom{h+r-i}{h}_q RT_{a-r,b-s-1,k-i}^{(h-1)}(q,t)
	\end{align*}
	
	and hence
	\begin{align*}
		& RT_{a,b,k}^{(r,s,i)}(q,t) = q^{\binom{i}{2}} t^{a-r+b-s-k} \qbinom{r+s}{r}_q \\
		& \times \sum_{h=1}^{b-s} \left( q^{i} \qbinom{h-1}{i}_q \qbinom{h+r-i-1}{h}_q + \qbinom{h-1}{i-1}_q \qbinom{h+r-i}{h}_q \right) RT_{a-r,b-s-1,k-i}^{(h-1)}(q,t).
	\end{align*}
	
	Using Lemma~\ref{lem:elementary4}, i.e. the identity \[ q^i \qbinom{h-1}{i}_q \qbinom{h+r-i-1}{h}_q + \qbinom{h-1}{i-1}_q \qbinom{h+r-i}{h}_q = \qbinom{r}{i}_q \qbinom{h+r-i-1}{h-i}_q, \] we have
	\begin{align*}
		RT_{a,b,k}^{(r,s,i)}(q,t) & = q^{\binom{i}{2}} t^{a-r+b-s-k+i} \qbinom{r+s}{r}_q \\ 
		& \times \sum_{h=1}^{b-s} \qbinom{r}{i}_q \qbinom{h+r-i-1}{h-i}_q RT_{a-r,b-s-1,k-i}^{(h-1)}(q,t)
	\end{align*}
	
	and now summing over $r$ and $i$ we get
	\begin{align*}
		RT_{a,b,k}^{(s)}(q,t) & = \sum_{r=1}^{a} \sum_{i=0}^{k} q^{\binom{i}{2}} t^{a-r+b-s-k+i} \qbinom{r+s}{r}_q \\
		& \times \sum_{h=1}^{b-s} \qbinom{r}{i}_q \qbinom{h+r-i-1}{h-i}_q RT_{a-r,b-s-1,k-i}^{(h-1)}(q,t).
	\end{align*}
	
	Making the substitutions
	
	\begin{multicols}{3}
		\begin{itemize}
			\item $a \mapsto n$
			\item $b \mapsto m$
			\item $s \mapsto r-1$
			\item $r \mapsto s$
			\item $i \mapsto h$
			\item $h \mapsto u$
		\end{itemize}		
	\end{multicols}
	
	we get
	\begin{align*}
		RT_{n,m,k}^{(r-1)}(q,t) & = \sum_{s=1}^{n} \sum_{h=0}^{k} \sum_{u=1}^{m-r+1} q^{\binom{h}{2}} t^{m+n-r-s-k+1} \\ & \times \qbinom{r+s-1}{s}_q \qbinom{s}{h}_q \qbinom{u+s-h-1}{s-1}_q RT_{n-s,m-r,k-h}^{(u-1)}(q,t)
	\end{align*}
	
	and now replacing $RT_{n,m,k}^{(r-1)}(q,t)$ with $\PF^2_{q,t}(m \backslash r, n)^{\ast k}$ we get the recursion in Theorem~\ref{th:recursion2cpf}, as desired.
\end{proof}

\section{Partially labelled Dyck paths}

There is an interesting connection between parallelogram polyominoes and a special subset of partially labelled Dyck paths.

\begin{figure}[!ht]
	\begin{center}
		\begin{tikzpicture}[scale=0.6]
		\draw[step=1.0, gray!60, thin] (0,0) grid (12,12);
		
		\draw[gray!60, thin] (0,0) -- (12,12);
		
		\draw[blue!60, line width=1.6pt] (0,0) -- (0,1) -- (0,2) -- (0,3) -- (1,3) -- (1,4) -- (2,4) -- (2,5) -- (3,5) -- (4,5) -- (4,6) -- (4,7) -- (4,8) -- (5,8) -- (6,8) -- (6,9) -- (6,10) -- (7,10) -- (7,11) -- (8,11) -- (8,12) -- (9,12) -- (10,12) -- (11,12) -- (12,12);
		
		\draw
		(0.5,0.5) circle (0.4cm) node {$1$}
		(0.5,1.5) circle (0.4cm) node {$2$}
		(0.5,2.5) circle (0.4cm) node {$5$}
		(1.5,3.5) circle (0.4cm) node {$0$}
		(2.5,4.5) circle (0.4cm) node {$0$}
		(4.5,5.5) circle (0.4cm) node {$0$}
		(4.5,6.5) circle (0.4cm) node {$4$}
		(4.5,7.5) circle (0.4cm) node {$6$}
		(6.5,8.5) circle (0.4cm) node {$0$}
		(6.5,9.5) circle (0.4cm) node {$3$}
		(7.5,10.5) circle (0.4cm) node {$0$}
		(8.5,11.5) circle (0.4cm) node {$0$};
		\end{tikzpicture}
	\end{center}
	
	\caption{A partially labelled Dyck path with $6$ nonzero labels and $6$ zero labels.}
	\label{fig:pldp}
\end{figure}
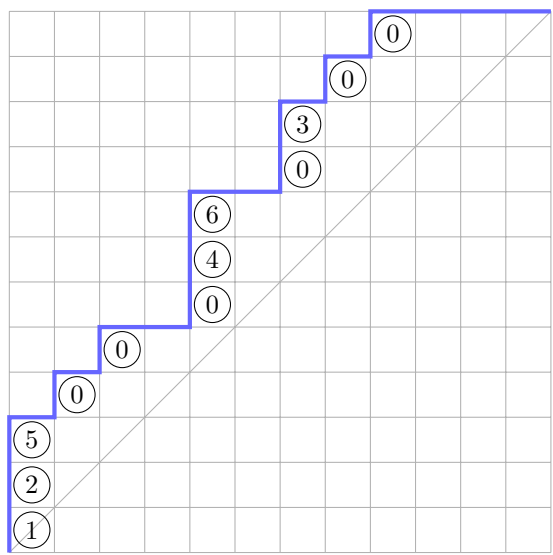

\begin{figure}[!ht]
	\begin{center}
		\begin{tikzpicture}[scale=0.6]
		\draw[step=1.0, gray!60, thin] (0,0) grid (7,6);
		
		\filldraw[yellow, opacity=0.3] (0,0) -- (1,0) -- (2,0) -- (3,0) -- (3,1) -- (4,1) -- (4,2) -- (5,2) -- (6,2) -- (7,2) -- (7,3) -- (7,4) -- (7,5) -- (7,6) -- (6,6) -- (5,6) -- (4,6) -- (4,5) -- (3,5) -- (3,4) -- (3,3) -- (2,3) -- (1,3) -- (0,3) -- (0,2) -- (0,1) -- (0,0);
		
		\draw[red, sharp <-sharp >, sharp angle = -45, line width=1.6pt] (0,0) -- (0,1) -- (0,2) -- (0,3) -- (1,3) -- (2,3) -- (3,3) -- (3,4) -- (3,5) -- (4,5) -- (4,6) -- (5,6) -- (6,6) -- (7,6);
		
		\draw[green, sharp <-sharp >, sharp angle = 45, line width=1.6pt] (0,0) -- (1,0) -- (2,0) -- (3,0) -- (3,1) -- (4,1) -- (4,2) -- (5,2) -- (6,2) -- (7,2) -- (7,3) -- (7,4) -- (7,5) -- (7,6);
		
		\node[red] at (0.5,0.5) {$0$};
		\node[red] at (0.5,1.5) {$1$};
		\node[red] at (0.5,2.5) {$2$};
		\node[green] at (1.5,2.5) {$2$};
		\node[green] at (2.5,2.5) {$2$};
		\node[green] at (3.5,2.5) {$1$};
		\node[red] at (3.5,3.5) {$2$};
		\node[red] at (3.5,4.5) {$3$};
		\node[green] at (4.5,4.5) {$2$};
		\node[red] at (4.5,5.5) {$3$};
		\node[green] at (5.5,5.5) {$3$};
		\node[green] at (6.5,5.5) {$3$};
		\end{tikzpicture}
		\begin{tikzpicture}[scale=0.6]
		\draw[step=1.0, gray!60, thin] (0,0) grid (7,6);
		
		\filldraw[yellow, opacity=0.3] (0,0) -- (1,0) -- (2,0) -- (3,0) -- (3,1) -- (4,1) -- (4,2) -- (5,2) -- (6,2) -- (7,2) -- (7,3) -- (7,4) -- (7,5) -- (7,6) -- (6,6) -- (5,6) -- (4,6) -- (4,5) -- (3,5) -- (3,4) -- (3,3) -- (2,3) -- (1,3) -- (0,3) -- (0,2) -- (0,1) -- (0,0);
		
		\draw[red, sharp <-sharp >, sharp angle = -45, line width=1.6pt] (0,0) -- (0,1) -- (0,2) -- (0,3) -- (1,3) -- (2,3) -- (3,3) -- (3,4) -- (3,5) -- (4,5) -- (4,6) -- (5,6) -- (6,6) -- (7,6);
		
		\draw[green, sharp <-sharp >, sharp angle = 45, line width=1.6pt] (0,0) -- (1,0) -- (2,0) -- (3,0) -- (3,1) -- (4,1) -- (4,2) -- (5,2) -- (6,2) -- (7,2) -- (7,3) -- (7,4) -- (7,5) -- (7,6);
		
		\node[white] at (-3,0.5) {};
		
		\draw[red]
		(0.5,0.5) circle(0.4 cm) node {$1$}
		(0.5,1.5) circle(0.4 cm) node {$2$}
		(0.5,2.5) circle(0.4 cm) node {$5$}
		(3.5,3.5) circle(0.4 cm) node {$4$}
		(3.5,4.5) circle(0.4 cm) node {$6$}
		(4.5,5.5) circle(0.4 cm) node {$3$};
		\end{tikzpicture}
	\end{center}
	
	\caption{The corresponding parallelogram polyomino with its new area word (left) or the regular labelling (right). In the picture on the left, each green number corresponds to the horizontal red step in the column to its immediate left.}
	\label{fig:bijection}
\end{figure}
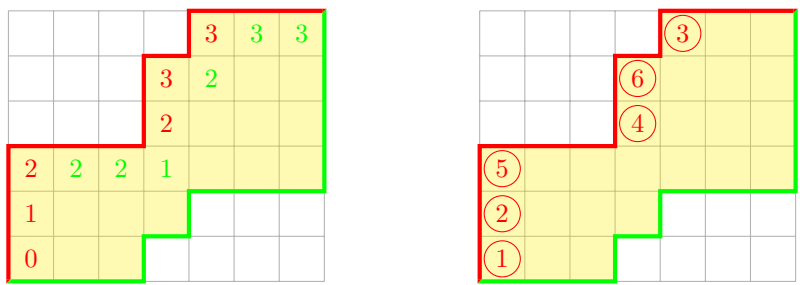

\begin{theorem}
	\label{th:bijection}
	For $m,n \geq 1$ there is a bijection \[ \eta \colon \PLD(m-1,n)^{\ast n-1}\to \LPP(m,n) \] such that $\area(D)= \area(\eta(D)) - (m+n-1)$ and $\pmaj(D)=\pmaj(\eta(D))$ for all $P\in \PLD(m-1,n)^{\ast n-1}$.
\end{theorem}

\begin{proof}
	Let $D \in \PLD(m-1,n)^{\ast n-1}$. Write its area word (the usual one for Dyck paths), colouring green the numbers corresponding to rows containing a label $0$, and red the other ones (for example, the Dyck path in Figure~\ref{fig:pldp} has area word ${\color{red} 0} {\color{red} 1} {\color{red} 2} {\color{green} 2} {\color{green} 2} {\color{green} 1} {\color{red} 2} {\color{red} 3} {\color{green} 2} {\color{red} 3} {\color{green} 3} {\color{green} 3}$).
	
	Now, draw the red path of the polyomino as follows: running over the letters of the area word, draw a vertical step if the letter is red, and a horizontal step if it is green. Whenever you draw a vertical step, label it with the corresponding nonzero label of the Dyck path. End with an extra horizontal step.
	
	Next, draw the green path as follows. Start with drawing a horizontal step. Then, for each horizontal red step $h$, draw a horizontal green step $x+1$ rows below $h$ but in the column to its immediate right, where $x$ is the value of the green letter corresponding to $h$. Then, connect them with vertical steps to get a lattice path from $(0,0)$ to $(m,n)$. See Figure~\ref{fig:bijection} for an example.
	
	It is not hard to see that this is a bijection (its inverse is quite straightforward) and it maps area to the sum of the green letters in the area word, which is exactly the area of the polyomino minus $m+n-1$. In fact, drawing the picture as in Figure~\ref{fig:bijection}, it is clear that each green letter denotes the number of squares in the polyomino that are strictly below the one containing it, and the red letters do not count since a letter is red if and only if it is either the first one (which is ${\color{red} 0}$) or it is a rise, and all the rises are decorated and thus do not contribute to the area. It also preserves pmaj by construction, since the algorithms to compute it coincide step by step.
	%
	%
\end{proof}

\section{A new $\dinv$ statistic on parallelogram polyominoes}

The bijection in Theorem~\ref{th:bijection} gives a new coding for parallelogram polyominoes (using the area word of the corresponding partially labelled Dyck path), together with a new $\dinv$ statistic, which can be seen directly on labelled parallelogram polyominoes as follows.

\begin{definition}
	We define the statistic $\mathsf{newdinv}$ on a labelled parallelogram polyomino with new area word $a_1 a_2 \cdots a_{m+n-1}$ as the number of \textit{inversions}, i.e. the pairs $(i,j)$ with $a_i$ red, such that one of the following holds:
	\begin{itemize}
		\item $i > j$, $a_i = a_j$, and either $a_j$ is green or the label attached to the vertical red step corresponding to $a_i$ is greater than the one attached to the step corresponding to $a_j$;
		\item $i < j$, $a_i = a_j + 1$, and either $a_j$ is green or the label attached to the vertical red step corresponding to $a_i$ is greater than the one attached to the step corresponding to $a_j$.
	\end{itemize}
\end{definition}

This definition also applies to unlabelled parallelogram polyominoes assuming that the third condition (the one regarding the label corresponding to $a_j$) always holds, or equivalently that the labels are attached in a suitable canonical ordering (i.e. attaching the labels from $1$ to $n$ starting from the red $0$, then to the $1$'s, next to the $2$'s, and so on, going left to right on each level). In other words, it is the number of pairs $(i,j)$ with $a_i$ red such that either $i > j$ and $a_i = a_j$, or $i < j$ and $a_i = a_j + 1$. For example, the $\mathsf{newdinv}$ of the parallelogram polyomino in Figure~\ref{fig:bijection} is $4$ (or $6$ if we disregard the labels), the inversions being $(3,6), (7,4), (7,5), (8,9)$ (and $(7,3)$ and $(10,8)$ if we disregard the labels).

This $\mathsf{newdinv}$ statistic on labelled polyominoes answers the question in \cite{Aval-Bergeron-Garsia-2015}*{Equation (8.14)}.

\section{A $\bounce$ statistic on partially labelled Dyck paths}

The bijection in Theorem~\ref{th:bijection} gives also an implicit definition of a $\bounce$ statistic on partially labelled Dyck paths in $\PLD(m-1,n)^{\ast n-1}$. Call \emph{blank valley} a vertical step of a partially labelled Dyck path whose label is $0$. Indeed, such a step must always be preceded by a horizontal step, i.e.\ it must be a \emph{valley}, since the labels must be strictly increasing in columns. Observe that for every element of $\PLD(m-1,n)^{\ast n-1}$ every valley is necessarily a blank valley, and every rise is necessarily decorated.

We can draw a bounce path as usual for Dyck paths, with the following difference. We can replace blank valleys (the vertical step together with the preceding horizontal step) with diagonal steps, and then when the bounce path should go horizontally, it goes parallel to the step of the path in the same column. See Figure~\ref{fig:plbounce} on the left for an example, noticing that in the picture we do not show the decorations (every rise is decorated) nor the labels: in order to construct the bounce path it is enough to remember that every valley is a blank valley.

\begin{figure}[!ht]
	\begin{center}
		\begin{tikzpicture}[scale=0.4]
		\draw[step=1.0, gray!60, thin] (0,0) grid (12,12);
		
		\draw[gray!60, thin] (0,0) -- (12,12);
		
		\draw[blue!60, transform canvas={xshift=0.15mm}, transform canvas={yshift=-0.15mm}, line width=1.6pt] (0,0) -- (0,1) -- (0,2) -- (0,3) -- (1,3) -- (1,4) -- (2,4) -- (2,5) -- (3,5) -- (4,5) -- (4,6) -- (4,7) -- (4,8) -- (5,8) -- (6,8) -- (6,9) -- (6,10) -- (7,10) -- (7,11) -- (8,11) -- (8,12) -- (9,12) -- (10,12) -- (11,12) -- (12,12);
		
		\draw[blue!20, transform canvas={xshift=-0.15mm}, transform canvas={yshift=0.15mm}, line width=1.6pt] (0,0) -- (0,1) -- (0,2) -- (0,3) -- (1,4) -- (2,5) -- (3,5) -- (4,6) -- (4,7) -- (4,8) -- (5,8) -- (6,9) -- (6,10) -- (7,11) -- (8,12) -- (9,12) -- (10,12) -- (11,12) -- (12,12);
		
		\draw[dashed, opacity=0.6, thick] (0,0) -- (0,3) -- (1,4) -- (2,5) -- (3,5) -- (4,6) -- (5,6) -- (6,7) -- (7,8) -- (8,9) -- (9,9) -- (9,12) -- (10,12) -- (11,12) -- (12,12);
		
		\node at (0.5,0.5) {$0$};
		\node at (0.5,1.5) {$0$};
		\node at (0.5,2.5) {$0$};
		\node at (1.5,3.5) {$0$};
		\node at (2.5,4.5) {$0$};
		\node at (4.5,5.5) {$0$};
		\node at (6.5,6.5) {$0$};
		\node at (7.5,7.5) {$0$};
		\node at (8.5,8.5) {$0$};
		\node at (9.5,9.5) {$1$};
		\node at (9.5,10.5) {$1$};
		\node at (9.5,11.5) {$1$};
		\end{tikzpicture}
		\begin{tikzpicture}[scale=0.6]
		\draw[step=1.0, gray!60, thin] (0,0) grid (7,6);
		
		\filldraw[yellow, opacity=0.3] (0,0) -- (1,0) -- (2,0) -- (3,0) -- (3,1) -- (4,1) -- (4,2) -- (5,2) -- (6,2) -- (7,2) -- (7,3) -- (7,4) -- (7,5) -- (7,6) -- (6,6) -- (5,6) -- (4,6) -- (4,5) -- (3,5) -- (3,4) -- (3,3) -- (2,3) -- (1,3) -- (0,3) -- (0,2) -- (0,1) -- (0,0);
		
		\draw[red, sharp <-sharp >, sharp angle = -45, line width=1.6pt] (0,0) -- (0,1) -- (0,2) -- (0,3) -- (1,3) -- (2,3) -- (3,3) -- (3,4) -- (3,5) -- (4,5) -- (4,6) -- (5,6) -- (6,6) -- (7,6);
		
		\draw[green, sharp <-sharp >, sharp angle = 45, line width=1.6pt] (0,0) -- (1,0) -- (2,0) -- (3,0) -- (3,1) -- (4,1) -- (4,2) -- (5,2) -- (6,2) -- (7,2) -- (7,3) -- (7,4) -- (7,5) -- (7,6);
		
		\draw[dashed, opacity=0.6, thick] (0,0) -- (1,0) -- (1,1) -- (1,2) -- (1,3) -- (2,3) -- (3,3) -- (4,3) -- (5,3) -- (6,3) -- (7,3) -- (7,4) -- (7,5) -- (7,6);
		
		\node[white] at (-2,-0.8) {};
		\node[above] at (0.5,0.0) {$\bar{0}$};
		\node[right] at (1.0,0.5) {$1$};
		\node[right] at (1.0,1.5) {$1$};
		\node[right] at (1.0,2.5) {$1$};
		\node[above] at (1.5,3.0) {$\bar{1}$};
		\node[above] at (2.5,3.0) {$\bar{1}$};
		\node[above] at (3.5,3.0) {$\bar{1}$};
		\node[above] at (4.5,3.0) {$\bar{1}$};
		\node[above] at (5.5,3.0) {$\bar{1}$};
		\node[above] at (6.5,3.0) {$\bar{1}$};
		\node[right] at (7.0,3.5) {$2$};
		\node[right] at (7.0,4.5) {$2$};
		\node[right] at (7.0,5.5) {$2$};
		\end{tikzpicture}
	\end{center}
	
	\caption{A partially labelled Dyck path in $\PLD(7-1,6)^{\ast 6-1}$ with its bounce path shown (left), and the corresponding polyomino (right). The darker path (left) is the original one, in the lighter one valleys have been replaced by diagonal steps. Vertical steps of the bounce path on the left correspond to vertical ones on the right, while diagonal steps on the left correspond to horizontal ones on the right.}
	\label{fig:plbounce}
\end{figure}
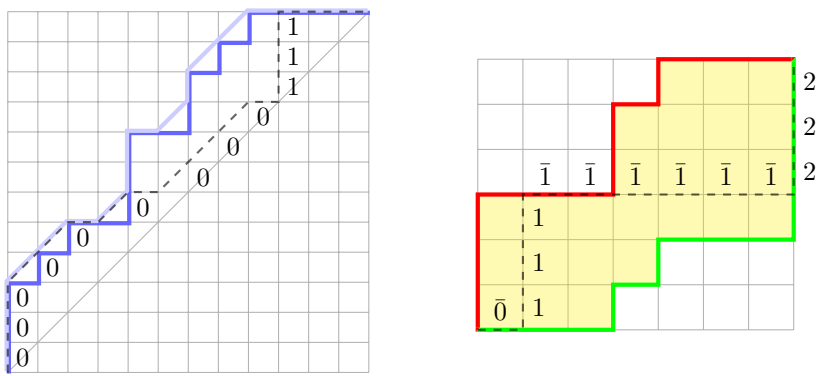

By construction, this $\bounce$ statistic agrees with the $\pmaj$ if the red labels are assigned from $1$ to $n$ going bottom to top. In fact, in each run of the pmaj word $w$, we have a number of letters different from $0$ equal to the number of vertical steps in the corresponding part of the bounce path, and a number of $0$'s equal to the number of diagonal steps in the same part: that is because $0$'s correspond to blank valleys, and blank valleys correspond to diagonal steps; any other letter corresponds to a vertical step of the path which is not a blank valley, and these correspond to vertical steps of the bounce path.

\chapter{Putting the pieces together}

In this chapter we combine the results from the previous three chapters to finally obtain the combinatorial interpretations of the plethystic formulae that we encountered. 

\section{Combinatorial interpretations of plethystic formulae}	

In this section we combine our previous results to provide combinatorial interpretations of several plethystic formulae. We are going to use the notations introduced in Section~\ref{sec:combinatorial_recursions_ddyck} .

\medskip
In the following theorem, the first equalities of \eqref{eq:Ftildenk_qt} and \eqref{eq:Fnk_qt} are due to Zabrocki, as they are immediate consequences of \cite{Zabrocki-4Catalan-2016}*{Theorem~10}. 
\begin{theorem} \label{thm:qt_enumerators_formulae}
	Let $d,\ell,n,k,i\in \mathbb N$ with $n\geq 1$, $n\geq k\geq 1$, $\ell \geq i$. Then
	\begin{align} 
		\label{eq:Ftildenk_qt}	\widetilde{F}_{n,k}^{(d,\ell)} & =\DDd_{q,t}(n\backslash k )^{\circ d, \ast \ell} =\DDb_{q,t}(n\backslash k )^{\circ \ell, \ast d},\\
		\label{eq:Fnk_qt}	F_{n,k}^{(d,\ell)} & =
		\DDd_{q,t}(n\backslash k  )^{\circ d, \ast \ell }+ \DDd_{q,t}(n\backslash k )^{\circ d-1, \ast \ell} 
		\\ \nonumber  &=\DDb_{q,t}(n\backslash k  )^{\circ \ell, \ast d }+ \DDb_{q,t}(n\backslash k )^{\circ \ell, \ast d-1},\\
		\label{eq:Gtildenk_qt}	\widetilde{G}_{n-k,k}^{(d,\ell)} & =
		\DDd_{q,t}(n\backslash k \backslash 0 )^{\circ d, \ast \ell}
		=\DDb_{q,t}(n\backslash k \backslash 0 )^{\circ \ell, \ast d},\\
		\label{eq:Gnk_qt}	G_{n-k,k}^{(d,\ell)} & =
		\DDd_{q,t}(n\backslash k \backslash 0 )^{\circ d, \ast \ell }+ \DDd_{q,t}(n\backslash k \backslash 0 )^{\circ d-1, \ast \ell}
		\\ \nonumber &= \DDb_{q,t}(n\backslash k \backslash 0 )^{\circ \ell, \ast d }+ \DDb_{q,t}(n\backslash k\backslash 0 )^{\circ \ell, \ast d-1},
		\\
		\label{eq:Hbarnk_qt}	\overline{H}_{n,k,i}^{(d,\ell)} & =
		\DDp_{q,t}(n\backslash k \backslash i )^{\circ d, \ast \ell},\\
		\label{eq:Hnk_qt}	H_{n,k,i}^{(d,\ell)} & =\DDp_{q,t}(n\backslash k \backslash i )^{\circ d, \ast \ell}+\DDp_{q,t}(n\backslash k \backslash i )^{\circ d-1, \ast \ell}.
	\end{align}
\end{theorem}
\begin{proof}
	Identity \eqref{eq:Ftildenk_qt} follows by comparing Theorem~\ref{thm:reco1_first_bounce} with Theorem~\ref{thm:reco1_F_and_Ftilde}, and Corollary~\ref{cor:zeta_map}.
	\begin{itemize}
		\item Identity \eqref{eq:Fnk_qt} follows from \eqref{eq:Ftildenk_qt}, and the definitions of the polynomials involved.
		
		\item Identity \eqref{eq:Gtildenk_qt} follows by comparing Theorem~\ref{thm:reco_R_first_bounce} with Theorem~\ref{thm:recoG_and_Gtilde}, and Corollary~\ref{cor:zeta_map}.
		
		\item Identity \eqref{eq:Gnk_qt} follows from \eqref{eq:Gtildenk_qt}, and the definitions of the polynomials involved.
		
		\item Identity \eqref{eq:Hbarnk_qt} follows by comparing Theorem~\ref{thm:reco_old_bounce} with Theorem~\ref{thm:reco_H_and_Htilde} and Remark~\ref{rem:rewriting_reco_H}.
		
		\item Identity \eqref{eq:Hnk_qt} follows from \eqref{eq:Hbarnk_qt}, and the definitions of the polynomials involved.
	\end{itemize}
\end{proof}

\section{Proof of the decorated $q,t$-Schr\"{o}der} \label{sec:deco_qt_Schroeder}

We are finally able to prove the decorated analogue of Haglund's $q,t$-Schr\"{o}der theorem \cite{Haglund-Schroeder-2004} for $\Delta_{e_k}e_n$.
\begin{theorem} \label{thm:decoqtSchroeder}
	For $a,b,k\in \mathbb{N}$, $a\geq 1$, $a+b\geq k+1$, we have
	\begin{align} \label{eq:qtSchroder1}
		\langle \Delta_{e_{a+b-k-1}}'e_{a+b},s_{b+1,1^{a-1}}\rangle &= \DDd_{q,t}(a+b)^{\circ b, \ast k} \\
		\nonumber &= \DDp_{q,t}(a+b)^{\circ b, \ast k} \\
		\nonumber &= \DDb_{q,t}(a+b)^{\circ k, \ast b}.
	\end{align}
	so that, for $a+b\geq k+1$, we have
	\begin{align} \label{eq:qtSchroder2}
		\langle \Delta_{e_{a+b-k-1}}'e_{a+b},e_a h_{b}\rangle &= \sum_{D\in \DD(a+b)^{\circ b, \ast k} } q^{\dinv(D)}t^{\area(D)} 
		\\ \nonumber &=\sum_{D\in \DD(a+b)^{\circ b, \ast k} } q^{\area(D)}t^{\bounce(D)}
		\\ \nonumber &=\sum_{D\in \DD(a+b)^{\circ b, \ast k} } q^{\area(D)}t^{\pbounce(D)}.
	\end{align}
\end{theorem}
\begin{proof}
	To prove \eqref{eq:qtSchroder1}, just combine Proposition~\ref{prop:SF_sums_F_H} and Theorem~\ref{thm:qt_enumerators_formulae}.
	
	To prove the first two equalities in \eqref{eq:qtSchroder2}, just add \eqref{eq:qtSchroder1} with itself where we replace $b$ by $b-1$ and $a$ by $a+1$. For the last equality in \eqref{eq:qtSchroder2}, doing the same operation we get
	\begin{equation}
		\langle \Delta_{e_{a+b-k-1}}'e_{a+b},e_a h_{b}\rangle= \DDb_{q,t}(a+b)^{\circ k, \ast b}+ \DDb_{q,t}(a+b)^{\circ k, \ast b-1}.
	\end{equation}
	But using again \eqref{eq:qtSchroder1} and Corollary~\ref{cor:comb_symmetry}, we have
	\begin{align}
		\DDb_{q,t}(a+b)^{\circ k, \ast b}&+ \DDb_{q,t}(a+b)^{\circ k, \ast b-1}  \\ &= \DDp_{q,t}(a+b)^{\circ b, \ast k}+ \DDp_{q,t}(a+b)^{\circ b-1, \ast k}\\
		& = \DDp_{q,t}(a+b)^{\circ k, \ast b}+ \DDp_{q,t}(a+b)^{\circ k, \ast b-1}\\
		& = \DDb_{q,t}(a+b)^{\circ b, \ast k}+ \DDb_{q,t}(a+b)^{\circ b-1, \ast k}\\
		& =  \sum_{D\in \DD(a+b)^{\circ b, \ast k} }q^{\area(D)}t^{\bounce(D)},
	\end{align}
	as claimed.
\end{proof}

We are now able to give a combinatorial proof of Theorem~\ref{thm:symmetry}.
\begin{proof}[Combinatorial proof of Theorem~\ref{thm:symmetry}]
	Using \eqref{eq:Schroeder_identity}, with $n=a+b$ and $\ell=b-1$,  \eqref{eq:Mac_hook_coeff_ss}, and Theorem~\ref{thm:decoqtSchroeder}, we have
	\begin{align}
		\langle \Delta_{h_k}\Delta_{e_{n-k-\ell-1}}' e_{n-k},e_{n-k}\rangle & =\langle \Delta_{e_{n-\ell-1}}'e_{n},s_{k+1,1^{n-k-1}}\rangle = \DDp_{q,t}(n)^{\circ k, \ast \ell}.
	\end{align}
	Now the symmetry in $k$ and $\ell$ is explained combinatorially by the bijection $\psi$ of Theorem~\ref{thm:psi_map} (cf. Corollary~\ref{cor:comb_symmetry}).
\end{proof}

\section{Proof of the decorated $q,t$-Narayana}\label{sec:2shuffle_delta}

The following theorem is the final step towards the proof of a decorated $q,t$-Narayana theorem, and hence of the two-shuffle case of the Delta conjecture.

\begin{theorem} \label{thm:main_qtenumerator}
	For $m \geq 0$, $n \geq 0$, $k \geq 0$, and $1 \leq r \leq m+1$, we have \[ \RP_{q,t}(m \backslash r, n)^{\ast k} = t^{m-k-r+1} \< \Delta_{h_{m-k-r+1}} \Delta_{e_k} e_n \left[ X \dfrac{1 - q^r}{1 - q} \right], e_n \> . \]
\end{theorem}

\begin{proof}
	It's enough to show that $\RP_{q,t}(m \backslash r, n)^{\ast k}$ and \[ t^{m-k-r+1} \< \Delta_{h_{m-k-r+1}} \Delta_{e_k} e_n \left[ X \dfrac{1 - q^r}{1 - q} \right], e_n \> \] satisfy the same recursion. We already gave the one for $\RP_{q,t}(m \backslash r, n)^{\ast k}$.
	
	First of all, we need to look at Theorem~\ref{thm:dvw}. Then we slightly change the statement. If we allow $h$ to be $0$ in the third sum, then since $F_{n,0}^{(d,\ell)} = 0$ unless $n = \ell = d = 0$, the only extra term in the sum is the one with $\ell = j$, $n = k + \ell$, $k = d + s$, and in that case its value agrees with the term \[ \delta_{n, k + \ell} \; q^{\binom{k-d}{2}} \qbinom{n-1}{\ell}_q \qbinom{k}{d}_q \] so it follows that we can allow that sum to start from $0$ if we delete that initial term. Then we make the following substitutions:
	
	\begin{multicols}{2}
		\begin{itemize}
			\item $n \mapsto n + m - k - r + 1$
			\item $\ell \mapsto m - k - r + 1$
			\item $d \mapsto n - k$
			\item $k \mapsto s$
		\end{itemize}		
	\end{multicols}
	
	and we get the following recursion (with $s \mapsto h$, $h \mapsto j$, and $j \mapsto u$ as indices).
	\begin{align*}
		F_{n+m-k-r+1,s}^{(n-k,m-k-r+1)} & = \sum_{u=0}^{n+m-k-r+1-s} \sum_{h=0}^k \sum_{j=0}^{n+m-k-r+1-s-u} t^{n+m-k-r+1-s-u} \\ & \times q^{\binom{h}{2}} \qbinom{s}{h}_q \qbinom{s+u-1}{u}_q \qbinom{h+u+j-1}{u}_q F_{n+m-k-r+1-s-u,j}^{(n-k-s+h,m-k-r+1-u)}
	\end{align*}
	
	Now, recall from \eqref{eq:en_q_sum_Enk} that \[ e_n \left[ X \dfrac{1 - q^r}{1 - q} \right] = \sum_{k=1}^{n} \qbinom{k+r-1}{k}_q E_{n,k} \] and the definition \[ F_{n,k}^{(d,\ell)} = \< \Delta_{h_\ell} \Delta_{e_{n-\ell-d}} E_{n-\ell,k}, e_{n-\ell} \> ,\] so, replacing this in the recursion, we get
	\begin{align*}
		& \< \Delta_{h_{m-k-r+1}} \Delta_{e_k} E_{n,s}, e_n \> = \sum_{u=0}^{n+m-k-r+1-s} \sum_{h=0}^k \sum_{j=0}^{n+m-k-r+1-s-u} t^{n+m-k-r+1-s-u} \\ & \times q^{\binom{h}{2}} \qbinom{s}{h}_q \qbinom{s+u-1}{u}_q \qbinom{h+u+j-1}{j}_q \< \Delta_{h_{m-k-r+1-u}} \Delta_{e_{k-h}} E_{n-s,j}, e_{n-s} \>
	\end{align*}
	
	that, since $E_{n,j} = 0$ if $j > n$, and $m-k-r+1 > 0$, we can rewrite as
	\begin{align*}
		\< \Delta_{h_{m-k-r+1}} & \Delta_{e_k} E_{n,s}, e_n \> = \sum_{u=0}^{n+m-k-r+1-s} \sum_{h=0}^k t^{n+m-k-r+1-s-u} \\ & \times q^{\binom{h}{2}} \qbinom{s}{h}_q \qbinom{s+u-1}{u}_q \< \Delta_{h_{m-k-r+1-u}} \Delta_{e_{k-h}} e_{n-s} \left[ X \dfrac{1 - q^{u+h}}{1 - q} \right], e_{n-s} \>
	\end{align*}
	
	Multiplying by $\qbinom{r+s-1}{s}_q$ and summing over $s$ from $1$ to $n$, we get
	\begin{align*}
		& \< \Delta_{h_{m-k-r+1}} \Delta_{e_k} e_n \left[ X \dfrac{1 - q^r}{1 - q} \right], e_n \> = \sum_{s=1}^{n} \sum_{u=0}^{n+m-k-r+1-s} \sum_{h=0}^k t^{n+m-k-r+1-s-u} \\ & \times q^{\binom{h}{2}} \qbinom{r+s-1}{s}_q \qbinom{s}{h}_q \qbinom{s+u-1}{j}_q \< \Delta_{h_{m-k-r+1-u}} \Delta_{e_{k-h}} e_{n-s} \left[ X \dfrac{1 - q^{u+h}}{1 - q} \right], e_{n-s} \>.
	\end{align*}
	
	and now multiplying by $t^{m-k-r+1}$ and noticing that $e_i = 0$ if $i < 0$, we get
	\begin{align*}
		t^{m-k-r+1} & \< \Delta_{h_{m-k-r+1}} \Delta_{e_k} e_n \left[ X \dfrac{1 - q^r}{1 - q} \right], e_n \> \\ & = \sum_{s=1}^{n} \sum_{u=0}^{m-k-r+1} \sum_{h=0}^k t^{m+n-r-s-k+1} q^{\binom{h}{2}} \qbinom{r+s-1}{s}_q \qbinom{s}{h}_q \qbinom{s+u-1}{u}_q \\ & \times t^{m-k-r+1-u} \< \Delta_{h_{m-k-r+1-u}} \Delta_{e_{k-h}} e_{n-s} \left[ X \dfrac{1 - q^{u+h}}{1 - q} \right], e_{n-s} \>.
	\end{align*}
	
	Finally, with the substitution $u \mapsto u - h$, and recalling that $h$ ranges from $0$ to $k$ and one of the $q$-binomials drops to $0$ if $u < h$, we get
	\begin{align*}
		t^{m-k-r+1} & \< \Delta_{h_{m-k-r+1}} \Delta_{e_k} e_n \left[ X \dfrac{1 - q^r}{1 - q} \right], e_n \> \\ & = \sum_{s=1}^{n} \sum_{u=1}^{m-r+1} \sum_{h=0}^k t^{m+n-r-s-k+1} \qbinom{r+s-1}{s}_q q^{\binom{h}{2}} \qbinom{s}{h}_q \qbinom{s+u-h-1}{u-h}_q \\ & \times t^{(m-r)-(k-h)-u+1} \< \Delta_{h_{(m-r)-(k-h)-u+1}} \Delta_{e_{k-h}} e_{n-s} \left[ X \dfrac{1 - q^{u}}{1 - q} \right], e_{n-s} \>.
	\end{align*}
	
	which is exactly the recursion for $\RP_{q,t}(m \backslash r, n)^{\ast k}$. The initial conditions are easy to check.	
\end{proof}

\begin{corollary} \label{cor:PF2_qtenum}
	For $m \geq 0$, $n \geq 0$, $k \geq 0$, and $1 \leq r \leq m+1$, we have 
	\[ \PF^2_{q,t}(m \backslash r, n)^{\ast k} = t^{m-k-r+1} \< \Delta_{h_{m-k-r+1}} \Delta_{e_k} e_n \left[ X \dfrac{1 - q^r}{1 - q} \right], e_n \> \]
\end{corollary}

The following corollary is an immediate consequence of Theorem~\ref{thm:main_qtenumerator}, Theorem~\ref{thm:magic_equality} and Theorem~\ref{th:redzetamap}. It extends the main results in \cite{Aval-DAdderio-Dukes-Hicks-LeBorgne-2014}, by providing a decorated version of the $q,t$-Narayana theorem.
\begin{corollary}
	For $m \geq 0$, $n \geq 0$, $k \geq 0$, we have
	\begin{align} 
		\RP_{q,t}(m,n)^{\ast k} = \RP_{q,t}(m,n)^{\circ k} = \< \Delta_{h_m} e_{n+1}, s_{k+1,1^{n-k}} \>
	\end{align}
\end{corollary}

The following corollary is an immediate consequence of Corollary~\ref{cor:PF2_qtenum} and Corollary~\ref{cor:SF_sum}. It settles the two-shuffle case of the Delta conjecture in \cite{Haglund-Remmel-Wilson-2015}.
\begin{corollary}
	For $m \geq 0$, $n \geq 0$, $k \geq 0$, we have
	\begin{align} 
		\PF^2_{q,t}(m,n)^{\ast k}= \< \Delta_{e_{m+n-k-1}}'e_{m+n},h_{m}h_{n} \> 
	\end{align}
\end{corollary}

\chapter{Square paths}

In this chapter we prove a new $q,t$-square theorem, similar to (but different from) the one proved in \cite{Can-Loehr-2006}, which turns out to be the case $\<\cdot , e_{n-d}h_d\>$ of our Conjecture~\ref{conj:new_square}. We also show an intriguing connection between our square conjecture and the one in \cite{Loehr-Warrington-square-2007} (proved in \cite{Leven-2016} after the breakthrough in \cite{Carlsson-Mellit-ShuffleConj-2015}).

\section{A new $q,t$-square} \label{sec:new_qt_square}

We refer to Section~\ref{sec:sq_path_defs} for the definitions about square paths.

\begin{proposition} \label{prop:qtsquare_bijections}
	There exist bijections
	\begin{align}
		\gamma_E'  : \mathsf{SQ^E}(n)\rightarrow 
		&\DD(n)^{\circ 0, \ast 0} \sqcup \DD(n)^{\circ 0, \ast 1}\\& \nonumber =\DDd(n)^{\circ 0, \ast 0} \sqcup \DDd(n)^{\circ 0, \ast 1} 
		\\
		\gamma_N'  : \mathsf{SQ^N}(n)\rightarrow &\DD(n)^{\circ 1, \ast 0} \\\nonumber &=
		\DDd(n)^{\circ 0, \ast 0} \sqcup \DDd(n)^{\circ 1, \ast 0}\\ \nonumber & = \DDp(n)^{\circ 0, \ast 0} \sqcup \DDp(n)^{\circ 1, \ast 0}
	\end{align}
	such that for all $P\in \mathsf{SQ^E}(n)$ and $P'\in \mathsf{SQ^N}(n)$ we have
	\begin{align}
		\area(\gamma_E'(P)) & =\area(P)  &   \area(\gamma_N'(P')) & =\area(P')  \\
		\dinv(\gamma_E'(P)) & =\dinv(P) &    \dinv(\gamma_N'(P')) & =\dinv(P')  \\
		\bounce(\gamma_E'(P)) & =\bounce(P) &   \pbounce(\gamma_N'(P')) & =\bounce(P').
	\end{align}
\end{proposition}
\begin{proof}
	We define $\gamma_E'$ in exactly the same way as the $\gamma_E$ map (\ref{gammamap}) defined in Section~\ref{sec:square_paths_map} for labelled objects. 
	
	We define a new map  $\gamma_N'$ similarly to $\gamma_E'$: the image of a path is constructed in the same way except that the portion of the path right before the added horizontal step is a north step, so we created a peak instead of a fall, which we decorate. See Figure~\ref{fig: gammaN} for an example. Note that the paths that get sent into paths with a decoration on the last peak are exactly the paths whose $j=0$.
	
	Now the statements of the proposition are straightforward to check.
\end{proof}

\begin{figure}
	\centering
	\begin{minipage}{.6 \textwidth}
		\centering
		\begin{tikzpicture}[scale=.7]	
		\draw[gray!60, thin] (0,0)grid (5,5);
		\draw[gray!60, thin]  (2,0)--(7,5);
		\draw[blue!60, line width = 1.6pt] (0,0)|-(2,1)|-(4,2)|-(5,3)--(5,5);	
		\filldraw (4,2) circle (3pt); 
		\draw (5.6,1) node {$j$} (5.6,4); 
		\draw [decorate,decoration={brace, mirror,amplitude=8pt},xshift=0pt,yshift=0pt] (5,0)--(5,2); 
		\draw[ultra thick, black] (3,2)--(4,2);
		\end{tikzpicture}
	\end{minipage}%
	\begin{minipage}{.4 \textwidth}
		\centering
		\begin{tikzpicture}[scale=.7]
		\draw[gray!60, thin]  (0,0)grid (5,5);
		\draw[gray!60, thin]  (0,0)--(5,5);
		\draw[blue!60, line width = 1.6pt] (0,0)|-(1,1)|-(1,3);
		\filldraw (0,0) circle (3pt);
		\draw[ultra thick, black](1,3)--(2,3);
		\draw[blue!60, line width = 1.6pt](2,3)|-(4,4)|-(5,5);
		\filldraw[fill=blue!60] (1,3) circle (3pt);
		\draw [decorate,decoration={brace, mirror,amplitude=8pt},xshift=0pt,yshift=0pt] (5,3)--(5,5); 
		\draw (5.6,4) node{$j$};
		\end{tikzpicture}
	\end{minipage}
	\caption{A path in $\SQN(5)$ (left) and its image by $\gamma'_N$ in $\DD(5)^{\circ 1, \ast 0}$ (right).} \label{fig: gammaN}
\end{figure}
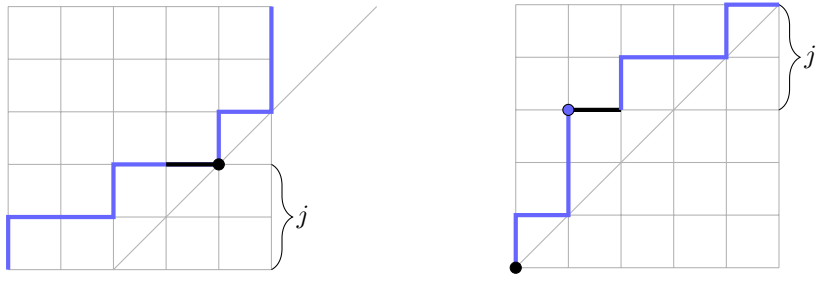

The following theorem with its proof provides the analogue of several results about the $q,t$-square in \cite{Loehr-Warrington-square-2007} and \cite{Can-Loehr-2006}. For this reason we call it a \emph{new $q,t$-square} theorem.
\begin{theorem} We have
	\begin{align}
		\langle \Delta_{e_{n-1}}e_n,e_n \rangle& =\sum_{P\in \mathsf{SQ^E}(n)}q^{\area(P)}t^{\bounce(P)} = \sum_{P\in \mathsf{SQ^E}(n)}q^{\dinv(P)} t^{\area(P)} \\ & =\sum_{P\in \mathsf{SQ^N}(n)}q^{\area(P)}t^{\bounce(P)} =\sum_{P\in \mathsf{SQ^N}(n)}q^{\dinv(P)} t^{\area(P)}. 
	\end{align}
\end{theorem}
\begin{proof}
	We observe that Theorem~\ref{thm:decoqtSchroeder} gives
	\begin{align}
		\langle \Delta_{e_{n-1}}e_n,e_n \rangle & =\langle \Delta_{e_{n-1}}'e_n,e_n \rangle+\langle \Delta_{e_{n-2}}'e_n,e_n \rangle\\ 
		& = \sum_{D\in \DD(n)^{\circ 0, \ast 0}}q^{\dinv(D)}t^{\area(D)}+\sum_{D\in \DD(n)^{\circ 0, \ast 1}}q^{\dinv(D)}t^{\area(D)}\\
		& = \DDd(n)^{\circ 0, \ast 0}   + \DDd(n)^{\circ 0, \ast 1}.
	\end{align}
	
	Using $\gamma'_E$ and $\gamma'_N$ of Proposition~\ref{prop:qtsquare_bijections}, Theorem~\ref{thm:decoqtSchroeder}, the $\zeta$ map of Theorem~\ref{thm:zeta_map} and the $\psi$ map of Theorem~\ref{thm:psi_map}, we prove the equalities in the diagram below.

	\begin{center}
		\begin{tikzpicture}
		\node[scale=0.72] (0,0) { \( 
			\begin{tikzcd}
			&&\sum_{P\in \SQE_n}q^{\dinv(P)}t^{\area(P)} \dar["\gamma'_E"] &\\ 
			
			&&\DDd_{q,t}(n)^{\circ 0, \ast 0}+\DDd_{q,t}(n)^{\circ 0, \ast 1} \dar["\zeta"]  &\\ 
			
			&&\DDb_{q,t}(n)^{\circ 0, \ast 0}+\DDb_{q,t}(n)^{\circ 1, \ast 0}\drar[equal] &\\  
			
			&\sum\limits_{D\in \DD(n)^{\circ 1, \ast 0} } q^{\dinv(D)}t^{\area(D)} \rar[equal] &\sum\limits_{D\in \DD(n)^{\circ 1, \ast 0} } q^{\area(D)}t^{\pbounce(D)} \rar[equal] &\sum\limits_{D\in \DD(n)^{\circ 1, \ast 0} } q^{\area(D)}t^{\bounce(D)}\\
			
			&\sum\limits_{P\in \SQN_n}q^{\dinv(P)}t^{\area(P)} \uar["\gamma'_N"]&
			\DDp_{q,t}(n)^{\circ 0, \ast 0}+\DDp_{q,t}(n)^{\circ 1, \ast 0} \uar[equal]\dar["\psi"]
			& \sum\limits_{P\in \SQN_n}q^{\area(P)}t^{\bounce(P)} \uar["\gamma'_N"]  \\
			
			&& \DDp_{q,t}(n)^{\circ 0, \ast 0}+\DDp_{q,t}(n)^{\circ 0, \ast 1} & \\
			
			&&\DDb_{q,t}(n)^{\circ 0, \ast 0}+\DDb_{q,t}(n)^{\circ 0, \ast 1}\uar[equal]& \\
			&&\sum_{P\in \SQE_n}q^{\area(P)}t^{\bounce(P)} \uar["\gamma'_E"]&
			\end{tikzcd}
			\) };
		\end{tikzpicture}
	\end{center}

	All this proves the result.
\end{proof}

\section{Observations when $t=1/q$} \label{sec:observ_one_over_q}

The most remarkable connection between the old $q,t$-square and the new one is the fact that they coincide at $t=1/q$. In fact, comparing \cite{Haglund-Book-2008}*{Corollary~4.8.1} (with $d=1$) with \cite[Theorem~10]{Loehr-Warrington-square-2007}, we deduce immediately the following proposition.
\begin{proposition} \label{prop:coinc}
	We have
	\begin{equation} \label{eq:t_1_q_id}
		\langle \Delta_{e_{n-1}}e_n,e_n \rangle\big{|}_{t=1/q} = \langle \Delta_{e_{n}}\omega (p_n),e_n \rangle\big{|}_{t=1/q}=q^{-\binom{n}{2}} \frac{1}{1+q^n}\begin{bmatrix}
			2n\\
			n\\
		\end{bmatrix}_q.
	\end{equation}
\end{proposition}

In fact, using \eqref{eq:lem_e_h_Delta}, we can rewrite \eqref{eq:t_1_q_id} as
\begin{equation}
	\langle \Delta_{e_{n}}e_n,e_{n-1}e_1 \rangle\big{|}_{t=1/q} = \langle \Delta_{e_{n}}\omega (p_n),e_n \rangle\big{|}_{t=1/q}.
\end{equation}
We now show that this is not an isolated phenomenon.

We proved in Theorem~\ref{thm:omegapn} that for any $f\in \Lambda^{(k)}$ we have
\begin{equation}
	{\Delta_f \omega(p_n)}_{\big{|}_{t=1/q}}=\frac{[n]_qf[[n]_q]}{[k]_qq^{k(n-1)}}e_n[X[k]_q].
\end{equation}

The following theorem, due to Haglund, Remmel and Wilson, is proved in a similar way.
\begin{theorem}[\cite{Haglund-Remmel-Wilson-2015} Theorem~5.1]
	For any $f\in \Lambda^{(k)}$ we have
	\begin{equation} \label{eq:HRW_lemma}
		{\Delta_f e_n}_{\big{|}_{t=1/q}}=\frac{f[[n]_q]}{[k+1]_qq^{k(n-1)}}e_n[X[k+1]_q].
	\end{equation}
\end{theorem}
We have an easy corollary, which is at the origin of our investigations. This corollary shows that Proposition~\ref{prop:coinc} is not an isolated coincidence.
\begin{corollary} \label{cor:omegapn}
	For any $f\in \Lambda^{(k)}$ we have
	\begin{equation}
		{\langle\Delta_f e_n,e_{n-1}e_1\rangle}_{\big{|}_{t=1/q}}={\langle\Delta_f \omega(p_n),e_{n}\rangle}_{\big{|}_{t=1/q}}.
	\end{equation}
\end{corollary}
\begin{proof}
	Using \eqref{eq:h_q_binomial}, we have
	\begin{equation}
		h_n[[k]_q]=\begin{bmatrix}
			n+k-1\\
			k-1
		\end{bmatrix}_q.
	\end{equation}
	Using the Cauchy identity \eqref{eq:Cauchy_identities}, we get
	\begin{equation}
		e[X[k]_q]=\sum_{\mu\vdash n}s_{\mu'}[X]s_{\mu}[[k]_q].
	\end{equation}
	Therefore
	\begin{equation}
		\langle e[X[k]_q],e_n\rangle = h_n[[k]_q]=\begin{bmatrix}
			n+k-1\\
			k-1
		\end{bmatrix}_q,
	\end{equation}
	and similarly
	\begin{equation}
		\langle e[X[k+1]_q],e_{n-1}e_1\rangle  = (h_{n-1}h_1)[[k+1]_q]=\begin{bmatrix}
			n+k-1\\
			k
		\end{bmatrix}_q [k+1]_q.
	\end{equation}
	So
	\begin{align*}
		{\langle\Delta_f \omega(p_n),e_{n}\rangle}_{\big{|}_{t=1/q}} & = \frac{[n]_qf[[n]_q]}{[k]_qq^{k(n-1)}}\langle e_n[X[k]_q],e_{n} \rangle\\
		& = \frac{f[[n]_q]}{q^{k(n-1)}}\frac{[n]_q}{[k]_q}\begin{bmatrix}
			n+k-1\\
			k-1
		\end{bmatrix}_q\\
		& = \frac{f[[n]_q]}{q^{k(n-1)}}\begin{bmatrix}
			n+k-1\\
			k
		\end{bmatrix}_q\\
		& = \frac{f[[n]_q]}{[k+1]_qq^{k(n-1)}}\begin{bmatrix}
			n+k-1\\
			k
		\end{bmatrix}_q[k+1]_q\\
		& = \frac{f[[n]_q]}{[k+1]_qq^{k(n-1)}}\langle e[X[k+1]_q],e_{n-1}e_1\rangle\\
		& = {\langle \Delta_f e_n,e_{n-1}e_1\rangle}_{\big{|}_{t=1/q}}.
	\end{align*}
	
\end{proof}

The following result shows a surprisingly tight relation between the square conjecture in \cite{Loehr-Warrington-square-2007} and our Conjecture~\ref{conj:new_square}.
\begin{theorem} \label{thm:square_cnjs_at_1_q}
	We have
	\begin{equation}
		{\Delta_{e_n} \omega(p_n)}_{\big{|}_{t=1/q}}={\Delta_{e_{n-1}}e_n}_{\big{|}_{t=1/q}}.
	\end{equation}
\end{theorem}
\begin{proof}
	Using \eqref{eq:e_q_binomial}, we have
	\begin{equation} \label{eq:aux_e_qbin}
		e_n[[n]_{q}]=q^{\binom{n}{2}}\quad \text{ and }\quad
		e_{n-1}[[n]_{q}]=q^{\binom{n-1}{2}}[n]_q.
	\end{equation}
	
	Using Theorem~\ref{thm:omegapn} we have
	\begin{align*}
		{\Delta_{e_n} \omega(p_n)}_{\big{|}_{t=1/q}} & =
		\frac{[n]_qe_{n}[[n]_q]}{[n]_qq^{n(n-1)}}e_n[X[n]_q]\\
		\text{(using \eqref{eq:aux_e_qbin})}& = \frac{[n]_qq^{\binom{n}{2}}}{[n]_qq^{n(n-1)}}e_n[X[n]_q]\\
		& = \frac{[n]_qq^{\binom{n-1}{2}+(n-1)}}{[n]_qq^{n(n-1)}}e_n[X[n]_q]
	\end{align*}
	\begin{align*}
		& = \frac{[n]_qq^{\binom{n-1}{2}}}{[n]_qq^{(n-1)^2}}e_n[X[n]_q]\\
		\text{(using \eqref{eq:aux_e_qbin})}& = \frac{e_{n-1} [[n]_q]}{[n]_qq^{(n-1)^2}}e_n[X[n]_q]\\
		\text{(using \eqref{eq:HRW_lemma})} & = {\Delta_{e_{n-1}} e_n}_{\big{|}_{t=1/q}}.
	\end{align*}
\end{proof}

\appendix

\chapter{Proof of the elementary lemmas}

In this chapter we prove some elementary lemma used in the text.

\subsection*{Proof of Lemma~\ref{lem:elementary4}}

Using \eqref{eq:qbin_recursion}, we have
\begin{equation*}
	q^{s}\begin{bmatrix}
		h-1\\
		s\\
	\end{bmatrix}_q \begin{bmatrix}
		h+k-s-1\\
		h\\
	\end{bmatrix}_q + \begin{bmatrix}
		h-1\\
		s-1\\
	\end{bmatrix}_q \begin{bmatrix}
		h+k-s\\
		h\\
	\end{bmatrix}_q=\qquad \qquad \qquad 
\end{equation*}
\begin{align*}
	& = q^{s}\begin{bmatrix}
		h-1\\
		s\\
	\end{bmatrix}_q \begin{bmatrix}
		h+k-s-1\\
		h\\
	\end{bmatrix}_q \\
	& \quad +  \begin{bmatrix}
		h-1\\
		s-1\\
	\end{bmatrix}_q\left( \begin{bmatrix}
		h+k-s-1\\
		k-s-1\\
	\end{bmatrix}_q+q^{k-s} \begin{bmatrix}
		h+k-s-1\\
		k-s\\
	\end{bmatrix}_q\right)\\
	\qquad \quad & = \begin{bmatrix}
		h\\
		s\\
	\end{bmatrix}_q \begin{bmatrix}
		h+k-s-1\\
		k-s-1\\
	\end{bmatrix}_q+ q^{k-s}\begin{bmatrix}
		h-1\\
		s-1\\
	\end{bmatrix}_q \begin{bmatrix}
		h+k-s-1\\
		k-s\\
	\end{bmatrix}_q \\
	& = \frac{[h]_q!}{[s]_q![h-s]_q!}\frac{[h+k-s-1]_q!}{[h]_q![k-s-1]_q!} +q^{k-s}\frac{[h-1]_q!}{[s-1]_q![h-s]_q!}\frac{[h+k-s-1]_q!}{[h-1]_q![k-s]_q!}\\
	& = \frac{[k]_q!}{[s]_q![k-s]_q!}\frac{[h+k-s-1]_q!}{[k-1]_q![h-s]_q!} \frac{1}{[k]_q} \left( [k-s]_q+q^{k-s}[s]_q\right)\\
	& =\begin{bmatrix}
		k\\
		s\\
	\end{bmatrix}_q \begin{bmatrix}
		h+k-s-1\\
		h-s\\
	\end{bmatrix}_q.
\end{align*}
This completes the proof of the lemma.

\subsection*{Proof of Lemma~\ref{lem:elementary1}}

Let us call $f(s,i,a)$ the left hand side of \eqref{eq:first_qlemma}. 
First we rewrite
\begin{align*}
	f(s,i,a) & = \sum_{r=1}^{i}\begin{bmatrix}
		i-1\\
		r-1
	\end{bmatrix}_q \begin{bmatrix}
		r+s+a-1\\
		s-1
	\end{bmatrix}_q q^{\binom{r}{2}+r-ir}(-1)^{i-r}\\
	& = \sum_{j=0}^{i-1} \begin{bmatrix}
		i-1\\
		j
	\end{bmatrix}_q  \begin{bmatrix}
		s+a+j\\
		s-1
	\end{bmatrix}_q q^{\binom{j+2}{2}-ij-i}(-1)^{i-j-1}.
\end{align*}
Then, using \eqref{eq:qbin_recursion}, for $i\geq 2$ and $a\geq -i+1$ we have
\begin{align*}
	f(s,i,a) & = \sum_{j=0}^{i-1} \begin{bmatrix}
		i-1\\
		j
	\end{bmatrix}_q  \begin{bmatrix}
		s+a+j\\
		s-1
	\end{bmatrix}_q q^{\binom{j+2}{2}-ij-i}(-1)^{i-j-1}\\
	& = \sum_{j=0}^{i-1}\left(q^j\begin{bmatrix}
		i-2\\
		j
	\end{bmatrix}_q+\begin{bmatrix}
		i-2\\
		j-1
	\end{bmatrix}_q\right) \begin{bmatrix}
		s+a+j\\
		s-1
	\end{bmatrix}_q q^{\binom{j+2}{2}-ij-i}(-1)^{i-j-1}
\end{align*}
\begin{align*}
	& = -q^{-1}\sum_{j=0}^{i-1}\begin{bmatrix}
		i-2\\
		j
	\end{bmatrix}_q \begin{bmatrix}
		s+a+j\\
		s-1
	\end{bmatrix}_q q^{\binom{j+2}{2}-(i-1)j-(i-1)}(-1)^{(i-1)-j-1}	\\
	& \quad + \sum_{j=0}^{i-1} \begin{bmatrix}
		i-2\\
		j-1
	\end{bmatrix}_q  \begin{bmatrix}
		s+a+j\\
		s-1
	\end{bmatrix}_q q^{\binom{j+2}{2}-ij-i}(-1)^{i-j-1}\\
	& =-q^{-1}f(s,i-1,a)\\
	& \quad + \sum_{h=0}^{i-2} \begin{bmatrix}
		i-2\\
		h
	\end{bmatrix}_q  \begin{bmatrix}
		s+a+1+h\\
		s-1
	\end{bmatrix}_q q^{\binom{h+2}{2}+h+2-ih-i-i}(-1)^{i-h-2}\\
	&=-q^{-1}f(s,i-1,a)+q^{1-i}f(s,i-1,a+1).
\end{align*}
So by induction,
\begin{align*}
	f(s,i,a)&=-q^{-1}f(s,i-1,a)+q^{1-i}f(s,i-1,a+1)\\
	& =-q^{-1}q^{\binom{i-1}{2}+(i-2)a} \begin{bmatrix}
		s+a\\
		i-1+a
	\end{bmatrix}_q+q^{1-i}q^{\binom{i-1}{2}+(i-2)(a+1)} \begin{bmatrix}
		s+a+1\\
		i+a
	\end{bmatrix}_q\\
	& =q^{\binom{i-1}{2}+(i-2)a-1}\left(\begin{bmatrix}
		s+a+1\\
		i+a
	\end{bmatrix}_q- \begin{bmatrix}
		s+a\\
		i-1+a
	\end{bmatrix}_q \right)\\
	& =q^{\binom{i-1}{2}+(i-2)a-1}q^{i+a}\begin{bmatrix}
		s+a\\
		i+a
	\end{bmatrix}_q \\
	& =q^{\binom{i}{2}+(i-1)a} \begin{bmatrix}
		s+a\\
		i+a
	\end{bmatrix}_q,
\end{align*}
as we wanted. 

For the base cases: clearly $f(s,i,a)=0$ if $a<-i$. Moreover
\begin{align*}
	f(s,i,-i) &  = \sum_{r=1}^{i} \qbinom{i-1}{r-1}_q \qbinom{r+s-i-1}{s-1}_q  q^{\binom{r}{2}+r-ir}(-1)^{i-r}\\
	&  =  q^{\binom{i}{2}+i-i^2}\\
	&  =  q^{-\binom{i}{2}}\\
	&  =  q^{1-i}q^{-\binom{i-1}{2}}\\
	&  =  q^{1-i}f(s,i-1,-i+1),
\end{align*}
so that our recursive step actually works for $a\geq -i$.

Finally, for $i=1$ we have
\begin{align*}
	f(s,1,a) &  = \sum_{r=1}^{1} \qbinom{1-1}{r-1}_q \qbinom{r+s+a-1}{s-1}_q  q^{\binom{r}{2}+r-r}(-1)^{1-r}\\
	&  = \qbinom{s+a}{s-1}_q\\
	&  = q^{\binom{1}{2}+(1-1)a} \qbinom{s+a}{1+a}_q,
\end{align*}
as we wanted. This completes the proof of the lemma.

\subsection*{Proof of Lemma~\ref{lem:elementary2}}

Let us call $f(\ell,b,d,k)$ the right hand side of \eqref{eq:qlemma3}. 

We proceed by induction on $b$, $d$ and $k$. 

For $d\geq b\geq 1$ and $k\geq 2$, using \eqref{eq:qbin_recursion}, we have

\begin{align*}
	& \hspace{-0.5cm} f(\ell,b,d,k):=\\
	& = \sum_{i=0}^{k} 
	\sum_{s =0}^{i+\ell-d+b} 
	h_{k-i}\left[\frac{1}{1-q}\right]  q^{\binom{i}{2}} \begin{bmatrix}
		s\\
		i
	\end{bmatrix}_q e_{s-b}\left[-\frac{1}{1-q}\right] 
	e_{i+\ell-s-d+b}\left[\frac{1}{1-q}\right]\\
	& = \sum_{i=0}^{k} 
	\sum_{s =0}^{i+\ell-d+b} 
	h_{k-i}\left[\frac{1}{1-q}\right]  q^{\binom{i}{2}} \qbinom{s-1}{i}_q   e_{s-b}\left[-\frac{1}{1-q}\right] 
	e_{i+\ell-s-d+b}\left[\frac{1}{1-q}\right] \\
	& + \sum_{i=0}^{k} 
	\sum_{s =0}^{i+\ell-d+b} 
	h_{k-i}\left[\frac{1}{1-q}\right]  q^{\binom{i}{2}}q^{s-i} \qbinom{s-1}{i-1}_q   e_{s-b}\left[-\frac{1}{1-q}\right] 
	e_{i+\ell-s-d+b}\left[\frac{1}{1-q}\right] \\
	& = \sum_{i=0}^{k} 
	\sum_{s =0}^{i-1+\ell-d+b} 
	h_{k-i}\left[\frac{1}{1-q}\right]  q^{\binom{i}{2}} \qbinom{s}{i}_q   e_{s+1-b}\left[-\frac{1}{1-q}\right] 
	e_{i+\ell-s-1-d+b}\left[\frac{1}{1-q}\right] \\
	& + \sum_{i=0}^{k-1} 
	\sum_{s =0}^{i+\ell-d+b} 
	h_{k-i-1}\left[\frac{1}{1-q}\right]  q^{\binom{i+1}{2}}q^{s-i}  \qbinom{s}{i}_q   e_{s+1-b}\left[-\frac{1}{1-q}\right] 
	e_{i+\ell-s-d+b}\left[\frac{1}{1-q}\right] \\
	& = f(\ell,b-1,d,k) \\
	& + q^{b-1}\sum_{i=0}^{k-1} 
	\sum_{s =0}^{i+\ell-d+b} 
	h_{k-i-1}\left[\frac{1}{1-q}\right]  q^{\binom{i}{2}}q^{s+1-b}  \qbinom{s}{i}_q   e_{s+1-b}\left[-\frac{1}{1-q}\right] 
	e_{i+\ell-s-d+b}\left[\frac{1}{1-q}\right] \\
	& = f(\ell,b-1,d,k) \\
	& + q^{b-1}\sum_{i=0}^{k-1} 
	\sum_{s =0}^{i+\ell-d+b} 
	h_{k-i-1}\left[\frac{1}{1-q}\right]  q^{\binom{i}{2}}  \qbinom{s}{i}_q   e_{s+1-b}\left[-\frac{1}{1-q}\right] 
	e_{i+\ell-s-d+b}\left[\frac{1}{1-q}\right]\\
	& + q^{b-1}\sum_{i=0}^{k-1} 
	\sum_{s =0}^{i+\ell-d+b} 
	h_{k-i-1}\left[\frac{1}{1-q}\right]  q^{\binom{i}{2}}(q^{s+1-b}-1)  \qbinom{s}{i}_q   e_{s+1-b}\left[-\frac{1}{1-q}\right] 
	e_{i+\ell-s-d+b}\left[\frac{1}{1-q}\right]\\
	& = f(\ell,b-1,d,k)+q^{b-1}f(\ell,b-1,d-1,k-1) \\
	& + q^{b-1}\sum_{i=0}^{k-1} 
	\sum_{s =0}^{i+\ell-d+b} 
	h_{k-i-1}\left[\frac{1}{1-q}\right]  q^{\binom{i}{2}} \qbinom{s}{i}_q  (q^{s+1-b}-1) (-1)^{s+1-b}h_{s+1-b}\left[\frac{1}{1-q}\right] 
	e_{i+\ell-s-d+b}\left[\frac{1}{1-q}\right]\\
	& = f(\ell,b-1,d,k)+q^{b-1}f(\ell,b-1,d-1,k-1) \\
	& + q^{b-1}\sum_{i=0}^{k-1} 
	\sum_{s =0}^{i+\ell-d+b} 
	h_{k-i-1}\left[\frac{1}{1-q}\right]  q^{\binom{i}{2}} \qbinom{s}{i}_q   e_{s-b}\left[-\frac{1}{1-q}\right] 
	e_{i+\ell-s-d+b}\left[\frac{1}{1-q}\right]\\
	& = f(\ell,b-1,d,k)+q^{b-1}f(\ell,b-1,d-1,k-1)+q^{b-1}f(\ell,b,d,k-1) ,
\end{align*}
where we used also \eqref{eq:h_q_prspec} and \eqref{eq:minusepsilon}.

By induction, using again \eqref{eq:qbin_recursion}, we get (notice that $d\geq b\geq 1$ and $k\geq 2$)
\begin{align*}
	f(\ell,b,d,k)	& = f(\ell,b-1,d,k)+q^{b-1}f(\ell,b-1,d-1,k-1)+q^{b-1}f(\ell,b,d,k-1)\\
	& = q^{\binom{d-\ell}{2}}\qbinom{b-1}{d-\ell}_q \qbinom{k-d+\ell+b-2}{k-d+\ell}_q\\
	& + q^{b-1}q^{\binom{d-1-\ell}{2}}\qbinom{b-1}{d-1-\ell}_q \qbinom{k-d+\ell+b-2}{k-d+\ell}_q\\
	& + q^{b-1}q^{\binom{d-\ell}{2}}\qbinom{b}{d-\ell}_q \qbinom{k-1-d+\ell+b-1}{k-1-d+\ell}_q\\
	& = q^{\binom{d-\ell}{2}}\left(\qbinom{b-1}{d-\ell}_q + q^{b-d+\ell} \qbinom{b-1}{d-1-\ell}_q \right)\qbinom{k-d+\ell+b-2}{k-d+\ell}_q\\
	& + q^{b-1}q^{\binom{d-\ell}{2}}\qbinom{b}{d-\ell}_q \qbinom{k-1-d+\ell+b-1}{k-1-d+\ell}_q\\
	& = q^{\binom{d-\ell}{2}}\qbinom{b}{d-\ell}_q \qbinom{k-d+\ell+b-2}{k-d+\ell}_q\\
	& + q^{b-1}q^{\binom{d-\ell}{2}}\qbinom{b}{d-\ell}_q \qbinom{k-d+\ell+b-2}{k-d+\ell-1}_q\\
	& = q^{\binom{d-\ell}{2}}\qbinom{b}{d-\ell}_q \qbinom{k-d+\ell+b-1}{k-d+\ell}_q,
\end{align*}
as we wanted.

For the base cases, using as usual \eqref{eq:e_h_sum_alphabets}, \eqref{eq:h_q_prspec}, and \eqref{eq:minusepsilon}, for $d\geq b\geq 0$ we have
\begin{align*}
	& \hspace{-0.5cm} f(\ell,b,d,1):=\\
	& = \sum_{i=0}^{1} 
	\sum_{s =0}^{i+\ell-d+b} 
	h_{1-i}\left[\frac{1}{1-q}\right]  q^{\binom{i}{2}} \qbinom{s}{i}_q  e_{s-b}\left[-\frac{1}{1-q}\right] 
	e_{i+\ell-s-d+b}\left[\frac{1}{1-q}\right]\\
	& = \sum_{s =0}^{\ell-d+b} h_{1}\left[\frac{1}{1-q}\right]   e_{s-b}\left[-\frac{1}{1-q}\right] 
	e_{\ell-s-d+b}\left[\frac{1}{1-q}\right] \\
	& +\sum_{s =0}^{1+\ell-d+b}   [s]_qe_{s-b}\left[-\frac{1}{1-q}\right] 
	e_{1+\ell-s-d+b}\left[\frac{1}{1-q}\right]\\
	& = \frac{1}{1-q} 
	e_{\ell-d}\left[\frac{1}{1-q}-\frac{1}{1-q}\right] \\
	& +\frac{1}{1-q} \sum_{s =0}^{1+\ell-d+b}   (1-q^s)e_{s-b}\left[-\frac{1}{1-q}\right] 
	e_{1+\ell-s-d+b}\left[\frac{1}{1-q}\right]\\
	& = \frac{1}{1-q} \delta_{\ell-d,0}  +\frac{1}{1-q} \sum_{s =0}^{1+\ell-d+b}   (1-q^{s-b})e_{s-b}\left[-\frac{1}{1-q}\right] 
	e_{1+\ell-s-d+b}\left[\frac{1}{1-q}\right]\\
	& +\frac{1}{1-q} \sum_{s =0}^{1+\ell-d+b}   q^{s-b}(1-q^b)e_{s-b}\left[-\frac{1}{1-q}\right] 
	e_{1+\ell-s-d+b}\left[\frac{1}{1-q}\right]
\end{align*}
\begin{align*}
	& = \frac{1}{1-q} \delta_{\ell-d,0}  -\frac{1}{1-q} \sum_{s =0}^{1+\ell-d+b}   e_{s-b-1}\left[-\frac{1}{1-q}\right] 
	e_{1+\ell-s-d+b}\left[\frac{1}{1-q}\right]\\
	& +[b]_q \sum_{s =0}^{1+\ell-d+b} e_{s-b}\left[-\frac{q}{1-q}\right] 
	e_{1+\ell-s-d+b}\left[\frac{1}{1-q}\right]\\
	& =  \frac{1}{1-q} \delta_{\ell-d,0}- \frac{1}{1-q} \delta_{\ell-d,0}+ [b]_q e_{1+\ell-d}\left[\frac{1}{1-q}-\frac{q}{1-q}\right]\\
	& = [b]_qe_{1+\ell-d}[1]\\
	& = [b]_q(\delta_{\ell,d}+\delta_{\ell,d-1}) \\
	& = q^{\binom{d-\ell}{2}}\qbinom{b}{d-\ell}_q \qbinom{1-d+\ell+b-1}{1-d+\ell}_q,
\end{align*}
%
while for $d\geq 0$ and $k\geq 1$ we have
\begin{align*}
	& \hspace{-0.5cm} f(\ell,0,d,k):=\\
	& = \sum_{i=0}^{k} 
	\sum_{s =0}^{i+\ell-d} 
	h_{k-i}\left[\frac{1}{1-q}\right]  q^{\binom{i}{2}} \qbinom{s}{i}_q  e_{s}\left[-\frac{1}{1-q}\right] 
	e_{i+\ell-s-d}\left[\frac{1}{1-q}\right]\\
	& = \sum_{i=0}^{k} 
	\sum_{s =0}^{i+\ell-d} 
	h_{k-i}\left[\frac{1}{1-q}\right]  q^{\binom{i}{2}} \qbinom{s}{i}_q  (-1)^sh_{s}\left[\frac{1}{1-q}\right] 
	e_{i+\ell-s-d}\left[\frac{1}{1-q}\right]\\
	& = \sum_{i=0}^{k} 
	\sum_{s =0}^{i+\ell-d} 
	h_{k-i}\left[\frac{1}{1-q}\right]  q^{\binom{i}{2}}  (-1)^sh_{s-i}\left[\frac{1}{1-q}\right] h_{i}\left[\frac{1}{1-q}\right]  
	e_{i+\ell-s-d}\left[\frac{1}{1-q}\right]\\
	& = \sum_{i=0}^{k} 
	h_{k-i}\left[\frac{1}{1-q}\right]  q^{\binom{i}{2}} (-1)^i h_{i}\left[\frac{1}{1-q}\right]  	\sum_{s =0}^{i+\ell-d} (-1)^{s-i}h_{s-i}\left[\frac{1}{1-q}\right] 
	e_{i+\ell-s-d}\left[\frac{1}{1-q}\right] \\
	& = \sum_{i=0}^{k} 
	h_{k-i}\left[\frac{1}{1-q}\right]  q^{\binom{i}{2}} (-1)^i h_{i}\left[\frac{1}{1-q}\right]  	\sum_{s =0}^{i+\ell-d} e_{s-i}\left[-\frac{1}{1-q}\right] 
	e_{i+\ell-s-d}\left[\frac{1}{1-q}\right]\\
	& = \sum_{i=0}^{k} 
	h_{k-i}\left[\frac{1}{1-q}\right]  q^{\binom{i}{2}} (-1)^i h_{i}\left[\frac{1}{1-q}\right]
	e_{\ell-d}\left[\frac{1}{1-q}-\frac{1}{1-q}\right]\\
	& = \delta_{\ell,d}\sum_{i=0}^{k} (-1)^i q^{\binom{i}{2}}  \qbinom{k}{i}_q 
	h_{k}\left[\frac{1}{1-q}\right]  \\
	& = \delta_{\ell,d} h_{k}\left[\frac{1}{1-q}\right] (1;q)_k\\
	& = 0 \\
	& = q^{\binom{d-\ell}{2}}\qbinom{0}{d-\ell}_q \qbinom{k-d+\ell-1}{k-d+\ell}_q,
\end{align*}
where in the third to last equality we used the well-known $q$-binomial theorem \cite[Theorem~3.3]{Andrews-Book_Partitions}.

This completes the proof of the lemma.

\subsection*{Proof of Lemma~\ref{lem:elementary3}}

We have
\begin{align*}
	\begin{bmatrix}
		k\\
		s
	\end{bmatrix}_q \begin{bmatrix}
		k+j-1\\
		j
	\end{bmatrix}_q (1-q^{s+j}) & = (1-q)\begin{bmatrix}
		k\\
		s
	\end{bmatrix}_q \begin{bmatrix}
		k+j-1\\
		j
	\end{bmatrix}_q [s+j]_q \frac{[s+j-1]_q!}{[s+j-1]_q!}\\
	& = (1-q)\frac{[k]_q!}{[k-s]_q![s]_q} \frac{[k+j-1]_q!}{[k-1]_q![j]_q}\frac{[s+j]_q!}{[s+j-1]_q!}\\
	& = (1-q)[k]_q\begin{bmatrix}
		s+j\\
		s
	\end{bmatrix}_q \begin{bmatrix}
		k+j-1\\
		k-s
	\end{bmatrix}_q \\
	& = (1-q^k) \begin{bmatrix}
		s+j\\
		s
	\end{bmatrix}_q \begin{bmatrix}
		k+j-1\\
		k-s
	\end{bmatrix}_q .
\end{align*}
This proves the lemma.

\backmatter


\bibliographystyle{amsalpha}
\bibliography{Biblebib}


\end{document}